\documentclass[11pt]{article}
\usepackage{amssymb,amsmath,epsfig}
\usepackage{graphics,psfrag,graphicx,color}
\usepackage{float,hhline}
\usepackage{multicol}
\usepackage{comment}
\usepackage[dvipsnames]{xcolor}
\usepackage{overpic}
\usepackage{caption}
\usepackage{cite}

\usepackage{hyperref}
\hypersetup{
	colorlinks=true,
	linkcolor=blue,
        citecolor=blue,
	filecolor=magenta,      
	urlcolor=cyan,
}

\numberwithin{equation}{section}
\allowdisplaybreaks

\newcommand{\ds}{\displaystyle}

\newcommand{\bgamma}{{\boldsymbol\gamma}}

\newcommand{\bLambda}{{\boldsymbol\Lambda}}
\newcommand{\bbeta}{{\boldsymbol\eta}}

\newcommand{\bsi}{{\boldsymbol\sigma}}
\newcommand{\bSigma}{{\boldsymbol\Sigma}}
\newcommand{\bphi}{{\boldsymbol\phi}}

\newcommand{\bvarphi}{{\boldsymbol\varphi}}

\newcommand{\bpsi}{{\boldsymbol\psi}}

\newcommand{\btau}{{\boldsymbol\tau}}
\newcommand{\bzeta}{{\boldsymbol\zeta}}

\newcommand{\bchi}{{\boldsymbol\chi}}
\newcommand{\bxi}{{\boldsymbol\xi}}

\newcommand{\ubsi}{\underline{\bsi}}
\newcommand{\ubtau}{\underline{\btau}}
\newcommand{\ubu}{\underline{\bu}}
\newcommand{\ubv}{\underline{\bv}}
\newcommand{\bv}{{\mathbf{v}}}
\newcommand{\bw}{{\mathbf{w}}}
\newcommand{\f}{\mathbf{f}}

\newcommand{\bi}{\mathbf{i}}

\newcommand{\bu}{\mathbf{u}}

\newcommand{\bz}{{\mathbf{z}}}
\newcommand{\bt}{{\mathbf{t}}}
\newcommand{\bn}{{\mathbf{n}}}
\newcommand{\be}{{\mathbf{e}}}

\newcommand{\0}{{\mathbf{0}}}
\def\bA{\mathbf{A}}

\def\bF{\mathbf{F}}
\def\bG{\mathbf{G}}
\def\bK{\mathbf{K}}
\def\bI{\mathbf{I}}

\def\bX{\mathbf{X}}

\def\bV{\mathbf{V}}
\def\bW{\mathbf{W}}
\def\bM{\mathbf{M}}
\def\bT{\mathbf{T}}

\def\bP{\mathbf{P}}
\def\bQ{\mathbf{Q}}
\def\bS{\mathbf{S}}

\def\bx{\mathbf{x}}

\def\bz{\mathbf{z}}

\newcommand{\bL}{\mathbf{L}}
\newcommand\bH{\mathbf{H}}

\newcommand\bbM{\mathbb{M}}

\newcommand\bbP{\mathbb{P}}
\newcommand\bbQ{\mathbb{Q}}

\newcommand\bbH{\mathbb{H}}

\newcommand\bbX{\mathbb{X}}

\newcommand\bbL{\mathbb{L}}

\newcommand{\cA}{\mathcal{A}}
\newcommand{\cB}{\mathcal{B}}
\newcommand{\cC}{\mathcal{C}}

\newcommand{\cE}{\mathcal{E}}
\newcommand{\cF}{\mathcal{F}}

\newcommand{\cJ}{\mathcal{J}}
\newcommand{\cN}{\mathcal{N}}
\newcommand{\cM}{\mathcal{M}}
\newcommand{\cT}{\mathcal{T}}
\newcommand{\cK}{\mathcal{K}}
\newcommand{\cL}{\mathcal{L}}
\newcommand{\cD}{\mathcal{D}}
\newcommand{\cO}{\mathcal{O}}
\newcommand{\cP}{\mathcal{P}}
\newcommand{\cQ}{\mathcal{Q}}
\newcommand{\cR}{\mathcal{R}}

\def\R{\mathrm{R}}

\def\H{\mathrm{H}}
\def\L{\mathrm{L}}
\def\M{\mathrm{M}}
\def\U{\mathrm{U}}
\def\V{\mathrm{V}}
\def\W{\mathrm{W}}
\def\X{\mathrm{X}}
\def\rd{\mathrm{d}}

\def\rD{\mathrm{D}}
\def\rN{\mathrm{N}}
\def\rP{\mathrm{P}}

\def\rp{\mathrm{p}}
\def\rq{\mathrm{q}}
\def\rt{\mathrm{t}}
\def\ttd{\mathtt{d}}

\def\esssup{\mathrm{ess\,sup}}
\def\bBDM{\mathbf{BDM}}
\def\bbBDM{\mathbb{BDM}}
\def\BJS{\mathtt{BJS}}

\def\bdiv{\mathbf{div}}

\def\tr{\mathrm{tr}}

\def\div{\mathrm{div}}
\def\dist{\mathrm{dist}\,}
\def\pil{\left<}
\def\pir{\right>}

\def\sk{\mathrm{sk}}

\def\qin{{\quad\hbox{in}\quad}}
\def\qon{{\quad\hbox{on}\quad}}
\def\qan{{\quad\hbox{and}\quad}}

\def\wt{\widetilde}
\def\wh{\widehat}
\def\dt{\partial_t}

\newtheorem{thm}{Theorem}[section]
\newtheorem{rem}{Remark}[section]
\newtheorem{lem}[thm]{Lemma}
\newtheorem{cor}[thm]{Corollary}

\newtheorem{assumption}{Assumption}
\newenvironment{proof}{\noindent{\it Proof.}}{\hfill$\square$}

\numberwithin{equation}{section}

\allowdisplaybreaks

\title{A Banach space mixed formulation for the fully dynamic \\ Navier--Stokes--Biot coupled problem}

\author{{\sc Sergio Caucao}\thanks{Departamento de Matem\'atica y F\'isica Aplicadas and Grupo de Investigaci\'on en An\'alisis Num\'erico y C\'alculo Cient\'ifico (GIANuC$^2$), 
Universidad Cat\'olica de la Sant\'isima Concepci\'on, Casilla 297, Concepci\'on, Chile, 
email: {\tt scaucao@ucsc.cl}. Supported in part by ANID-Chile through the projects {\sc Centro de Mode\-lamiento Matem\'atico} (FB210005) and Fondecyt 1250937.}
\quad
{\sc Aashi Dalal}\thanks{Department of Mathematics, University of Pittsburgh, Pittsburgh, PA 15260, USA, email: {\tt aad100@pitt.edu} and Department of Mathematics and Statistics, University of Ottawa, Ottawa, ON K1N 6N5, Canada, email: {\tt adalal@uottawa.ca}. Supported in part by the Fields Institute for Research in Mathematical Sciences.}
\quad
{\sc Ivan Yotov}\thanks{Department of Mathematics, University of Pittsburgh, Pittsburgh, PA 15260, USA, email: {\tt yotov@math.pitt.edu}. The second and third authors are supported in part by NSF grants DMS-2111129 and DMS-2410686.}}

\date{\today}

\begin{document}
	
\maketitle

\begin{abstract}
\noindent 
We introduce and analyze a mixed formulation for the coupled problem
arising in the interaction between a free fluid and a poroelastic medium. 
The flows in the free fluid and poroelastic regions are governed by the Navier--Stokes and Biot equations, respectively, and the transmission conditions are given by mass conservation, balance of stresses, and the Beavers--Joseph--Saffman law. 
We apply dual-mixed formulations for both the Navier--Stokes and Darcy equations and a displacement-based formulation for the elasticity equation. We develop a weak formulation in Banach space framework where the symmetry of the Navier--Stokes pseudostress tensor is imposed in a weak sense using a vorticity Lagrange multiplier. In turn, the transmission conditions are imposed weakly by introducing the traces of the fluid velocity and the poroelastic pressure on the interface as Lagrange multipliers. Existence and uniqueness of a weak solution are established using classical results for nonlinear monotone operators, complemented by a regularization technique and the Banach fixed-point theory, under a small data assumption. We then present a semidiscrete continuous-in-time numerical approximation based on conforming finite element spaces, allowing for nonmatching grids along the interface. We establish well posedness of the numerical scheme and present error analysis with corresponding rates of convergence within the Banach space framework. Numerical experiments are presented to verify the theoretical rates of convergence and to illustrate the method's performance in an application involving flow through a filter.
\end{abstract}
	
	
\maketitle
	
	
\section{Introduction}

The interaction between a free fluid and a deformable porous medium, which is referred to as fluid--poroelastic structure interaction (FPSI), exhibits characteristics of both coupled free fluid--porous media flows \cite{LaySchYot,DMQ,ErvJenSun,GalSar,ry2005,gos2011,DQ-NSD,Cesm-NS-Darcy-time-dep,GirRiv-NSD,badea2010-NSD} and fluid--structure interaction \cite{galdi2010fundamental,bazilevs2013computational,bungartz2006fluid,richter2017fluid}.
The modeling of FPSI has a wide range of applications. For instance, in hydrology, it aids in tracking how pollutants discharged into ground and surface water enter the water supply and in designing strategies for cleaning water contaminants. In petroleum engineering, it is essential for predicting and optimizing the processes involved in efficient oil and gas extraction from hydraulically fractured reservoirs. In biomedicine, it plays a crucial role in modeling blood flow, simulating low-density lipoprotein (LDL) transport, and optimizing drug delivery. The free fluid region is typically modeled by the Stokes (or Navier--Stokes) equations, while the flow through the deformable porous medium is governed by the Biot system of poroelasticity. The latter combines the equations for the deformation of the elastic structure with the Darcy equations, which describe the mass conservation of the fluid flowing through the porous medium. The two regions are coupled through dynamic and kinematic interface conditions, which include balance of forces, continuity of the normal velocity, and a no-slip or slip-with-friction condition for the tangential velocity. 

The modeling of FPSI modeling has gained popularity in recent years, with most works focusing on the Stokes--Biot model \cite{FPSI-LM,show2005,byz2015,byzz2015,Buk-Yot-Zun-fracture,aeny2019,fpsi-transport,Stokes-Biot-eye,fpsi-mixed-elast,fpsi-msfmfe,Bociu-etal-2021,Cesm-Chid,HDG-SB,Boon-precond-SB,hyper-SB,SB-nonlin-geom}. Significantly less attention has been given to the Navier--Stokes--Biot model, which is better suited for fast free fluid flows, such as blood flows and flows through industrial filters. The problem is more challenging due to the presence of the nonlinear convective term in the fluid momentum equation. The first mathematical analysis for the Navier--Stokes--Biot system was presented in \cite{cesmelioglu2017analysis}, where well posedness is established for the fully dynamic model using a velocity-pressure Navier--Stokes formulation, displacement elasticity formulation, and pressure Darcy formulation. Mathematical and numerical analysis for the Navier--Stokes--Biot system is presented in \cite{augmented} for a fully mixed formulation of the quasistatic model, with pseudostress--velocity for Navier--Stokes, stress--displacement--rotation for elasticity, and velocity--pressure for Darcy. The formulation is augmented with suitable redundant Galerkin-type terms derived from the equilibrium and constitutive equations. A hybridizable discontinuous Galerkin method for the Navier--Stokes--Biot problem is studied in \cite{CLR-NSB-HDG}, based on the time-dependent Navier--Stokes equations in the velocity-pressure formulation, and the quasistatic Biot system in a displacement--total pressure--velocity--pressure formulation.
In \cite{wy2022}, a formulation for the fully dynamic Navier--Stokes--Biot system is developed, which is based on velocity--pressure for Navier--Stokes and Darcy and displacement for elasticity. Analysis for both the weak formulation and a fully discrete finite element method is presented. Computational studies using the Navier--Stokes--Biot model are presented in \cite{badia2009coupling,byz2015,Bukac-JCP,cly2020,NSB-transport,MFE-NSB}.

The goal of the present paper is to develop and analyze a new fully dynamic Banach space mixed formulation of the Navier--Stokes--Biot model and to study a suitable conforming numerical discretization. We employ a classical dual mixed velocity--pressure formulation for the Darcy flow, a displacement-based formulation for the elasticity equation, and a weakly symmetric nonlinear pseudostress--velocity--vorticity formulation for the Navier--Stokes equations. This new formulation offers several advantages, including local conservation of mass for the Darcy flow, local momentum conservation for the Navier--Stokes equations, and accurate approximations with continuous normal components across element edges or faces for the Darcy velocity and the free fluid stress.
More precisely, similarly to \cite{cot2017,cgo2021,augmented}, we introduce a nonlinear pseudostress tensor that combines the fluid stress tensor with the convective term. 
Subsequently, we eliminate the pressure unknown by utilizing the deviatoric tensor. 
To impose the weak symmetry of the Navier--Stokes pseudostress, vorticity is introduced as an additional unknown. The approximation of the fluid pseudostress in the $\bbH(\bdiv)$ space ensures compatible enforcement of momentum conservation. 
The transmission conditions, including mass conservation, momentum conservation, and the Beavers--Joseph--Saffman slip with friction condition, are imposed weakly through the incorporation of two Lagrange multipliers: the traces of the fluid velocity and the Darcy pressure on the interface.
We note that the Banach space formulation leverages the natural function spaces that arise from applying the Cauchy--Schwarz and H\"older inequalities to the terms obtained by testing and integrating by parts the equations of the model. This is in contrast to the augmented formulation developed in \cite{augmented}, where redundant terms are added in order to ensure control of the fluid velocity in the $\bH^1$ norm.

The main contributions of this paper are as follows. In the first part we establish well posedness of the weak formulation. It is a nonlinear time-dependent system, which is challenging to analyze due to the presence of the time derivative of the displacement in certain non-coercive terms and the nonlinear convective Navier--Stokes term.
To address the first difficulty, we consider an alternative mixed elasticity formulation as in \cite{show2005,aeny2019}, with the structural velocity and elastic stress as the primary variables. This approach results in a system characterized by a degenerate evolution operator in time and a nonlinear saddle-point spatial operator. 
By employing techniques from \cite{Showalter} and \cite{cmo2018}, we combine classical monotone operator theory with a suitable regularization technique in Banach spaces to establish the well-posedness of the alternative formulation. To deal with the nonlinear convective term, we show that the fluid velocity is contained in a ball with a sufficiently small radius, under a small data assumption, and employ the Banach fixed-point theory. In addition, the existence proof requires the construction of compatible initial data for all variables. We then recover a solution of the original formulation, prove its uniqueness, and obtain a stability bound. In the second part of the paper we introduce a semidiscrete continuous-in-time formulation based on stable $\bH(\div)$-conforming mixed finite element spaces for the Navier--Stokes and Darcy equations and $\bH^1$-conforming finite element space for the displacement, along with conforming finite element spaces for the Lagrange multipliers. Well-posedness and stability results are established using arguments analogous to those applied in the continuous case. Additionally, we perform error analysis and establish rates of convergence for all variables.
Finally, we present numerical experiments employing a fully discrete finite element method with backward Euler time discretization to verify the theoretical rates of convergence and illustrate the method's behavior in modeling air flow through a filter with physically realistic parameters.

The rest of this work is organized as follows. 
The remainder of this section introduces the standard notation and functional spaces employed throughout the paper. 
In Section~\ref{sec:Navier-Stokes--Biot model problem}, we present the mathematical model along with the corresponding interface, boundary, and initial conditions. 
Section~\ref{sec:variational-formulation} focuses on the continuous weak formulation, which forms the basis of the numerical method, as well as an alternative formulation required for the analysis. 
Stability properties of the associated operators are also established. 
In Section~\ref{sec:well-posedness-model}, we prove the well-posedness of both the alternative and original formulations, employing a suitable fixed-point approach to establish the existence, uniqueness, and stability of the solution. 
Section~\ref{sec:Semidiscrete continuous-in-time approximation} introduces and analyzes the semidiscrete continuous-in-time approximation, including its well-posedness, stability, and error analysis. 
Section \ref{sec:numerical} presents the fully discrete scheme, which is then used to carry out the numerical experiments. We end with some conclusions in Section~\ref{sec:concl}.


\subsection*{Notation}

Let $\cO\subseteq \R^n$, $n\in \{2,3\}$, denote a domain with Lipschitz boundary $\Gamma$.
For $s\geq 0$ and $p\in [1,+\infty]$, we denote by $\L^p(\cO)$ and $\W^{s,p}(\cO)$ the usual Lebesgue and Sobolev spaces endowed with the norms $\|\cdot\|_{\L^p(\cO)}$ and $\|\cdot\|_{\W^{s,p}(\cO)}$, respectively.
Note that $\W^{0,p}(\cO) = \L^p(\cO)$.
If $p=2$ we write $\H^s(\cO)$ in place of $\W^{s,2}(\cO)$, and denote the corresponding norm by $\|\cdot\|_{\H^s(\cO)}$.
In addition, we denote by $\H^{1/2}(\Gamma)$ the trace space of $\H^1(\cO)$, and let $\H^{-1/2}(\Gamma)$ be the dual space of $\H^{1/2}(\Gamma)$ endowed with the norms $\|\cdot\|_{\H^{1/2}(\Gamma)}$ and $\|\cdot\|_{\H^{-1/2}(\Gamma)}$, respectively.
By $\bM$ and $\bbM$ we will denote the corresponding vectorial and tensorial counterparts of the generic scalar functional space $\M$.
The $\L^2(\cO)$ inner product for scalar, vector, or tensor valued functions is denoted by $(\cdot,\cdot)_{\cO}$. For a section of the boundary $\Gamma$, the $\L^2(\Gamma)$ inner product or duality pairing is denoted by $\pil\cdot,\cdot\pir_\Gamma$. For a Banach space $\V$, we denote its dual space by $\V'$. 
For an operator $\cA:\V \to \U'$, its adjoint operator is denoted by $\cA':\U \to \V'$. 
For any vector fields $\bv=(v_i)_{i=1,n}$ and $\bw=(w_i)_{i=1,n}$, 
we set the gradient, the symmetric part of the gradient, divergence,
and tensor product operators, as
\begin{equation*}
\nabla\bv := \left(\frac{\partial v_i}{\partial x_j}\right)_{i,j=1,n},\quad 
\be(\bv) := \frac{1}{2}\left\{ \nabla\bv + (\nabla\bv)^\rt \right\},\quad
\div(\bv) := \sum_{j=1}^n \frac{\partial v_j}{\partial x_j},\qan 
\bv\otimes\bw := (v_i w_j)_{i,j=1,n} \,.
\end{equation*}
Furthermore, for any tensor field $\btau:=(\tau_{ij})_{i,j=1,n}$ and $\bzeta:=(\zeta_{ij})_{i,j=1,n}$, we let $\bdiv(\btau)$ be the divergence operator $\div$ acting along the rows of $\btau$, and define the transpose,
the trace, the tensor inner product, and the deviatoric tensor, respectively, as
\begin{equation}\label{eq:transpose-trace-tensor-deviatoric}
\btau^\rt := (\tau_{ji})_{i,j=1,n},\quad \tr(\btau):=\sum_{i=1}^n \tau_{ii},\quad \btau:\bzeta:=\sum_{i,j=1}^n \tau_{ij}\zeta_{ij},\qan \btau^\rd:=\btau-\frac{1}{n}\,\tr(\btau)\,\bI\,,
\end{equation}
where $\bI$ is the identity matrix in $\R^{n\times n}$.
In addition, we recall that
$\bH(\div;\cO):=\Big\{ \bw\in\bL^2(\cO) :\quad \div(\bw)\in \L^2(\cO) \Big\}$,
equipped with the norm $\|\bw\|^2_{\bH(\div;\cO)} := \|\bw\|^2_{\bL^2(\cO)} + \|\div(\bw)\|^2_{\L^2(\cO)}$, is a standard Hilbert space in the realm of mixed problems.
The space of matrix valued functions whose rows belong to $\bH(\div;\cO)$ will be denoted by $\bbH(\bdiv;\cO)$ and endowed with the norm $\|\btau\|^2_{\bbH(\bdiv;\cO)} := \|\btau\|^2_{\bbL^2(\cO)} + \|\bdiv(\btau)\|^2_{\bL^2(\cO)}$. 
Moreover, given a separable Banach space $\V$ endowed with the norm $\| \cdot \|_{\V}$, we introduce
the Bochner spaces $\L^2(0,T;\V)$, $\H^s(0,T;\V)$, with integer $s \ge 1$,
and $\L^{\infty}(0,T;\V)$ and $\W^{1,\infty}(0,T;\V)$, endowed with the norms
\begin{equation*}
\begin{array}{c}
\ds\|f\|^{2}_{\L^{2}(0,T;\V)} \,:=\, \int^T_0 \|f(t)\|^{2}_{\V} \,dt \,,\quad
\ds\|f\|^2_{\H^s(0,T;\V)} \,:=\, \int^T_0 \sum^{s}_{i=0} \|\partial^{i}_t f(t)\|^2_{\V}\,dt\,, \\[3ex]
\ds\|f\|_{\L^\infty(0,T;\V)} \,:=\, \mathop{\esssup}\limits_{t\in [0,T]} \|f(t)\|_{\V}\, ,\quad \|f\|_{\W^{1,\infty}(0,T;\V)} \,:=\, \mathop{\esssup}\limits_{t\in [0,T]} \big\{\|f(t)\|_{\V} + \|\partial_t f(t)\|_{\V}\big\}.
\end{array}
\end{equation*}


\section{Navier--Stokes--Biot model problem}\label{sec:Navier-Stokes--Biot model problem}

Let $\Omega\subset \R^n$, $n\in\{2,3\}$ be a Lipschitz domain, which is subdivided into two polytopal non-overlapping and possibly non-connected regions: fluid region $\Omega_f$ and poroelastic region $\Omega_p$.
Let $\Gamma_{fp} := \partial\Omega_f\cap\partial\Omega_p$ denote the (nonempty) interface between these regions and let $\Gamma_f := \partial\Omega_f\setminus\Gamma_{fp}$ and $\Gamma_p := \partial\Omega_p\setminus\Gamma_{fp}$ denote the external parts of the boundary $\partial\Omega$.
In turn, given $\star\in\{f,p\}$, let $\Gamma_\star = \Gamma^\rD_\star\cup \Gamma^\rN_\star$, with $\Gamma^\rD_\star\cap \Gamma^\rN_\star = \emptyset$ and $|\Gamma^\rD_\star|, |\Gamma^\rN_\star| > 0$. We denote by $\bn_f$ and $\bn_p$ the unit normal vectors which point outward from $\partial\Omega_f$ and $\partial\Omega_p$, respectively, noting that $\bn_f = - \bn_p$ on $\Gamma_{fp}$. Figure \ref{fig:domain_sketch} gives a schematic representation of the geometry. 
\begin{figure}[t]
\centering\includegraphics[scale=1]{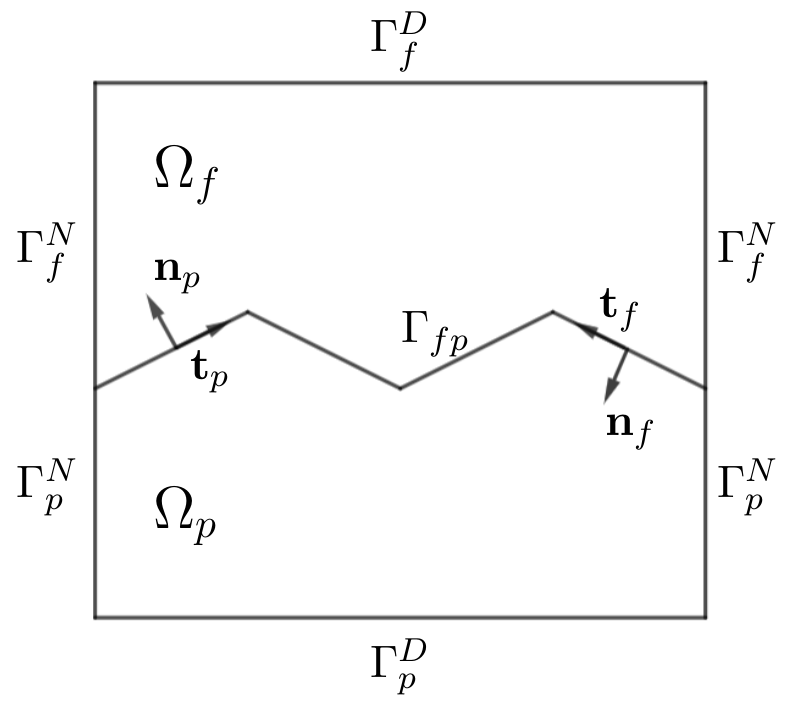}
	
\vspace{-0.2cm}
\caption{Schematic representation of a 2D computational domain.}
\label{fig:domain_sketch}
\end{figure}
Let $(\bu_\star,p_\star)$ be the velocity-pressure pair in $\Omega_\star$, and let $\bbeta_p$ be the displacement in $\Omega_p$.
Let $\mu>0$ be the fluid viscosity, let $\rho_f$ be the fluid density, let $\f_\star$ be 
body force terms, and let $q_p$ be an external source or sink term. The flow in $\Omega_f$ is governed by the Navier--Stokes equations:
\begin{align}
& \rho_f\,\left(\frac{\partial\,\bu_f}{\partial\,t} + (\nabla\bu_f)\,\bu_f\right) - \bdiv(\bT_f) \,=\, \f_f,\quad \div(\bu_f) \,=\, 0 \qin \Omega_f\times (0,T], \nonumber\\ 
&\left(\bT_f - \rho_f\,(\bu_f\otimes\bu_f) \right)\bn_f \,=\, \0 \qon \Gamma^\rN_f\times (0,T],\quad 
\bu_f \,=\, \0 \qon \Gamma^\rD_f\times (0,T],\label{eq:Navier-Stokes-1}
\end{align}
where $\bT_f := -p_f\,\bI + 2\,\mu\,\be(\bu_f)$ denotes the stress tensor.
In this work we make use of an equivalent version of \eqref{eq:Navier-Stokes-1} based on the introduction of a pseudostress tensor combining the stress tensor $\bT_f$ with the convective term.
More precisely, analogously to \cite{cgo2021}, we introduce the nonlinear-pseudostress tensor
\begin{equation}\label{eq:pseudostress-tensor-formulae-1}
\bsi_f := \bT_f - \rho_f(\bu_f\otimes\bu_f) = -p_f\,\bI + 2\,\mu\,\be(\bu_f) - \rho_f(\bu_f\otimes \bu_f) \qin \Omega_f\times (0,T]\,.
\end{equation}
In order to derive the weak formulation, similarly to \cite{cgot2016,cgos2017}, we first observe, owing to the fact that $\tr(\be(\bu_f)) = \div(\bu_f) = 0$, that
\begin{equation}\label{eq:div-tr-identities}
\bdiv(\bu_f\otimes\bu_f) \,=\, (\nabla\bu_f)\,\bu_f,\quad
\tr(\bsi_f) \,=\, -\,n\,p_f - \rho_f\,\tr(\bu_f\otimes\bu_f)\,.
\end{equation}
In particular, from the second equation in \eqref{eq:div-tr-identities} we observe that the pressure $p_f$ can be written in terms of the velocity $\bu_f$ and the nonlinear pseudostress tensor $\bsi_f$ as
\begin{equation}\label{eq:pseudostress-pressure-formulae}
p_f \,=\, -\frac{1}{n}\,\Big( \tr(\bsi_f) + \rho_f\,\tr(\bu_f\otimes\bu_f) \Big)  \qin \Omega_f\times (0,T] \,.
\end{equation}
Hence, plugging \eqref{eq:pseudostress-pressure-formulae} into \eqref{eq:pseudostress-tensor-formulae-1} and using the definition of the deviatoric operator \eqref{eq:transpose-trace-tensor-deviatoric}, we obtain
$\bsi^\rd_f \,=\, 2\,\mu\,\be(\bu_f) - \rho_f\,(\bu_f\otimes\bu_f)^\rd$.
In addition, to weakly impose the symmetry of the pseudostress tensor and apply the integration by parts formula, we introduce the additional unknown
\begin{equation*}
\bgamma_f \,:=\, \frac{1}{2}\,\left\{ \nabla\bu_f - (\nabla\bu_f)^\rt \right\}\,,
\end{equation*}
which represents the vorticity (or the skew-symmetric part of the velocity gradient). 
Thus, employing the identities \eqref{eq:div-tr-identities}, the Navier--Stokes system \eqref{eq:Navier-Stokes-1} can be rewritten as the set of equations with unknowns $\bsi_f, \bgamma_f$ and $\bu_f$, given by
\begin{equation}\label{eq:Navier-Stokes-3}
\begin{array}{c}
\ds \frac{1}{2\,\mu}\,\bsi^\rd_f = \nabla\bu_f - \bgamma_f - \frac{\rho_f}{2\,\mu}\,(\bu_f\otimes\bu_f)^\rd \,,\quad 
\rho_f\,\frac{\partial\,\bu_f}{\partial\,t} - \bdiv(\bsi_f) = \f_f \qin \Omega_f\times (0,T] \,, \\[2ex]
\ds \bsi_f\bn_f \,=\, \0 \qon \Gamma^\rN_f\times (0,T] \,,\quad 
\bu_f \,=\, \0 \qon \Gamma^\rD_f\times (0,T]\,,
\end{array}
\end{equation}
where $\bsi_f$ is a symmetric tensor in $\Omega_f\times (0,T]$.
Notice that, as suggested by \eqref{eq:pseudostress-pressure-formulae}, $p_f$ is eliminated from the present formulation and can be computed afterwards in terms of $\bsi_f$ and $\bu_f$.
In addition, the fluid stress $\bT_f$ can be recovered from \eqref{eq:pseudostress-tensor-formulae-1}.
For simplicity, we assume that $|\Gamma^\rN_f| > 0$, which allows us to control $\bsi_f$ by $\bsi^\rd_f$, cf. \eqref{eq:tau-H0div-Xf-inequality}. The case $|\Gamma^\rN_f| = 0$ can be handled as in \cite{gmor2014} and \cite{gov2020} by introducing an additional
variable corresponding to the mean value of $ \tr(\bsi_f)$. In addition, in order to simplify the characterization of the normal trace space $\bsi_f\bn_f|_{\Gamma_{fp}}$ (cf. \eqref{eq:trace-inequality-2}), we assume that $\Gamma^\rD_f$ is not adjacent to $\Gamma_{fp}$, i.e., $\dist(\Gamma^\rD_f,\Gamma_{fp}) \geq s >0$. We further note that one can also consider the boundary condition $\bT_f\bn_f = \0$ on $\Gamma^\rN_f$ in terms of the fluid stress tensor. This leads to a Robin boundary condition in terms of the pseudostress, $\bsi_f\bn_f + \rho_f(\bu_f\otimes \bu_f)\bn_f = 0$ on $\Gamma^\rN_f$. In this case, the space for $\bsi_f$ is unrestricted on $\Gamma^\rN_f$. The second term in \eqref{eq:continuous-weak-formulation-1a} becomes $\pil\btau_f\bn_f,\bvarphi \pir_{\Gamma_{fp} \cup \Gamma^\rN_f}$ and the first and third terms in \eqref{eq:continuous-weak-formulation-1h} become $\pil\bsi_f\bn_f,\bpsi\pir_{\Gamma_{fp} \cup \Gamma^\rN_f}$ and
$\rho_f\pil \bvarphi\cdot\bn_f, \bvarphi\cdot\bpsi \pir_{\Gamma_{fp} \cup \Gamma^\rN_f}$, which can be treated in a similar way. Moreover, the control of $\bsi_f$ by $\bsi_f^\rd$ can be achieved as in the case $|\Gamma^\rN_f| = 0$. In addition,
in order to avoid the issue of restricting the mean value of the fluid velocity $\bu_f$ (cf. \eqref{eq:inf-sup-vf-chif}), we assume that $|\Gamma^\rD_f| > 0$,  

In turn, in $\Omega_p$ we consider the fully dynamic Biot system \cite{b1941}:
\begin{equation}\label{eq:Biot-model}
\begin{array}{c}
\ds \rho_p\,\frac{\partial^2\bbeta_p}{\partial\,t^2} - \bdiv(\bsi_p) = \f_p,\quad 
\mu\,\bK^{-1}\bu_p + \nabla\,p_p = \0 \qin \Omega_p\times(0,T] \,, \\[2ex]
\ds \frac{\partial}{\partial\,t}\left( s_0\,p_p + \alpha_p\,\div(\bbeta_p) \right) + \div(\bu_p) = q_p \qin \Omega_p\times(0,T] \,, \\[2ex]
\ds \bu_p\cdot\bn_p \,=\, 0 \qon \Gamma^\rN_p\times (0,T],\quad 
p_p \,=\, 0 \qon \Gamma^\rD_p\times (0,T],\quad
\bbeta_p \,=\, \0 \qon \Gamma_p\times (0,T]\,,
\end{array}
\end{equation}
where $\rho_p$ represents the fluid density in the poroelastic region, 
$s_0 > 0$ is a storage coefficient, $0 < \alpha_p \leq 1$ is the Biot--Willis constant and $\bK$ the symmetric and uniformly positive definite permeability tensor, satisfying, for some constants $0< k_{\min}\leq k_{\max}$,
\begin{equation}\label{eq:permeability-bounds}
\forall\, \bw\in\R^n \quad k_{\min}\,|\bw|^2 \,\leq\, \bw^\rt\,\bK^{-1}(\bx)\bw \,\leq\, k_{\max}\,|\bw|^2 \quad \forall\, \bx\in\Omega_p \,.
\end{equation}
In addition, $\bsi_e$ and $\bsi_p$ denote the elastic and poroelastic stress tensors, respectively, and both satisfy
\begin{equation}\label{eq:bsie-bsip-definitions}
A(\bsi_e) \,= \be(\bbeta_p)  \qan
\bsi_p \,:=\, \bsi_e - \alpha_p\,p_p\,\bI \,,
\end{equation}
where $A$ is a symmetric and positive definite compliance tensor,
satisfying, for some $0 < a_{\min} \leq a_{\max} < \infty$,
\begin{equation}\label{eq:A-bounds}
\forall\, \btau\in\R^{n\times n}, \quad a_{\min}
\, \btau : \btau \, \leq \, A(\btau):\btau \, \leq \, a_{\max} \,
\btau : \btau \quad \forall\, \bx\in\Omega_p \,.
\end{equation}
In the case of an isotropic material, $A$ is given as
\begin{equation}\label{eq:elasticity-stress-isotropic}
A(\bsi_e) \,=\, \frac{1}{2\,\mu_p} \left(\bsi_e - \frac{\lambda_p}{2\,\mu_p + n\,\lambda_p}\,\tr(\bsi_e)\,\bI\right),\quad\mbox{ with }\quad 
A^{-1}(\bsi_e) \,=\, 2\,\mu_p\,\bsi_e + \lambda_p\,\tr(\bsi_e)\,\bI \,,
\end{equation}
where $\ds 0<\lambda_{\min}\leq \lambda_p(\bx)\leq \lambda_{\max}$ and 
$\ds 0<\mu_{\min} \leq\mu_p(\bx) \leq\mu_{\max}$ are the Lam\'e parameters.
In this case, 
$\bsi_e \,:=\, \lambda_p\,\div(\bbeta_p)\,\bI + 2\,\mu_p\,\be(\bbeta_p)$,  $\ds a_{\min}=1/(2 \mu_{\max} + n \, \lambda_{\max})$, and 
$\ds a_{\max}=1/(2 \mu_{\min})$.
In order to avoid the issue with restricting the mean value of the pressure $p_p$ (cf. \eqref{eq:inf-sup-qp-xi}), we assume that $|\Gamma^\rD_p| > 0$. Also, in order to simplify the characterization of the normal trace space $\bu_p\cdot\bn_p|_{\Gamma_{fp}}$ (cf. \eqref{eq:trace-inequality-1}), we assume that $\Gamma^\rD_p$ is not adjacent to $\Gamma_{fp}$, i.e., $\dist(\Gamma^\rD_p,\Gamma_{fp}) \geq s >0$, which also implies that $|\Gamma^\rN_p| > 0$.

Next, we introduce the transmission conditions on the interface $\Gamma_{fp}$:
\begin{equation}\label{eq:interface-conditions}
\begin{array}{c}
\ds \bu_f\cdot\bn_f + \left(\frac{\partial\,\bbeta_p}{\partial\,t} + \bu_p\right)\cdot\bn_p \,=\, 0,\quad 
\bT_f\bn_f + \bsi_p\bn_p \,=\, \0 \qon \Gamma_{fp}\times (0,T] \,, \\[2ex]
\ds \bT_f\bn_f + \mu\,\alpha_{\BJS}\sum^{n-1}_{j=1}\,\sqrt{\bK^{-1}_j}\left\{\left(\bu_f - \frac{\partial\,\bbeta_p}{\partial\,t}\right)\cdot\bt_{f,j}\right\}\,\bt_{f,j} \,=\, -\,p_p\bn_f \qon \Gamma_{fp}\times (0,T] \,,
\end{array}	
\end{equation}
where $\bt_{f,j}$, $1\leq j\leq n-1$, is an orthogonal system of unit tangent vectors on $\Gamma_{fp}$, $\bK_j = (\bK\,\bt_{f,j})\cdot\bt_{f,j}$, and $\alpha_{\BJS} \geq 0$ is an experimentally determined friction coefficient.
The first and second equations in \eqref{eq:interface-conditions} corresponds to mass conservation and conservation of momentum on $\Gamma_{fp}$, respectively, whereas the third one can be decomposed into its normal and tangential components, as follows:
\begin{equation}\label{eq:BJS-condition-split}
(\bT_f\bn_f)\cdot\bn_f \,=\, -\,p_p,\quad
(\bT_f\bn_f)\cdot\bt_{f,j} \,=\, -\,\mu\,\alpha_{\BJS}\,\sqrt{\bK^{-1}_j}\left(\bu_f - \frac{\partial\,\bbeta_p}{\partial\,t}\right)\cdot\bt_{f,j} \qon \Gamma_{fp}\times (0,T]\,.
\end{equation}
The first condition in \eqref{eq:BJS-condition-split} corresponds to the balance of normal stress, whereas the second one is known as the Beavers--Joseph--Saffman slip with friction condition.
Notice that the second and third equations in \eqref{eq:interface-conditions} can be rewritten in terms of tensor $\bsi_f$ as follows:
\begin{equation}\label{eq:BJS-condition-2}
\begin{array}{c}
\ds \bsi_f\bn_f + \rho_f(\bu_f\otimes \bu_f)\bn_f + \bsi_p\bn_p = \0 \,, \\[2ex]
\ds \bsi_f\bn_f + \rho_f(\bu_f\otimes \bu_f)\bn_f
+ \mu\,\alpha_{\BJS}\sum^{n-1}_{j=1}\,\sqrt{\bK^{-1}_j}\left\{\left(\bu_f - \frac{\partial\,\bbeta_p}{\partial\,t}\right)\cdot\bt_{f,j}\right\}\,\bt_{f,j} \,=\, -\,p_p\bn_f\,.
\end{array}	
\end{equation}

Finally, the above system of equations is complemented by a set of initial conditions:
\begin{equation*}
\begin{array}{c}
\ds \bu_f(\bx,0) \,=\, \bu_{f,0}(\bx) \qin \Omega_f\,,\quad	
p_p(\bx,0) \,=\, p_{p,0}(\bx) \qin \Omega_p \,, \\[1ex]
\ds \bbeta_p(\bx,0) = \bbeta_{p,0}(\bx),\qan
\partial_t\,\bbeta_p(\bx,0) = \bu_{s,0}(\bx) \qin \Omega_p,
\end{array}
\end{equation*}
where $\bu_{f,0}, p_{p,0}, \bbeta_{p,0}$, and $\bu_{s,0}$ are suitable initial data.
Conditions on these initial data are discussed later on in Lemma \ref{lem:sol0-in-M-operator}.


\section{The variational formulation}\label{sec:variational-formulation}

In this section we proceed analogously to \cite[Section~3]{aeny2019} (see also \cite{cgos2017,fpsi-msfmfe,gmor2014}) and derive a weak formulation of the coupled problem given by \eqref{eq:Navier-Stokes-3}, \eqref{eq:Biot-model}, and \eqref{eq:interface-conditions}--\eqref{eq:BJS-condition-2}.

\subsection{Preliminaries}

We first introduce further notations and definitions. In what follows, given $\star\in\big\{ f, p \big\}$, we set
\begin{equation*}
(p,q)_{\Omega_\star} \,:=\, \int_{\Omega_\star} p\,q,\quad 
(\bu,\bv)_{\Omega_\star} \,:=\, \int_{\Omega_\star} \bu\cdot\bv \qan
(\bsi,\btau)_{\Omega_\star} \,:=\, \int_{\Omega_\star} \bsi:\btau.
\end{equation*}
In addition, similarly to \cite{cgo2021} and \cite{cgos2017}, in the sequel we will employ suitable Banach spaces to deal with the nonlinear stress tensor and velocity of the Navier--Stokes equation, together with the subspace of skew-symmetric tensors of $\bbL^2(\Omega_f)$ for the vorticity, that is,
\begin{equation*}
\begin{array}{l}
\bbX_f \,:=\, \Big\{ \btau_f\in \bbL^2(\Omega_f) :\quad \bdiv(\btau_f)\in \bL^{4/3}(\Omega_f) \qan \btau_f\bn_f = \0 \qon \Gamma^\rN_f \Big\}, \\ [2ex] 
\bV_f \,:=\, \bL^4(\Omega_f)\,,\quad
\bbQ_f \,:=\, \Big\{ \bchi_f\in \bbL^2(\Omega_f) :\quad \bchi^\rt_f = - \bchi_f \Big\}\,,
\end{array}
\end{equation*}
endowed with the corresponding norms
\begin{equation*}
\|\btau_f\|^2_{\bbX_f} \,:=\, \|\btau_f\|^2_{\bbL^2(\Omega_f)} + \|\bdiv(\btau_f)\|^2_{\bL^{4/3}(\Omega_f)},\quad
\|\bv_f\|_{\bV_f} \,:=\, \|\bv_f\|_{\bL^4(\Omega_f)},\quad
\|\bchi_f\|_{\bbQ_f} \,:=\, \|\bchi_f\|_{\bbL^2(\Omega_f)}.
\end{equation*}
For the velocity, pressure and displacement in the poroelastic region $\Omega_p$ we will use Hilbert spaces, respectively,
\begin{equation*}
\begin{array}{l}
\bX_p \,:=\, \Big\{ \bv_p\in \bH(\div;\Omega_p) :\quad \bv_p\cdot\bn = 0 \qon \Gamma^\rN_p \Big\}, \\ [2ex]
\W_p \,:=\, \L^2(\Omega_p)\,,\quad
\bV_p \,:=\, \Big\{ \bxi_p\in \bH^1(\Omega_p) :\quad \bxi_p = \0 \qon \Gamma_p \Big\}, 
\end{array}
\end{equation*}
endowed with the norms
\begin{equation*}
\|\bv_p\|_{\bX_p} \,:=\, \|\bv_p\|_{\bH(\div;\Omega_p)},\quad
\|w_p\|_{\W_p} \,:=\, \|w_p\|_{\L^2(\Omega_p)},\quad
\|\bxi_p\|_{\bV_p} \,:=\, \|\bxi_p\|_{\bH^1(\Omega_p)}.
\end{equation*}
Finally, analogously to \cite{gmor2014,cgos2017} we need to introduce the spaces of traces $\Lambda_p := (\bX_p\cdot\bn_p|_{\Gamma_{fp}})'$ and $\bLambda_f := (\bbX_f\bn_f|_{\Gamma_{fp}})'$.
According to the normal trace theorem, since $\bv_p\in \bX_p\subset \bH(\div;\Omega_p)$, then $\bv_p\cdot\bn_p\in \H^{-1/2}(\partial\,\Omega_p)$.
It is shown in \cite{FPSI-LM} that, since $\bv_p\cdot\bn_p = 0$ on $\Gamma^\rN_p$ and $\dist(\Gamma^\rD_p,\Gamma_{fp}) \geq s> 0$, it holds that
$\bv_p\cdot\bn_p\in \H^{-1/2}(\Gamma_{fp})$ and
\begin{equation}\label{eq:trace-inequality-1}
\pil \bv_p \cdot \bn_p, \xi \pir_{\Gamma_{fp}} 
\,\leq\, C\,\| \bv_p \|_{\bH(\div; \Omega_p)} \| \xi \|_{\H^{1/2}(\Gamma_{fp})} \quad \forall\,\bv_p \in \bX_p, \, \xi \in \H^{1/2}(\Gamma_{fp}).
\end{equation}
Since $\dist(\Gamma^\rD_f,\Gamma_{fp}) \geq s> 0$, 
using similar arguments, in combination with \cite[Lemma~3.5]{cgo2021}, we have,
\begin{equation}\label{eq:trace-inequality-2}
\langle \btau_f \, \bn_f,\bpsi \rangle_{\Gamma_{fp}}
\le C\,\|\btau_f\|_{\bbH(\bdiv_{4/3};\Omega_f)}\|\bpsi\|_{\bH^{1/2}(\Gamma_{fp})}
\quad \forall \, \btau_f \in \bbX_f, \,
\bpsi \in \bH^{1/2}(\Gamma_{fp}).
\end{equation}
Therefore we take
$\Lambda_p = \H^{1/2}(\Gamma_{fp})$ and $\bLambda_f = \bH^{1/2}(\Gamma_{fp})$,
endowed with the norms
\begin{equation}\label{eq:H1/2-norms}
\|\xi\|_{\Lambda_p} \,:=\, \|\xi\|_{\H^{1/2}(\Gamma_{fp})} \quad \text{and} \quad
\|\bpsi\|_{\bLambda_f} \,:=\, \|\bpsi\|_{\bH^{1/2}(\Gamma_{fp})}\,.
\end{equation} 


\subsection{Lagrange multiplier formulation}

We now proceed with the derivation of our Lagrange multiplier variational formulation for the coupling of the Navier--Stokes and Biot problems.
To this end, we begin by introducing two Lagrange multipliers, which represent the Navier--Stokes velocity and Darcy pressure on the interface, respectively,
\begin{equation*}
\bvarphi := \bu_f|_{\Gamma_{fp}} \in \bLambda_f \qan 
\lambda := p_p|_{\Gamma_{fp}}\in \Lambda_p \,.
\end{equation*}
Then, similarly to \cite{cgos2017,aeny2019}, we test the first and second equation of \eqref{eq:Navier-Stokes-3}, and the first, second and third equations of \eqref{eq:Biot-model} with arbitrary $\btau_f\in \bbX_f$, $\bv_f\in \bV_f$, $\bxi_p\in \bV_p,\bv_p\in \bX_p$ and $w_p\in\W_p,$  integrate by parts, utilize the fact that $\bsi^\rd_f:\btau_f = \bsi^\rd_f:\btau^\rd_f$, and impose the remaining equations weakly, as well as the symmetry of $\bsi_f$ and the transmission conditions (cf. first equation of \eqref{eq:interface-conditions} and \eqref{eq:BJS-condition-2}) to obtain the variational problem: Given 
$\f_f : [0,T]\to \bL^2(\Omega_f)$, $\f_p : [0,T]\to \bL^2(\Omega_p)$, $q_p : [0,T]\to \L^2(\Omega_p)$
and $(\bu_{f,0},p_{p,0},\bbeta_{p,0},\bu_{s,0})\in \bV_f\times \W_p\times \bV_p\times \bV_p$\,,
find $(\bsi_f,\bu_p, \bbeta_p, \bu_f, p_p, \bgamma_f, \bvarphi, \lambda) : [0,T]\to \bbX_f\times \bX_p\times \bV_p\times \bV_f\times \W_p\times \bbQ_f\times \bLambda_f\times \Lambda_p$, such that $\bu_f(0)=\bu_{f,0}$, $p_p(0)=p_{p,0}$, $\bbeta_p(0)=\bbeta_{p,0}$ and $\partial_t\,\bbeta_p(0) = \bu_{s,0}$, and for a.e. $t\in (0,T)$:
\begin{subequations}\label{eq:continuous-weak-formulation-1}
\begin{align}
& \ds \frac{1}{2\,\mu}\,(\bsi^\rd_f,\btau^\rd_f)_{\Omega_f} - \pil\btau_f\bn_f,\bvarphi \pir_{\Gamma_{fp}} + (\bu_f,\bdiv(\btau_f))_{\Omega_f} + (\bgamma_f,\btau_f)_{\Omega_f} 
+ \frac{\rho_f}{2\,\mu}\,((\bu_f\otimes\bu_f)^\rd,\btau_f^\rd)_{\Omega_f} = 0, \label{eq:continuous-weak-formulation-1a}\\[1ex]
& \ds \rho_f\, (\partial_t\,\bu_f,\bv_f)_{\Omega_f} - (\bv_f,\bdiv(\bsi_f))_{\Omega_f} = (\f_f,\bv_f)_{\Omega_f}, \label{eq:continuous-weak-formulation-1b} \\[1ex]
& \ds -\,\,(\bsi_f,\bchi_f)_{\Omega_f} = 0, \label{eq:continuous-weak-formulation-1c} \\[1ex]
& \ds \rho_p (\partial_{tt}\bbeta_p,\bxi_p)_{\Omega_p} + 2\,\mu_p\,(\be(\bbeta_p),\be(\bxi_p))_{\Omega_p} + \lambda_p\,(\div(\bbeta_p),\div(\bxi_p))_{\Omega_p} -\,\,\alpha_p\,(p_p,\div(\bxi_p))_{\Omega_p} \nonumber \\[0.5ex] 
& \ds\quad +\pil\bxi_p\cdot\bn_p,\lambda\pir_{\Gamma_{fp}} - \mu\,\alpha_{\BJS}\,\sum_{j=1}^{n-1} \pil\sqrt{\bK^{-1}_j}\left( \bvarphi - \partial_t\,\bbeta_p \right)\cdot\bt_{f,j},\bxi_p\cdot\bt_{f,j} \pir_{\Gamma_{fp}} = (\f_p,\bxi_p)_{\Omega_p}, \label{eq:continuous-weak-formulation-1d} \\[1ex]
& \ds \mu\,(\bK^{-1}\bu_p,\bv_p)_{\Omega_p} - (p_p,\div(\bv_p))_{\Omega_p} + \pil\bv_p\cdot\bn_p,\lambda\pir_{\Gamma_{fp}} = 0, \label{eq:continuous-weak-formulation-1e} \\[1ex]
& \ds s_0\,(\partial_t\, p_p,w_p)_{\Omega_p} + \alpha_p\,(\div(\partial_t\,\bbeta_p), w_p)_{\Omega_p} + (w_p,\div(\bu_p))_{\Omega_p} = (q_p,w_p)_{\Omega_p},\label{eq:continuous-weak-formulation-1f} \\[1ex]
& \ds - \pil\bvarphi\cdot\bn_f + \left(\partial_t\,\bbeta_p + \bu_p\right)\cdot\bn_p,\xi\pir_{\Gamma_{fp}} = 0, \label{eq:continuous-weak-formulation-1g} \\[1ex]
& \ds \pil\bsi_f\bn_f,\bpsi\pir_{\Gamma_{fp}} + \mu\,\alpha_{\BJS}\,\sum_{j=1}^{n-1} \pil\sqrt{\bK^{-1}_j}\left( \bvarphi - \partial_t\,\bbeta_p \right)\cdot\bt_{f,j},\bpsi\cdot\bt_{f,j} \pir_{\Gamma_{fp}} \nonumber \\[0.5ex]
 & \ds\quad +\,\, \rho_f\pil \bvarphi\cdot\bn_f, \bvarphi\cdot\bpsi \pir_{\Gamma_{fp}} + \pil\bpsi\cdot\bn_f,\lambda\pir_{\Gamma_{fp}} = 0,\label{eq:continuous-weak-formulation-1h}
\end{align}
\end{subequations}
for all $(\btau_f, \bv_p, \bxi_p, \bv_f, w_p, \bchi_f, \bpsi, \xi)\in \bbX_f\times \bX_p\times \bV_p\times \bV_f\times \W_p\times \bbQ_f\times \bLambda_f\times \Lambda_p$.

We note that the fifth and sixth terms on the left-hand side of \eqref{eq:continuous-weak-formulation-1d} are obtained using the equation
\begin{equation}\label{eq:elasticity-stress-tensor-identity}
\pil\bsi_p\bn_p,\bxi_p\pir_{\Gamma_{fp}} 
= - \pil\bxi_p\cdot\bn_p,\lambda\pir_{\Gamma_{fp}} 
+ \mu\,\alpha_{\BJS}\,\sum_{j=1}^{n-1} \pil\sqrt{\bK^{-1}_j}\left( \bvarphi - \partial_t\,\bbeta_p \right)\cdot\bt_{f,j},\bxi_p\cdot\bt_{f,j} \pir_{\Gamma_{fp}},
\end{equation}
which follows from combining the second and third equations in \eqref{eq:interface-conditions} (or \eqref{eq:BJS-condition-2}). 
Note that \eqref{eq:continuous-weak-formulation-1a}--\eqref{eq:continuous-weak-formulation-1c} correspond to the Navier--Stokes equations, \eqref{eq:continuous-weak-formulation-1d} is the elasticity equation, \eqref{eq:continuous-weak-formulation-1e}--\eqref{eq:continuous-weak-formulation-1f} are the Darcy equations, whereas \eqref{eq:continuous-weak-formulation-1g}--\eqref{eq:continuous-weak-formulation-1h}, together with the interface terms in \eqref{eq:continuous-weak-formulation-1d}, enforce weakly the transmission conditions \eqref{eq:interface-conditions}. 
In particular, \eqref{eq:continuous-weak-formulation-1g} imposes the mass conservation, \eqref{eq:continuous-weak-formulation-1h} imposes the last equation in \eqref{eq:interface-conditions} which is a combination of balance of normal stress and the BJS condition, while the interface terms in \eqref{eq:continuous-weak-formulation-1d} imposes the conservation of momentum.

Analyzing \eqref{eq:continuous-weak-formulation-1} directly is challenging, due to the presence of $\partial_t\bbeta_p$ in several non-coercive terms (in addition to the nonlinear nature of the model). 
Motivated by \cite{ErvJenSun}, \cite{show2005}, and \cite{aeny2019}, we analyze an alternative formulation, which will then be used to establish the well-posedness of \eqref{eq:continuous-weak-formulation-1}.

 
\subsection{Alternative formulation}

We proceed analogously to \cite{aeny2019} and derive a system of evolutionary saddle point type, which fits the general framework studied in \cite{s2010}.
Following the approach from \cite{show2005}, we do this by considering a mixed elasticity formulation with the structure velocity $\bu_s := \partial_t\,\bbeta_p\in \bV_p$ and the elastic stress $\bsi_e$ (cf. \eqref{eq:bsie-bsip-definitions}) as primary variables.
The space and norm for the elastic stress are respectively
\begin{equation*}
\bSigma_e \,:=\, \Big\{ \btau_e\in \bbL^2(\Omega_p) :\quad \btau^\rt_e \,=\, \btau_e \Big\} \qan
\|\btau_e\|_{\bSigma_e} \,:=\, \|\btau_e\|_{\bbL^2(\Omega_p)} \,.
\end{equation*}
The derivation of the alternative variational formulation differs from the original one in the way the equilibrium equation $\rho_p\,\partial_{tt}\bbeta_p - \bdiv(\bsi_p) = \f_p$ (cf. \eqref{eq:Biot-model}) is handled.
As before, we multiply it by a test function $\bv_s\in \bV_p$ and integrate by parts.
However, instead of using the constitutive relation from the first equation in \eqref{eq:bsie-bsip-definitions}, we use only the second equation in \eqref{eq:bsie-bsip-definitions}, which results in
\begin{equation*}
\rho_p\,(\partial_t\bu_s,\bv_s)_{\Omega_p} + (\bsi_e, \be(\bv_s))_{\Omega_p} - \alpha_p\,(p_p,\div(\bv_s))_{\Omega_p} - \pil\bsi_p\bn_p, \bv_s\pir_{\Gamma_{fp}} \,=\, (\f_p, \bv_s)_{\Omega_p} \,,
\end{equation*}
where the term $\partial_{tt} \bbeta_p$ has been replaced by $\partial_t \bu_s$, and the last term on the left-hand side is treated as in \eqref{eq:continuous-weak-formulation-1d} by employing \eqref{eq:elasticity-stress-tensor-identity}.
Furthermore, we eliminate the displacement $\bbeta_p$ from the system by differentiating in time the first equation in \eqref{eq:bsie-bsip-definitions}. 
Multiplying by a test function $\btau_e\in \bSigma_e$ gives
\begin{equation*}
\partial_t(A(\bsi_e), \btau_e)_{\Omega_p} - (\be(\bu_s), \btau_e)_{\Omega_p} \,=\, 0 \,.
\end{equation*}
The rest of the equations are handled in the same way as in the original weak formulation \eqref{eq:continuous-weak-formulation-1}.

Next, given $\bw_f\in \bV_f$ and $\bzeta \in \bLambda_f$, we define the bilinear forms $a_f : \bbX_f\times \bbX_f\to \R, a^s_p : \bSigma_e\times \bSigma_e\to \R, \kappa_{\bw_f} : \bV_f\times \bbX_f\to \R, a^d_p : \bX_p\times \bX_p\to \R, a^e_p : \bV_p\times \bV_p\to \R, b_f : \bV_f\times \bbX_f\to \R$, $b_p : \W_p\times \bX_p\to \R, b_s : \bSigma_e \times \bV_p\to \R$, and $b_\sk : \bbQ_f\times \bbX_f\to \R$, as:
\begin{subequations}\label{eq:bilinear-forms}
\begin{align}
& a_f(\bsi_f,\btau_f) := \frac{1}{2\,\mu} (\bsi^\rd_f,\btau^\rd_f)_{\Omega_f} ,\quad a^s_p(\bsi_e,\btau_e) \,:=\, (A(\bsi_e),\btau_e)_{\Omega_p}, \label{bilinear-form-1} \\[0.5ex]
  & \ds \kappa_{\bw_f}(\bu_f,\btau_f)
  := \frac{\rho_f}{2\,\mu} ((\bw_f\otimes\bu_f)^\rd, \btau_f^\rd)_{\Omega_f},
  \quad a^d_p(\bu_p,\bv_p) := \mu (\bK^{-1}\bu_p,\bv_p)_{\Omega_p} , \label{bilinear-form-2}\\[0.5ex]
& \ds a^e_p(\bbeta_p,\bxi_p) := 2 \mu_p (\be(\bbeta_p),\be(\bxi_p))_{\Omega_p} + \lambda_p (\div(\bbeta_p),\div(\bxi_p))_{\Omega_p}, \label{bilinear-form-3}\\[0.5ex]
& \ds b_f(\btau_f,\bv_f) := (\bdiv(\btau_f),\bv_f)_{\Omega_f} ,\quad
b_p(\bv_p,w_p) := - (\div(\bv_p),w_p)_{\Omega_p} , \label{bilinear-form-4}\\[0.5ex]
& \ds b_s(\btau_e,\bv_s) := -\,(\btau_e,\be(\bv_s))_{\Omega_p},\quad
b_\sk(\btau_f,\bchi_f) := (\btau_f,\bchi_f)_{\Omega_f} , \label{bilinear-form-5}
\end{align}
\end{subequations}
and the forms on the interface $c_\BJS : (\bV_p\times \bLambda_f)\times (\bV_p\times \bLambda_f)\to \R$, $c_{\Gamma} : (\bV_p\times \bLambda_f)\times \Lambda_p\to \R$, $l_{\bzeta} : \bLambda_f \times \bLambda_f\to \R$, $b_{\bn_f} : \bbX_f\times \bLambda_f\to \R$, and $b_{\bn_p} : \bX_p\times \Lambda_p\to \R$
\begin{subequations}\label{eq:forms-on-interface-1}
\begin{align}
& c_{\BJS}(\bbeta_p, \bvarphi;\bxi_p, \bpsi) 
\,:=\, \mu\,\alpha_{\BJS}\,\sum^{n-1}_{j=1} \pil\sqrt{\bK^{-1}_j}( \bvarphi-\bbeta_p)\cdot\bt_{f,j},(\bpsi-\bxi_p)\cdot\bt_{f,j}\pir_{\Gamma_{fp}} \,, \label{interface-bilinear-form-1}\\[0.5ex]
& \ds c_{\Gamma}(\bxi_p,\bpsi;\xi) \,:=\, -\,\pil\bxi_p\cdot\bn_p,\xi\pir_{\Gamma_{fp}} - \pil\bpsi\cdot\bn_f,\xi\pir_{\Gamma_{fp}}\,,\quad
l_{\bzeta}(\bvarphi,\bpsi):=\rho_f\pil \bzeta\cdot\bn_f, \bvarphi\cdot\bpsi \pir_{\Gamma_{fp}} \,, \label{interface-bilinear-form-2}\\[0.5ex]
& \ds  b_{\bn_f}(\btau_f,\bpsi) \,:=-\, \pil\btau_f\bn_f,\bpsi\pir_{\Gamma_{fp}} \,,\quad
b_{\bn_p}(\bv_p,\xi) \,:=\, \pil\bv_p\cdot\bn_p, \xi\pir_{\Gamma_{fp}} 
\label{interface-bilinear-form-3} \,.
\end{align}
\end{subequations}
Hence the Lagrange variational formulation alternative to \eqref{eq:continuous-weak-formulation-1}, reads:
given $\f_f:[0,T]\to \bL^2(\Omega_f),\, \f_p : [0,T]\to \bL^2(\Omega_p),\, q_p:[0,T]\to \L^2(\Omega_p)$ and $(\bu_{f,0},p_{p,0},\bu_{s,0})\in \bV_f\times \W_p\times \bV_p$, find $(\bsi_f, \bu_p, \bsi_e, \bu_f, p_p, \bgamma_f, \bu_s$, $\bvarphi, \lambda) : [0,T]\to \bbX_f\times \bX_p\times \bSigma_e\times \bV_f\times \W_p\times \bbQ_f\times \bV_p\times \bLambda_f\times \Lambda_p$, such that $(\bu_f(0),p_p(0),\bu_s(0)) = (\bu_{f,0},p_{p,0},\bu_{s,0})$ and for a.e. $t\in (0,T)$:
\begin{subequations}\label{eq:continuous-alternative-weak-formulation-1}
\begin{align}
& \ds a_f(\bsi_f,\btau_f) + b_{\bn_f}(\btau_f,\bvarphi) + b_f(\btau_f,\bu_f) + b_\sk(\bgamma_f,\btau_f)+\kappa_{\bu_f}(\bu_f, \btau_f) = 0, \label{eq:continuous-alternative-weak-formulation-1a}\\[1ex]
& \ds \rho_f\, (\partial_t\,\bu_f,\bv_f)_{\Omega_f} -b_f(\bsi_f,\bv_f) = (\f_f,\bv_f)_{\Omega_f}, \label{eq:continuous-alternative-weak-formulation-1b} \\[1ex]
& \ds - b_\sk(\bsi_f,\bchi_f) = 0, \label{eq:continuous-alternative-weak-formulation-1c} \\[1ex]
& \ds \rho_p\,(\partial_t\bu_s,\bv_s)_{\Omega_p}  - b_s(\bsi_e,\bv_s)  + \alpha_p\,b_p(\bv_s,p_p) - c_{\Gamma}(\bv_s,\0;\lambda)+\,\, c_{\BJS}(\bu_s,\bvarphi;\bv_s,\0)  \nonumber \\[0.5ex] 
& \ds = (\f_p,\bv_s)_{\Omega_p},\label{eq:continuous-alternative-weak-formulation-1d1} \\[1ex]
& \ds a^s_p(\partial_t\,\bsi_e,\btau_e) + b_s(\btau_e,\bu_s)= 0, \label{eq:continuous-alternative-weak-formulation-1d2} \\[1ex]
& \ds a^d_p(\bu_p,\bv_p) + b_p(\bv_p,p_p) + b_{\bn_p}(\bv_p,\lambda) = 0, \label{eq:continuous-alternative-weak-formulation-1e} \\[1ex]
& \ds s_0\,(\partial_t\, p_p,w_p)_{\Omega_p}  - \alpha_p\,b_p(\bu_s,w_p)  - b_p(\bu_p,w_p) = (q_p,w_p)_{\Omega_p},\label{eq:continuous-alternative-weak-formulation-1f} \\[1ex]
& \ds c_{\Gamma}(\bu_s,\bvarphi;\xi) - b_{\bn_p}(\bu_p,\xi) = 0, \label{eq:continuous-alternative-weak-formulation-1g} \\[1ex]
& \ds - b_{\bn_f}(\bsi_f,\bpsi) + \,\, c_{\BJS}(\bu_s,\bvarphi;\0,\bpsi) + l_{\bvarphi}(\bvarphi,\bpsi) - c_{\Gamma}(\0,\bpsi;\lambda) = 0,\label{eq:continuous-alternative-weak-formulation-1h}
\end{align}
\end{subequations}
for all $(\btau_f, \bv_p, \btau_e, \bv_f, w_p, \bchi_f, \bv_s, \bpsi, \xi)\in \bbX_f\times \bX_p\times \bSigma_e\times \bV_f\times \W_p\times \bbQ_f\times \bV_p\times \bLambda_f\times \Lambda_p$.
We note that initial data for $\bsi_e$ and the remaining variables will be constructed in Lemma \ref{lem:sol0-in-M-operator}.

There are many different ways of ordering the variables in \eqref{eq:continuous-alternative-weak-formulation-1}.
For the sake of the subsequent analysis, we proceed as in \cite{aeny2019}, and adopt one leading to an evolution problem in a mixed form.
In particular, we group the spaces, unknowns, and test functions as follows:
\begin{equation*}
\begin{array}{c}
\ds \bQ := \bbX_f\times \bX_p\times \bSigma_e,\quad
\bS := \bV_f\times \W_p\times \bbQ_f\times \bV_p\times \bLambda_f\times \Lambda_p, \\ [1.5ex]
\ds \ubsi := (\bsi_f, \bu_p, \bsi_e)\in \bQ,\quad 
\ubu := (\bu_f, p_p, \bgamma_f, \bu_s, \bvarphi, \lambda)\in \bS, \\[1ex]
\ds \ubtau := (\btau_f, \bv_p, \btau_e)\in \bQ,\quad 
\ubv := (\bv_f, w_p, \bchi_f, \bv_s, \bpsi, \xi)\in \bS,
\end{array}
\end{equation*}
where the spaces $\bQ$ and $\bS$ are endowed with the norms, respectively,
\begin{equation}\label{norms}
\begin{split}
\|\ubtau\|^2_{\bQ} & = \|\btau_f\|^2_{\bbX_f} + \|\bv_p\|^2_{\bX_p} + \|\btau_e\|^2_{\bSigma_e}\,, \\[1ex]
\|\ubv\|^2_{\bS} & = \|\bv_f\|^2_{\bV_f} + \|w_p\|^2_{\W_p} + \|\bchi_f\|^2_{\bbQ_f} + \|\bv_s\|^2_{\bV_p} + \|\bpsi\|^2_{\bLambda_f} + \|\xi\|^2_{\Lambda_p}\,.
\end{split}
\end{equation}
Summing the equations in \eqref{eq:continuous-alternative-weak-formulation-1} corresponding to the test functions in $\bQ$ and $\bS$ results in 
\begin{align}
& a_f(\bsi_f,\btau_f) + \kappa_{\bu_f}(\bu_f, \btau_f) + a^d_p(\bu_p,\bv_p) + a^s_p(\partial_t\,\bsi_e,\btau_e) \nonumber \\[0.5ex] 
&\quad +\, b_f(\btau_f,\bu_f) + b_p(\bv_p,p_p) + b_\sk(\btau_f,\bgamma_f)
+ b_s(\btau_e,\bu_s) + b_{\bn_f}(\btau_f,\bvarphi) + b_{\bn_p}(\bv_p,\lambda)  \,=\, 0 \,,  \nonumber \\[0.5ex] 
& \rho_f\,(\partial_t\,\bu_f,\bv_f)_{\Omega_f} + s_0\,(\partial_t\,p_p,w_p)_{\Omega_p}
+ \rho_p\,(\partial_t\bu_s,\bv_s)_{\Omega_p}  \nonumber \\[0.5ex] 
&\quad +\, c_{\BJS}(\bu_s,\bvarphi;\bv_s,\bpsi) + c_{\Gamma}(\bu_s,\bvarphi;\xi) - c_{\Gamma}(\bv_s,\bpsi;\lambda) 
+ \alpha_p\,b_p(\bv_s,p_p) - \alpha_p\,b_p(\bu_s,w_p)  \nonumber \\[0.5ex] 
&\quad -\, b_f(\bsi_f,\bv_f) - b_p(\bu_p,w_p) - b_\sk(\bsi_f,\bchi_f)  
- b_s(\bsi_e,\bv_s) - b_{\bn_f}(\bsi_f,\bpsi) - b_{\bn_p}(\bu_p,\xi) \nonumber \\[0.5ex] 
&\quad +\, l_{\bvarphi}(\bvarphi,\bpsi) =\, (\f_f,\bv_f)_{\Omega_f} +(q_p,w_p)_{\Omega_p} + (\f_p,\bv_s)_{\Omega_p}\,. \label{eq:NS-Biot-formulation-2}
\end{align}
This can be expressed in operator notation as a degenerate evolution problem in mixed form:
\begin{align}
& \frac{\partial}{\partial\,t}\,\cE_1(\ubsi(t)) + \cA(\ubsi(t)) + \cB'(\ubu(t)) + \cK_{\bu_f(t)}(\ubu(t))  =  \bF(t) \qin \bQ', \nonumber\\ 
& \frac{\partial}{\partial\,t}\,\cE_2(\ubu(t)) - \cB(\ubsi(t)) + \cC(\ubu(t))+\cL_{\bvarphi(t)}(\ubu(t))  =  \bG(t) \qin \bS',\label{eq:alternative-formulation-operator-form}
\end{align}
where, given $\bu_f\in \bV_f$ and $\bvarphi \in \bLambda_f$, the operators
$\cE_1 : \bQ\to \bQ'$, $\cE_2 : \bS\to \bS'$, 
$\cA\,:\,\bQ\to \bQ'$, $\cB\,:\, \bQ\to \bS'$, $\cC\,:\,\bS\to \bS'$,  $\cK_{\bu_f}\,:\,\bS\to \bQ'$, and $\cL_{\bvarphi}\,:\,\bS\to \bS'$ are defined as follows:
\begin{subequations}\label{operators-1}
\begin{align}
& \cE_1(\ubsi)(\ubtau) \,:=\, a^s_p(\bsi_e,\btau_e)\,,\quad 
\cE_2(\ubu)(\ubv) \,:=\, \rho_f\,(\bu_f,\bv_f)_{\Omega_f} + s_0\,(p_p,w_p)_{\Omega_p}
+ \rho_p\,(\bu_s,\bv_s)_{\Omega_p}, \label{defn-E1-E2}  \\[1ex]  
&\ds \cA(\ubsi)(\ubtau) := a_f(\bsi_f,\btau_f) + a^d_p(\bu_p,\bv_p)\,, \label{defn-A} \\[1ex]
&\ds \cB(\ubtau)(\ubv) \,:=\, b_f(\btau_f,\bv_f) + b_p(\bv_p,w_p) + b_\sk(\btau_f,\bchi_f) + b_s(\btau_e,\bv_s) + b_{\bn_f}(\btau_f,\bpsi) + b_{\bn_p}(\bv_p,\xi)\,, \label{defn-B} \\[1ex]
&\ds \cC(\ubu)(\ubv) \,:=\, c_\BJS(\bu_s,\bvarphi;\bv_s,\bpsi) + c_{\Gamma}(\bu_s,\bvarphi;\xi) - c_{\Gamma}(\bv_s,\bpsi;\lambda) + \alpha_p\,b_p(p_p,\bv_s) - \alpha_p\,b_p(w_p,\bu_s)\,, \label{defn-C} \\[1ex]
&\ds \cK_{\bu_f}(\ubu)(\ubtau) := \kappa_{\bu_f}(\bu_f,\btau_f)\,,\quad
\cL_{\bvarphi}(\ubu)(\ubv) := l_{\bvarphi}(\bvarphi,\bpsi),\label{defn-K-L}
\end{align}
\end{subequations}
whereas the functionals $\bF\in \bQ'$ and $\bG\in \bS'$ are defined respectively as follows:
\begin{equation}\label{operators-3}
\bF(\ubtau) := 0 \qan
\bG(\ubv) := (\f_f,\bv_f)_{\Omega_f} + (q_p,w_p)_{\Omega_p} + (\f_p,\bv_s) \,.
\end{equation}


\subsection{Operator properties}

We next discuss boundness, continuity, coercivity, monotonicity, and inf-sup stability properties of the operators defined in \eqref{operators-1} (cf. \eqref{eq:bilinear-forms}--\eqref{eq:forms-on-interface-1}). We recall \cite[Definitions 9.39, 9.40 and 9.41]{Renardy-Rogers} that for a Banach space $\V$, an operator $\cA:\V \to \V'$ is bounded if it maps bounded sets in $\V$ to bounded sets in $\V'$. It is continuous if for every $u \in \V$,
$$
\|\cA(u) - \cA(v)\|_{\V'} \to 0 \mbox{ whenever } \|u - v\|_{\V} \to 0.
$$
It is coercive if
$$
\frac{\cA(v)(v)}{\|v\|_{\V}} \to \infty \mbox{ as } \|v\|_{\V} \to \infty,
$$
and it is monotone if
$$
(\cA(u) - \cA(v))(u - v) \ge 0 \ \ \forall\, u,v \in \V.
$$
If the operator is linear, then boundness and continuity are equivalent to $|\cA(u)(v)| \le C \|u\|_{\V}\|v\|_{\V} \ \forall\, u,v \in \V$ 
and monotonicity is equivalent to $\cA(v)(v) \ge 0 \ \forall\, v \in \V$.

We first note that for given $\bw_f$, the term $\kappa_{\bw_f}(\bu_f,\btau_f)$ requires $\bu_f$ to live in a space smaller than $\bL^2(\Omega_f)$. In particular, using Cauchy--Schwarz inequality twice gives
\begin{equation}\label{kappa-cont}
|\kappa_{\bw_f}(\bu_f,\btau_f)| 
\,\le\, \frac{\rho_f\,n^{1/2}}{2\,\mu} \|\bw_f\|_{\bL^4(\Omega_f)}\|\bu_f\|_{\bL^4(\Omega_f)}\|\btau^\rd_f\|_{\bbL^2(\Omega_f)}.
\end{equation}
Accordingly, we look for both $\bu_f, \bw_f \in \bL^4(\Omega_f)$. Furthermore, applying H\"older's inequality, we obtain
\begin{equation}\label{bf-cont}
|b_f(\btau_f,\bv_f)| \leq \|\bdiv(\btau_f)\|_{\bL^{4/3}(\Omega_f)} \|\bv_f\|_{\bL^4(\Omega_f)}.
\end{equation}
For the interface nonlinear term $l_{\bzeta}(\bvarphi,\bpsi)$, applying H\"older's inequality and using the continuous injection $\bi_\Gamma$ of $\bH^{1/2}(\partial \Omega_f)$ into $\bL^3(\partial \Omega_f)$, we have
\begin{equation}\label{eq:injection-H1/2-into-L3-interface}
\big|l_{\bzeta}(\bvarphi,\bpsi)\big|\leq\ \|\bi_{\Gamma}\|^3 \|\bzeta\|_{\bH^{1/2}(\Gamma_{fp})} \|\bvarphi\|_{\bH^{1/2}(\Gamma_{fp})} \|\bpsi\|_{\bH^{1/2}(\Gamma_{fp})}.
\end{equation}
In addition, using the continuous trace operator
$\gamma_0: \bH^1(\Omega_p) \to \bL^2(\Gamma_{fp})$, there hold
\begin{align}
  &  |\pil\bpsi\cdot\bt_{f,j},\bv_s\cdot\bt_{f,j}\pir_{\Gamma_{fp}}| \leq \|\gamma_0\|\|\bv_s\|_{\bH^1(\Omega_p)} \|\bpsi\|_{\bL^2(\Gamma_{fp})}, \label{bjs-cont} \\
  & |\pil\bv_s\cdot\bn_p,\xi\pir_{\Gamma_{fp}}| \leq \|\gamma_0\|\|\bv_s\|_{\bH^1(\Omega_p)} \|\xi\|_{\L^2(\Gamma_{fp})}. \label{cgamma-cont-1} 
\end{align}
Now, we establish continuity properties for  $\cE_1, \cE_2, \cA, \cB, \cC, \cK_{\bw_f}, \cL_{\bzeta}$, and $\bG(t)$. 
\begin{lem}\label{lem:cont}
The linear operators $\cE_1, \cE_2, \cA, \cB$ and $\cC$ are continuous. In particular, there exist positive constants $C_{\cE_1}$, $C_{\cE_2}$, $C_{\cA}$, $C_{\cB}$, and $C_f$ such that
\begin{gather}
 \big|\cE_1(\ubsi)(\ubtau)\big| \,\leq\, C_{\cE_1}\,\|\ubsi\|_{\bQ} \|\ubtau\|_{\bQ},\quad
 \big|\cE_2(\ubu)(\ubv)\big| \,\leq\, C_{\cE_2}\,\|\ubu\|_{\bS}\,\|\ubv\|_{\bS},
 \label{eq:continuity-cE1-cE2}\\[1ex]
    \big|\cA(\ubsi)(\ubtau)\big| \,\leq\, C_{\cA}\,\|\ubsi\|_{\bQ}\|\ubtau\|_{\bQ},\quad
\big|\cB(\ubtau)(\ubv)\big| \,\leq\, C_{\cB}\,\|\ubtau\|_{\bQ}\,\|\ubv\|_{\bS},
\quad \big|\cC(\ubu)(\ubv)\big| 
\,\leq\, C_f\,\|\ubu\|_{\bS}\,\|\ubv\|_{\bS}. \label{eq:continuity-cA-cB-cC}
  \end{gather}
Furthermore, given $\bw_f \in \bL^4(\Omega_f)$ and $\bzeta \in \bH^{1/2}(\Gamma_{fp})$, the operators $\cK_{\bw_f}$ and $\cL_{\bzeta}$ are continuous:
\begin{gather}
\big|\cK_{\bw_f}(\ubu)(\ubtau)\big|
\,\leq\, \frac{\rho_f\,n^{1/2}}{2\,\mu}\,\|\bw_f\|_{\bV_f}\|\ubu\|_{\bS}\,\|\ubtau\|_{\bQ},
\label{eq:continuity-cK-wf}\\[1ex]
\big|\cL_{\bzeta}(\ubu)(\ubv)\big|\leq C_{\cL}\|\bzeta\|_{\bLambda_f}\|\ubu\|_{\bS}\|\ubv\|_{\bS},
 \label{eq:continuity-cL-zeta}   
\end{gather}
and the linear functional $\bG(t) \in \bS'$ is continuous:
\begin{equation}\label{cont-F-G}  
  \big|\bG(t)(\ubv)\big| \leq C_{\bG}(t) \|\ubv\|_{\bS}.
\end{equation}
\end{lem}

\begin{proof}
We recall that the operators and functionals are defined in \eqref{operators-1} and \eqref{operators-3} with the associated bilinear forms defined in \eqref{eq:bilinear-forms} and \eqref{eq:forms-on-interface-1}. The continuity of $\cE_1$ follows from \eqref{eq:A-bounds} with $C_{\cE_1}:= a_{\max}$. 
For the continuity of $\cE_2$, applying the Cauchy--Schwarz and H\"older's inequalities, we obtain
\begin{equation*}
\big|\cE_2(\ubu)(\ubv)\big| 
\,\leq\, \rho_f |\Omega_f|^{1/2}\|\bu_f\|_{\bV_f}\|\bv_f\|_{\bV_f} 
+ s_0\|p_p\|_{\W_p}\|w_p\|_{\W_p} 
+ \rho_p \|\bu_s\|_{\bV_p}\|\bv_s\|_{\bV_p}
\,\leq\, C_{\cE_2}\,\|\ubu\|_\bS\,\|\ubv\|_\bS \,,
\end{equation*}
with $C_{\cE_2}:=\max\{\rho_f|\Omega_f|^{1/2},s_0,\rho_p\}$. The continuity of $\cA$ follows from \eqref{eq:permeability-bounds} with $C_{\cA}:=\max\{ \frac{1}{2\mu}, \mu\,k_{\max}\}$.
The continuity of $\cB$ follows from \eqref{eq:trace-inequality-1}, \eqref{eq:trace-inequality-2}, and \eqref{bf-cont}. The continuity of $\cC$ follows from \eqref{bjs-cont} and \eqref{cgamma-cont-1} with $C_f$ depending on $\mu$, $k_{\min}^{-1}$, $\alpha_{\BJS}$, and $\alpha_p$. Finally, the bounds for $\cK_{\bw_f}$ and $\cL_{\bzeta}$,  \eqref{eq:continuity-cK-wf} and \eqref{eq:continuity-cL-zeta}, follow from \eqref{kappa-cont} and \eqref{eq:injection-H1/2-into-L3-interface}, respectively, with $C_{\cL}:= \|\bi_{\Gamma}\|^3$, whereas the continuity of $\bG$ \eqref{cont-F-G} follows easily from its definition with $C_{\bG}(t) := \big(|\Omega_f|^{1/2}\|\f_f(t)\|^2_{\bL^2(\Omega_f)} + \|q_p(t)\|_{\L^2(\Omega_p)}^2 + \|\f_p(t)\|^2_{\bL^2(\Omega_p)}\big)^{1/2}$.
\end{proof}

Next, we establish the non-negativity of the forms $a_f$ and $a^d_p$, the positive semidefiniteness of $c_\BJS$, and the monotonicity of the operators $\cA, \cE_1, \cE_2$ and $\cC$.

\begin{lem}\label{lem:coercivity-properties-A-E2}
There hold
\begin{equation}\label{eq:coercivity-af}
a_f(\btau_f, \btau_f) 
\,\geq\, \frac{1}{2\,\mu}\,\|\btau^\rd_f\|^2_{\bbL^2(\Omega_f)} \quad \forall\,\btau_f\in \bbX_f \,,
\end{equation}
\begin{equation}\label{eq:coercivity-adp}
a^d_p(\bv_p, \bv_p) 
\,\geq\, \mu\,k_{\min}\,\|\bv_p\|^2_{\bL^2(\Omega_p)} \quad \forall\, \bv_p\in \bX_p \,,
\end{equation}
and there exists a constant $c_I > 0$, such that
\begin{equation}\label{eq:positivity-aBJS}
c_\BJS(\bv_s,\bpsi;\bv_s,\bpsi) 
\,\geq c_I\,\sum^{n-1}_{j=1} \|( \bpsi-\bv_s)\cdot\bt_{f,j}\|^2_{\L^2(\Gamma_{fp})} \quad \forall\, (\bv_s, \bpsi)\in \bV_p\times \bLambda_f \,.
\end{equation}
In addition, the operators $\cA, \cE_1, \cE_2$ and $\cC$ are monotone. 
\end{lem}
\begin{proof}
We begin by noting that the non-negativity properties \eqref{eq:coercivity-af} and \eqref{eq:coercivity-adp} follows easily from \eqref{eq:permeability-bounds}, and the definition of $a_f$ and $a^d_p$ (cf. \eqref{bilinear-form-1}, \eqref{bilinear-form-2}), respectively.
In turn, from the definition of $c_\BJS$ (cf. \eqref{interface-bilinear-form-1}), \eqref{eq:positivity-aBJS} follows with a positive constant $c_I:= \mu\,\alpha_{\BJS}/\sqrt{ k_{\max}}$.
On the other hand, combining \eqref{eq:coercivity-af} and \eqref{eq:coercivity-adp} we deduce that 
\begin{equation}\label{eq: operator A-monotone}
\cA(\ubtau)(\ubtau) 
\,\geq\, \frac{1}{2\mu}\,\|\btau^\rd_f\|^2_{\bbL^2(\Omega_f)} + \, \mu\,k_{\min}\,\|\bv_p\|^2_{\bL^2(\Omega_p)} \,,
\end{equation}
which implies the monotonicity of $\cA$. Finally, from the definition of the operators $\cE_1, \cE_2$ and $\cC$ (cf. \eqref{defn-E1-E2}, \eqref{defn-C}) we have that for all $\ubtau\in \bQ$ and for all $\ubv\in \bS$ there exist positive constants $C_A, C_{fps}>0$, such that
\begin{equation}\label{eq: operator E_1-monotone}
\cE_1(\ubtau)(\ubtau) 
\,\geq\, C_A\,\|\btau_e\|^2_{\bSigma_e} \,,
\end{equation}
\begin{equation}\label{eq: operator E_2-monotone}
\cE_2(\ubv)(\ubv) \,\geq\, C_{fps}\,\Big( \|\bv_f\|^2_{\bL^2(\Omega_f)} + \|w_p\|^2_{\W_p}+\|\bv_s\|^2_{\bL^2(\Omega_p)}  \Big)\,,
\end{equation}
with $C_A:= a_{\min}$ and $C_{fps}:=\min\{\rho_f,s_0,\rho_p\}$, and
\begin{equation}\label{eq: operator C-monotone}
\cC(\ubv)(\ubv) 
\,\geq c_I\,\sum^{n-1}_{j=1} \|( \bpsi-\bv_s)\cdot\bt_{f,j}\|^2_{\L^2(\Gamma_{fp})} \,,
\end{equation}
which implies the monotonicity of $\cE_1, \cE_2$ and $\cC$, completing the proof.
\end{proof}

We end this section by establishing the inf-sup conditions associated to the forms $b_\star$, with $\star\in \{f, p, s, \sk, \bn_f, \bn_p \}$.
\begin{lem}\label{lem:inf-sup-conditions}
There exist constants $\beta_1, \beta_2, \beta_3 > 0$ such that
\begin{equation}\label{eq:inf-sup-vs}
\sup_{\0\neq \btau_e\in \bSigma_e} \frac{b_s(\btau_e,\bv_s)}{\|\btau_e\|_{\bSigma_e}} 
\,\geq\, \beta_1\,\|\bv_s\|_{\bV_p} \quad \forall\,\bv_s\in \bV_p \,,
\end{equation}
\begin{equation}\label{eq:inf-sup-qp-xi}
\sup_{\0\neq \bv_p\in \bX_p} \frac{b_p(\bv_p,w_p) + b_{\bn_p}(\bv_p,\xi)}{\|\bv_p\|_{\bX_p}} 
\,\geq\, \beta_2\,\|(w_p,\xi)\|_{\W_p\times \Lambda_p} \quad \forall\,(w_p,\xi)\in \W_p\times \Lambda_p\,,
\end{equation}
and
\begin{equation}\label{eq:inf-sup-vf-chif}
\sup_{\0\neq \btau_f\in \bbX_f} \frac{B_f(\btau_f,(\bv_f,\bchi_f,\bpsi))}{\|\btau_f\|_{\bbX_f}} 
\,\geq\, \beta_3\,\|(\bv_f, \bchi_f, \bpsi)\|_{\bV_f\times \bbQ_f\times \bLambda_f} \quad \forall\,(\bv_f,\bchi_f,\bpsi)\in \bV_f\times \bbQ_f\times \bLambda_f\,,
\end{equation}
where 
$B_f(\btau_f,(\bv_f,\bchi_f,\bpsi)) \,:=\, b_f(\btau_f,\bv_f) + b_\sk(\btau_f,\bchi_f) + b_{\bn_f}(\btau_f,\bpsi)$.
\end{lem}
\begin{proof}
For the proof of \eqref{eq:inf-sup-vs} we refer the reader to \cite[Lemma ~4.1, eq. (4.1)]{aeny2019}, whereas combining a slight adaptation of \cite[Lemma 3.6, eq. (3.5)]{gos2011} and \cite[Lemma 3.8, eq. (3.10)]{gos2011} we deduce \eqref{eq:inf-sup-qp-xi}. Next, the proof \eqref{eq:inf-sup-vf-chif} is based on a modification of the proof of \cite[Lemma 3.5]{gobs2021}. In particular, by a suitable change of the boundary conditions in the auxiliary problems \cite[(3.54) and (3.58)]{gobs2021} to construct $\btau_f \in \bbX_f$ such that $\btau_f\bn_f = 0$ on $\Gamma_{fp}$, we can show that (cf. \cite[(3.57) and (3.61)]{gobs2021})
\begin{equation}\label{bf-bsk}
  \sup_{\0\neq \btau_f\in \bbX_f,\btau_f\bn_f = 0 \text{ on } \Gamma_{fp} }
  \frac{b_f(\btau_f,\bv_f) + b_{\sk}(\btau_f,\bchi_f)}
{\|\btau_f\|_{\bbX_f}} \geq C(\|\bv_f\|_{\bL^2(\Omega_f)} + \|\bchi_f\|_{\bbL^2(\Omega_f)}),
\end{equation}
where the constraint $\btau_{f}\bn_f = 0$ on $\Gamma_f^N \cup \Gamma_{fp}$ can be handled as in \cite[Lemma~4.3]{msmfe-simpl}.

In order to control $\|\bpsi\|_{\bLambda_f}$, we proceed as in \cite[Lemma 3.8, eq. (3.9)]{gos2011}.
In particular, given $\bphi_f\in \bH^{-1/2}(\Gamma_{fp})$, we define $\btau_1 := \be(\bz_1)$ in $\Omega_f$, where $\bz_1 \in \bH^1_{\Gamma^{\rD}_f}(\Omega_f)$ is the unique weak solution of the boundary value problem
\begin{equation}\label{eq:auxiliary-problem-inf-sup}
-\bdiv(\be(\bz_1)) = \0 \qin \Omega_f,\quad 
\bz_1 = \0 \qon \Gamma^{\rD}_f,\quad 
\be(\bz_1)\bn_f = \left\{\begin{array}{l}
\0 \qon \Gamma^{\rN}_f\,, \\ 
\bphi_f \qon \Gamma_{fp}\,,
\end{array} \right.
\end{equation}
which gives $b_{\bn_f}(\btau_1,\bpsi) = -\, \pil\btau_1\bn_f,\bpsi\pir_{\Gamma_{fp}} = -\, \pil\bphi_f,\bpsi\pir_{\Gamma_{fp}}$ for each $\bpsi\in \bH^{1/2}(\Gamma_{fp})$. 
It is clear that $\btau_1\in \bbL^2(\Omega_f)$, $\bdiv(\btau_1) = \0\in \bL^{4/3}(\Omega_f)$, and $\btau_1\bn_f = \0$ on $\Gamma^{\rN}_f$, so then $\btau_1\in \bbX_f$. Let $\bbX_f^0:= \big\{\btau_f\in \bbX_f: b_f(\btau_f,\bv_f) + b_{\sk}(\btau_f,\bchi_f) = 0 \ \ \forall (\bv_f,\bchi_f) \in \bV_f\times \bbQ_f\big\}$.
Using the fact that $\be(\bz_1):\bchi_f = 0$ in $\Omega_f$, we have that $\btau_1 \in \bbX_f^0$. In addition, an energy bound for \eqref{eq:auxiliary-problem-inf-sup} implies that $\|\btau_1\|_{\bbX_f} \leq C\,\|\bphi_f\|_{\bH^{-1/2}(\Gamma_{fp})}$, and hence we deduce that
\begin{equation*}
  \sup_{\0 \neq \btau_f\in \bbX_f^0} \frac{b_{\bn_f}(\btau_f,\bpsi)}
      {\|\btau_f\|_{\bbX_f}}
  \ge \sup_{\0 \neq \bphi_f\in \bH^{-1/2}(\Gamma_{fp})} \frac{\left|\pil\bphi_f,\bpsi\pir_{\Gamma_{fp}}\right|}{\|\btau_f\|_{\bbX_f}}    
\,\geq\, \frac{1}{C}\, \sup_{\0 \neq \bphi_f\in \bH^{-1/2}(\Gamma_{fp})} \frac{\left|\pil\bphi_f,\bpsi\pir_{\Gamma_{fp}}\right|}{\|\bphi_f\|_{\bH^{-1/2}(\Gamma_{fp})}} = \frac{1}{C}\,\|\bpsi\|_{\bLambda_f},
\end{equation*}
which, combined with \eqref{bf-bsk}, implies 
\eqref{eq:inf-sup-vf-chif}.
\end{proof}


\section{Well-posedness of the model}\label{sec:well-posedness-model}

In this section we establish the solvability of \eqref{eq:continuous-weak-formulation-1}. To that end, we start with the analysis of the alternative formulation \eqref{eq:continuous-alternative-weak-formulation-1} (cf. \eqref{eq:alternative-formulation-operator-form}).

\subsection{Preliminary results}

A key result that we use to establish the existence of a solution to \eqref{eq:alternative-formulation-operator-form} is the following theorem, which is a slight modification of \cite[Theorem~IV.6.1(b)]{Showalter}. In what follows, $Rg(\cA) $ denotes the range of $\cA$.
\begin{thm}\label{thm:auxiliary-theorem}
Let the linear, symmetric and monotone operator $\cN$ be given from the real vector space $E$ to its algebraic dual $E^*$, and let $E'_b$ be the Hilbert space which is the dual of $E$ with the seminorm
\begin{equation*}
|x|_b = \big(\cN\,(x)(x)\big)^{1/2} \quad x\in E.
\end{equation*}
Let $\cM:E\to E'_b$ be an operator with domain $\cD = \Big\{ x\in E \,:\, (\cN + \cM)(x) \in E'_b \Big\}$.
Assume that $\cM$ is monotone and $Rg(\cN + \cM) = E'_b$.
Then, for each $\cF\in \W^{1,1}(0,T;E'_b)$ and for each $u_0\in \cD$, there is a solution $u$ of
\begin{equation}\label{eq:parabolic-problem-Showalter}
\frac{d}{dt}\big(\cN(u(t))\big) + \cM\big(u(t)\big) = \cF(t) \quad a.e. \, \ 0 < t < T\,,
\end{equation}
with
\begin{equation*}
\cN(u)\in \W^{1,\infty}(0,T;E'_b),\quad u(t)\in \cD,\quad \mbox{ for all }\, 0\leq t\leq T,\qan \cN(u(0)) = \cN(u_0)\,.
\end{equation*}
\end{thm}

Recalling the definition of the operators $\cE_1, \cE_2, \cA, \cB, \cC, \cK_{\bw_f}$ and $\cL_{\bzeta}$ (cf. \eqref{operators-1}), problem \eqref{eq:alternative-formulation-operator-form} can be written in the form of \eqref{eq:parabolic-problem-Showalter} with
\begin{equation}\label{eq:defn-E-N-M}
E:= \bQ\times \bS, \quad u:= 
\left(\begin{array}{c}
\ubsi \\ \ubu
\end{array}\right),\quad 
\cN:= \begin{pmatrix}
\cE_1 & \0 \\
\0 & \cE_2
\end{pmatrix},\quad 
\cM:= \begin{pmatrix}
\cA & \cB' + \cK_{\bu_f} \\
-\cB  & \cC + \cL_{\bvarphi}
\end{pmatrix}.
\end{equation}
In particular, we observe from \eqref{defn-E1-E2} that
\begin{equation*}
\cN(u)(v) := (A(\bsi_e),\btau_{e})_{\Omega_p} + \rho_f\,(\bu_{f},\bv_{f})_{\Omega_f} + s_0\,(p_{p},w_{p})_{\Omega_p} + \rho_p\,(\bu_{s},\bv_{s})_{\Omega_p}\,,
\end{equation*}
for all $v=(\ubtau,\ubv)\in \bQ\times \bS$.
This gives the semi-norm:
\begin{equation*}
\|v\|_{\cN} := \big( \cN(v)(v) \big)^{1/2} 
= \big( (A(\btau_e),\btau_{e})_{\Omega_p} + \rho_f\,\|\bv_{f}\|^2_{\bL^2(\Omega_f)} + s_0\,\|w_{p}\|^2_{\L^2(\Omega_p)} + \rho_p\,\|\bv_{s}\|^2_{\bL^2(\Omega_p)} \big)^{1/2},
\end{equation*}
which is equivalent to $\|v\|_{\cN,2} := \big( \|\btau_{e}\|^2_{\bbL^2(\Omega_p)} + \|\bv_{f}\|^2_{\bL^2(\Omega_f)} + \|w_{p}\|^2_{\L^2(\Omega_p)} + \|\bv_{s}\|^2_{\bL^2(\Omega_p)} \big)^{1/2}$. Therefore we define the space $E'_b:= \bQ_2'\times \bS_2'$ where,
\begin{equation}\label{eq:defn-E'_b-D}
\bQ_2':=\{\0\}\times \{\0\} \times\bbL^2(\Omega_p) \quad \text{and}\quad \bS_2':=\bL^2(\Omega_f) \times \L^2(\Omega_p) \times  \{\0\}\times \bL^2(\Omega_p) \times \{\0\} \times \{0\}\,.
\end{equation}
In order to establish the range condition $Rg(\cN + \cM) = E'_b$,
we consider the following resolvent system: find $(\ubsi,\ubu)\in \bQ\times \bS$, such that
\begin{align}
(\cE_1 + \cA)(\ubsi) + (\cB' + \cK_{\bu_f})(\ubu) &= \wh{\bF} \qin \bQ', \nonumber\\ 
-\,\cB(\ubsi) + (\cE_2 + \cC +\cL_{\bvarphi})(\ubu) &= \wh{\bG} \qin \bS',\label{eq:T-auxiliary-problem-operator-A}
\end{align}
for all $(\ubtau,\ubv)\in \bQ\times \bS$ where $\wh{\bF} \in \bQ_2'$ and $\wh{\bG} \in \bS_2'$ are such that
\begin{align*}
\wh{\bF}(\ubtau) &\,:=\, (\wh{\f}_e,\btau_e)_{\Omega_p} \quad \forall\,\ubtau \in \bQ\,, \nonumber\\ 
\wh{\bG}(\ubv) &\,:=\, (\wh{\f}_f,\bv_f)_{\Omega_f} + (\wh{q}_p,w_p)_{\Omega_p} + (\wh{\f}_p,\bv_s)_{\Omega_p} \quad \forall\,\ubv \in \bS\,. \nonumber
\end{align*}
Using the definition of operators $\cE_1, \cE_2, \cA, \cB, \cC, \cK_{\bu_f}$ and $\cL_{\bvarphi}$, \eqref{eq:T-auxiliary-problem-operator-A} can be rewritten as: given functionals $\wh{\f}_e \in \bL^2(\Omega_p), \wh{\f}_f \in \bL^2(\Omega_f), \wh{q}_p \in \L^2(\Omega_p), \wh{\f}_p \in \bL^2(\Omega_p)$, find $(\ubsi,\ubu)\in \bQ\times \bS$, such that for all $(\ubtau,\ubv)\in \bQ\times \bS$,
\begin{align}
&  a_f(\bsi_f,\btau_f) + \kappa_{\bu_f}(\bu_f, \btau_f) + a^d_p(\bu_p,\bv_p) + a^s_p(\bsi_e,\btau_e) +\,\, b_f(\btau_f,\bu_f)+ b_p(\bv_p,p_p) + b_\sk(\btau_f,\bgamma_f)  \nonumber \\[0.5ex] 
&\quad +\, b_s(\btau_e,\bu_s) + b_{\bn_f}(\btau_f,\bvarphi) + b_{\bn_p}(\bv_p,\lambda)  \,=\, (\wh{\f}_e,\btau_e)_{\Omega_p} \,,  \nonumber \\[0.5ex]  
& \rho_f\,(\bu_f,\bv_f)_{\Omega_f} + s_0\,(p_p,w_p)_{\Omega_p}
+ \rho_p\,(\bu_s,\bv_s)_{\Omega_p}+\,\, c_{\BJS}(\bu_s,\bvarphi;\bv_s,\bpsi) + c_{\Gamma}(\bu_s,\bvarphi;\xi)- c_{\Gamma}(\bv_s,\bpsi;\lambda)    \nonumber \\[0.5ex]  
&\quad +\, \alpha_p\,b_p(\bv_s,p_p) - \alpha_p\,b_p(\bu_s,w_p) -\,\, b_f(\bsi_f,\bv_f) - b_p(\bu_p,w_p) - b_\sk(\bsi_f,\bchi_f) - b_s(\bsi_e,\bv_s)  \nonumber \\[0.5ex]  
&\quad -\, b_{\bn_f}(\bsi_f,\bpsi) - b_{\bn_p}(\bu_p,\xi)+l_{\bvarphi}(\bvarphi,\bpsi) =\, (\wh{\f}_f,\bv_f)_{\Omega_f} + (\wh{q}_p,w_p)_{\Omega_p} + (\wh{\f}_p,\bv_s)_{\Omega_p} \,.
\label{eq:T-auxiliary-problem-operator-B}
\end{align}

Using Theorem~\ref{thm:auxiliary-theorem}, we can show that the problem \eqref{eq:alternative-formulation-operator-form} (cf. \eqref{eq:NS-Biot-formulation-2}) has a solution.
We proceed in the following manner.

\medskip

\noindent{\bf Step 1.} Introduce a fixed-point operator $\bT$ associated to problem \eqref{eq:T-auxiliary-problem-operator-A}.

\noindent{\bf Step 2.} Prove that $\bT$ is a contraction mapping and  \eqref{eq:T-auxiliary-problem-operator-A} has unique solution.

\noindent{\bf Step 3.} Show that the alternative formulation \eqref{eq:alternative-formulation-operator-form} is well posed.

\medskip

\noindent Each step will be covered in detail in a corresponding subsection.


\subsubsection{Step 1: A fixed-point approach}\label{sec:fixed-point-approach}
We begin the proof of well-posedness of \eqref{eq:T-auxiliary-problem-operator-A} by defining the operator $\bT : \bV_f \times \bLambda_f\to \bV_f\times \bLambda_f$ by
\begin{equation}\label{eq:definition-operator-T}
\bT(\bw_f,\bzeta) \,:=\, (\bu_f,\bvarphi) \quad \forall\,(\bw_f,\bzeta)\in \bV_f\times \bLambda_f,
\end{equation}
where $\ubu := (\bu_f, p_p, \bgamma_f, \bu_s, \bvarphi, \lambda)\in \bQ$ is the second component of the unique solution (to be confirmed below) of the problem:
Find $(\ubsi,\ubu)\in \bQ\times \bS$, such that
\begin{align}
(\cE_1 + \cA)(\ubsi) + (\cB' + \cK_{\bw_f})(\ubu) &= \wh{\bF} \qin \bQ'\,, \nonumber\\ 
-\,\cB(\ubsi) + (\cE_2 + \cC +\cL_{\bzeta})(\ubu) &= \wh{\bG} \qin \bS'\,. \label{eq:T-auxiliary-problem-operator}
\end{align}
Hence, it is not difficult to see that $(\ubsi,\ubu)\in \bQ\times \bS$ is a solution of \eqref{eq:T-auxiliary-problem-operator-A} if and only if $(\bu_f,\bvarphi)\in \bV_f\times \bLambda_f$ is a fixed-point of $\bT$, that is
\begin{equation*}
\bT(\bu_f,\bvarphi) \,=\, (\bu_f,\bvarphi).
\end{equation*}
In what follows we focus on proving that $\bT$ possesses a unique fixed-point.
However, we remark in advance that the definition of $\bT$ will make sense only in a closed ball of $\bV_f\times \bLambda_f$.

Before continuing with the solvability analysis of \eqref{eq:T-auxiliary-problem-operator}, we provide the hypotheses under which $\bT$ is well defined.
To that end, we introduce operators that will be used to regularize the problem \eqref{eq:T-auxiliary-problem-operator}.
Let $R_f\,:\,\bbX_f\to \bbX'_f$ and $R_p\,:\,\bX_p\to \bX'_p$ be defined by
\begin{subequations}\label{eq:Rf-Rp-definition}
\begin{align}
& \ds R_f(\bsi_f)(\btau_f) \,:=\, r_f(\bsi_f, \btau_f) 
\,=\, (|\bdiv(\bsi_f)|^{-2/3}\,\bdiv(\bsi_f),\bdiv(\btau_f))_{\Omega_f}, \label{eq: R_f defn} \\[1ex]
& \ds R_p(\bu_p)(\bv_p) \,:=\, r_p(\bu_p, \bv_p)
\,=\, (\div(\bu_p),\div(\bv_p))_{\Omega_p}. \label{eq:R_p defn} 
\end{align}
\end{subequations}

\begin{lem}\label{lem:Rf-Rp-properties}
The operators $R_f$ and $R_p$ are bounded, continuous, and monotone. They 
satisfy the following bounds:
\begin{subequations}\label{eq:R_f R_p bounded-coerc-cont-monotone}
\begin{align}
& \ds R_f(\bsi_f)(\btau_f) \,\leq\, \|\bdiv(\bsi_f)\|^{1/3}_{\bL^{4/3}(\Omega_f)} \|\bdiv(\btau_f)\|_{\bL^{4/3}(\Omega_f)}\,, 
\quad  R_f(\btau_f)(\btau_f) \,\geq\, \|\bdiv(\btau_f)\|^{4/3}_{\bL^{4/3}(\Omega_f)}\,, \label{eq: R_f bounded-coerc} \\[1ex] 
& \ds \|\cR_f(\btau_1) - \cR_f(\btau_2)\|_{\bbX'_f} 
\,\leq\, C\,\|\btau_1 - \btau_2\|^{1/3}_{\bbX_f}\,, \label{eq:R_f cont} \\[1ex]
& \ds (R_f(\btau_1) - R_f(\btau_2))(\btau_1 - \btau_2)
\,\geq\, C\,\frac{\|\bdiv(\btau_1) - \bdiv(\btau_2)\|^2_{\bL^{4/3}(\Omega_f)}}{\|\btau_1\|^{2/3}_{\bbX_f} + \|\btau_2\|^{2/3}_{\bbX_f}}\,, \label{eq:R_f monotone} \\[1ex]
& \ds R_p(\bu_p)(\bv_p) \,\leq\, \|\div(\bu_p)\|_{\bL^2(\Omega_p)} \|\div(\bv_p)\|_{\bL^2(\Omega_p)}\,, 
\quad  R_p(\bv_p)(\bv_p) \,\geq\, \|\div(\bv_p)\|^2_{\bL^2(\Omega_p)}\,, \label{eq: R_p bounded-coerc} \\[1ex]
& \ds \|R_p(\bu_1) - R_p(\bu_2)\|_{\bX'_p} 
\,\leq\, C\,\|\bu_1 - \bu_2\|_{\bX_p}\,, \label{eq:R_p cont} \\[1ex]
& \ds (R_p(\bu_1) - R_p(\bu_2))(\bu_1 - \bu_2)
\,\geq\, \|\bdiv(\bu_1) - \bdiv(\bu_2)\|^2_{\bL^{2}(\Omega_p)}\,. \label{eq:R_p monotone}
\end{align}
\end{subequations}
\end{lem}
\begin{proof}
The proof of \eqref{eq: R_f bounded-coerc}--\eqref{eq:R_f monotone} follow from similar arguments to \cite[Lemma 3.4]{cy2021} for $\rq=4/3$ and $\rp=4$. 
Inequalities \eqref{eq: R_p bounded-coerc}--\eqref{eq:R_p monotone} follow immediately from the definition \eqref{eq:R_p defn}.
\end{proof}

\medskip
According to the above, we define the operator $\cR : \bQ \to \bQ'$ as
\begin{equation*}
\cR(\ubsi)(\ubtau) \,:=\, R_f(\bsi_f)(\btau_f) + R_p(\bu_p)(\bv_p)\,,
\end{equation*}
and notice from \eqref{eq: R_f bounded-coerc} and \eqref{eq: R_p bounded-coerc} that
\begin{subequations}\label{eq: R-continuous-coercive}
\begin{align}
& \ds \cR(\ubsi)(\ubtau) 
\,\leq\, (\,\|\bdiv(\bsi_{f})\|^{1/3}_{\bL^{4/3}(\Omega_f)}+\, \|\div(\bu_{p})\|_{\bL^2(\Omega_p)})\|\ubtau\|_{\bQ}\,, \label{eq: R-continuous} \\[1ex]
& \ds \cR(\ubtau)(\ubtau) 
\,\geq\, \|\bdiv(\btau_{f})\|^{4/3}_{\bL^{4/3}(\Omega_f)}+\, \|\div(\bv_{p})\|^2_{\bL^2(\Omega_p)}\,. \label{eq: R-coercive}
\end{align}
\end{subequations}
On the other hand, given $q \in (1,2]$, we recall that there exist constants $c_1(\Omega_f), c_2(\Omega_f) > 0$, such that
\begin{equation}\label{eq:tau-d-H0div-inequality}
c_1(\Omega_f)\,\|\btau_{f,0}\|^q_{\bbL^2(\Omega_f)} \,\leq\, \|\btau^\rd_f\|^q_{\bbL^2(\Omega_f)} + \|\bdiv(\btau_f)\|^q_{\bL^{4/3}(\Omega_f)} \quad \forall\,\btau_f = \btau_{f,0} + c\,\bI\in \bbH(\bdiv_{4/3};\Omega_f)
\end{equation}
\begin{equation}\label{eq:tau-H0div-Xf-inequality}
\mbox{and}\quad c_2(\Omega_f)\,\|\btau_f\|^q_{\bbX_f} \,\leq\, \|\btau_{f,0}\|^q_{\bbX_f} \quad \forall\,\btau_f = \btau_{f,0} + c\,\bI\in \bbX_f,
\end{equation}
where $\btau_{f,0}\in \bbH_0(\bdiv_{4/3};\Omega_f) := \Big\{ \btau_f\in \bbH(\bdiv_{4/3};\Omega_f) :\quad (\tr(\btau_f),1)_{\Omega_f} = 0 \Big\}$ and $c\in \R$.
For the proof of \eqref{eq:tau-d-H0div-inequality} we refer to \cite[Lemma 3.2]{cgo2021}, whereas \eqref{eq:tau-H0div-Xf-inequality} follows from a slight adaptation of \cite[Lemma 2.5]{Gatica} and the fact that $\btau_f\bn_f\in \bH^{-1/2}(\partial\Omega_f)$ for all $\btau_f\in \bbH(\bdiv_{4/3};\Omega_f)$ (cf. \cite[eq. (2.5)]{cmo2018}). We emphasize that the proof of \eqref{eq:tau-H0div-Xf-inequality} utilizes the boundary condition $\btau_f\bn_f = \0$ on $\Gamma^\rN_f$, with $|\Gamma^\rN_f|>0$ for $\btau_f \in \bbX_f.$

Thus, for $\epsilon > 0$, we consider a regularization of \eqref{eq:T-auxiliary-problem-operator}  defined by: Find $(\ubsi_\epsilon,\ubu_\epsilon)\in \bQ\times \bS$ satisfying
\begin{align}
(\cE_1 + \cA + \epsilon\,\cR)(\ubsi_\epsilon) + (\cB' + \cK_{\bw_f})(\ubu_\epsilon) &= \wh{\bF} \qin \bQ'\,, \nonumber\\ 
-\,\cB(\ubsi_\epsilon) + (\cE_2 + \cC+\cL_{\bzeta})(\ubu_\epsilon) &= \wh{\bG} \qin \bS'\,.\label{eq:regularization-T-auxiliary-problem-operator}
\end{align}
We will prove that \eqref{eq:regularization-T-auxiliary-problem-operator} is well-posed.
To than end, we first state the following preliminary lemma.
\begin{lem}\label{lem:auxiliary-problem-first-fow}
Given $\bw_f\in \bV_f$ and $\ubv\in \bS$, for each $\epsilon>0$, there exists a unique $\ubsi_\epsilon(\ubv)\in \bQ$, such that
\begin{equation}\label{eq:auxiliary-problem-first-row}
(\cE_1 + \cA + \epsilon\,\cR)(\ubsi_\epsilon(\ubv))(\ubtau) \,=\, \big(\wh{\bF} - (\cB' + \cK_{\bw_f})(\ubv)\big)(\ubtau) \quad \forall \, \ubtau \in \bQ.
\end{equation}
Moreover, there exists a constant $\hat{C} > 0$ depending on $\|\wh{\bF}\|_{\bQ'}$, $\|\bw_f\|_{\bV_f}$, and $\|\ubv\|_\bS$, such that
\begin{equation}\label{eq:auxiliary-bound-1}
  \|\ubsi_\epsilon(\ubv)\|_\bQ \,\leq\, \widehat{C}.
\end{equation}
Let $\beta:= \min\left\{ \beta_1, \beta_2, \beta_3\right\}$, where $\beta_i$ are the inf-sup constants in \eqref{eq:inf-sup-vs}--\eqref{eq:inf-sup-vf-chif}. In addition, let
\begin{equation}\label{eq:r_1^0 defn}
r_1^0:= \frac{\mu\,\beta}{3\,\rho_f\,n^{1/2}}.
\end{equation}
Then for each $\bw_f\in \bV_f$ satisfying $\|\bw_f\|_{\bV_f} \leq r_1^0$, there holds
\begin{equation}\label{eq:auxiliary-bound-2}
  \|\ubv\|_\bS \,\leq\, C\,\Big( \|\wh{\bF}\|_{\bQ'} + \|\ubsi_\epsilon(\ubv)\|_\bQ + \,\|\bdiv(\bsi_{f,\epsilon}(\ubv))\|^{1/3}_{\bL^{4/3}(\Omega_f)}
  \Big),
\end{equation}
and, given $\ubv_1, \ubv_2\in \bS$ for which $\ubsi_\epsilon(\ubv_1)$ and $\ubsi_\epsilon(\ubv_2)$ satisfy \eqref{eq:auxiliary-problem-first-row}, there hold
\begin{align}
  &\|\ubv_1 - \ubv_2\|_\bS \,\leq\, C\,\Big( \|\ubsi_\epsilon(\ubv_1) - \ubsi_\epsilon(\ubv_2)\|_\bQ + \,\|\bdiv(\bsi_{f,\epsilon}(\ubv_1) - \bsi_{f,\epsilon}(\ubv_2))\|^{1/3}_{\bL^{4/3}(\Omega_f)}
  \Big),\label{eq:auxiliary-bound-3}
\end{align}
\begin{equation}\label{eq:auxiliary-bound-4}
\mbox{and}\quad \|\ubsi_\epsilon(\ubv_1) - \ubsi_\epsilon(\ubv_2)\|_\bQ 
\,\leq\, C_{\ubsi_{\epsilon}}\,(1 + \|\bw_f\|_{\bV_f})\,\|\ubv_1 - \ubv_2\|_\bS\,,
\end{equation}
with $C_{\ubsi_{\epsilon}}$ depending on $\|\bsi_{f,\epsilon}(\ubv_1)\|_{\bbX_f}$ and $\|\bsi_{f,\epsilon}(\ubv_2)\|_{\bbX_f}$.
\end{lem}
\begin{proof}
We begin by noting that given $\bw_f\in \bV_f$ and $\ubv\in \bS$ the functional $\wh{\bF} - (\cB' + \cK_{\bw_f})(\ubv)$ is continuous (cf. \eqref{eq:continuity-cA-cB-cC}, \eqref{eq:continuity-cK-wf}). In addition, combining \eqref{eq:continuity-cE1-cE2}, \eqref{eq:continuity-cA-cB-cC}, and Lemmas \ref{lem:coercivity-properties-A-E2} and \ref{lem:Rf-Rp-properties}, we deduce that the operator $\cE_1 + \cA + \epsilon\,\cR$ is bounded, continuous, and monotone. 
Moreover, given $\ubtau = (\btau_f, \bv_p, \btau_e)\in \bQ$, employing the non-negativity bounds of $\cA, \cE_1$ and $ \cR$ (cf. \eqref{eq: operator A-monotone}, \eqref{eq: operator E_1-monotone}, \eqref{eq: R-coercive}), it follows that
\begin{equation}\label{eq:E1-A-epsR coercive}
(\cE_1 + \cA + \epsilon\,\cR)(\ubtau)(\ubtau) \geq \, C_1\,\Big( \|\btau^\rd_f\|^2_{\bbL^2(\Omega_f)} + \|\bdiv(\btau_f)\|^{4/3}_{\bL^{4/3}(\Omega_f)} \Big)
\,+\, C_2\,\Big(\|\bv_p\|^2_{\bX_p} + \|\btau_e\|^2_{\bSigma_e}\Big) \,,
\end{equation}
with $C_1:= \min\big\{ 1/(2\,\mu), \epsilon \big\}$ and $C_2:=\min\big\{\mu\,k_{\min},\epsilon, C_A \big\}$. In addition, inequalities \eqref{eq:tau-d-H0div-inequality}--\eqref{eq:tau-H0div-Xf-inequality} imply that
\begin{equation}\label{Q-norm}
  \|\ubtau\|_{\bQ} \le C_3\left(\|\btau^\rd_f\|_{\bbL^2(\Omega_f)} + \|\bdiv(\btau_f)\|_{\bL^{4/3}(\Omega_f)} + \|\bv_p\|_{\bX_p} + \|\btau_e\|_{\bSigma_e}\right),
\end{equation}
with $C_3$ depending on $c_1(\Omega_f), c_2(\Omega_f)$. Then
$\frac{(\cE_1 + \cA + \epsilon\,\cR)(\ubtau)(\ubtau)}{\|\ubtau\|_{\bQ}} \to \infty$ as $\|\ubtau\|_{\bQ} \to \infty$, since the numerator dominates the denominator, 
i.e., the operator $\cE_1 + \cA + \epsilon\,\cR$ is coercive.
Then, as a direct application of the Browder--Minty theorem \cite[Theorem 10.49]{Renardy-Rogers}, we obtain that problem \eqref{eq:auxiliary-problem-first-row} has a unique solution.

We proceed with deriving \eqref{eq:auxiliary-bound-1}.
Given $\ubv \in \bS$ for which $\ubsi_\epsilon(\ubv) := (\bsi_{f,\epsilon}(\ubv),\bu_{p,\epsilon}(\ubv),\bsi_{e,\epsilon}(\ubv))$ satisfies \eqref{eq:auxiliary-problem-first-row}, test with $\ubtau = (\0,\bu_{p,\epsilon}(\ubv), \bsi_{e,\epsilon}(\ubv))$. Using non-negativity bounds of $a^d_p, \cE_1$ and $\cR_p$ (cf. \eqref{eq:coercivity-adp}, \eqref{eq: operator E_1-monotone}, \eqref{eq: R_p bounded-coerc}), the continuity of $\cB$ (cf. \eqref{eq:continuity-cA-cB-cC}), and the Cauchy--Schwarz inequality, we obtain
\begin{align*}
&C_2\,\big(\|\bu_{p,\epsilon}(\ubv)\|^2_{\bX_p}
+ \|\bsi_{e,\epsilon}(\ubv)\|^2_{\bSigma_e} \big) 
\leq\, \big(\|\wh{\bF}\|_{\bQ'} + C_{\cB}\,\|\ubv\|_{\bS}\big)
\|(\0,\bu_{p,\epsilon}(\ubv),\bsi_{e,\epsilon}(\ubv))\|_{\bQ},
\end{align*} 
with $C_2$ as in \eqref{eq:E1-A-epsR coercive}, which implies
\begin{equation}\label{eq:bound-up-bsie}
\|(\0,\bu_{p,\epsilon}(\ubv),\bsi_{e,\epsilon}(\ubv))\|_{\bQ} \le \, C \,
(\|\wh{\bF}\|_{\bQ'}  + \|\ubv\|_{\bS}).
\end{equation}
In turn, testing \eqref{eq:auxiliary-problem-first-row} now with $\ubtau = (\bsi_{f,\epsilon}(\ubv), \0,\0)$, employing the non-negativity bounds of $a_f$ (cf. \eqref{eq:coercivity-af}) and $ \cR_f$ (cf. \eqref{eq: R_f bounded-coerc}), and the continuity of $\cB, \cK_{\bw_f}$ (cf. \eqref{eq:continuity-cA-cB-cC}, \eqref{eq:continuity-cK-wf}), we deduce that
\begin{align}\label{eq:bound-bsif}
  C_1\,\big(\|\bsi^\rd_{f,\epsilon}(\ubv)\|^2_{\bbL^2(\Omega_f)} + \|\bdiv(\bsi_{f,\epsilon}(\ubv))\|^{4/3}_{\bL^{4/3}(\Omega_f)}\big)
  \le \bigg(C_{\cB} + \frac{\rho_f\,n^{1/2}}{2\,\mu}\,\|\bw_f\|_{\bV_f}\bigg)
  \|\ubv\|_\bS\,\|\bsi_{f,\epsilon}(\ubv)\|_{\bbX_f},
\end{align}
with $C_1$ as in \eqref{eq:E1-A-epsR coercive}. If $\|\bsi^\rd_{f,\epsilon}(\ubv)\|_{\bbL^2(\Omega_f)} \ge 1$, then
$\|\bsi^\rd_{f,\epsilon}(\ubv)\|^2_{\bbL^2(\Omega_f)} \ge \|\bsi^\rd_{f,\epsilon}(\ubv)\|^{4/3}_{\bbL^2(\Omega_f)}$, hence \eqref{eq:bound-bsif} implies
\begin{equation}\label{eq:bound-bsif-1}
\|\bsi_{f,\epsilon}(\ubv)\|_{\bbX_f}^{1/3} \le C(1+ \|\bw_f\|_{\bV_f})\|\ubv\|_\bS.
\end{equation}
If $\|\bsi^\rd_{f,\epsilon}(\ubv)\|_{\bbL^2(\Omega_f)} \le 1$, then \eqref{eq:bound-bsif} implies
$$
\|\bdiv(\bsi_{f,\epsilon}(\ubv))\|^{4/3}_{\bL^{4/3}(\Omega_f)} \le C
(1+ \|\bw_f\|_{\bV_f})\|\ubv\|_\bS
(1+\|\bdiv(\bsi_{f,\epsilon}(\ubv))\|_{\bL^{4/3}(\Omega_f)}),
$$
If $\|\bdiv(\bsi_{f,\epsilon}(\ubv))\|_{\bL^{4/3}(\Omega_f)} \ge 1$, the above inequality implies
$$
\|\bdiv(\bsi_{f,\epsilon}(\ubv))\|^{1/3}_{\bL^{4/3}(\Omega_f)} \le C(1+ \|\bw_f\|_{\bV_f})\|\ubv\|_\bS\,,
$$
which holds trivially with $C=1$ if $\|\bdiv(\bsi_{f,\epsilon}(\ubv))\|_{\bL^{4/3}(\Omega_f)} \le 1$. The above inequality, combined with the assumption $\|\bsi^\rd_{f,\epsilon}(\ubv)\|_{\bbL^2(\Omega_f)} \le 1$, implies that \eqref{eq:bound-bsif-1} also holds in this case. A combination of \eqref{eq:bound-up-bsie} and \eqref{eq:bound-bsif-1} implies \eqref{eq:auxiliary-bound-1}.

Next, we establish \eqref{eq:auxiliary-bound-2}. We have from the inf-sup condition of $\cB$ (cf. Lemma \ref{lem:inf-sup-conditions}) and the continuity of $\cK_{\bw_f}$ (cf. \eqref{eq:continuity-cK-wf}), that
\begin{equation*}
\sup_{\underline{\0} \neq\ubtau\in \bQ} \frac{(\cB' + \cK_{\bw_f})(\ubv)(\ubtau)}{\|\ubtau\|_\bQ} 
\,\geq\, \Big( \beta - \frac{\rho_f\,n^{1/2}}{2\,\mu}\,\|\bw_f\|_{\bV_f} \Big)\,\|\ubv\|_\bS,
\end{equation*}
where $\beta := \min\big\{\beta_1, \beta_2, \beta_3\big\}$.
Consequently, by requiring that $\|\bw_f\|_{\bV_f} \leq r_1^0$,
and using \eqref{eq:auxiliary-problem-first-row} and the continuity of $\wh{\bF}, \cE_1, \cA$ and $\cR$ (cf. \eqref{eq:continuity-cE1-cE2}, \eqref{eq:continuity-cA-cB-cC}, \eqref{eq: R-continuous}), we obtain
\begin{align}
& \frac{5\beta}{6}\,\|\ubv\|_\bS 
\,\leq\, \sup_{\underline{\0} \neq\ubtau\in \bQ} \frac{\wh{\bF}(\ubtau) - (\cE_1 + \cA + \epsilon\,\cR)(\ubsi_\epsilon(\ubv))(\ubtau)}{\|\ubtau\|_\bQ} \nonumber\\ 
&\quad \leq \, \|\wh{\bF}\|_{\bQ'} + (C_{\cE_1} + C_{\cA})\,\|\ubsi_\epsilon(\ubv)\|_\bQ + \epsilon\,\|\bdiv(\bsi_{f,\epsilon}(\ubv))\|^{1/3}_{\bL^{4/3}(\Omega_f)}+\epsilon\, \|\div(\bu_{p,\epsilon}(\ubv))\|_{\bL^2(\Omega_p)}\,.\label{eq:inf-sup-B-1}
\end{align}
Recognizing that the last term on the right-hand side above is included in the second term, we 
conclude \eqref{eq:auxiliary-bound-2}. Next, we derive \eqref{eq:auxiliary-bound-3}. Given $\ubv_1, \ubv_2\in \bS$ for which $\ubsi_\epsilon(\ubv_1)$ and $\ubsi_\epsilon(\ubv_2)$ satisfy \eqref{eq:auxiliary-problem-first-row}, we deduce that
for all $\ubtau \in \bQ$ there holds
\begin{equation}\label{eq:identity-E1-A-R-2}
(\cB' + \cK_{\bw_f})(\ubv_2 - \ubv_1)(\ubtau) \,=\, (\cE_1 + \cA)(\ubsi_\epsilon(\ubv_1) - \ubsi_\epsilon(\ubv_2))(\ubtau) + \epsilon\,\big(\cR(\ubsi_\epsilon(\ubv_1)) - \cR(\ubsi_\epsilon(\ubv_2))\big)(\ubtau)\,.
\end{equation}
Bound \eqref{eq:auxiliary-bound-3} then follows by proceeding analogously to \eqref{eq:auxiliary-bound-2}.

Finally, we derive \eqref{eq:auxiliary-bound-4}. Given $\ubv_i\in \bS$ for which $\ubsi_\epsilon(\ubv_i) := (\bsi_{f,\epsilon}(\ubv_i),\bu_{p,\epsilon}(\ubv_i),\bsi_{e,\epsilon}(\ubv_i))$, with $i\in \{1,2\}$, satisfy \eqref{eq:auxiliary-problem-first-row}, we subtract the problems and then test with $\ubtau = (\0,\bu_{p,\epsilon}(\ubv_1) - \bu_{p,\epsilon}(\ubv_2), \bsi_{e,\epsilon}(\ubv_1) - \bsi_{e,\epsilon}(\ubv_2))$, 
and use non-negativity bounds of $a^d_p, \cE_1$ and $\cR_p$ (cf. \eqref{eq:coercivity-adp}, \eqref{eq: operator E_1-monotone}, \eqref{eq: R_p bounded-coerc}), the continuity of $\cB$ (cf. \eqref{eq:continuity-cA-cB-cC}), and the Cauchy--Schwarz inequality, to obtain
\begin{align*}
&C_2\,\Big(\|\bu_{p,\epsilon}(\ubv_1) - \bu_{p,\epsilon}(\ubv_2)\|^2_{\bX_p}
+ \|\bsi_{e,\epsilon}(\ubv_1) - \bsi_{e,\epsilon}(\ubv_2)\|^2_{\bSigma_e} \Big) \nonumber \\
&\quad \leq\, C_{\cB}\,\|\big(\0,\bu_{p,\epsilon}(\ubv_1) - \bu_{p,\epsilon}(\ubv_2),\bsi_{e,\epsilon}(\ubv_1) - \bsi_{e,\epsilon}(\ubv_2)\big)\|_{\bQ}\|\ubv_1 - \ubv_2\|_{\bS} \,, 
\end{align*} 
with $C_2$ as in \eqref{eq:E1-A-epsR coercive}, which implies
\begin{equation}\label{eq:bound-monotonicity-up-bsie}
\|\big(\0,\bu_{p,\epsilon}(\ubv_1) - \bu_{p,\epsilon}(\ubv_2),\bsi_{e,\epsilon}(\ubv_1) - \bsi_{e,\epsilon}(\ubv_2)\big)\|_{\bQ} \,\leq\, C\,\|\ubv_1 - \ubv_2\|_{\bS} \,.
\end{equation}
In turn, testing \eqref{eq:auxiliary-problem-first-row} now with $\ubtau = (\bsi_{f,\epsilon}(\ubv_1) - \bsi_{f,\epsilon}(\ubv_2), \0,\0)$, employing the monotonicity of $a_f$ (cf. \eqref{eq:coercivity-af}) and $\cR_f$ (cf. \eqref{eq:R_f monotone}), and the continuity of $\cB, \cK_{\bw_f}$ (cf. \eqref{eq:continuity-cA-cB-cC}, \eqref{eq:continuity-cK-wf}), we deduce that
\begin{align}
&\frac{1}{2\,\mu}\,\|\bsi^\rd_{f,\epsilon}(\ubv_1) - \bsi^\rd_{f,\epsilon}(\ubv_2)\|^2_{\bbL^2(\Omega_f)} + \epsilon\,C\,\frac{\|\bdiv(\bsi_{f,\epsilon}(\ubv_1)) - \bdiv(\bsi_{f,\epsilon}(\ubv_2))\|^2_{\bL^{4/3}(\Omega_f)}}{\|\bsi_{f,\epsilon}(\ubv_1)\|^{2/3}_{\bbX_f} + \|\bsi_{f,\epsilon}(\ubv_2)\|^{2/3}_{\bbX_f}} \nonumber\\ 
&\quad \, \leq \, \left(C_{\cB} + \frac{\rho_f\,n^{1/2}}{2\,\mu}\,\|\bw_f\|_{\bV_f}\right)\|\ubv_1 - \ubv_2\|_\bS\,\|\bsi_{f,\epsilon}(\ubv_1) - \bsi_{f,\epsilon}(\ubv_2)\|_{\bbX_f} \,,
\label{eq:bound-continuity-sigf}
\end{align} 
and, using \eqref{eq:tau-d-H0div-inequality}--\eqref{eq:tau-H0div-Xf-inequality}, with exponent $q=2$, to bound the left-hand side of \eqref{eq:bound-continuity-sigf}, we derive
\begin{equation}\label{eq:bound-sigf-2}
\|\bsi_{f,\epsilon}(\ubv_1) - \bsi_{f,\epsilon}(\ubv_2)\|_{\bbX_f} 
\,\leq\, C\,(1 + \|\bw_f\|_{\bV_f})\,\|\ubv_1 - \ubv_2\|_\bS \,,
\end{equation}
with $C$ depending on $\epsilon$, $\|\bsi_{f,\epsilon}(\ubv_1)\|_{\bbX_f}$ and $\|\bsi_{f,\epsilon}(\ubv_2)\|_{\bbX_f}$.
Finally, combining \eqref{eq:bound-monotonicity-up-bsie} and \eqref{eq:bound-sigf-2}, we obtain \eqref{eq:auxiliary-bound-4} and conclude the proof.
\end{proof}

\begin{cor}
The solution to \eqref{eq:auxiliary-problem-first-row} defines a one-to-one operator.
\end{cor}

\begin{proof}
If we assume that $\bsi_\epsilon(\ubv_1) = \bsi_\epsilon(\ubv_2)$, \eqref{eq:auxiliary-bound-3} implies that $\ubv_1 = \ubv_2$.
Equivalently, this shows that given $\ubv_1, \ubv_2\in \bS$ with $\ubv_1 \neq \ubv_2$, then $\bsi_\epsilon(\ubv_1) \neq \bsi_\epsilon(\ubv_2)$.
\end{proof}

We next consider the problem associated with the second equation in \eqref{eq:regularization-T-auxiliary-problem-operator}:
Find $\ubu_\epsilon\in \bS$ such that
\begin{equation}\label{eq:operator-J}
\cJ_{\bw_f,\bzeta}(\ubu_\epsilon)(\ubv) 
\,:=\, -\cB(\ubsi_\epsilon(\ubu_\epsilon))(\ubv) + (\cE_2 + \cC + \cL_{\bzeta})(\ubu_\epsilon)(\ubv) 
\,=\, \wh{\bG}(\ubv) \quad \forall\,\ubv\in \bS,
\end{equation}
where $\ubsi_\epsilon(\ubu_\epsilon)\in \bQ$ is the solution of \eqref{eq:auxiliary-problem-first-row}. In the next lemma we establish the properties of $\cJ_{\bw_f,\bzeta}$.

\begin{lem}\label{lem:J-bijective}
Let $\|\bw_f\|_{\bV_f}\leq r_1^0$ (cf. \eqref{eq:r_1^0 defn}) and $\|\bzeta\|_{\bLambda_f} \leq r_2^0$,  where
\begin{equation}\label{eq:r20-def}
r_2^0 :=\frac{\mu \beta^2}{12\,C_{\cL}} \,,
\end{equation}
where $\beta$ is defined in Lemma~\ref{lem:auxiliary-problem-first-fow} and $C_{\cL}$ is the continuity constant for operator $\cL_{\bzeta}$. Then, the operator $\cJ_{\bw_f,\bzeta}$ is bounded, continuous, coercive and monotone.
\end{lem}
\begin{proof}
We first prove that $\cJ_{\bw_f,\bzeta}$ is bounded.
In fact, given $\ubv\in \bS$, from the definition of $\cJ_{\bw_f,\bzeta}$ (cf. \eqref{eq:operator-J}), the continuity of $\cE_2, \cB, \cC$ and $\cL_{\bzeta}$ (cf. \eqref{eq:continuity-cE1-cE2}, \eqref{eq:continuity-cA-cB-cC}, \eqref{eq:continuity-cL-zeta}), in combination with \eqref{eq:auxiliary-bound-1}, yields
\begin{equation*}
\|\cJ_{\bw_f,\bzeta}(\ubv)\|_{\bS'}
\, \leq \, \left( C_{\cB}\,\wh{C} + C_{\cE_2} + C_f+ C_{\cL}\|\bzeta\|_{\bLambda_f}\right)\,\|\ubv\|_\bS\,,
\end{equation*}
which implies that $\cJ_{\bw_f,\bzeta}$ is bounded.
Next, we provide the continuity of $\cJ_{\bw_f,\bzeta}$. Let $\ubv_i \in \bS$, $i = 1,2$. Using again the definition of $\cJ_{\bw_f,\bzeta}$, the continuity of $\cB, \cE_2$ and $\cC$ (cf. \eqref{eq:continuity-cA-cB-cC}, \eqref{eq:continuity-cE1-cE2}), in  combination with \eqref{eq:auxiliary-bound-4}, it follows that
\begin{equation*}
\|\cJ_{\bw_f,\bzeta}(\ubv_1) - \cJ_{\bw_f,\bzeta}(\ubv_2)\|_{\bS'}
\,\leq \,  \Big( C_{\ubsi_{\epsilon}}\,C_{\cB}\,(1 + \|\bw_f\|_{\bV_f}) + C_{\cE_2} + C_f+C_{\cL}\|\bzeta\|_{\bLambda_f} \Big)\,\|\ubv_1 - \ubv_2\|_\bS \,.
\end{equation*}
Due to \eqref{eq:auxiliary-bound-1}, for any fixed $\ubv_1 \in \bS$, the above inequality implies that $\|\cJ_{\bw_f,\bzeta}(\ubv_1) - \cJ_{\bw_f,\bzeta}(\ubv_2)\|_{\bS'} \to 0$ as $\|\ubv_1 - \ubv_2\|_\bS \to 0$,
which proves the continuity of $\cJ_{\bw_f,\bzeta}$.

Now, for the coercivity of $\cJ_{\bw_f,\bzeta}$, we first notice from \eqref{eq:auxiliary-bound-2} that
\begin{equation}\label{eq:bount-to-coercivity}
  \|\ubv\|_\bS \,\leq\, C\,\max\big\{\|\wh{\bF}\|_{\bQ'},1\big\}\,\Big( 1 + \|\ubsi_\epsilon(\ubv)\|_\bQ + \,\|\bdiv(\bsi_{f,\epsilon}(\ubv))\|^{1/3}_{\bL^{4/3}(\Omega_f)}
  \Big)\,,
\end{equation}
and then it is clear that $\|\ubsi_\epsilon(\ubv)\|_\bQ \to \infty$ as $\|\ubv\|_\bS \to \infty$.
Next, using again the definition of $\cJ_{\bw_f,\bzeta}$ (cf. \eqref{eq:operator-J}), the monotonicity of $\cE_2$ and  $\cC$ (cf. Lemma \ref{lem:coercivity-properties-A-E2}) and the identity \eqref{eq:auxiliary-problem-first-row} with $\ubtau = \ubsi_{\epsilon}(\ubv) := (\bsi_{f,\epsilon}(\ubv),\bu_{p,\epsilon}(\ubv),\bsi_{e,\epsilon}(\ubv))$, we find that
\begin{align*}
& \cJ_{\bw_f,\bzeta}(\ubv)(\ubv) \,=\, -\,\cB(\bsi_\epsilon(\ubv))(\ubv) + (\cE_2 + \cC+\cL_{\bzeta})(\ubv)(\ubv) \,\geq\, -\,\cB'(\ubv)(\bsi_\epsilon(\ubv))+\cL_{\bzeta}(\ubv)(\ubv) \nonumber\\ 
& \quad \,=\, (\cE_1 + \cA + \epsilon\,\cR)(\ubsi_\epsilon(\ubv))(\ubsi_\epsilon(\ubv)) + \cK_{\bw_f}(\ubv)(\ubsi_\epsilon(\ubv))+\cL_{\bzeta}(\ubv)(\ubv) - \wh{\bF}(\ubsi_\epsilon(\ubv))\,,
\end{align*}
which combined with the  non-negativity bounds of $\cE_1, \cA$ and $\cR$ (cf. \eqref{eq: operator E_1-monotone}, \eqref{eq: operator A-monotone}, \eqref{eq: R-coercive}), and the continuity of $\cK_{\bw_f},\cL_{\bzeta}$ (cf. \eqref{eq:continuity-cK-wf}, \eqref{eq:continuity-cL-zeta}) and $\wh{\bF}$, yields
\begin{align}
&\cJ_{\bw_f,\bzeta}(\ubv)(\ubv) 
\,\geq\, \frac{1}{2\,\mu} \|\bsi^\rd_{f,\epsilon}(\ubv)\|^2_{\bbL^2(\Omega_f)} 
+ \epsilon\,\|\bdiv(\bsi_{f,\epsilon}(\ubv))\|^{4/3}_{\bL^{4/3}(\Omega_f)}
+ C_2\,\Big( \|\bu_{p,\epsilon}(\ubv)\|^2_{\bX_p} 
+ \|\bsi_{e,\epsilon}(\ubv)\|^2_{\bSigma_e} \Big) \nonumber\\
& \qquad- \frac{\rho_f\,n^{1/2}}{2\,\mu}\,\|\bw_f\|_{\bV_f}\,\|\bv_f\|_{\bV_f}\,\|\bsi^\rd_{f,\epsilon}(\ubv)\|_{\bbL^2(\Omega_f)}
- C_{\cL}\|\bzeta\|_{\bLambda_f}\|\bpsi\|^2_{\bLambda_f}
- \|\wh{\bF}\|_{\bQ'}\,\|\ubsi_\epsilon(\ubv)\|_{\bQ} \,, \label{eq:J-coercivity1}
\end{align}
with $C_2$ as in \eqref{eq:E1-A-epsR coercive}.
To control the terms involving $\|\bw_f\|_{\bV_f}$ and $\|\bzeta\|_{\bLambda_f}$, and  similarly to \eqref{eq:inf-sup-B-1}, we use the inf-sup condition \eqref{eq:inf-sup-vf-chif} along with the equation in \eqref{eq:auxiliary-problem-first-row} with test function $\ubtau=(\btau_f,\0,\0)$, to obtain
\begin{align}
&\frac{5\beta}{6}\,\|(\bv_f, \bchi_f, \bpsi)\|_{\bV_f\times \bbQ_f\times \bLambda_f}  \,\leq\, \sup_{\0 \neq\btau_f\in \bbX_f}\frac{-a_f(\bsi_f(\ubv),\btau_f) - \epsilon\,R_f(\bsi_f(\ubv),\btau_f)}{\|\btau_f\|_{\bchi_f}}\nonumber\\
&\qquad \leq\, \frac{1}{2\,\mu} \|\bsi^\rd_f(\ubv)\|_{\bbL^2(\Omega_f)}
+ \epsilon\,\|\bdiv(\bsi_f(\ubv))\|^{1/3}_{\bL^{4/3}(\Omega_f)} \,, \label{eq:bount-to-coercivity1}
\end{align}
and deduce that both $\|\bv_f\|_{\bV_f}$ and $\|\bpsi\|_{\bLambda_f}$ are bounded by
\begin{equation*}
\frac{3}{5\mu\,\beta}\|\bsi^\rd_{f,\epsilon}(\ubv)\|_{\bbL^2(\Omega_f)} + \frac{6\,\epsilon}{5\beta}\,\|\bdiv(\bsi_{f,\epsilon}(\ubv))\|^{1/3}_{\bL^{4/3}(\Omega_f)}\,.
\end{equation*}
%
Then, using the facts that $\|\bw_f\|_{\bV_f} \leq r_1^0$ (cf. \eqref{eq:r_1^0 defn}) and $\|\bzeta\|_{\bLambda_f} \leq r_2^0$ (cf. \eqref{eq:r20-def})  and applying Young's inequality, we get 
\begin{equation}\label{eq:kappa-continuity bound}
\frac{\rho_f\,n^{1/2}}{2\,\mu}\,\|\bw_f\|_{\bV_f}\,\|\bv_f\|_{\bV_f}\,\|\bsi^\rd_{f,\epsilon}(\ubv)\|_{\bbL^2(\Omega_f)} \leq \frac{1}{10}\left(\frac{1}{\mu}+\epsilon\right)\,\|\bsi^\rd_{f,\epsilon}(\ubv)\|^2_{\bbL^2(\Omega_f)} + \frac{\epsilon}{10}\|\bdiv(\bsi_{f,\epsilon}(\ubv))\|^{2/3}_{\bL^{4/3}(\Omega_f)}\,,
\end{equation}
\begin{equation}\label{eq:operator L continuity bound}
\mbox{and}\quad C_{\cL}\|\bzeta\|_{\bLambda_f}\|\bpsi\|^2_{\bLambda_f} \leq \frac {3}{50\mu}\,\|\bsi^\rd_{f,\epsilon}(\ubv)\|^2_{\bbL^2(\Omega_f)}
+ \frac{6\mu\epsilon^2}{25}\|\bdiv(\bsi_{f,\epsilon}(\ubv))\|^{2/3}_{\bL^{4/3}(\Omega_f)} \,.
\end{equation}
Replacing back \eqref{eq:kappa-continuity bound} and \eqref{eq:operator L continuity bound} into \eqref{eq:J-coercivity1}, using the definition of $r_2^0$ (cf. \eqref{eq:r20-def}), choosing
\begin{equation}\label{eq:epsilon-bound}
0<\epsilon \leq \frac{5}{12\mu} \,,
\end{equation} 
and after some algebraic manipulations, we deduce that
\begin{align}
&\cJ_{\bw_f,\bzeta}(\ubv)(\ubv) 
\,\geq \frac{179}{600\mu}\|\bsi^\rd_{f,\epsilon}(\ubv)\|^2_{\bbL^2(\Omega_f)} + C_2\,\Big( \|\bu_{p,\epsilon}(\ubv)\|^2_{\bX_p} + \|\bsi_{e,\epsilon}(\ubv)\|^2_{\bSigma_e} \Big) \nonumber\\
&\qquad +\,\, \epsilon\,\|\bdiv(\bsi_{f,\epsilon}(\ubv))\|^{4/3}_{\bL^{4/3}(\Omega_f)} 
- \frac{\epsilon}{5}\|\bdiv(\bsi_{f,\epsilon}(\ubv))\|^{2/3}_{\bL^{4/3}(\Omega_f)} - \|\wh{\bF}\|_{\bQ'}\,\|\ubsi_\epsilon(\ubv)\|_{\bQ}.\label{eq:J-coercivity3a}
\end{align}
If $\|\bdiv(\bsi_{f,\epsilon}(\ubv))\|_{\bL^{4/3}(\Omega_f)} \geq 1$
then the terms multiplied by $\epsilon$ on right-hand side of \eqref{eq:J-coercivity3a} are replaced by  $\frac{4\,\epsilon}{5}\|\bdiv(\bsi_{f,\epsilon}(\ubv))\|^{4/3}_{\bL^{4/3}(\Omega_f)}$.
Using \eqref{eq:bount-to-coercivity}, we obtain
\begin{align}
  & \frac{\cJ_{\bw_f,\bzeta}(\ubv)(\ubv)}{\|\ubv\|_\bS}
 \geq\, C\,\left(
 \frac{\|\bsi^\rd_{f,\epsilon}(\ubv)\|_{\bbL^2(\Omega_f)}^2
 + \|\bdiv(\bsi_{f,\epsilon}(\ubv))\|_{\bL^{4/3}(\Omega_f)}^{4/3}
 + \|\bu_{p,\epsilon}(\ubv)\|_{\bX_p}^2
 + \|\bsi_{e,\epsilon}(\ubv)\|_{\bSigma_e}^2}
 {1 + \|\ubsi_\epsilon(\ubv)\|_\bQ + \|\bdiv(\bsi_{f,\epsilon}(\ubv))\|^{1/3}_{\bL^{4/3}(\Omega_f)}}
\,-\, 1 \right)\,. \label{eq:J-coercivity-case-1}
\end{align}
If $\|\bdiv(\bsi_{f,\epsilon}(\ubv))\|_{\bL^{4/3}(\Omega_f)} < 1$, then \eqref{eq:J-coercivity3a} implies
\begin{equation}\label{eq:J-coercivity-case-2}
\frac{\cJ_{\bw_f,\bzeta}(\ubv)(\ubv)}{\|\ubv\|_\bS} 
\,\geq\, C \left(\bA(\ubsi_\epsilon(\ubv))
- \frac{1}{1 + \|\ubsi_\epsilon(\ubv)\|_\bQ + \|\bdiv(\bsi_{f,\epsilon}(\ubv))\|^{1/3}_{\bL^{4/3}(\Omega_f)}}
  \right),
\end{equation}
where $\bA(\ubsi_\epsilon(\ubv))$ denotes the quantity on the right-hand side of \eqref{eq:J-coercivity-case-1}.
We now take $\|\ubv\|_\bS\to \infty$ in both \eqref{eq:J-coercivity-case-1} and \eqref{eq:J-coercivity-case-2}. Using that $\|\ubsi_\epsilon(\ubv)\|_\bQ \to \infty$ as $\|\ubv\|_\bS \to \infty$ (cf. \eqref{eq:bount-to-coercivity}) and \eqref{Q-norm}, we conclude that $\frac{\cJ_{\bw_f,\bzeta}(\ubv)(\ubv)}{\|\ubv\|_\bS} \to \infty$, since the numerator dominates the denominator, hence 
$\cJ_{\bw_f,\bzeta}$ is coercive.

Finally, employing the identity \eqref{eq:identity-E1-A-R-2} with $\ubtau = \ubsi_\epsilon(\ubv_1) - \ubsi_\epsilon(\ubv_2)\in \bQ$ and the monotonicity of $\cE_2 + \cC$, we deduce
\begin{equation*}
\begin{array}{l}
\ds (\cJ_{\bw_f,\bzeta}(\ubv_1) - \cJ_{\bw_f,\bzeta}(\ubv_2))(\ubv_1 - \ubv_2) 
\,\geq\, \cB'(\ubv_2 - \ubv_1)(\ubsi_\epsilon(\ubv_1) - \ubsi_\epsilon(\ubv_2))+\cL_{\bzeta}(\ubv_1 - \ubv_2) (\ubv_1 - \ubv_2)  \\ [2ex]
\ds\quad \,=\, (\cE_1 + \cA)(\ubsi_\epsilon(\ubv_1) - \ubsi_\epsilon(\ubv_2))(\ubsi_\epsilon(\ubv_1) - \ubsi_\epsilon(\ubv_2)) + \cK_{\bw_f}(\ubv_1 - \ubv_2)(\ubsi_\epsilon(\ubv_1) - \ubsi_\epsilon(\ubv_2)) \\ [2ex]
\ds\quad \,+\,\,\epsilon\,(\cR(\ubsi_\epsilon(\ubv_1)) - \cR(\ubsi_\epsilon(\ubv_2)))(\ubsi_\epsilon(\ubv_1) - \ubsi_\epsilon(\ubv_2))+\cL_{\bzeta}(\ubv_1 - \ubv_2) (\ubv_1 - \ubv_2) ,
\end{array}
\end{equation*}
which together with the monotonicity of $\cE_1, \cA, \cR$ (cf. \eqref{eq: operator A-monotone}, \eqref{eq: operator E_1-monotone}, \eqref{eq:R_f monotone}, \eqref{eq:R_p monotone}), the continuity of $\cK_{\bw_f}$ and $\cL_{\bzeta}$ (cf. \eqref{eq:continuity-cK-wf}, \eqref{eq:continuity-cL-zeta}) with $\|\bw_f\|_{\bV_f} \leq r_1^0$ and $\|\bzeta\|_{\bLambda_f} \leq r_2^0$ (cf. \eqref{eq:r_1^0 defn}, \eqref{eq:r20-def}), and similar arguments to the coercivity bound (cf. \eqref{eq:J-coercivity-case-1}, \eqref{eq:J-coercivity-case-2}), we conclude the monotonicity of $\cJ_{\bw_f,\bzeta}$ and complete the proof.
\end{proof}

\begin{cor}\label{regularized-well-posed}
For each $(\bw_f,\bzeta)\in \bV_f\times \bLambda_f$ such that $\|\bw_f\|_{\bV_f} \leq r_1^0$ and $\|\bzeta\|_{\bLambda_f} \leq r_2^0$ (cf. \eqref{eq:r_1^0 defn}, \eqref{eq:r20-def}), the regularized problem \eqref{eq:regularization-T-auxiliary-problem-operator} has a unique solution.
\end{cor}

\begin{proof}
  Problem \eqref{eq:regularization-T-auxiliary-problem-operator} is equivalent to \eqref{eq:operator-J}.   More precisely, given $\ubsi_\epsilon(\ubu_\epsilon)\in \bQ$, solution of \eqref{eq:auxiliary-problem-first-row}, with $\ubu_\epsilon\in \bS$, solution of \eqref{eq:operator-J}, the vector $(\ubsi_\epsilon,\ubu_\epsilon) = (\ubsi_\epsilon(\ubu_\epsilon),\ubu_\epsilon)\in \bQ\times \bS$ solves \eqref{eq:regularization-T-auxiliary-problem-operator}. The converse is straightforward. Existence and uniqueness of a solution to \eqref{eq:operator-J}
follows from Lemma~\ref{lem:J-bijective} in combination with the Browder--Minty theorem \cite[Theorem 10.49]{Renardy-Rogers}.
\end{proof}

We are now ready to prove the well-posedness of problem \eqref{eq:T-auxiliary-problem-operator}.

\begin{lem}\label{thm:well-posedness-1}
For each $\wh{\f}_f\in \bL^2(\Omega_f), \wh{\f}_p\in \bL^2(\Omega_p), \wh{\f}_e \in \bbL^2(\Omega_p)$, and $\wh{q}_p\in \L^2(\Omega_p)$, the problem \eqref{eq:T-auxiliary-problem-operator} has a unique solution $(\ubsi,\ubu)\in \bQ\times \bS$ for each $(\bw_f,\bzeta)\in \bV_f\times \bLambda_f$ such that $\|\bw_f\|_{\bV_f} \leq r_1^0$ and $\|\bzeta\|_{\bLambda_f} \leq r_2^0$ (cf. \eqref{eq:r_1^0 defn}, \eqref{eq:r20-def}).
Moreover, there exists a constant $c_\bT > 0$, independent of $\bw_f,\bzeta$ and the data $\wh{\f}_f, \wh{\f}_p, \wh{\f}_e $, and $\wh{q}_p$, such that
\begin{equation}\label{eq:ubsi-ubu-bound-solution}
\, \|(\ubsi, \ubu)\|_{\bQ\times \bS} 
\,\leq\, {c_\bT\,\Big\{ \|\wh{\f}_f\|_{\bL^2(\Omega_f)} +  \|\wh{\f}_e\|_{\bbL^2(\Omega_p)}+ \|\wh{\f}_p\|_{\bL^2(\Omega_p)} + \| \wh{q}_p\|_{\L^2(\Omega_p)}\Big\}}.
\end{equation}
\end{lem}
\begin{proof}
In Corollary~\ref{regularized-well-posed} we established the existence and uniqueness of a solution $(\ubsi_\epsilon, \ubu_\epsilon)\in \bQ\times \bS$ to the regularized problem \eqref{eq:regularization-T-auxiliary-problem-operator}, where $\ubsi_\epsilon = (\bsi_{f,\epsilon}, \bu_{p,\epsilon}, \bsi_{e,\epsilon})$ and $\ubu_\epsilon = (\bu_{f,\epsilon}, p_{p,\epsilon}, \bvarphi_{\epsilon}, \bu_{s,\epsilon}, \bgamma_{f,\epsilon}, \lambda_{\epsilon})$. In order to bound $\|\ubsi_\epsilon\|_\bQ$ and $\|\ubu_\epsilon\|_\bS$ independently of $\epsilon$, we proceed similarly to \cite[Lemma~4.6]{aeny2019}. To that end, we begin testing \eqref{eq:regularization-T-auxiliary-problem-operator} with $\ubtau = \ubsi_\epsilon$ and $\ubv = \ubu_\epsilon$, to obtain
\begin{equation*}
(\cE_1 + \cA + \epsilon\,\cR)(\ubsi_\epsilon)(\ubsi_\epsilon) + \cK_{\bw_f}(\ubu_\epsilon)(\ubsi_\epsilon) + (\cE_2 + \cC +\cL_{\bzeta})(\ubu_\epsilon)(\ubu_\epsilon) = \wh{\bF}(\ubsi_\epsilon) + \wh{\bG}(\ubu_\epsilon)\,,
\end{equation*}
which, together with the non-negativity and coercivity estimates of the operators $\cE_1, \cA, \cR, \cE_2, \cC$ (cf. \eqref{eq: operator E_1-monotone}, \eqref{eq: operator A-monotone}, \eqref{eq: R-coercive}, \eqref{eq: operator E_2-monotone}, \eqref{eq: operator C-monotone}), and the continuity of $\cK_{\bw_f}$ and $\cL_{\bzeta}$ (cf. \eqref{eq:continuity-cK-wf}, \eqref{eq:continuity-cL-zeta}), yields
\begin{align}
& \frac{1}{2\,\mu} \|\bsi^\rd_{f,\epsilon}\|^2_{\bbL^2(\Omega_f)} 
+ \mu\,k_{\min} \|\bu_{p,\epsilon}\|^2_{\bL^2(\Omega_p)} 
+ C_A \|\bsi_{e,\epsilon}\|^2_{\bSigma_e} 
+ \epsilon\,\|\bdiv(\bsi_{f,\epsilon})\|^{4/3}_{\bL^{4/3}(\Omega_f)}  
+ \epsilon\,\|\div(\bu_{p,\epsilon})\|^2_{\L^2(\Omega_p)}
\nonumber \\
&\quad +\, \rho_f\|\bu_{f,\epsilon}\|^2_{\bL^2(\Omega_f)} + s_0 \|p_{p,\epsilon}\|^2_{\W_p}
+ \rho_p\|\bu_{s,\epsilon}\|^2_{\bL^2(\Omega_p)} 
+ c_I\,\sum^{n-1}_{j=1} \|( \bvarphi_{\epsilon}-\bu_{s,\epsilon})\cdot\bt_{f,j}\|^2_{\L^2(\Gamma_{fp})} \nonumber \\
&\quad -\, \frac{\rho_f\,n^{1/2}}{2\,\mu}\,\|\bw_f\|_{\bV_f}\,\|\bu_{f,\epsilon}\|_{\bV_f}\,\|\bsi^\rd_{f,\epsilon}\|_{\bbL^2(\Omega_f)}
- C_{\cL}\|\bzeta\|_{\bLambda_f}\|\bvarphi_{\epsilon}\|_{\bLambda_f}^2 \nonumber \\
& \leq\,   
\|\wh{\f}_f\|_{\bL^2(\Omega_f)} \,\|\bu_{f,\epsilon}\|_{\bL^2(\Omega_f)} + \|\wh{\f}_p\|_{\bL^2(\Omega_p)} \|\bu_{s,\epsilon}\|_{\bL^2(\Omega_p)}\, + \,\,\|\wh{q}_p\|_{\L^2(\Omega_p)}\|p_{p,\epsilon}\|_{\W_p}+\|\wh{\f}_e\|_{\bbL^2(\Omega_p)}\,\|\bsi_{e,\epsilon}\|_{\bSigma_e}. \label{eq:bound-solution-data-1}
\end{align}
Then, choosing $\epsilon$ as in \eqref{eq:epsilon-bound} and using \eqref{eq:kappa-continuity bound} and \eqref{eq:operator L continuity bound},
we deduce that
\begin{align}
& \frac{179}{600\mu}\|\bsi^\rd_{f,\epsilon}\|^2_{\bbL^2(\Omega_f)} 
+ \mu\,k_{\min} \|\bu_{p,\epsilon}\|^2_{\bL^2(\Omega_p)}
+ C_A\,\|\bsi_{e,\epsilon}\|^2_{\bSigma_e}
+ \epsilon\,\|\bdiv(\bsi_{f,\epsilon})\|^{4/3}_{\bL^{4/3}(\Omega_f)} 
- \frac{\epsilon}{5}\,\|\bdiv(\bsi_{f,\epsilon})\|^{2/3}_{\bL^{4/3}(\Omega_f)}
\nonumber \\
&\quad +\, \epsilon\,\|\div(\bu_{p,\epsilon})\|^2_{\L^2(\Omega_p)}
+ \rho_f \|\bu_{f,\epsilon}\|^2_{\bL^2(\Omega_f)} 
+ s_0\|p_{p,\epsilon}\|^2_{\W_p}
+ \rho_p\|\bu_{s,\epsilon}\|^2_{\bL^2(\Omega_p)} 
\nonumber \\
& \leq  \|\wh{\f}_f\|_{\bL^2(\Omega_f)} \|\bu_{f,\epsilon}\|_{\bL^2(\Omega_f)} 
+ \|\wh{\f}_p\|_{\bL^2(\Omega_p)} \|\bu_{s,\epsilon}\|_{\bL^2(\Omega_p)}
+ \|\wh{q}_p\|_{\L^2(\Omega_p)}\|p_{p,\epsilon}\|_{\W_p} 
+ \|\wh{\f}_e\|_{\bbL^2(\Omega_p)}\,\|\bsi_{e,\epsilon}\|_{\bSigma_e} \,.
\label{eq:bound-solution-data-2}
\end{align}

On the other hand, from the first row of \eqref{eq:regularization-T-auxiliary-problem-operator}, employing inf-sup condition of $\cB$ (cf. Lemma ~\ref{lem:inf-sup-conditions}), continuity of $\cK_{\bw_f}$ (cf. \eqref{eq:continuity-cK-wf}) along with $\|\bw_f\|_{\bV_f}\leq r_1^0$ (cf. \eqref{eq:r_1^0 defn}) and continuity of $\cR$ (cf. \eqref{eq: R-continuous}), we deduce that
\begin{align}
&\|\ubu_\epsilon\|_\bS
\, \leq \, C\,\Big( \|\bsi^\rd_{f,\epsilon}\|_{\bbL^2(\Omega_f)} + \|\bu_{p,\epsilon}\|_{\bL^2(\Omega_p)} + \|\bsi_{e,\epsilon}\|_{\bSigma_e} + \|\wh{\f}_e\|_{\bbL^2(\Omega_p)}  \nonumber \\ 
&\qquad +\, \epsilon\,\|\bdiv(\bsi_{f,\epsilon})\|^{1/3}_{\bL^{4/3}(\Omega_f)}
+ \epsilon\,\|\div(\bu_{p,\epsilon})\|_{\bL^2(\Omega_p)} \Big)\,.\label{eq:ubu-bound}
\end{align}
In turn, from the second row of \eqref{eq:regularization-T-auxiliary-problem-operator} and recalling that $\bdiv(\bbX_f) = (\bV_f)'$ and $\div(\bX_p) = (\W_p)'$, it follows that
\begin{align}
& \|\bdiv(\bsi_{f,\epsilon})\|_{\bL^{4/3}(\Omega_f)} + \|\div(\bu_{p,\epsilon})\|_{\L^2(\Omega_p)} \nonumber \\ 
&\quad \leq\, C\Big(\|\wh{\f}_f\|_{\bL^2(\Omega_f)} + \|\wh{q}_p\|_{\L^2(\Omega_p)} + \|\bu_{f,\epsilon}\|_{\bL^2(\Omega_f)} + s_0\|p_{p,\epsilon}\|_{\W_p} + \|\bu_{s,\epsilon}\|_{\bV_p} \Big).\label{eq:bound-solution-data-2a}
\end{align}
Next, as in Lemma \ref{lem:J-bijective}, we consider the following two cases:

\medskip
\noindent {\bf Case 1:} If $\|\bdiv(\bsi_{f,\epsilon}(\ubv))\|_{\bL^{4/3}(\Omega_f)} \geq 1$
then the fourth and fifth terms on right-hand side of \eqref{eq:bound-solution-data-2} are replaced by  $\frac{4\,\epsilon}{5}\|\bdiv(\bsi_{f,\epsilon}(\ubv))\|^{4/3}_{\bL^{4/3}(\Omega_f)}$.
Thus, using similar arguments to the ones employed in Lemma \ref{lem:J-bijective}, Young's inequality with appropriate weights on right-hand side of \eqref{eq:bound-solution-data-2} in combination with \eqref{eq:ubu-bound} and \eqref{eq:bound-solution-data-2a}, and estimates \eqref{eq:tau-d-H0div-inequality}--\eqref{eq:tau-H0div-Xf-inequality}, we obtain
\begin{align}
& \|\bsi_{f,\epsilon}\|^2_{\bbX_f} 
+ \|\bu_{p,\epsilon}\|^2_{\bX_p}
+ \|\bsi_{e,\epsilon}\|^2_{\bSigma_e}
+ \|\ubu_\epsilon\|^2_\bS 
+ s_0\,\|p_{p,\epsilon}\|^2_{\W_p} 
+ \epsilon\,\|\bdiv(\bsi_{f,\epsilon})\|^{4/3}_{\bL^{4/3}(\Omega_f)}
+ \epsilon\,\|\div(\bu_{p,\epsilon})\|^2_{\L^2(\Omega_p)} 
\nonumber \\
&\quad \,\leq\,  C\big(\|\wh{\f}_f\|^2_{\bL^2(\Omega_f)}  + \|\wh{\f}_p\|^2_{\bL^2(\Omega_p)} + \,\, \|\wh{q}_p\|^2_{\L^2(\Omega_p)}+ \|\wh{\f}_e\|^2_{\bbL^2(\Omega_p)}\big) \,.
\label{eq:bound-solution-data-5}
\end{align}

\medskip
\noindent {\bf Case 2:} If $\|\bdiv(\bsi_{f,\epsilon})\|_{\bL^{4/3}(\Omega_f)} < 1$ then the negative term in \eqref{eq:bound-solution-data-2} is moved to the right-hand side and bounded by $\epsilon/5$.
Thus, analogously to \eqref{eq:bound-solution-data-5}, we use Young's inequality, \eqref{eq:ubu-bound} and \eqref{eq:bound-solution-data-2a}, and estimates \eqref{eq:tau-d-H0div-inequality}--\eqref{eq:tau-H0div-Xf-inequality}, to derive
\begin{align}
&\|\bsi_{f,\epsilon}\|^2_{\bbX_f} 
+ \|\bu_{p,\epsilon}\|^2_{\bX_p}
+ \|\bsi_{e,\epsilon}\|^2_{\bSigma_e}
+ \|\ubu_\epsilon\|^2_\bS 
+ s_0\,\|p_{p,\epsilon}\|^2_{\W_p} 
+ \epsilon\,\|\bdiv(\bsi_{f,\epsilon})\|^{4/3}_{\bL^{4/3}(\Omega_f)} 
+ \epsilon\,\|\div(\bu_{p,\epsilon})\|^2_{\L^2(\Omega_p)}
\nonumber \\ 
&\quad \,\leq\,  C\,\Big(\|\wh{\f}_f\|^2_{\bL^2(\Omega_f)}  + \|\wh{\f}_p\|^2_{\bL^2(\Omega_p)} + \,\, \|\wh{q}_p\|^2_{\L^2(\Omega_p)}+ \|\wh{\f}_e\|^2_{\bbL^2(\Omega_p)}+ \epsilon \Big)\,.
\label{eq:bound-solution-data-9}
\end{align}
From \eqref{eq:bound-solution-data-5} and \eqref{eq:bound-solution-data-9}, we deduce that for both cases, there exists $\wt{c}_\bT > 0$ independent of $\epsilon$ and $s_0$, such that
\begin{equation}\label{eq:bound-solution-independently-epsilon-1}
\|(\ubsi_\epsilon,\ubu_\epsilon)\|_{\bQ\times \bS}
\,\leq\,\wt{c}_\bT\,\Big\{ \|\wh{\f}_f\|_{\bL^2(\Omega_f)} +  \|\wh{\f}_e\|_{\bbL^2(\Omega_p)} + \|\wh{\f}_p\|_{\bL^2(\Omega_p)} + \| \wh{q}_p\|_{\L^2(\Omega_p)}  + \epsilon^{1/2} \Big\},
\end{equation}
which prove that the solution $(\ubsi_\epsilon, \ubu_\epsilon)\in \bQ\times \bS$ of \eqref{eq:regularization-T-auxiliary-problem-operator} is bounded independently of $\epsilon$.

Given $\epsilon>0$, $\{(\ubsi_\epsilon, \ubu_\epsilon)\}$ solution of \eqref{eq:regularization-T-auxiliary-problem-operator} is a bounded sequence independently of $\epsilon$.
Then, since $\bQ$ and $\bS$ are reflexive Banach spaces, as $\epsilon\to 0$ we can extract weakly convergent subsequences $\{\ubsi_{\epsilon,n}\}^\infty_{n=1}$, $\{\ubu_{\epsilon,n}\}^\infty_{n=1}$, such that $\ubsi_{\epsilon,n} \rightharpoonup \ubsi$ in $\bQ$ and $\ubu_{\epsilon,n}\rightharpoonup \ubu$ in $\bS$, which together with the fact that $\cE_1, \cE_2, \cA, \cB, \cK_{\bw_f}$, $\cL_\bzeta$, and $\cC$ are continuous, imply that $(\ubsi,\ubu)$ solve \eqref{eq:T-auxiliary-problem-operator}.
Moreover, proceeding analogously to \eqref{eq:bound-solution-independently-epsilon-1} we derive \eqref{eq:ubsi-ubu-bound-solution}, with $c_\bT$ independent of $\bw_f$ and $\bzeta$.

Finally, we prove that the solution of \eqref{eq:T-auxiliary-problem-operator} is unique. Since \eqref{eq:T-auxiliary-problem-operator} is linear, it is sufficient to prove that the problem with zero data has only the zero solution. 
Taking $(\wh{\bF},\wh{\bG}) = (\0,\0) $ in \eqref{eq:T-auxiliary-problem-operator}, testing it with solution $(\ubsi,\ubu)$ and employing non-negativity and coercivity estimates of $\cE_1, \cA,\cE_2, \cC$ (cf. Lemma \ref{lem:coercivity-properties-A-E2}) and continuity of $\cK_{\bw_f}$ and $\cL_{\bzeta}$ (cf. \eqref{eq:continuity-cK-wf}, \eqref{eq:continuity-cL-zeta}), yields 
\begin{align}
&\frac{1}{2\,\mu}\,\| \bsi^\rd_f\|^2_{\bbL^2(\Omega_f)} 
+ \mu\,k_{\min}\,\|\bu_p\|^2_{\bL^2(\Omega_p)} 
+ C_A\,\|\bsi_e\|^2_{\bSigma_e} 
+ \rho_f \|\bu_f\|^2_{\bL^2(\Omega_f)} 
+ s_0\|p_p\|^2_{\W_p}
+ \rho_p\|\bu_s\|^2_{\bL^2(\Omega_p)}
\nonumber \\
&\quad +\, c_I\,\sum^{n-1}_{j=1} \|( \bvarphi-\bu_s)\cdot\bt_{f,j}\|^2_{\L^2(\Gamma_{fp})}
- \frac{\rho_f\,n^{1/2}}{2\,\mu}\,\|\bw_f\|_{\bV_f}\|\bu_f\|_{\bV_f}\|\bsi^\rd_f\|_{\bbL^2(\Omega_f)}
- C_{\cL}\|\bzeta\|_{\bLambda_f}\|\bvarphi\|^2_{\bLambda_f} \,\leq\, 0 \,.
\label{eq:sol-uniqueness-4.8-1}
\end{align}
For the last two terms we proceed as in \eqref{eq:bount-to-coercivity1}. Testing 
the first row of \eqref{eq:T-auxiliary-problem-operator} with $\ubtau=(\btau_f, \0, \0)$ in combination with the inf-sup condition of $B_f$ (cf. \eqref{eq:inf-sup-vf-chif}), the continuity of $a_f$ (cf. \eqref{bilinear-form-1}), and the fact that $\|\bw_f\|_{\bV_f}\leq r_1^0$ (cf. \eqref{eq:r_1^0 defn}), we obtain
\begin{equation}\label{eq:ubu-bound-1}
\frac{5\,\beta}{6}\,\|(\bu_f, \bgamma_f, \bvarphi)\|_{\bV_f\times \bbQ_f\times \bLambda_f}
\,\leq \frac{1}{2\,\mu}\,\| \bsi^\rd_f \|_{\bbL^2(\Omega_f)}\,.
\end{equation}
Thus, omitting the seventh term in \eqref{eq:sol-uniqueness-4.8-1}, using \eqref{eq:ubu-bound-1} to bound 
$\|\bu_f\|_{\bV_f}$ and $\|\bvarphi\|_{\bLambda_f}$ by $\frac{3}{5\mu\beta}\| \bsi^\rd_f\|_{\bbL^2(\Omega_f)}$ in \eqref{eq:sol-uniqueness-4.8-1}, and recalling that $\|\bw_f\|_{\bV_f}\leq r_1^0$  (cf. \eqref{eq:r_1^0 defn}) and $\|\bzeta\|_{\bLambda_f} \leq r_2^0$ (cf. \eqref{eq:r20-def}), \eqref{eq:sol-uniqueness-4.8-1}, we deduce that
\begin{align}
& \dfrac{37}{100\,\mu}\,\|\bsi^\rd_f\|^2_{\bbL^2(\Omega_f)} 
+ C\,\Big(\|\bu_p\|^2_{\bL^2(\Omega_p)} 
+ \|\bsi_e\|^2_{\bSigma_e}
+ \|\bu_f\|^2_{\bL^2(\Omega_f)}
+ \|p_p\|^2_{\W_p}
+ \|\bu_s\|^2_{\bL^2(\Omega_p)}\Big) \,\leq\, 0\,, \label{eq:sol-uniqueness-sigfd-others}
\end{align}
with $C>0$ depending on $\mu, k_{\min}, C_A, \rho_f, s_0$, and $\rho_p$,
implying $\bsi^\rd_f = \0, \bu_p = \0, \bsi_e = \0, \bu_f = \0, p_p = 0$ and $\bu_s = \0$. 
On the other hand, we test the first row of \eqref{eq:T-auxiliary-problem-operator} with $\ubtau=(\btau_f,\bv_p,\0)$, employ the inf-sup conditions of $\cB$ (cf. Lemma \ref{lem:inf-sup-conditions}), and the continuity of $\kappa_{\bw_f}$ (cf. \eqref{eq:continuity-cK-wf}) with $\|\bw_f\|_{\bV_f}\leq r_1^0$ (cf. \eqref{eq:r_1^0 defn}), to obtain
\begin{equation}\label{eq:sol-uniqueness-4.8-3}
\frac{5\beta}{6}\,\|(\bu_f ,p_p, \bgamma_f, \bvarphi, \lambda)\|
\,\leq\, \sup_{\0\neq (\btau_f,\bv_p)\in \bbX_{f}\times \bX_p}  \frac{-a_f(\bsi_f,\btau_f) - a^d_p(\bu_p ,\bv_p)}{\|(\btau_f,\bv_p)\|_{\bbX_f\times \bX_p}} \,=\, 0\,.
\end{equation}
Therefore, $\bgamma_f = \0, \bvarphi = \0,$ and  $\lambda = 0 $. Finally, from the second row in \eqref{eq:T-auxiliary-problem-operator}, we have the identity
\begin{equation*}
b_f(\bsi_f ,\bv_f) 
\, = \,  \rho_f\,(\bu_f,\bv_f)_{\Omega_f} \,=\, 0 \quad \forall\,\bv_f\in \bV_f, 
\end{equation*}
which together with the property $\bdiv(\bbX_f) = (\bV_f)'$ allow us to deduce that $\bdiv(\bsi_f) = \0$. Employing the inequalities \eqref{eq:tau-d-H0div-inequality} and \eqref{eq:tau-H0div-Xf-inequality}, we conclude that $\bsi_f = \0$. Therefore
\eqref{eq:T-auxiliary-problem-operator} has a unique solution.
\end{proof}

As an immediate consequence we have the following corollary.
\begin{cor}\label{eq:T-well-delfined}
Assume that the conditions of Lemma ~\ref{thm:well-posedness-1} are satisfied. The operator $\bT$ defined in \eqref{eq:definition-operator-T} is well defined and it satisfies
\begin{equation}\label{eq:operator T-bound-solution}
\|\bT(\bw_f,\bzeta)\|_{\bV_f\times \bLambda_f} 
\,\leq\, {c_\bT\,\Big\{ \|\wh{\f}_f\|_{\bL^2(\Omega_f)} +  \|\wh{\f}_e\|_{\bbL^2(\Omega_p)}+ \|\wh{\f}_p\|_{\bL^2(\Omega_p)} +  \| \wh{q}_p\|_{\L^2(\Omega_p)} \Big\}}\,,
\end{equation}
where $c_\bT$ is the constant introduced in \eqref{eq:ubsi-ubu-bound-solution}.
\end{cor}


\subsubsection[]{Step 2: The resolvent system \eqref{eq:T-auxiliary-problem-operator-A} is well-posed}\label{sec:domain-D-nonempty}

In this section we proceed analogously to \cite{cgos2017} (see also \cite{cgo2021}) and establish the existence and uniqueness of a fixed-point of operator $\bT$ (cf. \eqref{eq:definition-operator-T}) by means of the well known Banach fixed-point theorem which implies that the resolvent system \eqref{eq:T-auxiliary-problem-operator-A} has a unique solution.

\begin{lem}\label{lem:T-contraction-mapping}
Let $r_1\in (0,r_1^0]$ and $r_2\in (0,r_2^0]$, with $r_1^0$ and $r_2^0$ defined in \eqref{eq:r_1^0 defn} and \eqref{eq:r20-def}, respectively. Let $\bW_{r_1,r_2}$ be the closed set defined by
\begin{equation}\label{eq:Wr-definition}
\bW_{r_1,r_2} := \Big\{ (\bw_f,\bzeta)\in \bV_f\times \bLambda_f \,:\quad \| \bw_f\|_{\bV_f} \leq r_1,\quad \|\bzeta\|_{\bLambda_f} \leq r_2 \Big\}\,,
\end{equation}
and set $r_0:=\min\big\{r_1^0, r_2^0\big\}\,.$ Then, for all $(\bw_f,\bzeta), (\wt{\bw}_f,\wt{\bzeta}) \in \bW_{r_1,r_2}$, there holds
\begin{align}
&\|\bT(\bw_f,\bzeta) - \bT(\wt{\bw}_f,\wt{\bzeta}) \|_{\bV_f\times \bLambda_f} \nonumber \\
& \quad \leq\, \frac{c_\bT}{r_0}\,\Big\{ \|\wh{\f}_f\|_{\bL^2(\Omega_f)} +  \|\wh{\f}_e\|_{\bbL^2(\Omega_p)}+ \|\wh{\f}_p\|_{\bL^2(\Omega_p)} + \| \wh{q}_p\|_{\L^2(\Omega_p)} \Big\} \|(\bw_f,\bzeta) - (\wt{\bw}_f,\wt{\bzeta})\|_{\bV_f\times \bLambda_f}.\label{eq:Lipschitz-continuity}
\end{align}
\end{lem}
\begin{proof}
Given $(\bw_f,\bzeta), (\wt{\bw}_f,\wt{\bzeta}) \in \bW_{r_1,r_2}$, we let $(\bu_f,\bvarphi) := \bT(\bw_f,\bzeta)$ and $(\wt{\bu}_f,\wt{\bvarphi}) := \bT(\wt{\bw}_f,\wt{\bzeta})$. 
According to the definition of $\bT$ (cf. \eqref{eq:definition-operator-T}) and some algebraic computations, it follows that
\begin{align}
(\cE_1 + \cA)(\ubsi - \wt{\bsi})(\ubtau) + (\cB' + \cK_{\wt{\bw}_f})(\ubu - \wt{\ubu})(\ubtau) &= -\,\cK_{\bw_f - \wt{\bw}_f}(\ubu)(\ubtau) \quad \forall\,\ubtau\in \bQ\,, \nonumber \\[0.5ex]
\ds -\,\cB(\ubsi - \wt{\ubsi})(\ubv) + (\cE_2 + \cC +\cL_{\wt{\bzeta}})(\ubu - \wt{\ubu})(\ubv) &= -\cL_{\bzeta-\wt{\bzeta}}(\ubu )(\ubv) \quad \forall\,\ubv\in \bS\,. \label{eq:subtracting-auxiliary-problemas-T}
\end{align}
Testing with $\ubtau = \ubsi - \wt{\ubsi}$ and $\ubv = \ubu - \wt{\ubu}$ and employing the monotonicity of $a_f, a^d_p$ and $\cE_1$ (cf. \eqref{eq:coercivity-af}, \eqref{eq:coercivity-adp}, \eqref{eq: operator E_1-monotone}) and the continuity of  $\kappa_{\bw_f}$ and $l_{\bzeta}$ (cf. \eqref{eq:continuity-cK-wf}, \eqref{eq:continuity-cL-zeta}), we deduce that
\begin{align}
& \|\bsi^\rd_f - \wt{\bsi}^\rd_f\|^2_{\bbL^2(\Omega_f)}
\,\leq\, \rho_f\,n^{1/2}\,\Big( \|\bu_f\|_{\bV_f}\,\|\bw_f - \wt{\bw}_f\|_{\bV_f} + \|\wt{\bw}_f\|_{\bV_f}\,\|\bu_f - \wt{\bu}_f\|_{\bV_f} \Big)\|\bsi^\rd_f - \wt{\bsi}^\rd_f\|_{\bbL^2(\Omega_f)} \nonumber \\ 
& \qquad  +\, 2\mu C_{\cL}\Big(\|\wt{\bzeta}\|_{\bLambda_f}\|\bvarphi-\wt{\bvarphi}\|^2_{\bLambda_f}+\|\bvarphi\|_{\bLambda_f}\|\bzeta-\wt{\bzeta}\|_{\bLambda_f} \|\bvarphi-\wt{\bvarphi}\|_{\bLambda_f}\Big)\,. \label{eq:cont-bound-1a}
\end{align}
Using Young's inequality with $\epsilon = 8\,C_{\cL}/(\mu\beta^2)$ on the last term of right-hand side of \eqref{eq:cont-bound-1a}, we obtain
\begin{align}
&\|\bsi^\rd_f - \wt{\bsi}^\rd_f\|^2_{\bbL^2(\Omega_f)}
\,\leq\, \rho_f^2\,n\,\Big(2 \|\bu_f\|^2_{\bV_f}\,\|\bw_f - \wt{\bw}_f\|^2_{\bV_f} + 2\|\wt{\bw}_f\|^2_{\bV_f}\,\|\bu_f - \wt{\bu}_f\|^2_{\bV_f} \Big) \nonumber \\ 
& \qquad +\, 4\mu C_{\cL}\|\wt{\bzeta}\|_{\bLambda_f}\|\bvarphi-\wt{\bvarphi}\|^2_{\bLambda_f}+\frac{16 C_{\cL}^2}{\beta^2}\|\bvarphi\|^2_{\bLambda_f}\|\bzeta-\wt{\bzeta}\|^2_{\bLambda_f} + \frac{\mu^2\beta^2}{4}\|\bvarphi-\wt{\bvarphi}\|^2_{\bLambda_f}.\label{eq:sigfbound-1}
\end{align}
In turn, rearranging the first row of \eqref{eq:subtracting-auxiliary-problemas-T}, we obtain
\begin{equation*}
\ds (\cB' + \cK_{\wt{\bw}_f})(\ubu - \wt{\ubu})(\ubtau) 
\,=\, -\,\cK_{\bw_f - \wt{\bw}_f}\,(\ubu)(\ubtau) - (\cE_1 + \cA)\,(\ubsi - \wt{\ubsi})(\ubtau) \quad \forall\,\ubtau\in \bQ.
\end{equation*}
Testing with $\ubtau=(\btau_f, \0, \0)$ and combining with the inf-sup condition of $B_f$ (cf. \eqref{eq:inf-sup-vf-chif}), and the continuity of $a_f$ and $\cK_{\wt{\bw}_f}$ (cf. \eqref{bilinear-form-1}, \eqref{eq:continuity-cK-wf}) along with the fact that $\|\wt{\bw}_f\|_{\bV_f} \leq r_1$, yields
\begin{equation}\label{eq:Lipschitz-2}
\frac{25 \mu^2 \beta^2}{9}(\,\|\bu_f - \wt{\bu}_f\|^2_{\bV_f} +\|\bvarphi-\wt{\bvarphi}\|^2_{\bLambda_f})
\,\leq\, 2\,\rho_f^2\,n\,\|\bu_f\|^2_{\bV_f}\,\|\bw_f - \wt{\bw}_f\|^2_{\bV_f} + 2\,\|\bsi^\rd_f - \wt{\bsi}^\rd_f\|^2_{\bbL^2(\Omega_f)}.
\end{equation}
Combining \eqref{eq:sigfbound-1} and \eqref{eq:Lipschitz-2}, and using the bounds $\|\wt{\bw}_f\|_{\bV_f} \leq r_1$ and $\|\wt{\bzeta}\|_{\bLambda_f}\leq r_2$, we get
\begin{equation*}
\|\bu_f - \wt{\bu}_f\|^2_{\bV_f} +\|\bvarphi-\wt{\bvarphi}\|^2_{\bLambda_f}
\,\leq \frac{1}{(r_1^0)^2}\|\bu_f\|^2_{\bV_f}\,\|\bw_f - \wt{\bw}_f\|^2_{\bV_f} + \frac{1}{(r_2^0)^2}\|\bvarphi\|^2_{\bLambda_f}\|\bzeta-\wt{\bzeta}\|^2_{\bLambda_f}\,,
\end{equation*}
and applying simple algebraic computations, we obtain
\begin{align}
&\|\bT(\bw_f,\bzeta) - \bT(\wt{\bw}_f,\wt{\bzeta}) \|_{\bV_f\times \bLambda_f} \leq \frac{1}{r_0}\|(\bu_f,\bvarphi)\|_{\bV_f\times\bLambda_f}\|(\bw_f-\wt{\bw_f},\bzeta-\wt{\bzeta})\|_{\bV_f\times \bLambda_f}.\label{eq:Lipschitz-4}
\end{align}
Using \eqref{eq:ubsi-ubu-bound-solution} in \eqref{eq:Lipschitz-4}, we obtain \eqref{eq:Lipschitz-continuity} and complete the proof.
\end{proof}

\medskip

We are now in position of establishing the main result of this section.

\begin{lem}\label{thm:well-posed-domain-D}
Let $\bW_{r_1,r_2}$ be as in \eqref{eq:Wr-definition} and let $r:=\min\big\{r_1, r_2\big\}$. 
Assume that the data satisfy
\begin{equation}\label{eq:T-maps-Wr-into-Wr}
{c_\bT\,\Big\{ \|\wh{\f}_f\|_{\bL^2(\Omega_f)} +  \|\wh{\f}_e\|_{\bbL^2(\Omega_p)}+ \|\wh{\f}_p\|_{\bL^2(\Omega_p)} + \| \wh{q}_p\|_{\L^2(\Omega_p)} \Big\}} 
\,\leq\, r\,.
\end{equation}	
Then, the resolvent problem \eqref{eq:T-auxiliary-problem-operator-A} has a unique solution $(\ubsi,\ubu)\in \bQ\times \bS$ with $(\bu_f,\bvarphi)\in \bW_{r_1,r_2}$, and there holds
\begin{equation}\label{eq:bound-solution-steady-state}
\|(\ubsi,\ubu)\|_{\bQ\times \bS} 
\,\leq\, c_\bT\,\Big\{ \|\wh{\f}_f\|_{\bL^2(\Omega_f)} +  \|\wh{\f}_e\|_{\bbL^2(\Omega_p)}+ \|\wh{\f}_p\|_{\bL^2(\Omega_p)} + \|\wh{q}_p\|_{\L^2(\Omega_p)} \Big\} \,.
\end{equation} 
\end{lem}
\begin{proof}
Employing \eqref{eq:T-maps-Wr-into-Wr} in \eqref{eq:operator T-bound-solution} implies that $\bT:\bW_{r_1,r_2}\to \bW_{r_1,r_2}$. In particular, for $(\bu_f,\bvarphi) = \bT(\bw_f,\bzeta)$ we have that
$\|\bu_f\|_{\bV_f} \leq r \leq r_1$ and $\|\bvarphi\|_{\bLambda_f} \leq r \leq r_2$.
In turn, from \eqref{eq:Lipschitz-continuity} and assumption \eqref{eq:T-maps-Wr-into-Wr} we obtain that
\begin{equation*}
\|\bT(\bw_f,\bzeta) - \bT(\wt{\bw}_f,\wt{\bzeta}) \|_{\bV_f\times \bLambda_f}  \leq \frac{r}{r_0}\|(\bw_f,\bzeta) - (\wt{\bw}_f,\wt{\bzeta})\|_{\bV_f\times \bLambda_f}\,,
\end{equation*}
hence $\bT$ is a contraction mapping.
Therefore, from a direct application of the classical Banach fixed-point theorem we conclude that $\bT$ possesses a unique fixed-point $(\bu_f,\bvarphi)\in \bW_{r_1,r_2}$, or equivalently, problem \eqref{eq:T-auxiliary-problem-operator-A} has a unique solution. In addition, \eqref{eq:bound-solution-steady-state} follows directly from \eqref{eq:ubsi-ubu-bound-solution}.
\end{proof}


\subsubsection[]{Step 3: Existence of a solution of the alternative formulation \eqref{eq:alternative-formulation-operator-form}} \label{sec:wellposedness-alternative-formulation}

In this section we establish the existence of a solution to \eqref{eq:alternative-formulation-operator-form}.
We begin by showing that $\cM$, defined in \eqref{eq:defn-E-N-M}, is a monotone operator on the domain $\cD$ given by
\begin{equation}\label{eq:domain-D-and-Eb-tilde}
\cD := \big\{u \in E: \quad (\cN + \cM)(u)\in \wt{E}'_b \big\}\,,\quad\mbox{with }\, 
\wt{E}'_b:= \big\{ \wh{\mathcal{F}} \in E'_b:\quad \mbox{\eqref{eq:T-maps-Wr-into-Wr} holds} \big\}\,,
\end{equation}
where $u$, $\cN$, and $E'_b$ are as in \eqref{eq:defn-E-N-M} and \eqref{eq:defn-E'_b-D}, respectively, and $\wh{\mathcal{F}} = (\0,\0,\wh{\f}_e,\wh{\f}_f,\wh{q}_p,\0,\wh{\f}_p,\0,0)$.
Notice that \eqref{eq:T-maps-Wr-into-Wr} implies $(\bu_f,\bvarphi)\in \bW_{r_1,r_2}$ (cf. \eqref{eq:Wr-definition}).

\begin{lem}\label{lem:M-monotone-operator}
The operator $\cM$ defined by \eqref{eq:defn-E-N-M} is monotone in $\cD$.
\end{lem}
\begin{proof}
First, for each $u^i \in \cD$, $i\in \{1,2\}$, and using the definition of $\cM$ (cf. \eqref{eq:defn-E-N-M}), we have
\begin{align*}
&(\cM(u^1)-\cM(u^2))(u^1-u^2) =  \,  a_f(\bsi^1_f-\bsi^2_f,\bsi^1_f-\bsi^2_f) +  a^d_p(\bu_p^1-\bu_p^2,\bu_p^1-\bu_p^2) \nonumber\\ 
& \quad +\, \kappa_{\bu_f^1}(\bu_f^1, \bsi^1_f-\bsi^2_f)
- \kappa_{\bu_f^2}(\bu_f^2, \bsi^1_f-\bsi^2_f) + c_{\BJS}(\bu_s^1-\bu_s^2,\bvarphi^1-\bvarphi^2;\bu_s^1-\bu_s^2,\bvarphi^1-\bvarphi^2) \nonumber\\ 
& \quad +\, l_{\bvarphi^1}(\bvarphi^1,\bvarphi^1-\bvarphi^2)
- l_{\bvarphi^2}(\bvarphi^2,\bvarphi^1-\bvarphi^2)\,,
\end{align*}
which, together with the non-negativity bounds of $a_f, a^d_p$ and $c_\BJS$ (cf. \eqref{eq:coercivity-af}, \eqref{eq:coercivity-adp}, \eqref{eq:positivity-aBJS}), 
the continuity of $\kappa_{\bw_f}$ and $l_{\bzeta}$ (cf. \eqref{eq:continuity-cK-wf}, \eqref{eq:continuity-cL-zeta}), 
the estimates $\|\bu_f^1\|_{\bV_f}, \|\bu_f^2\|_{\bV_f} \leq r_1$, and $\|\bvarphi^1\|_{\bLambda_f}, \|\bvarphi^2\|_{\bLambda_f} \leq r_2$, where $r_1\in (0,r_1^0]$ (cf. \eqref{eq:r_1^0 defn}) and $r_2\in (0,r_2^0]$ (cf. \eqref{eq:r20-def}), and some algebraic computations, yields
\begin{align}
&(\cM(u^1)-\cM(u^2))(u^1-u^2)\geq \, \frac{1}{2 \mu} \|(\bsi^1_f -\bsi^2_f)^\rd\|^2_{\bbL^2(\Omega_f)}-C_{\cL}\big(\|\bvarphi^1\|_{\bLambda_f}+\|\bvarphi^2\|_{\bLambda_f}\big)\|\bvarphi^1-\bvarphi^2\|^2_{\bLambda_f} \nonumber\\ 
& \quad\,-\,\frac{\rho_f\,n^{1/2}}{2\,\mu}\,\big(\|\bu_f^1\|_{\bV_f} + \|\bu_f^2\|_{\bV_f}\big)\,\|\bu_f^1 - \bu_f^2\|_{\bV_f} \|(\bsi^1_f - \bsi^2_f)^\rd\|_{\bbL^2(\Omega_f)} \nonumber\\ 
& \quad \geq\,\frac{1}{2 \mu} \|(\bsi^1_f -\bsi^2_f)^\rd\|^2_{\bbL^2(\Omega_f)}
- \dfrac{\beta}{3}\,\|\bu_f^1 - \bu_f^2\|_{\bV_f}\|(\bsi^1_f - \bsi^2_f)^\rd\|_{\bbL^2(\Omega_f)}
- \dfrac{\mu\,\beta^2}{6}\,\|\bvarphi^1-\bvarphi^2\|^2_{\bLambda_f} \,.
\label{eq:mono f bound1}
\end{align}
In turn, since $u^1$ and $u^2$ belong to $\cD$ (cf. \eqref{eq:domain-D-and-Eb-tilde}), we proceed as in \eqref{eq:bount-to-coercivity1} and use the row corresponding to $\btau_f$ in the resolvent system \eqref{eq:T-auxiliary-problem-operator-A} (cf. \eqref{eq:continuous-alternative-weak-formulation-1a}), together with the inf-sup condition of $B_f$ (cf. \eqref{eq:inf-sup-vf-chif}), the continuity of $a_f$ and $\kappa_{\bw_f}$ (cf. \eqref{bilinear-form-1}, \eqref{eq:continuity-cK-wf}), and the bound $\|\bu_f^1\|_{\bV_f}, \|\bu_f^2\|_{\bV_f} \leq r_1$, to deduce that
\begin{equation}\label{eq:inf-sup-B1}
\begin{array}{l}
\ds \dfrac{2\,\beta}{3}\,\|(\bu_f^1, \bgamma_f^1, \bvarphi^1)-(\bu_f^2, \bgamma_f^2, \bvarphi^2)\| \\[2ex] 
\ds\quad \leq	
\big(\beta - \frac{\rho_f n^{1/2}}{2\,\mu}(\|\bu^1_f\|_{\bV_f}+\|\bu^2_f\|_{\bV_f})\big)\,\|(\bu_f^1, \bgamma_f^1, \bvarphi^1)-(\bu_f^2, \bgamma_f^2, \bvarphi^2)\| 
\leq \frac{1}{2\,\mu} \|(\bsi^1_f - \bsi^2_f)^\rd\|_{\bbL^2(\Omega_f)}\,.
\end{array}
\end{equation}
Thus, from \eqref{eq:inf-sup-B1}, both $\|\bu_f^1 - \bu_f^2\|_{\bV_f}$ and $\|\bvarphi^1-\bvarphi^2\|_{\bLambda_f}$ are bounded by $\frac{3}{4\beta\,\mu}\|(\bsi^1_f - \bsi^2_f)^\rd\|_{\bbL^2(\Omega_f)}$, which in turn allows us to deduce from \eqref{eq:mono f bound1} that
\begin{equation}\label{eq:mono u bound2}
(\cM(u^1)-\cM(u^2))(u^1-u^2)
\,\geq \frac{5}{32\mu}\|(\bsi^1_f - \bsi^2_f)^\rd\|^2_{\bbL^2(\Omega_f)}\geq 0\,,
\end{equation}
which implies the monotonicity of $\cM$ in $\cD$.
\end{proof}

Now, we establish a suitable initial condition result, which is needed to apply Theorem \ref{thm:auxiliary-theorem} to the
context of \eqref{eq:alternative-formulation-operator-form}.
\begin{lem}\label{lem:sol0-in-M-operator}
Assume the initial data $(\bu_{f,0},p_{p,0},\bbeta_{p,0},\bu_{s,0}) \in \bH$, where
\begin{subequations}\label{eq:H-space-initial-condition}
\begin{align}
& \ds \bH := \Big\{ (\bv_f,w_p,\bxi_p,\bv_s)\in \bH^1(\Omega_f) \times \H^1(\Omega_p) \times \bV_p\times \bV_p  \,:
  \nonumber \\[1ex] 
& \ds \qquad \bdiv(\be(\bv_f)) \in \bL^2(\Omega_f),\quad 
\bdiv(\bv_f\otimes \bv_f)\in \bL^2(\Omega_f), \nonumber \\[1ex]
& \ds \qquad \div(\bv_f) = 0\,\mbox{ in }\, \Omega_f, \quad \big( 2\mu\be(\bv_f) - \rho_f(\bv_f\otimes\bv_f) \big)\bn_f = \0 \,\mbox{ on }\, \Gamma^{\rN}_f,\quad 
\bv_f = \0 \qon \Gamma^{\rD}_f, \label{bdycond-1}\\[1ex]
& \ds\qquad \bK\,\nabla\,w_p\in \bH^1(\Omega_p),\quad \bK\,\nabla\,w_p\cdot\bn_p = 0 \,\mbox{ on }\, \Gamma^{\rN}_p,\quad w_p = 0 \,\mbox{ on }\, \Gamma^{\rD}_p, \label{bdycond-2}\\[1ex]
& \ds \qquad  \ds \bv_f\cdot\bn_f + \left(\bv_s -\frac{1}{\mu}  \bK\,\nabla\,w_p\right)\cdot\bn_p \,=\, 0 \mbox{ on } \Gamma_{fp}, \label{bdycond-3} \\[1ex]
& \ds\qquad  2\mu\,\be(\bv_f)\bn_f + (A^{-1}(\, \be(\bxi_{p})) - \alpha_p\,w_{p}\,\bI)\bn_p \,=\, 0 \mbox{ on } \Gamma_{fp},\label{bdycond-4}\\
& \ds \qquad  2\mu\,\be(\bv_f)\bn_f+ \mu\,\alpha_{\BJS}\sum^{n-1}_{j=1}\,\sqrt{\bK^{-1}_j}\left\{\left(\bv_f - \bv_s\right)\cdot\bt_{f,j}\right\}\,\bt_{f,j} \,=\, -\,w_p\bn_f \mbox{ on } \Gamma_{fp} \label{bdycond-5}\Big\}.
\end{align}
\end{subequations}
Furthermore, assume that there holds
\begin{align}
&\|\bu_{f,0}\|_{\bL^{2}(\Omega_f)}
+ \|\bdiv(\be(\bu_{f,0}))\|_{\bL^{2}(\Omega_f)} 
+ \|\bdiv(\bu_{f,0}\otimes \bu_{f,0})\|_{\bL^2(\Omega_f)}
+ \|\bu_{s,0}\|_{\bV_p} \nonumber\\ 
&\quad +\, \|\bbeta_{p,0}\|_{\bV_p}
+ \|\bdiv(A^{-1}(\be(\bbeta_{p,0})))\|_{\bbL^2(\Omega_p)} 
+ \|p_{p,0}\|_{\H^1(\Omega_p)}
+ \|\div(\bK\nabla p_{p,0})\|_{\L^2(\Omega_p)} 
\,<\, \frac{1}{C_0}\frac{r}{c_\bT} \,, \label{eq:extra-assumption}
\end{align}
with $C_0$ satisfying \eqref{eq: initial data bound} below and
$r\in (0,r_0)$, with $r_0$ and $c_\bT$ defined in Lemmas \ref{lem:T-contraction-mapping} and \ref{thm:well-posedness-1}, respectively. Then, there exist $\ubsi_0 := (\bsi_{f,0}, \bu_{p,0},\bsi_{e,0})\in \bQ$ and   $\ubu_0 := (\bu_{f,0},p_{p,0},\bgamma_{f,0},\bu_{s,0},\bvarphi_0, \lambda_0)\in \bS$  with $(\bu_{f,0},\bvarphi_0)\in \bW_{r_1,r_2}$ (cf. \eqref{eq:Wr-definition}) such that 
\begin{align}
(\cE_1 + \cA)(\ubsi_0) + (\cB' + \cK_{\bu_{f,0}})(\ubu_0) &=  \wh{\bF}_0 \quad \mbox{ in }\quad  \bQ', \nonumber\\ 
-\,\cB(\ubsi_0) + (\cE_2 + \cC +\cL_{\bvarphi_0})(\ubu_0) &=  \wh{\bG}_0  \quad \mbox{ in } \quad \bS', \label{eq:initial-data-system}
\end{align}
where $\wh{\bF}_0(\ubtau) \,:=\, (\wh{\f}_{e,0},\btau_e)_{\Omega_p}$ and 
$\wh{\bG}_0(\ubv) \,:=\, (\wh{\f}_{f,0},\bv_f)_{\Omega_f} + (\wh{q}_{p,0},w_p)_{\Omega_p} + (\wh{\f}_{p,0},\bv_s)_{\Omega_p}\,\,  \forall  (\ubtau,\ubv)\in \bQ\times \bS$,
with some $(\wh{\f}_{e,0}, \wh{\f}_{f,0}, \wh{q}_{p,0},\wh{\f}_{p,0})\in \bbL^2(\Omega_p)\times \bL^2(\Omega_f)\times \L^2(\Omega_p)\times \bL^2(\Omega_p)$ satisfying
\begin{equation}\label{eq:initial-data-bound-1}
c_\bT\,\Big\{ \|\wh{\f}_{f,0}\|_{\bL^2(\Omega_f)} +  \|\wh{\f}_{e,0}\|_{\bbL^2(\Omega_p)}+ \|\wh{\f}_{p,0}\|_{\bL^2(\Omega_p)} + \|\wh{q}_{p,0}\|_{\L^2(\Omega_p)} \Big\}
\,\leq\, r \,.
\end{equation}
\end{lem}
\begin{proof}
Starting with the given data $(\bu_{f,0},p_{p,0},\bu_{s,0},\bbeta_{p,0})\in \bH$, we take a sequence of steps associated with individual subproblems in order to obtain complete initial data that satisfy the coupled problem \eqref{eq:initial-data-system}. 

\noindent 1. Define $\bu_{p,0} := -\dfrac{1}{\mu}\,\bK\nabla p_{p,0}$. It follows that
\begin{equation}\label{eq:sol0-up0-pp0}
\mu\,\bK^{-1}\bu_{p,0} = -\nabla p_{p,0},\quad 
\div(\bu_{p,0}) = -\frac{1}{\mu}\,\div(\bK\nabla p_{p,0}) \qin \Omega_p,\quad
\bu_{p,0}\cdot\bn_p = 0 \qon \Gamma^{\rN}_p.
\end{equation}
Defining $\lambda_0 := p_{p,0}|_{\Gamma_{fp}}\in \Lambda_p$, integrating by parts the first equation in \eqref{eq:sol0-up0-pp0}, employing \eqref{bdycond-2}, and imposing in a weak sense the second equation of \eqref{eq:sol0-up0-pp0}, we obtain
\begin{equation}\label{eq:system-sol0-1}
\begin{array}{ll}
a^d_p(\bu_{p,0},\bv_p) + b_p(\bv_p,p_{p,0}) + b_{\bn_p}(\bv_p,\lambda_0) = 0 & \forall\,\bv_p\in \bV_p\,, \\[2ex]
s_0\,(p_{p,0},w_p)_{\Omega_p} - b_p(\bu_{p,0},w_p) = \ds ( s_0\,p_{p,0} - \frac{1}{\mu}\,\div(\bK\nabla p_{p,0}),w_p)_{\Omega_p} & \forall\,w_p\in \W_p\,.
\end{array} 
\end{equation}

\noindent 2. Defining $\bsi_{f,0} := 2\mu\be(\bu_{f,0}) - \rho_f(\bu_{f,0}\otimes\bu_{f,0})$, $\bgamma_{f,0} := \dfrac{1}{2}\left( \nabla\bu_{f,0} - (\nabla\bu_{f,0})^\rt \right)$, and $\bvarphi_0 := \bu_{f,0}|_{\Gamma_{fp}}$, it follows that $\bsi_{f,0}\in \bbX_f$, and
\begin{align}
& \frac{1}{2\mu}\bsi^\rd_{f,0} = \nabla\bu_{f,0} - \bgamma_{f,0} - \frac{\rho_f}{2\mu}(\bu_{f,0}\otimes\bu_{f,0})^\rd\,,\quad
-\bdiv(\bsi_{f,0}) = -\bdiv(2\mu\be(\bu_{f,0}) - \rho_f(\bu_{f,0}\otimes \bu_{f,0}))\,. \label{eq:sol0-sigmaf0-uf0-gammaf0}
\end{align}
Then, integrating by parts the first equation in \eqref{eq:sol0-sigmaf0-uf0-gammaf0}, employing  \eqref{bdycond-1}, and imposing in a weak sense the second equation of \eqref{eq:sol0-sigmaf0-uf0-gammaf0}, we get
\begin{equation}\label{eq:system-sol0-2}
\begin{array}{ll}
a_f(\bsi_{f,0},\btau_f) + b_f(\btau_f,\bu_{f,0}) + b_{\sk}(\btau_f,\bgamma_{f,0}) + b_{\bn_f}(\btau_f,\bvarphi_0) + \kappa_{\bu_{f,0}}(\bu_{f,0},\btau_f) = 0 & \forall\,\btau_f\in \bbX_f\,, \\[2ex]
\rho_f(\bu_{f,0},\bv_f)_{\Omega_f} - b_f(\bsi_{f,0},\bv_f) 
= ( \rho_f\bu_{f,0} - \bdiv(2\mu\be(\bu_{f,0}) - \rho_f(\bu_{f,0}\otimes\bu_{f,0})),\bv_f)_{\Omega_f} & \forall\,\bv_f\in \bV_f\,, \\[2ex]
-\,b_{\sk}(\bsi_{f,0},\bchi_f) = 0 & \forall\,\bchi_f\in \bbQ_f\,.
\end{array} 
\end{equation}

\noindent 3. From \eqref{bdycond-3} and \eqref{bdycond-5}, and using the data $\bu_{p,0}$, $\lambda_{0}$, $\bsi_{f,0}$, and $\bvarphi_0$ defined in the previous steps, we have
\begin{align*}
& \bu_{s,0}\cdot\bn_p \,=\, -\bvarphi_0\cdot\bn_f - \bu_{p,0}\cdot\bn_p \qon \Gamma_{fp}\,, \nonumber\\ 
& \mu\,\alpha_{\BJS}\sum^{n-1}_{j=1} \sqrt{\bK^{-1}_j}(\bu_{s,0}\cdot\bt_{f,j})\bt_{f,j} \nonumber\\ 
& \qquad
= \bsi_{f,0}\bn_f + \mu\alpha_{\BJS}\sum^{n-1}_{j=1} \sqrt{\bK^{-1}_j}(\bvarphi_0\cdot\bt_{f,j})\bt_{f,j}+ \rho_f(\bvarphi_0\otimes \bvarphi_0)\bn_f
+ p_{p,0}\bn_f \qon \Gamma_{fp}\,,
\end{align*}
which imply
\begin{equation}\label{eq:system-sol0-3}
\begin{array}{ll}
c_{\Gamma}(\bu_{s,0},\bvarphi_0;\xi) - b_{\bn_p}(\bu_{p,0},\xi) = 0 & \forall\,\xi\in \Lambda_p\,, \\[2ex]
c_{\BJS}(\bu_{s,0},\bvarphi_0;\0,\bpsi) - c_{\Gamma}(\0,\bpsi;\lambda_0) - b_{\bn_f}(\bsi_{f,0},\bpsi) +l_{\bvarphi_0}(\bvarphi_0,\bpsi)= 0 & \forall\,\bpsi\in \bLambda_{f} \,.
\end{array}
\end{equation}

\noindent 4. Define $\bsi_{e,0} \in \bSigma_e$ such that
\begin{equation}\label{eq:system-sol0-4}
\bsi_{e,0} := A^{-1}(\be(\bbeta_{p,0}))\,.
\end{equation}
It follows that $- \bdiv(\bsi_{e,0} - \alpha_p\,p_{p,0}\,\bI)
= - \bdiv(A^{-1}(\be(\bbeta_{p,0})) - \alpha_p\,p_{p,0}\,\bI)$.
Thus, multiplying this expression by $\bv_s \in \bV_p$, integrating by parts, and employing \eqref{bdycond-4} and \eqref{bdycond-5}, we obtain

\begin{align}
& \rho_p(\bu_{s,0},\bv_s)_{\Omega_p} - b_s(\bsi_{e,0},\bv_s) + \alpha_p\,b_p(\bv_s,p_{p,0}) - c_{\Gamma}(\bv_s,\0;\lambda_0) + c_{\BJS}(\bu_{s,0},\bvarphi_0;\bv_s,\0) \nonumber \\
& \qquad  =  
(\rho_p\bu_{s,0} - \bdiv(A^{-1}(\be(\bbeta_{p,0})) - \alpha_p\,p_{p,0}\,\bI),\bv_s)_{\Omega_p}\,. \label{eq:system-sol0-4a}
\end{align}
In addition, using \eqref{eq:system-sol0-4} and simple computations we deduce that
\begin{equation}\label{eq:system-sol0-5}
a^s_p(\bsi_{e,0},\btau_e) + b_s(\btau_e,\bu_{s,0})
= (\be(\bbeta_{p,0}) - \be(\bu_{s,0}),\btau_e)_{\Omega_p} \,.
\end{equation}	

Now, using \eqref{bdycond-1}, \eqref{eq:system-sol0-1}, \eqref{eq:system-sol0-2}, \eqref{eq:system-sol0-3}, \eqref{eq:system-sol0-4a}, and \eqref{eq:system-sol0-5}, we find
$(\bsi_{f,0}, \bu_{p,0}, \bsi_{e,0})\in \bQ$ and  $(\bu_{f,0},p_{p,0},\bgamma_{f,0},\bu_{s,0},\bvarphi_0,\lambda_0)\in \bS$ satisfying \eqref{eq:initial-data-system} with data $(\wh{\f}_{f,0},\wh{\f}_{e,0},\wh{q}_{p,0},\wh{\f}_{p,0})\in 
\bL^2(\Omega_f)\times\bbL^2(\Omega_p)\times \L^2(\Omega_p) \times \bL^2(\Omega_p)$ defined as
\begin{align*}
& \wh{\f}_{f,0} := \rho_f\,\bu_{f,0} - \bdiv(2\mu\be(\bu_{f,0}) - \rho_f(\bu_{f,0}\otimes\bu_{f,0})) \,,\quad 
{\f}_{e,0} := \be(\bbeta_{p,0}) - \be(\bu_{s,0}) \,,\\[0.5ex]
& \wh{\f}_{p,0} := \rho_p\,\bu_{s,0} - \bdiv(A^{-1}(\be(\bbeta_{p,0})) - \alpha_p\,p_{p,0}\,\bI) \,,\quad 
\wh{q}_{p,0} := s_0\,p_{p,0} + \alpha_p\,\div(\bu_{s,0}) - \frac{1}{\mu}\,\div(\bK\nabla p_{p,0}) \,.
\end{align*}
It follows from the above definition that there exists $C_0 > 0$ such that
\begin{align}
&\|\wh{\f}_{f,0}\|_{\bL^2(\Omega_f)} + \|\wh{\f}_{e,0}\|_{\bbL^2(\Omega_p)} +  \|\wh{\f}_{p,0}\|_{\bL^2(\Omega_p)} + \|\wh{q}_{p,0}\|_{\L^2(\Omega_p)} \nonumber\\ 
&\quad \leq C_0 \Big( \|\bu_{f,0}\|_{\bL^{2}(\Omega_f)}
+ \|\bdiv(\be(\bu_{f,0}))\|_{\bL^{2}(\Omega_f)} 
+ \|\bdiv(\bu_{f,0}\otimes \bu_{f,0})\|_{\bL^2(\Omega_f)}
+ \|\bu_{s,0}\|_{\bH^1(\Omega_p)} \nonumber\\ 
&\qquad +\, \|\bbeta_{p,0}\|_{\bH^1(\Omega_p)}
+ \|\bdiv(A^{-1}(\be(\bbeta_{p,0})))\|_{\bbL^2(\Omega_p)} 
+ \|p_{p,0}\|_{\H^1(\Omega_p)}+\|\div(\bK\nabla p_{p,0})\|_{\L^2(\Omega_p)}
\Big)\,. \label{eq: initial data bound}
\end{align}
Thus, \eqref{eq:extra-assumption} guarantee \eqref{eq:initial-data-bound-1} and proceeding as in Lemma~\ref{thm:well-posed-domain-D} we are able to deduce that
\begin{equation}\label{eqn:soln-1-bound}
\|(\ubsi_{0}, \ubu_{0})\|_{\bQ\times \bS} 
\,\leq\, c_\bT\,\Big\{ \|\wh{\f}_{f,0}\|_{\bL^2(\Omega_f)} +  \|\wh{\f}_{e,0}\|_{\bbL^2(\Omega_p)}+ \|\wh{\f}_{p,0}\|_{\bL^2(\Omega_p)} + \|\wh{q}_{p,0}\|_{\L^2(\Omega_p)} \Big\} \,,
\end{equation}
which together with \eqref{eq: initial data bound} and \eqref{eq:initial-data-bound-1}, implies that  $(\bu_{f,0},\bvarphi_{0})\in \bW_{r_{1},r_{2}}$, completing the proof.
\end{proof}

We are now in position to establish existence of a solution of the alternative formulation \eqref{eq:alternative-formulation-operator-form}.

\begin{lem}\label{lem:parabolic-solution}
For each 
$\f_f\in \W^{1,1}(0,T;\bL^2(\Omega_f))$, $q_p\in \W^{1,1}(0,T;\L^2(\Omega_p))$, and $\f_p\in \W^{1,1}(0,T;\bL^2(\Omega_p))$ satisfying for all $t \in [0,T]$
\begin{equation}\label{small-data}
\|\f_f(t)\|_{\bL^2(\Omega_f)} + \|\f_p(t)\|_{\bL^2(\Omega_p)} + \|q_p(t)\|_{\L^2(\Omega_p)} 
< \frac{r}{c_\bT}\,,
\end{equation}
where $r \in (0,r_0)$, with $r_0$ and $c_\bT$ defined in Lemmas \ref{lem:T-contraction-mapping} and \ref{thm:well-posedness-1}, 
and initial data $(\bu_{f,0},p_{p,0},\bbeta_{p,0},\bu_{s,0})$
satisfying the assumptions of Lemma \ref{lem:sol0-in-M-operator}, there exists a solution of \eqref{eq:alternative-formulation-operator-form}, $(\ubsi,\ubu): [0,T] \to \bQ\times\bS$ with $(\bu_f(t),\bvarphi(t))\in \bW_{r_1,r_2}$ (cf. \eqref{eq:Wr-definition}),
\begin{equation*}
(\bsi_e, \bu_f, p_p,\bu_s)\in \W^{1,\infty}(0,T;\bSigma_e)\times \W^{1,\infty}(0,T;\bL^2(\Omega_f))\times \W^{1,\infty}(0,T;\W_p)\times\W^{1,\infty}(0,T; \bL^2(\Omega_p))\,,
\end{equation*}
and $(\bsi_e(0),\bu_f(0),p_p(0),\bu_s(0)) = (\bsi_{e,0},\bu_{f,0},p_{p,0},\bu_{s,0})$, where $\bsi_{e,0}$ is constructed in Lemma \ref{lem:sol0-in-M-operator}.
\end{lem}
\begin{proof}
Problem \eqref{eq:alternative-formulation-operator-form} fits within the framework of Theorem \ref{thm:auxiliary-theorem}, with $E$, $u$, $\cN$, and $\cM$ defined in \eqref{eq:defn-E-N-M} and the restricted range space $\wt{E}'_b$ and domain $\cD$ given in \eqref{eq:domain-D-and-Eb-tilde}.
Indeed, under assumption \eqref{small-data}, in problem \eqref{eq:alternative-formulation-operator-form} we have $\cF = (\0,\0,\0,\f_f,q_p,\0,\f_p,\0,0) \in \wt{E}'_b$.
From the definition of the operators $\cE_1$ and $\cE_2$ (cf. \eqref{defn-E1-E2}), it follows that $\cN$ is linear, symmetric, and monotone.
Moreover, Lemma \ref{lem:M-monotone-operator} shows that $\cM$ is monotone on the domain $\cD$.
The range condition $Rg(\cN+\cM) = \wt{E}'_b$ is established in Lemma \ref{thm:well-posed-domain-D}, which proves that for each $\wh{\cF} \in \wt{E}'_b$ there exists $\wh{u}=(\wh{\ubsi},\wh{\ubu}) \in \cD$ solving \eqref{eq:T-auxiliary-problem-operator-A}.
Finally, initial data $u_0 \in \cD$ with $(\cN + \cM)(u_0) \in \wt{E}'_b$ is constructed in Lemma \ref{lem:sol0-in-M-operator}.
The statement of the lemma then follows by applying Theorem \ref{thm:auxiliary-theorem}.
\end{proof}


\subsection{Existence and uniqueness of solution of the original formulation}

In this section we discuss how the well-posedness of the original formulation \eqref{eq:continuous-weak-formulation-1} follows from the existence of a solution of the alternative formulation \eqref{eq:NS-Biot-formulation-2} (cf. \eqref{eq:alternative-formulation-operator-form}).
Recall that $\bu_s$ is the structure velocity, so the displacement solution can be recovered from
\begin{equation}\label{eq:etap-us-relation}
\bbeta_p(t) = \bbeta_{p,0} + \int^t_0 \bu_s(s)\,ds \quad \forall\, t\in (0,T],
\end{equation}
and then, by construction, $\bu_s = \partial_t\,\bbeta_p$ and $\bbeta_p(0) = \bbeta_{p,0}$.

We note that $a^e_p(\cdot,\cdot)$ satisfies the bounds, for some $c_e, C_e > 0$, for all $\bbeta_p, \bxi_p\in \bV_p$,
\begin{equation}\label{eq:aep-coercivity-bound}
c_e\,\|\bxi_p\|^2_{\bV_p} \,\leq\, a^e_p(\bxi_p, \bxi_p),\quad
a^e_p(\bbeta_p, \bxi_p) \,\leq\, C_e\,\|\bbeta_p\|_{\bV_p}\,\|\bxi_p\|_{\bV_p},
\end{equation}
where the coercivity bound follows from Korn's inequality.
We now state the aforementioned result.
\begin{thm}\label{thm:unique soln}
For each $(\bu_f(0), p_p(0), \bbeta_p(0), \partial_t\bbeta_p(0)) = (\bu_{f,0}, p_{p,0}, \bbeta_{p,0}, \bu_{s,0})\in \bV_f\times \W_p\times \bV_p\times \bV_p$, where $\bu_{f,0}, p_{p,0}, \bbeta_{p,0}$ and $\bu_{s,0}$ are compatible initial data satisfying Lemma \ref{lem:sol0-in-M-operator} (cf. \eqref{eq:extra-assumption}), and for each
\begin{equation*}
\f_f\in \W^{1,1}(0,T;\bL^2(\Omega_f)),\quad \f_p\in \W^{1,1}(0,T;\bL^2(\Omega_p)),\quad q_p\in \W^{1,1}(0,T;\L^2(\Omega_p))\,,
\end{equation*}
under the assumptions of Lemma \ref{lem:parabolic-solution} (cf. \eqref{small-data}), there exists a unique solution of \eqref{eq:continuous-weak-formulation-1}$, (\bsi_f,\bu_p, \bbeta_p,$
$\bu_f, p_p, \bgamma_f, \bvarphi, \lambda): [0,T]\to \bbX_f\times \bX_p\times\bV_p\times\bV_f\times \W_p\times \bbQ_f\times \bLambda_f\times \Lambda_p$ with $(\bu_f(t),\bvarphi(t)):[0,T]\to \bW_{r_1,r_2}$. In addition, $\bsi_f(0) = \bsi_{f,0}, \bu_p(0) = \bu_{p,0},  \bgamma_f(0) = \bgamma_{f,0}, \bvarphi(0) = \bvarphi_{0}$ and $\lambda(0) = \lambda_{0}.$
\end{thm}
\begin{proof}
We begin by using the existence of a solution of the alternative formulation \eqref{eq:NS-Biot-formulation-2} to establish solvability of the original formulation \eqref{eq:continuous-weak-formulation-1}.
Let $(\bsi_f, \bu_p, \bsi_e, \bu_f, p_p, \bgamma_f, \bu_s, \bvarphi, \lambda)$ be a solution to \eqref{eq:NS-Biot-formulation-2}.
Let $\bbeta_p$ be defined in \eqref{eq:etap-us-relation}, so $\bu_s = \partial_t\,\bbeta_p$.
Then \eqref{eq:NS-Biot-formulation-2} with $\btau_e = \0$ and $\bv_s = \0$ implies \eqref{eq:continuous-weak-formulation-1} with $\bxi_p = \0$.
In turn, testing the first equation in \eqref{eq:NS-Biot-formulation-2} with $\btau_e\in \bSigma_e$ gives $\big(\partial_t\,(A(\bsi_e) - \be(\bbeta_p)), \btau_e\big)_{\Omega_p} = 0$, which, using that $\be(\bX_p) \subset \bSigma_e$, implies that $\partial_t\,(A(\bsi_e) - \be(\bbeta_p)) = \0$.
Integrating from $0$ to $t\in (0,T]$ and using that $\bsi_e(0) = A^{-1}(\be(\bbeta_p(0)))$ implies that $\bsi_e(t) = A^{-1}(\be(\bbeta_p(t)))$.
Therefore, with \eqref{eq:elasticity-stress-isotropic}, we deduce
\begin{equation*}
b_s(\bsi_e,\bv_s) \,=\, -\,(\bsi_e, \be(\bv_s))_{\Omega_p} 
\,=\, -\,(A^{-1}(\be(\bbeta_p)),\be(\bv_s))_{\Omega_p} 
\,=\, -\,a^e_p(\bbeta_p, \bv_s),
\end{equation*} 
and then replacing back into the second equation in \eqref{eq:NS-Biot-formulation-2} with $\bv_s\in \bV_p$, yields
\begin{equation*}
 \rho_p\,(\partial_{tt}\bbeta_p,\bv_s)_{\Omega_p}+\,\, c_{\BJS}(\partial_{t}\bbeta_p,\bvarphi;\bv_s,\0)- c_{\Gamma}(\bv_s,\0;\lambda) +a^e_p(\bbeta_p, \bv_s) + \alpha_p\,b_p(p_p, \bv_s) \,=\, (\f_p, \bv_s)_{\Omega_p}.
\end{equation*}
Therefore \eqref{eq:NS-Biot-formulation-2} implies \eqref{eq:continuous-weak-formulation-1}, which establishes that $(\bsi_f, \bu_p, \bbeta_{p,0} + \int^t_0 \bu_s(s)\,ds, \bu_f, p_p, \bgamma_f, \bvarphi, \lambda)$ is a solution to \eqref{eq:continuous-weak-formulation-1}.

Now, assume that the solution of \eqref{eq:continuous-weak-formulation-1} is not unique.
Let $(\bsi^i_f, \bu^i_p, \bbeta^i_p, \bu^i_f, p^i_p, \bgamma^i_f, \bvarphi^i, \lambda^i)$, with $i\in \{1,2\}$, be two solutions corresponding to the same data. Taking \eqref{eq:continuous-weak-formulation-1} with $\btau_f = \bsi^1_f - \bsi^2_f, \bv_p = \bu^1_p - \bu^2_p, \bxi_p = \partial_t\,\bbeta^1_p - \partial_t\,\bbeta^2_p, \bv_f = \bu^1_f - \bu^2_f, w_p = p^1_p - p^2_p, \bchi_f = \bgamma^1_f - \bgamma^2_f, \bpsi = \bvarphi^1 - \bvarphi^2$ and $\xi = \lambda^1 - \lambda^2$, using arguments similar to those in \eqref{eq:mono f bound1}--\eqref{eq:mono u bound2}, that is, by combining \eqref{eq:continuous-weak-formulation-1a} with the inf-sup condition of $B_f$ (cf. \eqref{eq:inf-sup-vf-chif}), the continuity of $\kappa_{\bw_f}$ and $l_{\bzeta}$ (cf. \eqref{eq:continuity-cK-wf}, \eqref{eq:continuity-cL-zeta}), the fact that $(\bu^i_f(t),\bvarphi^i(t)) \in \bW_{r_1,r_2}$ (cf. \eqref{eq:Wr-definition}), and \eqref{eq:aep-coercivity-bound}, we obtain
\begin{align}
&\frac{1}{2}\,\partial_t\,\Big( c_e\|\bbeta^1_p - \bbeta^2_p\|^2_{\bV_p}  + \rho_f\,\|\bu^1_f - \bu^2_f\|^2_{\bL^2(\Omega_f)} + s_0\,\|p^1_p - p^2_p\|^2_{\W_p} +\rho_p\,\|\partial_t\,\bbeta^1_p - \partial_t\,\bbeta^2_p\|^2_{\bL^2(\Omega_p)}\Big)  \nonumber\\ 
&\quad + \frac{5}{32\,\mu}\,\|(\bsi^1_f - \bsi^2_f)^\rd\|^2_{\bbL^2(\Omega_f)} + \mu\,k_{\min}\,\|\bu^1_p - \bu^2_p\|^2_{\bL^2(\Omega_p)}\,\leq\, 0 \,,\label{eq:thm4.11-coercivitybounds}
\end{align}
and integrating in time from $0$ to $t\in (0,T]$,  and using $\bu^1_f(0) = \bu^2_f(0), p^1_p(0) = p^2_p(0), \bbeta^1_p(0) = \bbeta^2_p(0)$,
$\partial_t\,\bbeta^1_p(0) = \partial_t\,\bbeta^2_p(0)$, yields
\begin{align*}
&\frac{1}{2}\,\Big( c_e\|\bbeta^1_p - \bbeta^2_p\|^2_{\bV_p}  + \rho_f\,\|\bu^1_f - \bu^2_f\|^2_{\bL^2(\Omega_f)} + s_0\,\|p^1_p - p^2_p\|^2_{\W_p} +\rho_p\,\|\partial_t\,\bbeta^1_p - \partial_t\,\bbeta^2_p\|^2_{\bL^2(\Omega_p)}\Big)  \nonumber\\ 
& \quad +\,\, C\,\int^t_0 \Big(\|(\bsi^1_f - \bsi^2_f)^\rd\|^2_{\bbL^2(\Omega_f)} + \|\bu^1_p - \bu^2_p\|^2_{\bL^2(\Omega_p)} \Big)\, ds \,\leq\, 0 \,.
\end{align*}
Thus, $(\bsi^1_f(t))^\rd = (\bsi^2_f(t))^\rd, \bu^1_p(t) = \bu^2_p(t), \bbeta^1_p(t) = \bbeta^2_p(t), \bu_f^1(t) = \bu_f^2(t), p^1_p(t) = p^2_p(t)$, and $\partial_t\,\bbeta^1_p(t) = \partial_t\,\bbeta^2_p(t)$ for all $t\in (0,T]$.

On the other hand, for $(\bu^i_f, p^i_p, \bgamma^i_f, \bvarphi^i, \lambda^i) \in \bV_f\times \W_p\times \bbQ_f\times \bLambda_f\times \Lambda_p$ and $(\bsi^i_f, \bu^i_p) \in \bbX_f\times \bX_p$  with $i\in \{1,2\}$ satisfying \eqref{eq:continuous-weak-formulation-1a} and \eqref{eq:continuous-weak-formulation-1e}, using similar arguments to \eqref{eq:inf-sup-B1}, the inf-sup conditions of $\cB$ (cf. Lemma \ref{lem:inf-sup-conditions}) and continuity of $\kappa_{\bw_f}$ (cf. \eqref{eq:continuity-cK-wf}) along with $(\bu^1_f,\bvarphi^1), (\bu^2_f,\bvarphi^2)\in \bW_{r_1,r_2}$ (cf. \eqref{eq:Wr-definition}), we obtain
\begin{align*}
& \frac{2\beta}{3}\,\|(\bu^1_f - \bu^2_f, p^1_p - p^2_p, \bgamma^1_f - \bgamma^2_f, \bvarphi^1 - \bvarphi^2, \lambda^1 - \lambda^2)\| \nonumber\\
& \quad \,\leq \sup_{\0\neq (\btau_f,\bv_p)\in \bbX_{f}\times \bX_p}   \frac{-a_f(\bsi^2_f - \bsi^1_f,\btau_f) - a^d_p(\bu^2_p - \bu^1_p,\bv_p)}{\|(\btau_f,\bv_p)\|_{\bbX_f\times \bX_p}} \,=\, 0 \,.
\end{align*}
Therefore, $\bgamma^1_f(t) = \bgamma^2_f(t), \bvarphi^1(t) = \bvarphi^2(t)$, and $\lambda^1(t) = \lambda^2(t)$ for all $t\in (0,T]$.
In turn, from \eqref{eq:continuous-weak-formulation-1b}, we have the identity
\begin{equation*}
b_f(\bsi^1_f - \bsi^2_f,\bv_f)  
\,=\, \rho_f\,(\partial_t\,(\bu^1_f - \bu^2_f),\bv_f)_{\Omega_f} 
\,=\, 0 \quad \forall\,\bv_f\in \bV_f, 
\end{equation*}
which together with the property $\bdiv(\bbX_f) = (\bV_f)'$ allow us to deduce: $\bdiv(\bsi^1_f(t)) = \bdiv(\bsi^2_f(t))$ for all $t\in (0,T]$, and employing again the inequalities \eqref{eq:tau-d-H0div-inequality} and \eqref{eq:tau-H0div-Xf-inequality}, we conclude that $\bsi^1_f(t) = \bsi^2_f(t)$ for all $t\in (0,T]$, so then we can conclude that \eqref{eq:continuous-weak-formulation-1} has a unique solution.

It remains to establish the initial values for $\bsi_f$, $\bu_p$, $\bgamma_f$, $\bvarphi$, and $\lambda$. Let $\overline{\bsi}_f =\bsi_f(0) - \bsi_{f,0}$, with a similar definition and notation for the rest of the variables, except for $\overline{\partial_t\,\bbeta}_{p} = \partial_t\,\bbeta_{p}(0) - \bu_{s,0}$. 
We can take $t \to 0$ in \eqref{eq:continuous-weak-formulation-1}. 
Using that the initial data $(\ubsi_0,\ubu_0)$ constructed in Lemma \ref{lem:sol0-in-M-operator} satisfies \eqref{eq:initial-data-system}, which corresponds to \eqref{eq:T-auxiliary-problem-operator-B} at $t=0$ without the operator $\cE_1$, and that $\overline{\bu}_{f} = \0, \overline{p}_{p} = 0, \overline{\bbeta}_{p} = \0 $ and $\overline{\partial_t\,\bbeta}_{p} = \0,$ we obtain
\begin{subequations}\label{eq:init-0}
\begin{align}
&\ds  \frac{1}{2\,\mu}\,(\overline{\bsi}^\rd_f,\btau^\rd_f)_{\Omega_f} - \pil\btau_f\bn_f,\overline{\bvarphi} \pir_{\Gamma_{fp}} + (\overline{\bgamma}_f,\btau_f)_{\Omega_f} =0, \label{init-1} \\[1ex]
& \ds -\,(\overline{\bsi}_f,\bchi_f)_{\Omega_f} = 0, \label{init-2} \\[1ex]
& \ds \mu\,(\bK^{-1}\overline{\bu}_p,\bv_p)_{\Omega_p} + \pil\bv_p\cdot\bn_p,\overline{\lambda}\pir_{\Gamma_{fp}} = 0, \label{init-3} \\[1ex]
& \ds - \pil\overline{\bvarphi}\cdot\bn_f + \overline{\bu}_p\cdot\bn_p,\xi\pir_{\Gamma_{fp}} = 0, \label{init-4} \\[1ex]
& \ds \pil\overline{\bsi}_f\bn_f,\bpsi\pir_{\Gamma_{fp}} + \mu\,\alpha_{\BJS}\,\sum_{j=1}^{n-1} \pil\sqrt{\bK^{-1}_j} \overline{\bvarphi} \cdot\bt_{f,j},\bpsi\cdot\bt_{f,j} \pir_{\Gamma_{fp}}   +\,\, \rho_f\pil \overline{\bvarphi}\cdot\bn_f, \bvarphi(0)\cdot\bpsi \pir_{\Gamma_{fp}} \nonumber \\[2ex]
& \ds\quad +\, \rho_f\pil \bvarphi_0\cdot\bn_f, \overline{\bvarphi}\cdot\bpsi \pir_{\Gamma_{fp}}+ \pil\bpsi\cdot\bn_f,\overline{\lambda}\pir_{\Gamma_{fp}} = 0.\label{init-5} 
\end{align}
\end{subequations}
Taking $(\btau_f,\bv_p,\bchi_f,\bpsi, \xi) = (\overline{\bsi}_f,\overline{\bu}_p,\overline{\bgamma}_f,\overline{\bvarphi},\overline{\lambda})$ in \eqref{eq:init-0}, using that $(\bu_{f,0},\bvarphi_0), (\bu_{f}(0),\bvarphi(0))\in \bW_{r_1,r_2}$ (cf. \eqref{eq:Wr-definition}), combining the equations and proceeding as in  \eqref{eq:thm4.11-coercivitybounds}, we get
\begin{equation*}
\,\|\overline{\bsi}^\rd_f \|^2_{\bbL^2(\Omega_f)} + \,\|\overline{\bu}_p\|^2_{\bL^2(\Omega_p)} + \,\sum^{n-1}_{j=1} \|\overline{\bvarphi}\cdot \bt_{f,j} \|^2_{\L^2(\Gamma_{fp})}\,\leq\, 0,  
\end{equation*}
which implies that $\overline{\bsi}^\rd_f = \0, \overline{\bu}_p=\0$ and $\overline{\bvarphi}\cdot \bt_{f,j} = 0.$ In addition, \eqref{init-4} implies that $\pil\overline{\bvarphi}\cdot\bn_f,\xi\pir_{\Gamma_{fp}} = 0$ for all $\xi \in \H^{1/2}(\Gamma_{fp}).$  Since $\H^{1/2}(\Gamma_{fp})$ 
is dense in $\L^2(\Gamma_{fp})$, it follows that $\overline{\bvarphi}\cdot\bn_f= 0;$ hence $\overline{\bvarphi}=\0$. Again, combining \eqref{init-1} and \eqref{init-3}, employing the inf-sup conditions of $\cB$ (cf. Lemma~\ref{lem:inf-sup-conditions}), together with  $\overline{\bsi}^\rd_f = \0$, $\overline{\bvarphi} = \0$ and $\overline{\bu}_p=\0$, implies that $\overline{\bgamma}_f = \0$ and $\overline{\lambda} =0$. Finally, from \eqref{eq:continuous-weak-formulation-1b} at t = 0, we obtain
\begin{equation}\label{div-sigf-init}
  b_f(\overline\bsi_{f},\bv_f) =  \rho_f\, (\partial_t\,\bu_f(0),\bv_f)_{\Omega_f} - (\f_f(0),\bv_f)_{\Omega_f} - b_f(\bsi_{f,0},\bv_f) = 0,
\end{equation}
where we have chosen $\f_f(0)= \rho_f\,\partial_t\,\bu_f(0) - \bdiv(\bsi_{f,0})$. Applying the property $\bdiv(\bbX_f) = (\bV_f)'$ allow us to deduce that $\bdiv(\overline{\bsi}_f) = \0$, which, combined with the inequalities \eqref{eq:tau-d-H0div-inequality} and \eqref{eq:tau-H0div-Xf-inequality}, yields $\overline{\bsi}_f = \0$, completing the proof.
\end{proof}

We next recall from \cite[Lemma 3.3]{cwy2025} a result that will be employed to derive the stability bound for the solution of \eqref{eq:continuous-weak-formulation-1} without relying on Gr\"onwall's inequality.
\begin{lem}\label{xing-lemma}
Suppose that for all $t \in (0,T]$,
\begin{equation*}
H^2(t) + R(t) \leq A(t) + 2 \int_0^t B(s) H(s) \, ds,
\end{equation*}
where $H, R, A$ and $B$ are non-negative functions. Then
\begin{equation*}
\sqrt{H^2(T) + R(T)} \leq \sup_{0\leq t\leq T} \sqrt{A(t)} + \int_0^T B(t) dt.
\end{equation*}
\end{lem}
%
%

We conclude with the aforementioned stability bound.
\begin{thm}\label{thm: continuous stability}
Under the assumptions of Theorem \ref{thm:unique soln}, and assuming that $\f_f \in \H^1(0,T;\bL^2(\Omega_f))$, $\f_p \in \H^1(0,T;\bL^2(\Omega_p))$, and $q_p \in \H^1(0,T;\L^2(\Omega_p))$, 
there exists a positive constant $C$, independent of $s_0$, such that
\begin{align}
&\|\bsi_f\|_{\L^2(0,T;\bbX_f)} 
+ \|\partial_t\bsi^\rd_f\|_{\L^2(0,T;\bbL^2(\Omega_f))} 
+ \|\bu_p\|_{\L^2(0,T;\bX_p)} 
+ \| \partial_t\bu_p\|_{\L^2(0,T;\bL^2(\Omega_p))}
+ \|\bbeta_p\|_{\L^\infty(0,T;\bV_p)} 
\nonumber \\
&\quad +\, \| \partial_t\bbeta_p\|_{\L^\infty(0,T;\bL^2(\Omega_p))}
+ \sum^{n-1}_{j=1} \|(\bvarphi-\partial_t\,\bbeta_p)\cdot\bt_{f,j}\|_{\H^1(0,T;\L^2(\Gamma_{fp}))}
+\, \| \partial_{tt}\bbeta_p\|_{\L^\infty(0,T;\bL^2(\Omega_p))}
\nonumber \\
&\quad +\, \|\bu_f\|_{\H^1(0,T;\bV_f)}
+ \|\bu_f\|_{\W^{1,\infty}(0,T;\bL^2(\Omega_f))}
+ \sqrt{s_0}\,\|p_p\|_{\W^{1,\infty}(0,T;\W_p)}
+ \|p_p\|_{\H^1(0,T;\W_p)}
\nonumber \\[1ex]
&\quad 
+\, \|\bgamma_f\|_{\H^1(0,T;\bbQ_f)} 
+ \|\bvarphi\|_{\H^1(0,T;\bLambda_f)} 
+ \|\lambda\|_{\H^1(0,T;\Lambda_p)} 
\nonumber \\
&\leq\, C \sqrt{T}\,\Bigg( \|\f_f\|_{\H^1(0,T;\bL^2(\Omega_f))}  
+ \|\f_p\|_{\H^1(0,T;\bL^2(\Omega_p))}
+ \|q_p\|_{\H^1(0,T;\L^2(\Omega_p))}
+ \frac{1}{\sqrt{s_0}}\|q_p(0)\|_{\L^2(\Omega_p)}
\nonumber \\
& \quad +\, \|\bu_{f,0}\|_{\bL^2(\Omega_f)}
+ \|\bdiv( \be(\bu_{f,0}))\|_{\bL^{2}(\Omega_f)}
+ \|\bdiv(\bu_{f,0}\otimes \bu_{f,0})\|_{\bL^2(\Omega_f)} 
+ \sqrt{s_0}\,\|p_{p,0}\|_{\W_p}
\nonumber \\
& \quad +\, \|p_{p,0}\|_{\H^1(\Omega_p)}
+ \frac{1}{\sqrt{s_0}}\|\div(\bK\nabla p_{p,0})\|_{\bL^2(\Omega_p)}
+ \|\bbeta_{p,0}\|_{\bV_p} 
+ \|\bdiv(A^{-1}(\, \be(\bbeta_{p,0})))\|_{\bL^2(\Omega_p)} 
\nonumber \\
&\quad 
+\, \left(1+\frac{1}{\sqrt{s_0}}\right)\|\bu_{s,0}\|_{\bV_p}  
\Bigg) \,.
\label{eq:continuous-stability}
\end{align}
\end{thm}
\begin{proof}
We begin by choosing $(\btau_f, \bv_p, \bxi_p, \bv_f, w_p, \bchi_f, \bpsi, \xi) = (\bsi_f, \bu_p, \partial_t\,\bbeta_p, \bu_f, p_p, \bgamma_f, \bvarphi, \lambda)$ in \eqref{eq:continuous-weak-formulation-1} to get
\begin{align}
&\frac{1}{2}\,\partial_t\,\Big( \rho_f\,\|\bu_f\|^2_{\bL^2(\Omega_f)} 
+ s_0\,\|p_p\|^2_{\W_p} 
+ a^e_p(\bbeta_p, \bbeta_p)  
+ \rho_p\,\|\partial_t\bbeta_p\|^2_{\bL^2(\Omega_p)}\Big) 
+ a_f(\bsi_f, \bsi_f) 
+ a^d_p(\bu_p,\bu_p) \nonumber \\
&\quad +\, c_{\BJS}(\partial_t\,\bbeta_p, \bvarphi;\partial_t\,\bbeta_p, \bvarphi) 
+ \kappa_{\bu_f}(\bu_f,\bsi_f) 
+ l_{\bvarphi}(\bvarphi,\bvarphi)
\,=\, (\f_p,\partial_t\,\bbeta_p)_{\Omega_p} 
+ (\f_f,\bu_f)_{\Omega_f} 
+ (q_p,p_p)_{\Omega_p} \,.\label{eq:stability-bound-1}
\end{align}
Next, we integrate \eqref{eq:stability-bound-1} from $0$ to $t\in (0,T]$, use the coercivity bound of $a^e_p$ (cf. \eqref{eq:aep-coercivity-bound}), the non-negativity bounds of $a_f, a^d_p, \cE_2$ and $c_\BJS$ (cf. \eqref{eq:coercivity-af}, \eqref{eq:coercivity-adp}, \eqref{eq: operator E_2-monotone}, \eqref{eq: operator C-monotone}), the identity
\begin{equation*}
\int_0^t (\f_p,\partial_t\,\bbeta_p)_{\Omega_p}\, ds = (\f_p,\,\bbeta_p)_{\Omega_p}\Big|_0^t - \int_0^t (\partial_t \f_p,\\,\bbeta_p)_{\Omega_p}\, ds,
\end{equation*}
and arguments similar to those in \eqref{eq:sol-uniqueness-4.8-1}--\eqref{eq:sol-uniqueness-sigfd-others} to deal with the terms $\kappa_{\bu_f}$ and $l_\bvarphi$, namely, using \eqref{eq:continuous-weak-formulation-1a} together with the inf-sup condition of $B_f$ (cf. \eqref{eq:inf-sup-vf-chif}), the continuity of $\kappa_{\bw_f}$ and $l_{\bzeta}$ (cf. \eqref{eq:continuity-cK-wf}, \eqref{eq:continuity-cL-zeta}), and the fact that $(\bu_f(t),\bvarphi(t)):[0,T]\to \bW_{r_1,r_2}$ (cf. \eqref{eq:Wr-definition}), to obtain
\begin{align}\label{eq:bound-unsteady-state-solution-1}
& \ds \frac{\rho_f}{2}\,\|\bu_f(t)\|^2_{\bL^2(\Omega_f)} + \frac{s_0}{2}\,\|p_p(t)\|^2_{\W_p} + \frac{c_e}{2}\|\bbeta_p(t)\|^2_{\bV_p} + \frac{\rho_p}{2}\,\| \partial_t\bbeta_p(t)\|^2_{\bL^2(\Omega_p)}\nonumber \\[1ex]
&\ds\quad +\,
\int^t_0 \Big( \frac{37}{100\mu}\| \bsi^\rd_f\|^2_{\bbL^2(\Omega_f)}
+ \mu\,k_{\min} \|\bu_p\|^2_{\bL^2(\Omega_p)}
+ c_I\sum^{n-1}_{j=1} \|( \bvarphi - \partial_t\,\bbeta_p)\cdot\bt_{f,j}\|^2_{\L^2(\Gamma_{fp})} \Big)\, ds \nonumber \\[1ex]
& \ds \leq\, \frac{\rho_f}{2}\,\|\bu_f(0)\|^2_{\bL^2(\Omega_f)} +  \frac{s_0}{2}\,\|p_p(0)\|^2_{\W_p} + \frac{c_e}{2}\|\bbeta_p(0)\|^2_{\bV_p} + \frac{\rho_p}{2}\,\| \partial_t\bbeta_p(0)\|^2_{\bL^2(\Omega_p)} + (\f_p(t),\bbeta_p(t))_{\Omega_p}  \nonumber \\[1ex]
&\ds\quad -\, (\f_p(0),\bbeta_p(0))_{\Omega_p}\,-\,\, \int^t_0 \Big( (\partial_t\,\f_p,\bbeta_p)_{\Omega_p} -  (\f_f,\bu_f)_{\Omega_f}  - (q_p,p_p)_{\Omega_p} \Big)\, ds \,.
\end{align}

On the other hand, similarly to \eqref{eq:sol-uniqueness-4.8-3}, adding \eqref{eq:continuous-weak-formulation-1a} and \eqref{eq:continuous-weak-formulation-1e}, applying the inf-sup condition of $\cB$ in Lemma \ref{lem:inf-sup-conditions} for $(\bu_f, p_p, \bgamma_f, \bvarphi, \lambda)$, the continuity of $\kappa_{\bw_f}$ (cf. \eqref{eq:continuity-cK-wf}), with $(\bu_f(t),\bvarphi(t)):[0,T]\to \bW_{r_1,r_2}$ (cf. \eqref{eq:Wr-definition}), and the continuity bounds of $a_f, a^d_p$ (cf. \eqref{defn-A}, \eqref{eq:continuity-cA-cB-cC}), we deduce that
\begin{equation}\label{eq:bound-unsteady-state-solution-3}
\int^t_0 \|(\bu_f, p_p, \bgamma_f, \bvarphi, \lambda)\|^2 \, ds 
\,\leq\, C\,\int^t_0 \Big( \| \bsi^\rd_f\|^2_{\bbL^2(\Omega_f)} + \|\bu_p\|^2_{\bL^2(\Omega_p)} \Big)\, ds.
\end{equation}
%

Thus, applying Cauchy--Schwarz and Young's inequalities on the right-hand side of \eqref{eq:bound-unsteady-state-solution-1} in combination with \eqref{eq:bound-unsteady-state-solution-3}, we obtain
\begin{align}
&\|\bu_f(t)\|^2_{\bL^2(\Omega_f)} + s_0\,\|p_p(t)\|^2_{\W_p} + \|\bbeta_p(t)\|^2_{\bV_p}+ \,\| \partial_t\bbeta_p(t)\|^2_{\bL^2(\Omega_p)} \nonumber \\
&\quad +\, \int^t_0 \Bigg(  \| \bsi^\rd_f\|^2_{\bbL^2(\Omega_f)} + \|\bu_p\|^2_{\bL^2(\Omega_p)} + \sum^{n-1}_{j=1} \|( \bvarphi - \partial_t\,\bbeta_p )\cdot\bt_{f,j}\|^2_{\L^2(\Gamma_{fp})} + \|(\bu_f, p_p, \bgamma_f, \bvarphi, \lambda)\|^2 \Bigg)\, ds \nonumber \\
&\leq\, C\,\Bigg( \int^t_0 \Big( \|\f_f\|^2_{\bL^2(\Omega_f)} + \|q_p\|^2_{\L^2(\Omega_p)} \Big)\,ds 
+ \|\f_p(t)\|^2_{\bL^2(\Omega_p)} + \|\f_p(0)\|^2_{\bL^2(\Omega_p)} +\, \|\bu_f(0)\|^2_{\bL^2(\Omega_f)} 
\nonumber \\
&\quad 
+ s_0\,\|p_p(0)\|^2_{\W_p} + \|\bbeta_p(0)\|^2_{\bV_p} 
+ \| \partial_t\bbeta_p(0)\|^2_{\bL^2(\Omega_p)}  +\, \int^t_0 \|\partial_t\,\f_p\|_{\bL^2(\Omega_p)}\|\bbeta_p\|_{\bV_p} \, ds\Bigg)\,.
\label{eq:bound-unsteady-state-solution-6}
\end{align}

Next, we obtain a stability bound for $\|\bdiv(\bsi_f)\|_{\bL^{4/3}(\Omega_f)}$ and $\|\div(\bu_{p})\|_{\L^2(\Omega_p)}$. From \eqref{eq:continuous-weak-formulation-1b}, \eqref{eq:continuous-weak-formulation-1f}, and recalling that $\bdiv(\bbX_f) = (\bV_f)'$ and $\div(\bX_p) = (\W_p)'$, it follows that
\begin{align}
&\|\bdiv(\bsi_f)\|_{\bL^{4/3}(\Omega_f)} 
\,\leq\, |\Omega_f|^{1/4}\,\big(\|\f_f \|_{\bL^2(\Omega_f)} + \rho_f\,\|\partial_t\,\bu_f \|_{\bL^2(\Omega_f)}\big), \nonumber \\[1ex]
&\mbox{and}\quad \|\div(\bu_{p})\|_{\L^2(\Omega_p)} 
\,\leq\, \|q_p\|_{\L^2(\Omega_p)} + s_0\,\|\partial_t p_{p}\|_{\W_p} + \alpha_p\,n^{1/2}\,\|\partial_t\bbeta_{p}\|_{\bV_p} .\label{eq:bound-unsteady-state-solution-5}
\end{align}
\noindent{\bf Bounds on time derivatives on the right-hand side of \eqref{eq:bound-unsteady-state-solution-5}.}

\noindent In order to bound the time derivative terms in \eqref{eq:bound-unsteady-state-solution-5}, we take a finite difference in time of the whole system \eqref{eq:continuous-weak-formulation-1}. 
In particular, given $t \in [0,T)$ and $s > 0$ with $t+s \le T$, let $\partial_t^s\phi := \frac{\phi(t+s) - \phi(t)}{s}$. 
Thus, applying the operator to \eqref{eq:continuous-weak-formulation-1}, and testing with
$(\btau_f, \bv_p, \bxi_p, \bv_f, w_p, \bchi_f, \bpsi, \xi) = (\partial_t^s\bsi_f, \partial_t^s\bu_p, \partial_t^s\partial_t^s\,\bbeta_p, \partial_t^s\bu_f, \partial_t^s p_p, \partial_t^s\bgamma_f$, $\partial_t^s\bvarphi, \partial_t^s\lambda)$, similarly to \eqref{eq:stability-bound-1}, we get 
\begin{align}
& \frac{1}{2}\,\partial_t\,\Big( \rho_f\,\|\partial_t^s\,\bu_f\|^2_{\bL^2(\Omega_f)} 
+ s_0\,\|\partial_t^s\,p_p\|^2_{\W_p} 
+ a^e_p(\partial_t^s\bbeta_p, \partial_t^s\bbeta_p) 
+ \rho_p\,\|\partial_t^s\partial_t^s\bbeta_p\|^2_{\bL^2(\Omega_p)}\Big) 
\nonumber \\
&\quad +\, a_f(\partial_t^s\bsi_f, \partial_t^s\bsi_f)
+ a^d_p(\partial_t^s\bu_p,\partial_t^s\bu_p) 
+ c_{\BJS}(\partial_t^s\partial_t^s\,\bbeta_p,\partial_t^s \bvarphi;\partial_t^s\partial_t^s\,\bbeta_p, \partial_t^s\bvarphi) 
+ \kappa_{\partial_t^s\bu_f}(\bu_f(t),\partial_t^s\bsi_f)  
\nonumber \\
&\quad +\, \kappa_{\bu_f(t+s)}(\partial_t^s\bu_f,\partial_t^s\bsi_f)
+ l_{\partial_t^s\bvarphi}(\bvarphi(t),\partial_t^s\bvarphi) 
+ l_{\bvarphi(t+s)}(\partial_t^s\bvarphi,\partial_t^s\bvarphi)
\nonumber \\
& =\, (\partial_t^s\f_p,\partial_t^s\partial_t^s\,\bbeta_p)_{\Omega_p} 
+ (\partial_t^s\f_f,\partial_t^s\bu_f)_{\Omega_f} 
+ (\partial_t^s q_p,\partial_t^s p_p)_{\Omega_p}\,.
\label{eq:stability-bound-7}
\end{align}
Next, we integrate from 0 to t $\in (0,T)$, use the coercivity bound in \eqref{eq:aep-coercivity-bound}, the non-negativity bounds of $a_f, a^d_p$ and $c_{\BJS}$ in Lemma \ref{lem:coercivity-properties-A-E2}, arguments similar to those in \eqref{eq:mono f bound1}--\eqref{eq:mono u bound2}, that is, by combining \eqref{eq:continuous-weak-formulation-1a} with the inf-sup condition of $B_f$ (cf. \eqref{eq:inf-sup-vf-chif}), the continuity of  $\kappa_{\bw_f}$ and $l_{\bzeta}$ (cf. \eqref{eq:continuity-cK-wf}, \eqref{eq:continuity-cL-zeta}), the fact that $(\bu_f(t),\bvarphi(t)):[0,T]\to \bW_{r_1,r_2}$ (cf. \eqref{eq:Wr-definition}), and take $s\to 0$ to obtain
\begin{align}\label{eq:bound-unsteady-state-solution-8}
&\ds \frac{\rho_f}{2}\,\|\partial_t \bu_f(t)\|^2_{\bL^2(\Omega_f)} 
+ \frac{s_0}{2}\,\| \partial_t p_p(t)\|^2_{\W_p} 
+ \frac{c_e}{2}\| \partial_t\bbeta_p(t)\|^2_{\bV_p} 
+ \frac{\rho_p}{2}\,\| \partial_{tt}\bbeta_p(t)\|^2_{\bL^2(\Omega_p)}
\nonumber \\[1ex]
&\ds\quad +\, \int^t_0 \Big( \frac{5}{32\mu}\,\| \partial_t\bsi^\rd_f\|^2_{\bbL^2(\Omega_f)}  
+ \mu\,k_{\min}\|\partial_t\bu_p\|^2_{\bL^2(\Omega_p)} 
+ c_I\,\sum^{n-1}_{j=1} \|( \partial_t\bvarphi - \partial_{tt}\,\bbeta_p)\cdot\bt_{f,j}\|^2_{\L^2(\Gamma_{fp})} \Big)\, ds 
\nonumber \\[1ex]
&\ds\leq\, \frac{\rho_f}{2}\,\|\partial_t\bu_f(0)\|^2_{\bL^2(\Omega_f)} 
+ \frac{s_0}{2}\,\|\partial_t p_p(0)\|^2_{\W_p} 
+ \frac{c_e}{2}\|\partial_t\bbeta_p(0)\|^2_{\bV_p} 
+ \frac{\rho_p}{2}\,\|\partial_{tt}\bbeta_p(0)\|^2_{\bL^2(\Omega_p)}  
\nonumber \\[1ex]
& \ds\quad +\, \int^t_0 \Big( (\partial_t\f_p,\partial_{tt}\,\bbeta_p)_{\Omega_p} 
+ (\partial_t\f_f,\partial_t\bu_f)_{\Omega_f} 
+ (\partial_t q_p,\partial_t p_p)_{\Omega_p} \Big)\, ds \,. 
\end{align}
In turn, differentiating in time \eqref{eq:continuous-weak-formulation-1a} and \eqref{eq:continuous-weak-formulation-1e} and applying the inf-sup condition of $\cB$ in Lemma \ref{lem:inf-sup-conditions} (cf. \eqref{eq:inf-sup-qp-xi}, \eqref{eq:inf-sup-vf-chif}) for $(\partial_t\bu_f, \partial_t p_p, \partial_t\bgamma_f, \partial_t\bvarphi, \partial_t\lambda)$, and the continuity of $\kappa_{\bw_f}$ (cf. \eqref{eq:continuity-cK-wf}), with $(\bu_f(t),\bvarphi(t)):[0,T]\to \bW_{r_1,r_2}$  (cf. \eqref{eq:Wr-definition}) as in \eqref{eq:bound-unsteady-state-solution-3}, to derive
\begin{equation}\label{eq:bound-unsteady-state-solution-14}
\int^t_0 \|(\partial_t\bu_f, \partial_t p_p, \partial_t\bgamma_f, \partial_t\bvarphi, \partial_t\lambda)\|^2 \, ds 
\,\leq\, C\,\int^t_0 \Big( \| \partial_t\bsi^\rd_f\|^2_{\bbL^2(\Omega_f)} + \|\partial_t\bu_p\|^2_{\bL^2(\Omega_p)} \Big)\,ds \,.
\end{equation}
Thus, using Cauchy--Schwarz and Young's inequalities on the right-hand side of \eqref{eq:bound-unsteady-state-solution-8} in combination with \eqref{eq:bound-unsteady-state-solution-14}, we get
\begin{align}\label{eq:bound-unsteady-state-solution-10}
&\ds \|\partial_t\bu_f(t)\|^2_{\bL^2(\Omega_f)} 
+ s_0\,\| \partial_t p_p(t)\|^2_{\W_p} 
+ \| \partial_t \bbeta_p(t)\|^2_{\bV_p} 
+ \| \partial_{tt}\bbeta_p(t)\|^2_{\bL^2(\Omega_p)}
+ \int^t_0 \Big( \| \partial_t\bsi^\rd_f\|^2_{\bbL^2(\Omega_f)} 
\nonumber \\[1ex]
&\ds\quad +\, \| \partial_t\bu_p\|^2_{\bL^2(\Omega_p)} + \sum^{n-1}_{j=1} \|( \partial_t\bvarphi - \partial_{tt}\,\bbeta_p)\cdot\bt_{f,j}\|^2_{\L^2(\Gamma_{fp})} + \|(\partial_t\bu_f, \partial_t p_p, \partial_t\bgamma_f, \partial_t\bvarphi, \partial_t\lambda)\|^2 \Big)\, ds \nonumber \\[1ex]
&\ds\leq\, C\,\Bigg( \,\int^t_0 \big(   \|\partial_t\f_f\|^2_{\bL^2(\Omega_f)} 
+ \|\partial_t q_p\|^2_{\L^2(\Omega_p)}\big)\,ds
+ \| \partial_t\bu_f(0)\|^2_{\bL^2(\Omega_f)} 
+ s_0\,\|\partial_t p_p(0)\|^2_{\W_p} 
\nonumber \\[1ex]
&\ds\quad +\, \|\partial_t\bbeta_p(0)\|^2_{\bV_p}
+ \|\partial_{tt}\bbeta_p(0)\|^2_{\bL^2(\Omega_p)} 
+ \int^t_0 \|\partial_{t}\,\f_p\|_{\bL^2(\Omega_p)}\|\partial_{tt}\bbeta_p\|_{\bL^2(\Omega_p)}\, ds  \Bigg)\,.
\end{align}

\noindent{\bf Bound on initial data.}

\noindent Recall that $(\bu_f(0), p_p(0), \bbeta_p(0), \partial_t\bbeta_p(0)) = (\bu_{f,0}, p_{p,0}, \bbeta_{p,0}, \bu_{s,0})$ is the initial data  given to us.
From \eqref{eq:continuous-weak-formulation-1b}, \eqref{eq:continuous-weak-formulation-1d} and \eqref{eq:continuous-weak-formulation-1f} at time $t=0$, integrating by parts backwardly \eqref{eq:continuous-weak-formulation-1d} in combination with the fact that the initial data satisfy \eqref{bdycond-4}--\eqref{bdycond-5} to cancel the terms on the interface, and choosing $(\bv_f,\bxi_p,w_p) = (\partial_t \bu_f(0), \partial_{tt}\bbeta_p(0), \partial_t p_p(0))$, we get
\begin{align}\label{eq:bound-unsteady-state-solution-11}
&\ds \rho_f\,\| \partial_t\bu_f(0)\|^2_{\bL^2(\Omega_f)} 
+ s_0\,\| \partial_t p_p(0)\|^2_{\W_p} 
+ \rho_p\,\| \partial_{tt}\bbeta_p(0)\|^2_{\bL^2(\Omega_p)} 
\,=\, (\partial_t\bu_f(0),\bdiv(\bsi_f)(0))_{\Omega_f} 
\nonumber \\[1ex]
&\ds\quad +\, ( \bdiv(A^{-1}(\be(\bbeta_p)) - \alpha_p\,p_p\,\bI)(0),\partial_{tt}\bbeta_p(0))_{\Omega_p}
-\alpha_p\,(\div(\partial_t\,\bbeta_p)(0), \partial_t p_p(0))_{\Omega_p}   
\nonumber \\[1ex]
&\ds\quad -\, (\partial_t p_p(0),\div(\bu_p)(0))_{\Omega_p}
+ (\f_f(0),\partial_t\bu_f(0))_{\Omega_f}
+ (\f_p(0), \partial_{tt}\bbeta_p(0))_{\Omega_p} 
+ (q_p(0),\partial_t p_p(0))_{\Omega_p} \,.
\end{align}
Then, applying Cauchy--Schwarz and Young's inequalities with appropriate weights on the right-hand side of \eqref{eq:bound-unsteady-state-solution-11}, we obtain
\begin{align*}
& \ds \|\partial_t\bu_f(0)\|^2_{\bL^2(\Omega_f)} 
+ s_0\,\| \partial_t p_p(0)\|^2_{\W_p} 
+ \| \partial_{tt}\bbeta_p(0)\|^2_{\bL^2(\Omega_p)}
\,\leq\, C\bigg( \|\bdiv(A^{-1}(\, \be(\bbeta_{p})))(0)\|^2_{\bL^2(\Omega_p)} 
\nonumber \\
&\ds\quad +\, \|p_{p}(0)\|^2_{\H^1(\Omega_p)}
+ \|\bdiv(\bsi_{f}(0))\|^2_{\bL^2(\Omega_f)}
+ \frac{1}{s_0}\,\|\partial_t\,\bbeta_p(0)\|^2_{\bV_p} 
+ \frac{1}{s_0}\,\|\div(\bu_p)(0)\|^2_{\bL^2(\Omega_p)}   
+ \|\f_f(0) \|^2_{\bL^2(\Omega_f)}
\nonumber \\
&\ds\quad +\, \|\f_p(0)\|^2_{\bL^2(\Omega_p)} 
+ \frac{1}{s_0}\|q_p(0)\|^2_{\L^2(\Omega_p)}\Bigg)
+ \delta\,\Bigg( \|\partial_t\bu_f(0)\|^2_{\bL^2(\Omega_f)} 
+ s_0\,\| \partial_t p_p(0)\|^2_{\W_p} 
+ \|\partial_{tt}\bbeta_p(0)\|^2_{\bL^2(\Omega_p)}\Bigg)\,,
\end{align*}
and taking $\delta$ small enough, using \eqref{eq:sol0-up0-pp0} and \eqref{eq:sol0-sigmaf0-uf0-gammaf0} to bound $\div(\bu_p)(0)$ and $\bdiv(\bsi_{f})(0)$, respectively, implies 
\begin{align}\label{eq:bound-unsteady-state-solution-13}
&\ds \|\partial_t\bu_f(0)\|^2_{\bL^2(\Omega_f)} 
+ s_0\,\|\partial_t p_p(0)\|^2_{\W_p} 
+ \| \partial_{tt}\bbeta_p(0)\|^2_{\bL^2(\Omega_p)} 
\nonumber \\
&\ds\quad \leq\, C\Bigg( \|\bdiv(\be(\bu_{f,0}))\|^2_{\bL^{2}(\Omega_f)}
+ \|\bdiv(\bu_{f,0}\otimes \bu_{f,0})\|^2_{\bL^2(\Omega_f)}
+ \|\bdiv(A^{-1}(\, \be(\bbeta_{p,0})))\|^2_{\bL^2(\Omega_p)} 
+ \|p_{p,0}\|^2_{\H^1(\Omega_p)}
\nonumber \\ 
&\ds\quad \quad+\, \frac{1}{s_0}\|\bu_{s,0}\|^2_{\bV_p} 
+ \frac{1}{s_0}\|\div(\bK\nabla p_{p,0})\|^2_{\bL^2(\Omega_p)} 
+ \|\f_f(0) \|^2_{\bL^2(\Omega_f)} 
+ \|\f_p(0)\|^2_{\bL^2(\Omega_p)} 
+ \frac{1}{s_0}\|q_p(0)\|^2_{\L^2(\Omega_p)}\Bigg)\,.
\end{align}

Finally, combining \eqref{eq:bound-unsteady-state-solution-6}, \eqref{eq:bound-unsteady-state-solution-10}, and \eqref{eq:bound-unsteady-state-solution-13}, using the Sobolev embedding of $\H^1(0,T)$ into $\L^\infty(0,T)$, and applying Lemma \ref{xing-lemma} in the context of the non-negative functions $H=\big( \|\bbeta_p\|^2_{\bV_p} + \|\partial_{tt}\bbeta_p\|^2_{\bL^2(\Omega_p)}\big)^{1/2}$ and $B=\|\partial_{t}\,\f_p\|_{\bL^2(\Omega_p)}$, with $R$ and $A$ representing the remaining terms, to control $\int^t_0 B(s)\,H(s)\,ds$, along with \eqref{eq:bound-unsteady-state-solution-5} and the Sobolev embedding of $\L^\infty(0,T)$ into $\L^2(0,T)$,
\begin{equation}\label{sobolev-L-infty-L2}
\|\Phi\|_{\L^2(0,T)} \,\leq\, \sqrt{T}\,\|\Phi\|_{\L^\infty(0,T)} \,,
\end{equation}
for any scalar or vector-valued function $\Phi$ to bound the terms $\|\bdiv(\bsi_f)\|_{\bL^{4/3}(\Omega_f)}$ and $\|\div(\bu_{p})\|_{\L^2(\Omega_p)}$ in $\L^2(0,T)$, applying \eqref{eq:tau-d-H0div-inequality}--\eqref{eq:tau-H0div-Xf-inequality}, and performing some algebraic manipulations, we obtain \eqref{eq:continuous-stability}.
\end{proof}


\section{Semidiscrete continuous-in-time approximation}\label{sec:Semidiscrete continuous-in-time approximation}

In this section we introduce and analyze the semidiscrete continuous-in-time approximation of \eqref{eq:continuous-weak-formulation-1}. 
We analyze its solvability by employing the strategy developed in Section \ref{sec:well-posedness-model}. 
In addition, we derive error estimates with rates of convergence. 

Let $\cT_h^f$ and $\cT_h^p$ be shape-regular and quasi-uniform affine finite element partitions of $\Omega_f$ and $\Omega_p$, respectively, where $h$ is the maximum element diameter. The two partitions may be non-matching along the interface $\Gamma_{fp}$. For the discretization, we consider the following conforming finite element spaces:
\begin{equation*}
\bbX_{fh}\times \bV_{fh}\times \bbQ_{fh}\subset \bbX_f \times \bV_f\times \bbQ_{f}, \quad
\bV_{ph}\subset\bV_p, \quad \bX_{ph}\times\W_{ph} \subset \bX_p\times\W_p \,.
\end{equation*}
We choose $(\bbX_{fh}, \bV_{fh}, \bbQ_{fh})$ to be any stable finite element spaces for mixed elasticity with weakly imposed stress symmetry, such as the
Amara--Thomas \cite{at1979}, PEERS \cite{abd1984}, Stenberg
\cite{stenberg1988}, Arnold--Falk--Winther \cite{afw2007,awanou2013},
or Cockburn--Gopalakrishnan--Guzman \cite{cgg2010} families of spaces. 
We take $(\bX_{ph},\W_{ph})$ to be any stable mixed finite element Darcy spaces, such as the Raviart--Thomas (RT) or Brezzi--Douglas--Marini (BDM) spaces \cite{Brezzi-Fortin}. 
We note that these spaces satisfy
\begin{equation}\label{eq: div-prop} 
\div(\bX_{ph})=\W_{ph}, \quad
\bdiv(\bbX_{fh}) = \bV_{fh}.
\end{equation}
For the Lagrange multipliers, we choose the conforming approximations
\begin{equation}\label{defn-Lambda-h}
\Lambda_{ph} \subset \Lambda_{p}\,,\quad 
\bLambda_{fh} \subset \bLambda_{f}\,,
\end{equation}
equipped with $\H^{1/2}$-norms as in \eqref{eq:H1/2-norms}. If the normal traces of the spaces $\bX_{ph},$ or $\bbX_{fh}$ contain piecewise polynomials in $\cP_k$ on simplices or $\cQ_k$ on cubes with $k \geq 1$, where $\cP_k$ denotes polynomials of total degree $k$ and $\cQ_k$ stands for polynomials of degree $k$ in each variable, we take
$\Lambda_{ph}$ and $\bLambda_{fh}$ to be continuous piecewise polynomials in $\cP_k$ or $\cQ_k$ on the traces of $\cT_h^p$ and $\cT_h^f$, respectively. In the case $k = 0$, we take $\Lambda_{ph}$ and $\bLambda_{fh}$ to be continuous piecewise polynomials in $\cP_1$ or $\cQ_1$ on grids obtained by coarsening by two the traces of $\cT_h^p$ and $\cT_h^f$, respectively. These choices guarantee the inf-sup conditions given below in Lemma ~\ref{lem: discrete inf-sup}.

The semidiscrete continuous-in-time approximation to \eqref{eq:continuous-weak-formulation-1} is: Find $(\bsi_{fh},\bu_{ph}, \bbeta_{ph}, \bu_{fh}, p_{ph}, \bgamma_{fh}, \bvarphi_{h},$ $\lambda_{h}) : [0,T]\to \bbX_{fh}\times \bX_{ph}\times \bV_{ph}\times \bV_{fh}\times \W_{ph}\times \bbQ_{fh}\times \bLambda_{fh}\times \Lambda_{ph}$, such that for a.e. $t\in (0,T)$:
\begin{align}
& \rho_f (\partial_t\,\bu_{fh},\bv_{fh})_{\Omega_f}+ a_f(\bsi_{fh},\btau_{fh})+b_{\bn_f}(\btau_{fh},\bvarphi_{h}) + b_f(\btau_{fh},\bu_{fh}) + b_\sk(\bgamma_{fh},\btau_{fh})+\kappa_{\bu_{fh}}(\bu_{fh}, \btau_{fh}) \nonumber \\ 
&\quad -\, b_f(\bsi_{fh},\bv_{fh}) - b_\sk(\bsi_{fh},\bchi_{fh})  \, = \, (\f_{f},\bv_{fh})_{\Omega_f}\,, \nonumber \\ 
& \rho_p(\partial_{tt}\bbeta_{ph},\bxi_{ph})_{\Omega_p} + a^e_p(\bbeta_{ph},\bxi_{ph})+ \alpha_p\,b_p(\bxi_{ph},p_{ph}) +c_{\BJS}(\partial_t\,\bbeta_{ph}, \bvarphi_{h};\bxi_{ph}, \bpsi_{h})- c_{\Gamma}(\bxi_{ph},\bpsi_{h};\lambda_{h}) \nonumber \\ 
&\quad -\,b_{\bn_f}(\bsi_{fh},\bpsi_{h})+l_{\bvarphi_{h}}(\bvarphi_{h},\bpsi_{h}) \, = \,(\f_{p},\bxi_{ph})_{\Omega_p}\,, \nonumber \\ 
& s_0 (\partial_t\,p_{ph},w_{ph})_{\Omega_p} +a^d_p(\bu_{ph},\bv_{ph}) +b_p(\bv_{ph},p_{ph})+b_{\bn_p}(\bv_{ph},\lambda_{h})  - \alpha_p\,b_p(\partial_t\,\bbeta_{ph},w_{ph})- b_p(\bu_{ph},w_{ph}) \nonumber \\ 
&\quad = \,(q_{p},w_{ph})_{\Omega_p}\,, \nonumber \\ 
& c_{\Gamma}(\partial_t\,\bbeta_{ph},\bvarphi_{h};\xi_{h})-b_{\bn_p}(\bu_{ph},\xi_{h})\, = \,0\,,  
\label{eq:NS-Biot-semiformulation-1}
\end{align}
for all $(\btau_{fh}, \bv_{ph}, \bxi_{ph}, \bv_{fh}, w_{ph}, \bchi_{fh}, \bpsi_{h}, \xi_{h})\in \bbX_{fh}\times \bX_{ph}\times \bV_{ph}\times \bV_{fh}\times \W_{ph}\times \bbQ_{fh}\times \bLambda_{fh}\times \Lambda_{ph}$. The system is complemented with 
initial conditions for $\bu_{fh}(0),p_{ph}(0),\bbeta_{ph}(0)$ and $\partial_t\,\bbeta_{ph}(0)$, which will be chosen as suitable approximations of $\bu_{f,0},p_{p,0},\bbeta_{p,0}$ and $\bu_{s,0}$ that give 
compatible initial data for all variables, constructed in Lemma~\ref{lem: discrete initial condition} below.

\subsection{Alternative semidiscrete formulation}

To establish the well-posedness of the semidiscrete formulation \eqref{eq:NS-Biot-semiformulation-1}, we follow the same strategy employed for its continuous counterpart.
For analysis purposes only, we consider an alternative semidiscrete formulation \eqref{eq:NS-Biot-formulation-2}.
Let $\bV_{ph}$ consist of polynomials of degree at most $s_{\bbeta_p}\geq 1$. We introduce the stress finite element space $\bSigma_{eh} \subset \bSigma_e$, as symmetric tensors with elements that are discontinuous polynomials of degree at most $s_{\bbeta_p}-1$:
\begin{equation*}
\bSigma_{eh} \,:=\, \Big\{ \bsi_e\in \bSigma_{e} :\quad\bsi_e|_{T\in \cT_h^p}\in \cP^{sym}_{s_{\bbeta_p}-1}(T)^{n\times n} \Big\}.
\end{equation*}
The alternative semidiscrete formulation is: Find $(\bsi_{fh}, \bu_{ph}, \bsi_{eh}, \bu_{fh}, p_{ph}, \bgamma_{fh}, \bu_{sh}, \bvarphi_{h}, \lambda_{h}) : [0,T]\to \bbX_{fh}\times \bX_{ph}\times \bSigma_{eh}\times \bV_{fh}\times \W_{ph}\times \bbQ_{fh}\times \bV_{ph}\times \bLambda_{fh}\times \Lambda_{ph}$, such that for a.e. $t\in (0,T)$:
\begin{align}
& a_f(\bsi_{fh},\btau_{fh}) + \kappa_{\bu_{fh}}(\bu_{fh}, \btau_{fh}) + a^d_p(\bu_{ph},\bv_{ph}) + a^s_p(\partial_t\,\bsi_{eh},\btau_{eh}) \nonumber\\ 
&\quad +\, b_f(\btau_{fh},\bu_{fh}) + b_p(\bv_{ph},p_{ph}) + b_\sk(\btau_{fh},\bgamma_{fh})
+ b_s(\btau_{eh},\bu_{sh}) + b_{\bn_f}(\btau_{fh},\bvarphi_{h}) + b_{\bn_p}(\bv_{ph},\lambda_{h})  
\,=\, 0\,, \nonumber\\ 
& \rho_f\,(\partial_t\,\bu_{fh},\bv_{fh})_{\Omega_f} + s_0\,(\partial_t\,p_{ph},w_{ph})_{\Omega_p}
+ \rho_p\,(\partial_t\bu_{sh},\bv_{sh})_{\Omega_p} \nonumber\\ 
&\quad +\, c_{\BJS}(\bu_{sh},\bvarphi_{h};\bv_{sh},\bpsi_{h}) + c_{\Gamma}(\bu_{sh},\bvarphi_{h};\xi_{h}) - c_{\Gamma}(\bv_{sh},\bpsi_{h};\lambda_{h}) 
+ \alpha_p\,b_p(\bv_{sh},p_{ph}) - \alpha_p\,b_p(\bu_{sh},w_{ph}) \nonumber\\  
&\quad -\, b_f(\bsi_{fh},\bv_{fh}) - b_p(\bu_{ph},w_{ph}) - b_\sk(\bsi_{fh},\bchi_{fh})  
- b_s(\bsi_{eh},\bv_{sh}) - b_{\bn_f}(\bsi_{fh},\bpsi_{h}) - b_{\bn_p}(\bu_{ph},\xi_{h}) \nonumber\\  
&\quad +\, l_{\bvarphi_{h}}(\bvarphi_{h},\bpsi_{h}) =\, (\f_f,\bv_{fh})_{\Omega_f} + (q_p,w_{ph})_{\Omega_p} + (\f_p,\bv_{sh})_{\Omega_p}\,, 
\label{eq:NS-Biot-semi-formulation-2}
\end{align}
for all
$(\btau_{fh}, \bv_{ph}, \btau_{eh}, \bv_{fh}, w_{ph}, \bchi_{fh}, \bv_{sh}, \bpsi_{h}, \xi_{h})\in \bbX_{fh}\times \bX_{ph}\times \bSigma_{eh}\times \bV_{fh}\times \W_{ph}\times \bbQ_{fh}\times \bV_{ph}\times \bLambda_{fh}\times \Lambda_{ph}$. 
The initial conditions for $\bu_{fh}(0)$, $p_{ph}(0)$, $\bu_{sh}(0)$, and $\bsi_{eh}(0)$ are approximations of $\bu_{f,0}$, $p_{p,0}$, $\bu_{s,0}$, and $\bsi_{e,0}$, respectively, chosen to ensure compatible initial data for all variables.

Now, we group the spaces, unknowns and test functions similarly to the continuous case:
\begin{equation*}
\begin{array}{c}
\ds \bQ_h := \bbX_{fh}\times \bX_{ph}\times \bSigma_{eh},\quad
\bS_h := \bV_{fh}\times \W_{ph}\times \bbQ_{fh}\times \bV_{ph}\times \bLambda_{fh}\times \Lambda_{ph}, \\ [1.5ex]
\ds \ubsi_h := (\bsi_{fh}, \bu_{ph}, \bsi_{eh})\in \bQ_h,\quad 
\ubu_h := (\bu_{fh}, p_{ph}, \bgamma_{fh}, \bu_{sh}, \bvarphi_{h}, \lambda_{h})\in \bS_h, \\[1ex]
\ds \ubtau_h := (\btau_{fh}, \bv_{ph}, \btau_{eh})\in \bQ_h,\quad 
\ubv_h := (\bv_{fh}, w_{ph}, \bchi_{fh}, \bv_{sh}, \bpsi_{h}, \xi_{h})\in \bS_h,
\end{array}
\end{equation*}
where the spaces $\bQ_h$ and $\bS_h$ are endowed with the norms defined in \eqref{norms}.
Thus, \eqref{eq:NS-Biot-semi-formulation-2} reads: Find $(\ubsi_{h},\ubu_{h}): [0,T]\to \bQ_{h}\times \bS_{h}$ such that 
for all $(\ubtau_h,\ubv_{h}) \in \bQ_{h}\times \bS_{h}$, and for a.e. $t\in (0,T),$
\begin{align}
\frac{\partial}{\partial\,t}\,\cE_1(\ubsi_{h})(\ubtau_{h}) + \cA(\ubsi_{h})(\ubtau_{h}) + \cB'(\ubu_{h})(\ubtau_{h}) + \cK_{\bu_{fh}}(\ubu_{h})(\ubtau_{h}) &= \bF(\ubtau_{h})\,, \nonumber\\ 
\frac{\partial}{\partial\,t}\,\cE_2(\ubu_{h})(\ubv_{h}) - \cB(\ubsi_{h})(\ubv_{h}) + \cC(\ubu_{h})(\ubv_{h})+\cL_{\bvarphi_{h}}(\ubu_{h})(\ubv_{h}) &= \bG(\ubv_{h})\,,
\label{eq:alternative-discrete-formulation-operator-form}
\end{align}
where the operators and functionals are defined in \eqref{operators-1} and \eqref{operators-3}.

Existence of a solution of \eqref{eq:alternative-discrete-formulation-operator-form} (cf. \eqref{eq:NS-Biot-semi-formulation-2}) follows by using Theorem \ref{thm:auxiliary-theorem} and similar arguments to the ones employed in Section \ref{sec:well-posedness-model}. In particular, we need to establish the range condition and construct compatible initial data. For the range condition, we consider the resolvent system: Find $(\ubsi_{h},\ubu_{h})\in \bQ_{h}\times \bS_{h}$, such that
\begin{align}
\cE_1(\ubsi_{h})(\ubtau_{h}) + \cA(\ubsi_{h})(\ubtau_{h}) + \cB'(\ubu_{h})(\ubtau_{h}) + \cK_{\bu_{fh}}(\ubu_{h})(\ubtau_{h}) &= \wh{\bF}(\ubtau_{h}), \nonumber\\ 
\cE_2(\ubu_{h})(\ubv_{h}) - \cB(\ubsi_{h})(\ubv_{h}) + \cC(\ubu_{h})(\ubv_{h})+\cL_{\bvarphi_{h}}(\ubu_{h})(\ubv_{h}) &=  \wh{\bG}(\ubv_{h}),\label{eq:T-auxiliary-discrete-problem-operator-A}
\end{align}
for all $(\ubtau_{h},\ubv_{h})\in \bQ_{h}\times \bS_{h}$, where $\wh{\bF} \in \bQ_{h}'$ and $\wh{\bG} \in \bS_{h}'$ are defined as
\begin{equation}\label{source-terms}
\begin{split}  
& \wh{\bF}(\ubtau_{h}) \,:=\, (\wh{\f}_{e},\btau_{eh})_{\Omega_p} \quad \forall\,\ubtau_{h} \in \bQ_{h}\,, \\ 
& \wh{\bG}(\ubv_{h}) \,:=\, (\wh{\f}_{f},\bv_{fh})_{\Omega_f} + (\wh{q}_{p},w_{ph})_{\Omega_p} + (\wh{\f}_{p},\bv_{sh})_{\Omega_p} \quad \forall\,\ubv_{h} \in \bS_{h}\,,
\end{split}
\end{equation}
for some $\wh{\f}_{e} \in \bbL^2(\Omega_p)$, $\wh{\f}_{f} \in \bL^2(\Omega_f)$, $\wh{q}_{p} \in \L^2(\Omega_p)$, and $\wh{\f}_{p} \in \bL^2(\Omega_p)$.

As in the continuous case, the well posedness of \eqref{eq:T-auxiliary-discrete-problem-operator-A} is established by formulating it as a fixed point problem and using the Banach fixed-point theorem. To that end, 
let $\bT_{\tt d}$ be the discrete version of the operator $\bT$ defined in \eqref{eq:definition-operator-T}, i.e., $\bT_{\tt d} : \bV_{fh} \times \bLambda_{fh} \to \bV_{fh}\times \bLambda_{fh}$ is such that
\begin{equation}\label{eq:definition-operator-T-h}
\bT_{\tt d}(\bw_{fh},\bzeta_h) \,:=\, (\bu_{fh},\bvarphi_h) \quad \forall\,(\bw_{fh},\bzeta_h)\in \bV_{fh}\times \bLambda_{fh},
\end{equation}
where $\ubu_h := (\bu_{fh}, p_{ph}, \bgamma_{fh}, \bu_{sh}, \bvarphi_h, \lambda_h)\in \bQ_h$ is the second component of the unique solution (to be confirmed below) of the problem: Find $(\ubsi_h,\ubu_h)\in \bQ_h\times \bS_h$, such that
\begin{align}
\cE_1(\ubsi_{h})(\ubtau_{h}) + \cA(\ubsi_{h})(\ubtau_{h}) + \cB'(\ubu_{h})(\ubtau_{h}) + \cK_{\bw_{fh}}(\ubu_{h})(\ubtau_{h}) &= \wh{\bF}(\ubtau_{h}), \nonumber\\ 
\cE_2(\ubu_{h})(\ubv_{h}) - \cB(\ubsi_{h})(\ubv_{h}) + \cC(\ubu_{h})(\ubv_{h})+\cL_{\bzeta_{h}}(\ubu_{h})(\ubv_{h}) &=  \wh{\bG}(\ubv_{h}),\label{eq:Td-lin}
\end{align}
for all $(\ubtau_{h},\ubv_{h})\in \bQ_{h}\times \bS_{h}$.
We proceed as in Sections \ref{sec:fixed-point-approach}--\ref{sec:domain-D-nonempty}.
We observe that the continuity of all bilinear forms in the discrete case follows directly from their continuous counterparts (cf. Lemma \ref{lem:cont}).
In addition, the discrete inf-sup conditions satisfied by the finite element spaces are established in the following lemma.
\begin{lem}\label{lem: discrete inf-sup}
There exist constants $\wt{\beta}_1, \wt{\beta}_2, \wt{\beta}_3 > 0$ such that
\begin{equation}\label{eq:discrete inf-sup-vs}
\sup_{\0\neq \btau_{eh}\in \bSigma_{eh}} \frac{b_s(\btau_{eh},\bv_{sh})}{\|\btau_{eh}\|_{\bSigma_{e}}} 
\geq \wt{\beta}_1\,\|\bv_{sh}\|_{\bV_{p}} \quad \forall\,\bv_{sh}\in \bV_{ph} \,,
\end{equation}
\begin{equation}\label{eq:discrete inf-sup-qp-xi}
\sup_{\0\neq \bv_{ph}\in \bX_{ph}} \frac{b_p(\bv_{ph},w_{ph}) + b_{\bn_p}(\bv_{ph},\xi_{h})}{\|\bv_{ph}\|_{\bX_{p}}} 
\geq \wt{\beta}_2\,\|(w_{ph},\xi_{h})\|_{\W_{p}\times \Lambda_{p}} \quad \forall\,(w_{ph},\xi_{h})\in \W_{ph}\times \Lambda_{ph} \,,
\end{equation}
and 
\begin{equation}\label{eq:discrete inf-sup-vf-chif}
\sup_{\0\neq \btau_{fh}\in \bbX_{fh}} \frac{B_f(\btau_{fh},(\bv_{fh},\bchi_{fh},\bpsi_{h}))}{\|\btau_{fh}\|_{\bbX_{f}}} 
\geq \wt{\beta}_3\,\|(\bv_{fh}, \bchi_{fh}, \bpsi_{h})\|_{\bV_{f}\times \bbQ_{f}\times \bLambda_{f}}\,,
\end{equation}
for all $(\bv_{fh},\bchi_{fh},\bpsi_{h})\in \bV_{fh}\times \bbQ_{fh}\times \bLambda_{fh}$ and $B_f$ defined in \eqref{eq:inf-sup-vf-chif}.
\end{lem}
\begin{proof}
For the proof of \eqref{eq:discrete inf-sup-vs} we refer the reader to \cite[eq. (5.18) in Theorem 5.1]{aeny2019}.
By combining a slight adaptation of \cite[eqs. (4.13) and (4.22)]{gos2011}, we deduce \eqref{eq:discrete inf-sup-qp-xi}. 
For the proof of \eqref{eq:discrete inf-sup-vf-chif}, since $(\bbX_{fh}, \bV_{fh}, \bbQ_{fh})$ is a stable elasticity triple, it follows that
\begin{equation}\label{bf-bsk-h1}
  \sup_{\0\neq \btau_{fh}\in \bbX_{fh},\btau_{fh}\bn_f = 0 \text{ on } \Gamma_{fp} }
  \frac{b_f(\btau_{fh},\bv_{fh}) + b_{\sk}(\btau_{fh},\bchi_{fh})}
{\|\btau_{fh}\|_{\bbH(\bdiv,\Omega_f)}} \geq C(\|\bv_{fh}\|_{\bL^2(\Omega_f)} + \|\bchi_{fh}\|_{\bbL^2(\Omega_f)}),
\end{equation}
where the constraint $\btau_{fh}\bn_f = 0$ on $\Gamma_f^N \cup \Gamma_{fp}$ can be handled as in \cite[Lemma~4.3]{msmfe-simpl}. Then, following the argument from 
\cite[Section 4.4.2]{gobs2021} and combining \cite[eqs. (4.28) and (4.29)]{gobs2021}, we conclude that
\begin{equation}\label{bf-bsk-h2}
\sup_{\0\neq \btau_{fh}\in \bbX_{fh},\btau_{fh}\bn_f = 0 \text{ on } \Gamma_{fp} }   
\frac{b_f(\btau_{fh},\bv_{fh}) + b_{\sk}(\btau_{fh},\bchi_{fh})}
{\|\btau_{fh}\|_{\bbX_{f}}} 
\geq C(\|\bv_{fh}\|_{\bV_f} + \|\bchi_{fh}\|_{\bbQ_f}).
\end{equation}
On the other hand, let $\bbX_{fh}^0:= \big\{\btau_{fh}\in \bbX_{fh}: b_f(\btau_{fh},\bv_{fh}) + b_{\sk}(\btau_{fh},\bchi_{fh}) = 0 \ \ \forall (\bv_{fh},\bchi_{fh}) \in \bV_{fh}\times \bbQ_{fh}\big\}$. Observing that $\|\btau_{fh}\|_{\bbX_{f}} \leq C\,\|\btau_{fh}\|_{\bbH(\bdiv;\Omega_{f})}$, we have
\begin{equation}\label{bnf-h}
\sup_{\0\neq \btau_{fh}\in \bbX_{fh}^0} \frac{b_{\bn_f}(\btau_{fh},\bpsi_h)}{\|\btau_{fh}\|_{\bbX_{f}}}
\geq \sup_{\0\neq \btau_{fh}\in \bbX_{fh}^0} \frac{b_{\bn_f}(\btau_{fh},\bpsi_h)}{\|\btau_{fh}\|_{\bbH(\bdiv;\Omega_{f})}}
\geq C \|\bpsi_h\|_{\bLambda_f}\,,
\end{equation}
where the last inequality follows from \cite[eqs. (5.26)--(5.29)]{gmor2014}. Then \eqref{eq:discrete inf-sup-vf-chif} follows from combining \eqref{bf-bsk-h2} and \eqref{bnf-h}.
\end{proof}

Define $\wt\beta := \min\left\{ \wt{\beta}_1, \wt{\beta}_2, \wt{\beta}_3\right\}$ (cf. \eqref{eq:discrete inf-sup-vs}, \eqref{eq:discrete inf-sup-qp-xi}, \eqref{eq:discrete inf-sup-vf-chif}) and the radii
\begin{equation}\label{eq:discrete r_1^0 r_2^0 defn}
r_{1\ttd}^0:= \frac{\mu\,\wt\beta}{3\,\rho_f\,n^{1/2}} \qan 
r_{2\ttd}^0 := \frac{\mu \wt\beta^2}{12\,C_{\cL}}\,.
\end{equation}
Proceeding as in the proof of Lemma \ref{thm:well-posedness-1},
we obtain the existence and uniqueness of a solution $(\ubsi_{h},\ubu_{h})\in \bQ_{h}\times \bS_{h}$ of the linearized problem \eqref{eq:Td-lin}.
The aforementioned result is stated next.
\begin{lem}\label{thm:discrete-well-posedness-1}
  For each $\wh{\f}_{f}\in \bL^2(\Omega_f), \wh{\f}_{p}\in \bL^2(\Omega_p), \wh{\f}_{e} \in \bbL^2(\Omega_p)$, and $\wh{q}_{p}\in \L^2(\Omega_p)$, the problem \eqref{eq:Td-lin}
  has a unique solution $(\ubsi_{h},\ubu_{h})\in \bQ_{h}\times \bS_{h}$ for each $(\bw_{fh},\bzeta_{h})\in \bV_{fh}\times \bLambda_{fh}$ such that $\|\bw_{fh}\|_{\bV_f} \leq r_{1\ttd}^0$ and $\|\bzeta_{h}\|_{\bLambda_f} \leq r_{2\ttd}^0$ (cf. \eqref{eq:discrete r_1^0 r_2^0 defn}).
  Moreover, there exists a constant $c_{\bT_{\ttd}} > 0$, independent of $\bw_{fh},\bzeta_{h}$ and the data $\wh{\f}_{f}, \wh{\f}_{p}, \wh{\f}_{e} $, and $\wh{q}_{p}$, such that
\begin{equation}\label{eq:discrete-ubsi-ubu-bound-solution}
\, \|(\ubsi_{h}, \ubu_{h})\|_{\bQ\times \bS} 
\,\leq\, {c_{\bT_{\ttd}}\,\Big\{ \|\wh{\f}_{f}\|_{\bL^2(\Omega_f)} +  \|\wh{\f}_{e}\|_{\bbL^2(\Omega_p)}+ \|\wh{\f}_{p}\|_{\bL^2(\Omega_p)} + \| \wh{q}_{p}\|_{\L^2(\Omega_p)}\Big\}}.
\end{equation}
\end{lem}
As an immediate consequence we have the following corollary.
\begin{cor}\label{eq:T-discrete-well-delfined}
Assume that the conditions of Lemma \ref{thm:discrete-well-posedness-1} are satisfied. The operator $\bT_{\ttd}$ is well defined and it satisfies
\begin{equation*}
\|\bT_{\tt d}(\bw_{fh},\bzeta_{h})\|_{\bV_f\times \bLambda_f} 
\,\leq\, {c_{\bT_{\tt d}}\,\Big\{ \|\wh{\f}_{f}\|_{\bL^2(\Omega_f)} +  \|\wh{\f}_{e}\|_{\bbL^2(\Omega_p)}+ \|\wh{\f}_{p}\|_{\bL^2(\Omega_p)} + \| \wh{q}_{p}\|_{\L^2(\Omega_p)} \Big\}}.
\end{equation*}
\end{cor}

We next state the following lemma establishing that $\bT_{\tt d}$ is a continuous operator. The proof employs similar arguments to the ones used in the proof of Lemma \ref{lem:T-contraction-mapping}.
\begin{lem}\label{lem:T-discrete-contraction-mapping}
Let $r_{1\ttd}\in (0,r_{1\ttd}^0]$ and $r_{2\ttd}\in (0,r_{2\ttd}^0]$, where $r_{1\ttd}^0$ and $r_{2\ttd}^0$ are defined in \eqref{eq:discrete r_1^0 r_2^0 defn}. Let $\bW_{r_{1\ttd},r_{2\ttd}}$ be the closed set defined by
\begin{equation}\label{eq:Wr-discrete-definition}
\bW_{r_{1\ttd},r_{2\ttd}} := \Big\{ (\bw_{fh},\bzeta_{h})\in \bV_{fh}\times \bLambda_{fh} \,:\quad \| \bw_{fh}\|_{\bV_f} \leq r_{1\ttd},\quad \|\bzeta_{h}\|_{\bLambda_f} \leq r_{2\ttd} \Big\} \,,
\end{equation}
and define $r_\ttd^0:= \min\big\{r_{1\ttd}^0, r_{2\ttd}^0\big\}$. 
Then, for all $(\bw_{fh},\bzeta_{h}), (\wt{\bw}_{fh},\wt{\bzeta}_{h}) \in \bW_{r_{1\ttd},r_{2\ttd}}$, there holds 
\begin{align*}
&\|\bT_{\tt d}(\bw_{fh},\bzeta_{h}) - \bT_{\tt d}(\wt{\bw}_{fh},\wt{\bzeta}_{h}) \|_{\bV_f\times \bLambda_f}  \nonumber \\[1ex]
&\quad \leq\, \frac{c_{\bT_{\tt d}}}{r_\ttd^0}\,\Big\{ \|\wh{\f}_{f}\|_{\bL^2(\Omega_f)} +  \|\wh{\f}_{e}\|_{\bbL^2(\Omega_p)}+ \|\wh{\f}_{p}\|_{\bL^2(\Omega_p)} + \| \wh{q}_{p}\|_{\L^2(\Omega_p)} \Big\}\|(\bw_{fh},\bzeta_{h}) - (\wt{\bw}_{fh},\wt{\bzeta}_{h})\|_{\bV_f\times \bLambda_f}\,. 
\end{align*}
\end{lem}

Next, using the arguments employed in Lemma~\ref{thm:well-posed-domain-D} we obtain the following well-posedness result for the discrete resolvent problem \eqref{eq:T-auxiliary-discrete-problem-operator-A}. 
\begin{lem}\label{thm:well-posed-discrete}
Let $\bW_{r_{1\ttd},r_{2\ttd}}$ be as in \eqref{eq:Wr-discrete-definition} and let $r_\ttd:=\min\big\{r_{1\ttd}, r_{2\ttd}\big\}$. 
Assume that the data satisfy
\begin{equation}\label{eq:discrete-T-maps-Wr-into-Wr}
c_{\bT_{\tt d}}\,\Big\{ \|\wh{\f}_{f}\|_{\bL^2(\Omega_f)} 
+ \|\wh{\f}_{e}\|_{\bbL^2(\Omega_p)}
+ \|\wh{\f}_{p}\|_{\bL^2(\Omega_p)} 
+ \|\wh{q}_{p}\|_{\L^2(\Omega_p)} \Big\} \,\leq\, r_\ttd\,.
\end{equation}	
Then, the discrete resolvent problem \eqref{eq:T-auxiliary-discrete-problem-operator-A} has a unique solution $(\ubsi_{h},\ubu_{h})\in \bQ_{h}\times \bS_{h}$ with $(\bu_{fh},\bvarphi_{h})\in \bW_{r_{1\ttd},r_{2\ttd}}$, and there holds
\begin{equation*}
\|(\ubsi_{h},\ubu_{h})\|_{\bQ\times \bS} 
\,\leq\, c_{\bT_{\tt d}}\,\Big\{ \|\wh{\f}_{f}\|_{\bL^2(\Omega_f)} 
+  \|\wh{\f}_{e}\|_{\bbL^2(\Omega_p)}
+ \|\wh{\f}_{p}\|_{\bL^2(\Omega_p)} 
+ \|\wh{q}_{p}\|_{\L^2(\Omega_p)}\Big\}\,.
\end{equation*} 
\end{lem}

The next step is to construct compatible discrete initial data, which is carried out in the following lemma.
\begin{lem}\label{lem: discrete initial condition}
Assume that the conditions of Lemma \ref{lem:sol0-in-M-operator} are satisfied.
Assume in addition that the data satisfy
\begin{align}
&\|\bu_{f,0}\|_{\bL^2(\Omega_f)}
+ \|\bdiv(\be(\bu_{f,0}))\|_{\bL^{2}(\Omega_f)} 
+ \|\bdiv(\bu_{f,0}\otimes\bu_{f,0})\|_{\bL^{2}(\Omega_f)} 
+ \|\bu_{s,0}\|_{\bH^1(\Omega_p)} \nonumber \\
&\quad +\, \|\bbeta_{p,0}\|_{\bH^1(\Omega_p)}
+ \|\bdiv(A^{-1}(\be(\bbeta_{p,0})))\|_{\bL^2(\Omega_p)}
+ \|p_{p,0}\|_{\H^1(\Omega_p)} 
+ \|\div(\bK\nabla p_{p,0})\|_{\bL^2(\Omega_p)}
\leq \frac{1}{\wh{C}_0}\frac{r_\ttd}{c_{\bT_{\tt d}}}\,,
\label{eq:extra-assumption-discrete}
\end{align}
where the constant $\wh{C}_0 > 0$ is introduced below in \eqref{eq: discrete initial data bound-1}.
Then, there exist $\ubsi_{h,0} := (\bsi_{fh,0}, \bu_{ph,0}, \bsi_{eh,0})\in \bQ_h$ and   $\ubu_{h,0} := (\bu_{fh,0},p_{ph,0},\bgamma_{fh,0},\bu_{sh,0},\bvarphi_{h,0}, \lambda_{h,0})\in \bS_h$  with $(\bu_{fh,0},\bvarphi_{h,0})\in \bW_{r_{1\ttd},r_{2\ttd}}$ (cf. \eqref{eq:Wr-discrete-definition}) such that 
\begin{align}
(\cE_1 + \cA)(\ubsi_{h,0}) + (\cB' + \cK_{\bu_{fh,0}})(\ubu_{h,0}) &= \wh\bF_{0} \qin \bQ_{h}'\,, \nonumber\\ 
-\,\cB(\ubsi_{h,0}) + ( \cE_2 + \cC +\cL_{\bvarphi_{h,0}})(\ubu_{h,0}) &= \wh\bG_{0} \qin \bS_{h}'\,, \label{eq:initial-discrete-data-system}
\end{align}
where $\wh\bF_{0}(\ubtau_h) := (\wh{\f}_{e,0},\btau_{eh})_{\Omega_p}$ and  $\wh\bG_{0}(\ubv_h) := (\wh{\f}_{f,0},\bv_{fh})_{\Omega_f} + (\wh{q}_{p,0},w_{ph})_{\Omega_p} + (\wh{\f}_{p,0},\bv_{sh})_{\Omega_p} \,\,  \forall\,(\ubtau_{h},\ubv_{h})\in \bQ_h\times \bS_h $
with some $\wh{\f}_{e,0} \in \bbL^2(\Omega_p)$, $\wh{\f}_{f,0} \in \bL^2(\Omega_f)$, $\wh{q}_{p,0} \in \L^2(\Omega_p)$, and $\wh{\f}_{p,0} \in \bL^2(\Omega_p)$ satisfying
\begin{equation}\label{eq:initial-discrete-data-bound-1}
c_{\bT_{\tt d}}\,\Big\{ \|\wh{\f}_{f,0}\|_{\bL^2(\Omega_f)} 
+ \|\wh{\f}_{e,0}\|_{\bbL^2(\Omega_p)}
+ \|\wh{\f}_{p,0}\|_{\bL^2(\Omega_p)} 
+ \|\wh{q}_{p,0}\|_{\L^2(\Omega_p)}\Big\} \,\leq\, r_{\ttd}\,.
\end{equation}
\end{lem}
\begin{proof}
The discrete initial data is defined as a suitable elliptic projection of the continuous initial data $(\ubsi_0,\ubu_0)$ constructed in Lemma \ref{lem:sol0-in-M-operator}. 
Indeed, let $(\bsi_{fh,0}, \bu_{ph,0}, \bsi_{eh,0}, \bu_{fh,0}, p_{ph,0}, \bgamma_{fh,0}, \bu_{sh,0}, \bvarphi_{h,0}, \lambda_{h,0}) \in \bbX_{fh}\times \bX_{ph}\times \bSigma_{eh}\times \bV_{fh}\times \W_{ph}\times \bbQ_{fh}\times \bV_{ph}\times \bLambda_{fh}\times \Lambda_{ph}$ be such that, for all $(\btau_{fh}, \bv_{ph}, \btau_{eh}, \bv_{fh}, w_{ph}, \bchi_{fh}, \linebreak \bv_{sh}, \bpsi_{h}, \xi_{h})\in \bbX_{fh}\times \bX_{ph}\times \bSigma_{eh}\times \bV_{fh}\times \W_{ph}\times \bbQ_{fh}\times \bV_{ph}\times \bLambda_{fh}\times \Lambda_{ph}$,
\begin{subequations}\label{eq:system-discrete-sol0-1}
\begin{align}
& \ds a_f(\bsi_{fh,0},\btau_{fh}) + b_f(\btau_{fh},\bu_{fh,0}) + b_{\sk}(\btau_{fh},\bgamma_{fh,0}) + b_{\bn_f}(\btau_{fh},\bvarphi_{h,0}) + \kappa_{\bu_{fh,0}}(\bu_{fh,0},\btau_{fh})\nonumber \\[1ex] 
& \ds \quad = a_f(\bsi_{f,0},\btau_{fh}) + b_f(\btau_{fh},\bu_{f,0}) + b_{\sk}(\btau_{fh},\bgamma_{f,0}) + b_{\bn_f}(\btau_{fh},\bvarphi_{0}) + \kappa_{\bu_{f,0}}(\bu_{f,0},\btau_{fh}) = 0\,, \label{eq:semi-discrete-weak-formulation-1a}  \\[1ex] 
& \ds \rho_f(\bu_{fh,0},\bv_{fh})_{\Omega_f} - b_f(\bsi_{fh,0},\bv_{fh})
= \rho_f(\bu_{f,0},\bv_{fh})_{\Omega_f} - b_f(\bsi_{f,0},\bv_{fh}) \nonumber \\[1ex]
& \ds\quad = (\rho_f\,\bu_{f,0} - \bdiv(2\mu\be(\bu_{f,0}) - \rho_f(\bu_{f,0}\otimes \bu_{f,0})),\bv_{fh})_{\Omega_f}\,, \label{eq:semi-discrete-weak-formulation-1b}  \\[1ex] 
& \ds -\,b_{\sk}(\bsi_{fh,0},\bchi_{fh}) =  -\,b_{\sk}(\bsi_{f,0},\bchi_{fh}) = 0 \,, \label{eq:semi-discrete-weak-formulation-1c} \\[1ex]
& \ds a^s_p(\bsi_{eh,0},\btau_{eh}) + b_s(\btau_{eh},\bu_{sh,0})= a^s_p(\bsi_{e,0},\btau_{eh}) + b_s(\btau_{eh},\bu_{s,0})
= ( \be(\bbeta_{p,0}) - \be(\bu_{s,0}),\btau_{eh})_{\Omega_p}\,,
\label{eq:semi-discrete-weak-formulation-1d} \\[1ex] 
& \ds \rho_p(\bu_{sh,0},\bv_{sh})_{\Omega_p} - b_s(\bsi_{eh,0},\bv_{sh})  + c_{\BJS}(\bu_{sh,0},\bvarphi_{h,0};\bv_{sh},\0) - c_{\Gamma}(\bv_{sh},\0;\lambda_{h,0}) + \alpha_p\,b_p(\bv_{sh},p_{ph,0}) \nonumber \\[1ex] 
& \ds \quad =
\rho_p(\bu_{s,0},\bv_{sh})_{\Omega_p} - b_s(\bsi_{e,0},\bv_{sh})  + c_{\BJS}(\bu_{s,0},\bvarphi_{0};\bv_{sh},\0) - c_{\Gamma}(\bv_{sh},\0;\lambda_{0}) + \alpha_p\,b_p(\bv_{sh},p_{p,0}) \nonumber \\[1ex]
&\ds \quad = (\rho_p\,\bu_{s,0} - \bdiv(A^{-1}(\be(\bbeta_{p,0})) - \alpha_p\,p_{p,0}\,\bI),\bv_{sh})_{\Omega_p}\,, \label{eq:semi-discrete-weak-formulation-1e} \\[1ex] 
& \ds a^d_p(\bu_{ph,0},\bv_{ph}) + b_p(\bv_{ph},p_{ph,0}) + b_{\bn_p}(\bv_{ph},\lambda_{h,0}) \nonumber\\[1ex] 
& \ds \quad = a^d_p(\bu_{p,0},\bv_{ph}) + b_p(\bv_{ph},p_{p,0}) + b_{\bn_p}(\bv_{ph},\lambda_{0}) = 0\,, \label{eq:semi-discrete-weak-formulation-1f} \\[1ex] 
& \ds s_0(p_{ph,0},w_{ph})_{\Omega_p} - \alpha_p b_p(\bu_{sh,0},w_{ph})  - b_p(\bu_{ph,0},w_{ph}) 
= s_0(p_{p,0},w_{ph})_{\Omega_p} - \alpha_p b_p(\bu_{s,0},w_{ph})
- b_p(\bu_{p,0},w_{ph}) \nonumber \\[1ex] 
& \ds \quad = ( s_0\,p_{p,0} + \alpha_p \div (\,\bu_{s,0})
- \frac{1}{\mu}\,\div(\bK\nabla p_{p,0}),w_{ph})_{\Omega_p}\,, \label{eq:semi-discrete-weak-formulation-1g} \\[1ex] 
& \ds c_{\Gamma}(\bu_{sh,0},\bvarphi_{h,0};\xi_h) - b_{\bn_p}(\bu_{ph,0},\xi_h) = c_{\Gamma}(\bu_{s,0},\bvarphi_{0};\xi_h) - b_{\bn_p}(\bu_{p,0},\xi_h) = 0\,, \label{eq:semi-discrete-weak-formulation-1h} \\[1ex] 
& \ds c_{\BJS}(\bu_{sh,0},\bvarphi_{h,0};\0,\bpsi_h) - c_{\Gamma}(\0,\bpsi_h;\lambda_{h,0}) - b_{\bn_f}(\bsi_{fh,0},\bpsi_h) +l_{\bvarphi_{h,0}}(\bvarphi_{h,0},\bpsi_h) \nonumber \\[1ex]
& \ds \quad =  c_{\BJS}(\bu_{s,0},\bvarphi_{0};\0,\bpsi_h) - c_{\Gamma}(\0,\bpsi_h;\lambda_{0}) - b_{\bn_f}(\bsi_{f,0},\bpsi_h) +l_{\bvarphi_{0}}(\bvarphi_{0},\bpsi_h) = 0\,. \label{eq:semi-discrete-weak-formulation-1i} 
\end{align}
\end{subequations}
We observe that \eqref{eq:system-discrete-sol0-1} fit in the form of the discrete resolvent system \eqref{eq:T-auxiliary-discrete-problem-operator-A} with source terms
\begin{align*}
& \wh{\f}_{f,0} = \rho_f\,\bu_{f,0} - \bdiv(2\mu\be(\bu_{f,0}) - \rho_f(\bu_{f,0}\otimes \bu_{f,0})) \,,\quad 
\wh{\f}_{e,0} =\be(\bbeta_{p,0}) - \be(\bu_{s,0}) \,, \\ \nonumber
& \wh{\f}_{p,0} = \rho_p\,\bu_{s,0} - \bdiv(A^{-1}(\be(\bbeta_{p,0})) - \alpha_p\,p_{p,0}\,\bI) \,,\quad 
\wh{q}_{p,0} = s_0\,p_{p,0} + \alpha_p \div (\bu_{s,0}) - \frac{1}{\mu}\,\div(\bK\nabla p_{p,0}) \,,
\end{align*}
which satisfy, for some $\wh{C}_0 > 0$,
\begin{align}
&\|\wh{\f}_{f,0}\|_{\bL^2(\Omega_f)} 
+ \|\wh{\f}_{e,0}\|_{\bbL^2(\Omega_p)} 
+ \|\wh{\f}_{p,0}\|_{\bL^2(\Omega_p)} 
+ \|\wh{q}_{p,0}\|_{\L^2(\Omega_p)} \nonumber \\
&\ds\quad \leq\,
\wh{C}_0\,\Big( \|\bu_{f,0}\|_{\bL^2(\Omega_f)}
+ \|\bdiv(\be(\bu_{f,0}))\|_{\bL^{2}(\Omega_f)} 
+ \|\bdiv(\bu_{f,0}\otimes\bu_{f,0})\|_{\bL^{2}(\Omega_f)} 
+ \|\bu_{s,0}\|_{\bV_p} \nonumber \\
&\qquad +\, \|\bbeta_{p,0}\|_{\bV_p}
+ \|\bdiv(A^{-1}(\be(\bbeta_{p,0})))\|_{\bL^2(\Omega_p)}
+ \|p_{p,0}\|_{\H^1(\Omega_p)} 
+ \|\div(\bK\nabla p_{p,0})\|_{\bL^2(\Omega_p)}\Big) \,, 
\label{eq: discrete initial data bound-1}
\end{align}
Thus, \eqref{eq:extra-assumption-discrete} guarantee \eqref{eq:initial-discrete-data-bound-1} and hence Lemma~\ref{thm:well-posed-discrete} yields the well-posedness of 
\eqref{eq:system-discrete-sol0-1} and 
\begin{equation}\label{eqn:discrete-soln-1-bound}
\|(\ubsi_{h,0}, \ubu_{h,0})\|_{\bQ\times \bS} 
\,\leq\, c_{\bT_{\ttd}}\,\Big\{ \|\wh{\f}_{f,0}\|_{\bL^2(\Omega_f)} 
+ \|\wh{\f}_{e,0}\|_{\bbL^2(\Omega_p)}
+ \|\wh{\f}_{p,0}\|_{\bL^2(\Omega_p)} + \| \wh{q}_{p,0}\|_{\L^2(\Omega_p)}\Big\} \,,
\end{equation}
which together with \eqref{eq: discrete initial data bound-1} and \eqref{eq:extra-assumption-discrete}, implies that  $(\bu_{fh,0},\bvarphi_{h,0})\in \bW_{r_{1\ttd},r_{2\ttd}}$, completing the proof.
\end{proof}

We are now in position to establish existence of a solution of the alternative discrete formulation \eqref{eq:alternative-discrete-formulation-operator-form}

\begin{lem}\label{lem:parabolic-solution-h}
For each 
$\f_f\in \W^{1,1}(0,T;\bL^2(\Omega_f))$, $q_p\in \W^{1,1}(0,T;\L^2(\Omega_p))$, and $\f_p\in \W^{1,1}(0,T;\bL^2(\Omega_p))$ satisfying for all $t \in [0,T]$
\begin{equation}\label{small-data-h}
\|\f_f(t)\|_{\bL^2(\Omega_f)} + \|\f_p(t)\|_{\bL^2(\Omega_p)} + \|q_p(t)\|_{\L^2(\Omega_p)} 
< \frac{r_\ttd}{c_{\bT_{\tt d}}\wh{C}_0}\,,
\end{equation}
where $r_\ttd \in (0,r_\ttd^0)$, with $r_\ttd^0$ and $c_{\bT_{\tt d}}$
from Lemmas \ref{lem:T-discrete-contraction-mapping} and \ref{thm:discrete-well-posedness-1}, respectively,  
and initial data $(\bu_{f,0},p_{p,0},\bbeta_{p,0}$, $\bu_{s,0})$
satisfying the assumption \eqref{eq:extra-assumption-discrete} of Lemma \ref{lem: discrete initial condition}, there exists a solution of \eqref{eq:alternative-discrete-formulation-operator-form}, $(\ubsi_h,\ubu_h): [0,T] \to \bQ_h\times\bS_h$ with $(\bu_{fh}(t),\bvarphi_h(t))\in \bW_{r_{1\ttd},r_{2\ttd}}$ (cf. \eqref{eq:Wr-discrete-definition}),
\begin{equation*}
(\bsi_{eh}, \bu_{fh}, p_{ph},\bu_{sh})\in \W^{1,\infty}(0,T;\bSigma_{eh})\times \W^{1,\infty}(0,T;\bV_{fh})\times \W^{1,\infty}(0,T;\W_{ph})\times\W^{1,\infty}(0,T; \bV_{ph})\,,
\end{equation*}
and $(\bsi_{eh}(0),\bu_{fh}(0),p_{ph}(0),\bu_{sh}(0)) = (\bsi_{eh,0},\bu_{fh,0},p_{ph,0},\bu_{sh,0})$, where the initial data are constructed in Lemma \ref{lem: discrete initial condition}.
\end{lem}

\begin{proof}
The proof follows the proof of Lemma~\ref{lem:parabolic-solution}.
We apply Theorem~\ref{thm:auxiliary-theorem} with $E$, $u$, $\cN$, $\cM$, and $E'_b$ being the discrete counterparts of the spaces and operators defined in \eqref{eq:defn-E-N-M} and \eqref{eq:defn-E'_b-D}, with domain 
$\cD := \{u \in E: (\cN + \cM)(u)\in \wt{E}'_b \}$,
and restricted range space
$\wt{E}'_b := \{ (\0,\0,\wh{\f}_{e},\wh{\f}_{f},\wh{q}_{p},\0,\wh{\f}_p,\0,0) \in E'_b: \mbox{\eqref{eq:discrete-T-maps-Wr-into-Wr} holds}\}$ (cf. \eqref{eq:domain-D-and-Eb-tilde}).
By assumption \eqref{small-data-h}, in problem \eqref{eq:alternative-discrete-formulation-operator-form} we have $\cF = (\0,\0,\0,\f_f,q_p,\0, \f_p,\0,0) \in \wt{E}'_b$. As in the proof of Lemma~\ref{lem:parabolic-solution}, $\cN$ is linear, symmetric, and monotone, while $\cM$ is monotone on the domain $\cD$.
The range condition $Rg(\cN+\cM) = \wt{E}'_b$ is established in Lemma~\ref{thm:well-posed-discrete}.
Finally, initial data $u_{h,0} \in \cD$ with $(\cN + \cM)(u_{h,0}) \in \wt{E}'_b$ is constructed in Lemma~\ref{lem: discrete initial condition}.
Hence, the statement follows directly by applying Theorem~\ref{thm:auxiliary-theorem} in this discrete setting.
\end{proof}


\subsection{Existence and uniqueness of a solution of the semidiscrete method}

We define $\bbeta_{ph,0} \in \bV_{ph}$ as the unique solution to the problem
\begin{equation}\label{discrete etap0-bound}
a^e_p(\bbeta_{ph,0},\bxi_{ph}) = - b_s(\bsi_{eh,0},\bxi_{ph})\quad \forall\, \bxi_{ph} \in \bV_{ph} \,,
\end{equation}
where $\bsi_{eh,0}$ satisfies problem \eqref{eq:system-discrete-sol0-1}. Notice that \eqref{discrete etap0-bound} is well-posed by a direct application of Lax-Milgram theorem and satisfies
\begin{equation}\label{discrete etap0-bound-1}
\|\bbeta_{ph,0}\|_{\bV_{p}} \,\leq\, C\,\|\bsi_{eh,0}\|_{\bSigma_{e}} \,.
\end{equation}
Thus, as in the continuous case \eqref{eq:etap-us-relation}, we can recover the displacement solution from 
\begin{equation*}
\bbeta_{ph}(t) = \bbeta_{ph,0} + \int_0^t \bu_{sh}(s) \, ds, \quad \forall \, t \in [0,T] \,.
\end{equation*}

Now, we establish the well-posedness of the semidiscrete method \eqref{eq:NS-Biot-semiformulation-1} 
and the corresponding stability bound.
\begin{thm}\label{thm: well-posedness main result semi}
For each compatible discrete initial data satisfying Lemma \ref{lem: discrete initial condition} and
\begin{equation*}
\f_f\in \W^{1,1}(0,T;\bL^2(\Omega_f)),\quad \f_p\in \W^{1,1}(0,T;\bL^2(\Omega_p)),\quad q_p\in \W^{1,1}(0,T;\L^2(\Omega_p)),
\end{equation*}
under the assumptions of Lemma~\ref{lem:parabolic-solution-h}, there exists a unique solution of \eqref{eq:NS-Biot-semiformulation-1}, $(\bsi_{fh},\bu_{ph},\bbeta_{ph},$ 
$\bu_{fh}, p_{ph}, \bgamma_{fh}, \\ \bvarphi_h, \lambda_h) : [0,T]\to \bbX_{fh}\times \bX_{ph}\times\bV_{ph}\times\bV_{fh}\times \W_{ph}\times \bbQ_{fh}\times \bLambda_{fh}\times \Lambda_{ph}$  with $(\bu_{fh}(0), p_{ph}(0), \bbeta_{ph}(0), \partial_t\bbeta_{ph}(0)) = (\bu_{fh,0}, p_{ph,0}, \bbeta_{ph,0}, \bu_{sh,0})$ and 
$(\bu_{fh}(t),\bvarphi_{h}(t)):[0,T]\to \bW_{r_{1\ttd},r_{2\ttd}}$. 
In addition, $\bsi_{fh}(0) = \bsi_{fh,0}, \bu_{ph}(0) = \bu_{ph,0},  \bgamma_{fh}(0) = \bgamma_{fh,0}, \bvarphi_{h}(0) = \bvarphi_{h,0}$, and $\lambda_{h}(0) = \lambda_{h,0}$. 
Moreover, assuming sufficient regularity of the data, there exists a positive constant $C$, independent of $h$ and $s_0$, such that
\begin{align}
&\|\bsi_{fh}\|_{\L^2(0,T;\bbX_f)} 
+ \|\partial_t\bsi^\rd_{fh}\|_{\L^2(0,T;\bbL^2(\Omega_f))}
+ \|\bu_{ph}\|_{\L^2(0,T;\bX_p)} 
+ \| \partial_t\bu_{ph}\|_{\L^2(0,T;\bL^2(\Omega_p))}
+ \|\bbeta_{ph}\|_{\L^\infty(0,T;\bV_p)} 
\nonumber \\
&\quad +\, \|\partial_t\bbeta_{ph}\|_{\L^\infty(0,T;\bL^2(\Omega_p))}
+ \sum^{n-1}_{j=1} \|(\bvarphi_h-\partial_t\,\bbeta_{ph})\cdot\bt_{f,j}\|_{\H^1(0,T;\L^2(\Gamma_{fp}))}
+ \|\partial_{tt}\bbeta_{ph}\|_{\L^\infty(0,T;\bL^2(\Omega_p))}
\nonumber \\[1ex]
&\quad +\, \|\bu_{fh}\|_{\H^1(0,T;\bV_f)}
+ \|\bu_{fh}\|_{\W^{1,\infty}(0,T;\bL^2(\Omega_f))}
+ \sqrt{s_0}\,\|p_{ph}\|_{\W^{1,\infty}(0,T;\W_p)}
+ \|p_{ph}\|_{\H^1(0,T;\W_p)}
\nonumber \\[1ex]
&\quad +\, \|\bgamma_{fh}\|_{\H^1(0,T;\bbQ_f)} 
+ \|\bvarphi_h\|_{\H^1(0,T;\bLambda_f)} 
+ \|\lambda_h\|_{\H^1(0,T;\Lambda_p)} 
\nonumber \\
&\leq\, C \sqrt{T}\,\Bigg( \|\f_f\|_{\H^1(0,T;\bL^2(\Omega_f))}  
+ \|\f_p\|_{\H^1(0,T;\bL^2(\Omega_p))}
+ \|q_p\|_{\H^1(0,T;\L^2(\Omega_p))}
+ \frac{1}{\sqrt{s_0}}\|q_p(0)\|_{\L^2(\Omega_p)}
\nonumber \\
& \quad +\, \|\bu_{f,0}\|_{\bL^2(\Omega_f)}
+ \|\bdiv( \be(\bu_{f,0}))\|_{\bL^{2}(\Omega_f)}
+ \|\bdiv(\bu_{f,0}\otimes \bu_{f,0})\|_{\bL^2(\Omega_f)} 
+ \sqrt{s_0}\,\|p_{p,0}\|_{\W_p}
\nonumber \\
& \quad +\, \|p_{p,0}\|_{\H^1(\Omega_p)}
+ \frac{1}{\sqrt{s_0}}\|\div(\bK\nabla p_{p,0})\|_{\bL^2(\Omega_p)}
+ \|\bbeta_{p,0}\|_{\bV_p} 
+ \|\bdiv(A^{-1}(\, \be(\bbeta_{p,0})))\|_{\bL^2(\Omega_p)} 
\nonumber \\
&\quad 
+\, \left(1+\frac{1}{\sqrt{s_0}}\right)\|\bu_{s,0}\|_{\bV_p}  
\Bigg) \,.
\label{eq:discrete-stability}
\end{align} 
\end{thm}
\begin{proof}
From the fact that $\bQ_h \subset \bQ$, $\bS_h \subset \bS,$ and $\div(\bX_{ph})=\W_{ph}, 
\bdiv(\bbX_{fh}) = \bV_{fh}$, considering $(\ubsi_{h,0},\ubu_{h,0})$ satisfying \eqref{eq:initial-discrete-data-system}, and employing the continuity properties of $\cK_{\bw_{fh}}, \cL_{\bzeta_{h}}$ (cf. \eqref{eq:continuity-cK-wf}, \eqref{eq:continuity-cL-zeta}) with $(\bu_{fh}(t),\bvarphi_{h}(t)):[0,T]\to \bW_{r_{1\ttd},r_{2\ttd}}$ (cf. \eqref{eq:Wr-discrete-definition}) and monotonicity properties of  $\cA, \cE_1, \cE_2$ and $\cC$ (cf. Lemma \ref{lem:coercivity-properties-A-E2}) and coercivity bounds in \eqref{eq:aep-coercivity-bound}, as well as discrete inf-sup conditions \eqref{eq:discrete inf-sup-vs}, \eqref{eq:discrete inf-sup-qp-xi}, and \eqref{eq:discrete inf-sup-vf-chif}, the proof is similar to the proofs of Theorems \ref{thm:unique soln} and \ref{thm: continuous stability}.
We note that the proof of Theorem \ref{thm:unique soln} works at the discrete level due to the choice of the discrete initial data as the elliptic projection of the continuous initial data (cf. \eqref{eq:system-discrete-sol0-1}).
Note also that the discrete version of the stability bounds \eqref{eq:bound-unsteady-state-solution-6} and \eqref{eq:bound-unsteady-state-solution-10} can be derived following the proof of Theorem \ref{thm: continuous stability}, but with the corresponding discrete initial data on the right-hand side. Thus, we need to bound
\begin{equation}\label{eq:bound-discrete-initial-solution-1}
\|\bu_{fh}(0)\|^2_{\bL^2(\Omega_f)} 
+ s_0\,\|p_{ph}(0)\|^2_{\W_p} 
+ \|\bbeta_{ph}(0)\|^2_{\bV_p} 
+ \| \partial_t\bbeta_{ph}(0)\|^2_{\bL^2(\Omega_p)}  
\end{equation}
and
\begin{equation}\label{eq:bound-discrete-initial-solution-2}
\| \partial_t\bu_{fh}(0)\|^2_{\bL^2(\Omega_f)} 
+ s_0\,\|\partial_t p_{ph}(0)\|^2_{\W_p} 
+ \|\partial_t\bbeta_{ph}(0)\|^2_{\bV_p}
+ \|\partial_{tt}\bbeta_{ph}(0)\|^2_{\bL^2(\Omega_p)} \,.
\end{equation}
in terms of the corresponding continuous initial data.
In particular, the bound for $\|\bbeta_{ph}(0)\|^2_{\bV_p}$ follows by combining \eqref{discrete etap0-bound-1} with \eqref{eq: discrete initial data bound-1}--\eqref{eqn:discrete-soln-1-bound}.
Noting that $\partial_t\bbeta_{ph}(0) = \bu_{sh,0}$, the remaining terms in \eqref{eq:bound-discrete-initial-solution-1} and the third one in \eqref{eq:bound-discrete-initial-solution-2} can be bounded also by \eqref{eq: discrete initial data bound-1}--\eqref{eqn:discrete-soln-1-bound},
whereas to bound $\|\partial_t\bu_{fh}(0)\|^2_{\bL^2(\Omega_f)}$ and  
$s_0\,\|\partial_t p_{ph}(0)\|^2_{\W_p}$ in \eqref{eq:bound-discrete-initial-solution-2}, we first consider \eqref{eq:NS-Biot-semiformulation-1} for the test function $(\bv_{fh}, w_{ph})$ at $t=0$, and use \eqref{eq:semi-discrete-weak-formulation-1b} and \eqref{eq:semi-discrete-weak-formulation-1g}, to get 
\begin{align}\label{eq: discrete initial 1}
&  \rho_f (\partial_t\,\bu_{fh}(0),\bv_{fh})_{\Omega_f} +  s_0 (\partial_t\,p_{ph}(0),w_{ph})_{\Omega_p} 
 \nonumber \\[1ex]
&\quad \,=\, (\rho_f\,\bu_{fh,0} - \rho_f\,\bu_{f,0} +  \bdiv(2\mu\be(\bu_{f,0}) - \rho_f(\bu_{f,0}\otimes \bu_{f,0})),\bv_{fh})_{\Omega_f}
+ (\f_{f}(0),\bv_{fh})_{\Omega_f} \nonumber \\[1ex]
&\qquad +\,(s_0\,p_{ph,0} - s_0\,p_{p,0} - \alpha_p \div(\bu_{s,0})
+ \frac{1}{\mu}\,\div(\bK\nabla p_{p,0}),w_{ph})_{\Omega_p} 
+ (q_{p}(0),w_{ph})_{\Omega_p} \,.
\end{align}
Similarly, to bound $\|\partial_{tt}\bbeta_{ph}(0)\|^2_{\bL^2(\Omega_p)}$, we consider \eqref{eq:NS-Biot-semiformulation-1} at $t=0$ for the test function $\bxi_{ph}$ and use \eqref{discrete etap0-bound} and \eqref{eq:semi-discrete-weak-formulation-1e}, to obtain
\begin{align}\label{eq: discrete initial 4}
&\rho_p(\partial_{tt}\bbeta_{ph}(0),\bxi_{ph})_{\Omega_p} 
=  (\rho_p\,\bu_{sh,0} - \rho_p\,\bu_{s,0} + \bdiv(A^{-1}(\be(\bbeta_{p,0})) - \alpha_p\,p_{p,0}\,\bI),\bxi_{ph})_{\Omega_p} + (\f_{p}(0),\bxi_{ph})_{\Omega_p} \,.
\end{align}
Thus, taking $(\bv_{fh}, w_{ph}) = (\partial_t\bu_{fh}(0), \partial_t p_{ph}(0))$ and $\bxi_{ph} = \partial_{tt}\bbeta_{ph}(0)$ in \eqref{eq: discrete initial 1} and \eqref{eq: discrete initial 4}, respectively, using Cauchy--Schwarz and Young's inequalities and some simple manipulations, we deduce 
\begin{align}\label{eq: discrete initial 2}
&  \|\partial_t\,\bu_{fh}(0)\|^2_{\bL^2(\Omega_f)} 
+  s_0 \|\partial_t\,p_{ph}(0)\|^2_{\L^2(\Omega_p)}  
+ \|\partial_{tt}\bbeta_{ph}(0)\|^2_{\bL^2(\Omega_p)} \nonumber \\[1ex]
&\quad \leq\, C\,\Big( 
\|\bu_{f,0}\|^2_{\bL^2(\Omega_f)} 
+ \|\bdiv(\be(\bu_{f,0}))\|^2_{\bL^{2}(\Omega_f)}
+ \|\bdiv(\bu_{f,0}\otimes\bu_{f,0})\|^2_{\bL^{2}(\Omega_f)}
+ \|p_{p,0}\|^2_{\L^2(\Omega_p)}
+ \frac{1}{s_0} \|\bu_{s,0}\|^2_{\bV_p} \nonumber \\[1ex]
&\quad +\, \frac{1}{s_0}\|\div(\bK\nabla p_{p,0})\|^2_{\bL^2(\Omega_p)}
+ \|\bu_{s,0}\|^2_{\bL^2(\Omega_p)} 
+ \|\bdiv(A^{-1}(\be(\bbeta_{p,0})))\|^2_{\bL^2(\Omega_p)} 
+ \|p_{p,0}\|^2_{\H^1(\Omega_p)}
+ \|\f_{f}(0)\|^2_{\bL^2(\Omega_f)}  \nonumber \\[1ex]
&\quad +\,\|\f_{p}(0)\|^2_{\bL^2(\Omega_p)}
+ \frac{1}{s_0}\|q_p(0)\|^2_{\L^2(\Omega_p)} 
+ \|\bu_{fh,0}\|^2_{\bL^2(\Omega_f)} 
+ \|\bu_{sh,0}\|^2_{\bL^2(\Omega_p)} 
+ \|p_{ph,0}\|^2_{\L^2(\Omega_p)} \Big)  \,,
\end{align}
where, the last three terms can be bounded by \eqref{eq: discrete initial data bound-1}--\eqref{eqn:discrete-soln-1-bound}.

Finally, \eqref{eq:discrete-stability} follows by combining the discrete versions of \eqref{eq:bound-unsteady-state-solution-6} and \eqref{eq:bound-unsteady-state-solution-10} with \eqref{eq: discrete initial 2}, using the Sobolev embedding of $\H^1(0,T)$ into $\L^\infty(0,T)$, applying Lemma~\ref{xing-lemma} with the non-negative functions $H=\big( \|\bbeta_{ph}\|^2_{\bV_p} + \|\partial_{tt}\bbeta_{ph}\|^2_{\bL^2(\Omega_p)}\big)^{1/2}$ and $B=\|\partial_{t}\,\f_p\|_{\bL^2(\Omega_p)}$, \eqref{eq:tau-d-H0div-inequality}--\eqref{eq:tau-H0div-Xf-inequality}, and the arguments described at the end of the proof of Theorem \ref{thm: continuous stability}.
\end{proof}


\subsection{Error analysis}

We proceed with establishing rates of convergence.
To that end, let us set $\V\in \big\{ \W_p, \bV_f, \bbQ_f, \bSigma_{e} \big\}$, 
$\Lambda\in \big\{ \bLambda_f, \Lambda_p \big\}$ and let
$\V_h, \Lambda_h$ be the discrete counterparts. Let $P_h^{\V}: \V\to \V_h$ and 
$P_h^{\Lambda}: \Lambda\to \Lambda_h$ be the $\L^2$-projection operators, satisfying
\begin{equation}\label{eq:projection1}
( u - P_h^{\V}(u), v_{h} )_{\Omega_\star} = 0 \quad \forall\,v_h\in \V_{h}\,,\quad
\langle \varphi - P_h^{\Lambda}(\varphi), \psi_h \rangle_{\Gamma_{fp}} = 0  \quad\forall\, \psi_h\in \Lambda_{h}\,,
\end{equation}
where $\star\in \{f,p\}$, $u\in \big\{ p_p, \bu_f, \bgamma_f, \bsi_e \big\}$, $\varphi\in \big\{ \bvarphi, \lambda
\big\}$, and $v_h, \psi_h$ are the corresponding discrete test
functions. We have the approximation properties \cite{ciarlet1978}:
\begin{equation}\label{eq:approx-property1}
\|u - P^{\V}_h(u)\|_{\L^p(\Omega_\star)} 
\,\leq\, C\,h^{s_{u} + 1}\, \|u\|_{\W^{s_{u} + 1,p}(\Omega_\star)}\,,\quad
\|\varphi - P^{\Lambda}_h(\varphi)\|_{\Lambda} 
\,\leq\, C\,h^{s_{\varphi} + 1/2}\,
\|\varphi\|_{\H^{s_{\varphi} + 1}(\Gamma_{fp})} \,,
\end{equation}
where $p \in \{2,4\}$ and $s_{u}\in \big\{ s_{p_p}, s_{\bu_f}, s_{\bgamma_f}, s_{\bsi_e} \big\}$, $s_{\varphi}\in \big\{ s_{\bvarphi}, s_{\lambda} \big\}$ are the degrees of polynomials in the
spaces $\V_h$ and $\Lambda_h$, respectively. 

Next, denote $\X\in \big\{ \bbX_f, \bX_p \big\}$, $\sigma\in
\big\{ \bsi_f, \bu_p \big\}\in \X$ and let $\X_h$ and $\tau_h$ be
their discrete counterparts. 
Let $I^{\X}_h : \X \cap \H^{1}(\Omega_{\star})\to \X_{h}$ be the mixed finite element
projection operator \cite{Brezzi-Fortin} satisfying
\begin{equation}\label{eq:projection2}
(\div(I^{\X}_h(\sigma)), w_h)_{\Omega_\star} = (\div(\sigma), w_h)_{\Omega_\star}\,,\quad 
\pil I^{\X}_h(\sigma)\bn_{\star}, \tau_h\bn_{\star} \pir_{\Gamma_{fp}} = \pil \sigma \bn_{\star}, \tau_{h}\bn_{\star} \pir_{\Gamma_{fp}} \,, 
\end{equation}
for all $w_h\in \W_h$ and $\tau_h\in \X_h$, and
\begin{align}
&\|\sigma - I^{\X}_h(\sigma)\|_{\L^2(\Omega_{\star})} \,\leq\, C\,h^{s_{\sigma} + 1} \| \sigma \|_{\H^{s_{\sigma} + 1}(\Omega_{\star})},  \nonumber \\
&\|\div(\sigma - I^{\X}_h(\sigma))\|_{\L^q(\Omega_{\star})} \,\leq\, C\,h^{s_{\sigma} + 1} \|\div(\sigma)\|_{\W^{s_{\sigma} + 1,q}(\Omega_{\star})},\label{eq:approx-property2}
\end{align}
where $q \in \{2,4/3\}$,  $w_h\in \big\{ \bv_{fh}, w_{ph} \big\}$, $\W_h\in
\big\{ \bV_f, \W_p \big\}$, and $s_{\sigma}\in \big\{
s_{\bsi_f}, s_{\bu_p} \big\}$ -- the degrees of polynomials
in the spaces $\X_h$.

Finally, let $S_h^{\bV_p}$ be the Scott--Zhang interpolation operators onto $\bV_{ph}$, satisfying \cite{sz1990}\,:
\begin{equation}\label{eq: approx property 3}
\|\bbeta_p - S^{\bV_p}_h(\bbeta_p)\|_{\bH^1(\Omega_p)} 
\,\leq\, C\,h^{s_{\bbeta_p}} \| \bbeta_p \|_{\bH^{s_{\bbeta_p} + 1}(\Omega_p)} \,.
\end{equation}

Now, let $(\bsi_f,\bu_p, \bbeta_p, \bu_f, p_p, \bgamma_f, \bvarphi, \lambda)$ and $(\bsi_{fh},\bu_{ph}, \bbeta_{ph}, \bu_{fh}, p_{ph}, \bgamma_{fh}, \bvarphi_{h}, \lambda_{h})$ be the solutions of \eqref{eq:continuous-weak-formulation-1} and \eqref{eq:NS-Biot-semiformulation-1}, respectively.   We
introduce the error terms as the differences of these two solutions and
decompose them into approximation and discretization errors
using the interpolation operators:
\begin{align}
& \be_{\sigma} \,:=\, \sigma - \sigma_{h} \,=\,
(\sigma - I^{\X}_h(\sigma)) + (I^{\X}_h(\sigma) - \sigma_{h})
\,:=\, \be^I_{\sigma} + \be^h_{\sigma}\,,
\quad \sigma\in \big\{ \bsi_f, \bu_p \big\}\,, \nonumber \\
& \be_{\varphi} \,:=\, \varphi - \varphi_h
\,=\, (\varphi - P^{\Lambda}_h(\varphi)) + (P^{\Lambda}_h(\varphi) - \varphi_h)
\,:=\, \be^I_{\varphi} + \be^h_{\varphi}\,,
\quad \varphi\in \big\{ \bvarphi, \lambda \big\}\,,\nonumber \\
&\be_{u} \,:=\, u - u_{h} \,=\, (u - P^{\V}_h(u)) + (P^{\V}_h(u) - u_{h})
\,:=\, \be^I_{u} + \be^h_{u}\,, \quad
u\in \big\{ p_p, \bu_f, \bgamma_f, \bsi_e \big\}\,,\nonumber \\
& \be_{\bbeta_p} \,:=\, \bbeta_p - \bbeta_{ph} \,=\, (\bbeta_p - S^{\bV_p}_h(\bbeta_p)) + (S^{\bV_p}_h(\bbeta_p) - \bbeta_{ph})
\,:=\, \be^I_{\bbeta_p} + \be^h_{\bbeta_p} \,.
\label{eq:error-decomposition}
\end{align}
Then, we set the global errors endowed with the above decomposition:
\begin{equation*}
\be_{\ubsi} := (\be_{\bsi_f}, \be_{\bu_p}),\qan
\be_{\ubu} := (\be_{\bbeta_p}, \be_{\bu_f}, \be_{p_p}, \be_{\bgamma_f}, \be_{\bvarphi},  \be_{\lambda})\,.
\end{equation*}
We form the error equation by subtracting the discrete equations \eqref{eq:NS-Biot-semiformulation-1} from the continuous one  \eqref{eq:continuous-weak-formulation-1}:
\begin{align}
&\rho_f (\partial_t\be_{\bu_f},\bv_{fh})_{\Omega_f}+ a_f(\be_{\bsi_f},\btau_{fh})+b_{\bn_f}(\btau_{fh},\be_{\bvarphi}) + b_f(\btau_{fh},\be_{\bu_f}) + b_\sk(\be_{\bgamma_f},\btau_{fh}) \nonumber\\ 
&\quad +\, \kappa_{\bu_{f}}(\bu_{f}, \btau_{fh}) - \kappa_{\bu_{fh}}(\bu_{fh}, \btau_{fh})-b_f(\be_{\bsi_f},\bv_{fh}) - b_\sk(\be_{\bsi_f},\bchi_{fh})  \, = \,0,\nonumber\\ 
& \rho_p(\partial_{tt}\be_{\bbeta_p},\bxi_{ph})_{\Omega_p} + a^e_p(\be_{\bbeta_p},\bxi_{ph})+ \alpha_p\,b_p(\bxi_{ph},\be_{p_p}) +c_{\BJS}(\partial_t\be_{\bbeta_p}, \be_{\bvarphi};\bxi_{ph}, \bpsi_{h})- c_{\Gamma}(\bxi_{ph},\bpsi_{h};\be_{\lambda}) \nonumber\\ 
&\quad -\, b_{\bn_f}(\be_{\bsi_f},\bpsi_{h})+l_{\bvarphi}(\bvarphi,\bpsi_{h}) - l_{\bvarphi_{h}}(\bvarphi_{h},\bpsi_{h})\, = \,0,\nonumber\\ 
& s_0 (\partial_t\be_{p_p},w_{ph})_{\Omega_p} +a^d_p(\be_{\bu_p},\bv_{ph}) +b_p(\bv_{ph},\be_{p_p})+b_{\bn_p}(\bv_{ph},\be_{\lambda})  - \alpha_p\,b_p(\partial_t\be_{\bbeta_p},w_{ph}) \nonumber\\ 
&\quad -\, b_p(\be_{\bu_p},w_{ph}) \,=\, 0,\nonumber\\ 
&c_{\Gamma}(\partial_t\be_{\bbeta_p},\be_{\bvarphi};\xi_{h})-b_{\bn_p}(\be_{\bu_p},\xi_{h})\, = \,0,  \label{eq:NS-Biot-errorformulation-1}
\end{align}
for all $(\btau_{fh}, \bv_{ph}, \bxi_{ph}, \bv_{fh}, w_{ph}, \bchi_{fh}, \bpsi_{h}, \xi_{h})\in \bbX_{fh}\times \bX_{ph}\times \bV_{ph}\times \bV_{fh}\times \W_{ph}\times \bbQ_{fh}\times \bLambda_{fh}\times \Lambda_{ph}$.

\begin{rem}
Let $\wh{\beta}:= \min\left\{ \wh{\beta}_1, \wh{\beta}_2, \wh{\beta}_3\right\}$, with $\wh{\beta}_i = \min\big\{\beta_i, \wt{\beta}_i\big\}, i\in \{1,2,3\}$,
where $\beta_i$ and $\wt{\beta}_i$ are the inf-sup constants from Lemmas \ref{lem:inf-sup-conditions} and \ref{lem: discrete inf-sup}, respectively. 
It follows that $(\bu_f(t),\bvarphi(t)), (\bu_{fh}(t),\bvarphi_{h}(t)) \in \bW_{\wh r_1,\wh r_2} \quad \forall\,t\in [0,T]$, with $\bW_{\wh r_1,\wh r_2}$ defined by
\begin{equation}\label{eq:wh-Wr-definition}
\bW_{\wh r_1,\wh r_2} := \Big\{ (\bw_{f},\bzeta)\in \bV_{f}\times \bLambda_{f} \,:\quad \| \bw_{f}\|_{\bV_f} \leq \wh{r}_{1},\quad \|\bzeta\|_{\bLambda_f} \leq\wh{r}_{2} \Big\}\,,
\end{equation}
where $\wh{r}_1 \in (0,\wh{r}_{1}^0]$, $\wh{r}_{2}\in (0,\wh{r}_{2}^0]$, and $\wh{r}_1^0, \wh{r}_{2}^0$ are defined as $\wh{r}_{1}^0 := \min\{r_1^0,r_{1\ttd}^0\} = 
\mu\,\wh{\beta}/(3\,\rho_f\,n^{1/2})$ and  
$\wh{r}_{2}^0 := \min\{r_2^0,r_{2\ttd}^0\} = \mu \wh{\beta}^2/(12C_{\cL})$.
\end{rem}

We now establish the main result of this section. We make an additional assumption on the smoothness of the solution.

\begin{assumption}\label{extra-reg}
Assume that the solutions of the continuous and semidiscrete problems \eqref{eq:continuous-weak-formulation-1} and \eqref{eq:NS-Biot-semiformulation-1}, respectively, satisfy
\begin{equation}\label{eq:extra-solution-assumption}
\partial_t\,\bu_f \in \L^{\infty}(0,T;\bV_f), \ \
\partial_t\,\bvarphi \in \L^{\infty}(0,T;\bLambda_f), \ \
\partial_t\,\bu_{fh} \in \L^{\infty}(0,T;\bV_{fh}), \ \
\partial_t\,\bvarphi_h \in \L^{\infty}(0,T;\bLambda_{fh}).
\end{equation}
\end{assumption}
The increased regularity in \eqref{eq:extra-solution-assumption} can be obtained by assuming that the data are sufficiently smooth in time and using arguments similar to the proofs of Theorems \ref{thm: continuous stability} and \ref{thm: well-posedness main result semi}. 

\begin{thm}\label{thm: error analysis}
Let the assumptions in Theorem \ref{thm: well-posedness main result semi} and Assumption~\ref{extra-reg} hold. Consider the solutions of the continuous and semidiscrete problems \eqref{eq:continuous-weak-formulation-1} and \eqref{eq:NS-Biot-semiformulation-1}, respectively, and assume sufficient regularity of the exact solution as specified in \eqref{eq:approx-property1} and \eqref{eq:approx-property2}. Then, there exists a positive constant $C$, depending on the solution regularity but independent of $h$, such that
\begin{align}\label{eq:errror-rate-of-convergence}
&\|\be_{\bsi_f}\|_{\L^2(0,T;\bbX_f)} 
+ \| \partial_t(\be_{\bsi_f})^{\rd}\|_{\L^2(0,T;\bbL^2(\Omega_f))} 
+ \|\be_{\bu_p}\|_{\L^2(0,T;\bX_p)} 
+ \| \partial_t\be_{\bu_p}\|_{\L^2(0,T;\bL^2(\Omega_p))}
+ \|\be_{\bbeta_p}\|_{\L^\infty(0,T;\bV_p)} \nonumber \\
&\quad +\, \|\partial_t\be_{\bbeta_p}\|_{\L^\infty(0,T;\bL^2(\Omega_p))} 
+ \sum^{n-1}_{j=1} \|( \be_{\bvarphi}-\partial_t\be_{\bbeta_p})\cdot\bt_{f,j}\|_{\H^1(0,T;\L^2(\Gamma_{fp}))}
+ \|\partial_{tt}\be_{\bbeta_p}\|_{\L^\infty(0,T;\bL^2(\Omega_p))}
\nonumber \\
&\quad +\, \|\be_{\bu_f}\|_{\H^1(0,T;\bV_f)}
+ \|\be_{\bu_f}\|_{\W^{1,\infty}(0,T;\bL^2(\Omega_f))}
+ \sqrt{s_0}\,\|\be_{p_p}\|_{\W^{1,\infty}(0,T;\W_p)} 
+ \|\be_{p_p}\|_{\H^1(0,T;\W_p)} 
\nonumber \\[1ex]
&\quad +\, \|\be_{\bgamma_f}\|_{\H^1(0,T;\bbQ_f)} 
+ \|\be_{\bvarphi}\|_{\H^1(0,T;\bLambda_f)} 
+ \|\be_{\lambda}\|_{\H^1(0,T;\Lambda_p)}
\nonumber \\[1ex]
&\leq\, C\,\sqrt{T}\,\left( h^{s_{\underline{\sigma}}+1} + h^{s_{\bbeta_p}} 
+ h^{s_{\underline{u}}+1} +  h^{s_{\underline{\varphi}}+  1/2} \right) \,,
\end{align}
where $s_{\underline{\sigma}}= \min\{ s_{\bsi_f}, s_{\bu_p}\}$, $s_{\underline{u}}= \min\{ s_{p_p}, s_{\bu_f}, s_{\bgamma_f}, s_{\bsi_e} \}$, and  $s_{\underline{\varphi}}= \min\{ s_{\varphi}, s_{\lambda}\}$.
\end{thm}
\begin{proof}
  We start by taking $(\btau_{fh}, \bv_{ph}, \bxi_{ph}, \bv_{fh}, w_{ph}, \bchi_{fh}, \bpsi_{h}, \xi_{h}) = (\be_{\bsi_f}^h, \be_{\bu_p}^h, \partial_t\be_{\bbeta_p}^h, \be_{\bu_f}^h, \be_{p_p}^h, \be_{\bgamma_f}^h, \be_{\bvarphi}^h, \be_{\lambda}^h)$ in \eqref{eq:NS-Biot-errorformulation-1}, to obtain
\begin{align}\label{eq: error equation 1}
&\ds\frac{1}{2}\,\partial_t\,\Big( \rho_f\,\|\be_{\bu_f}^h\|^2_{\bL^2(\Omega_f)} 
+ s_0\,\|\be_{p_p}^h\|^2_{\W_p} 
+ a^e_p(\be_{\bbeta_p}^h,\be_{\bbeta_p}^h) 
+ \rho_p\,\|\partial_{t}\be_{\bbeta_p}^h\|^2_{\bL^2(\Omega_p)} \Big) 
+ a_f(\be_{\bsi_f}^h,\be_{\bsi_f}^h) 
+ a^d_p(\be_{\bu_p}^h,\be_{\bu_p}^h) \nonumber \\[1ex]
&\ds\quad +\, c_{\BJS}(\partial_t\be_{\bbeta_p}^h, \be_{\bvarphi}^h;\partial_t\be_{\bbeta_p}^h, \be_{\bvarphi}^h)
+ \kappa_{\bu_{fh}}(\be_{\bu_{f}}^h, \be_{\bsi_f}^h)
+ \kappa_{\be_{\bu_{f}}^h}(\bu_{f}, \be_{\bsi_f}^h)
+ l_{\bvarphi_h}(\be_{\bvarphi}^h,\be_{\bvarphi}^h) 
+ l_{\be_{\bvarphi}^h}(\bvarphi,\be_{\bvarphi}^h)\nonumber \\[1ex]
&\ds = -\,  a^e_p(\be_{\bbeta_p}^I,\partial_t\be_{\bbeta_p}^h) 
-\rho_p(\partial_{tt}\be_{\bbeta_p}^I,\partial_{t}\be_{\bbeta_p}^h)_{\Omega_p}
- a_f(\be_{\bsi_f}^I,\be_{\bsi_f}^h) 
- a^d_p(\be_{\bu_p}^I,\be_{\bu_p}^h) 
- c_{\BJS}(\partial_t\be_{\bbeta_p}^I, \be_{\bvarphi}^I;\partial_t\be_{\bbeta_p}^h, \be_{\bvarphi}^h) \nonumber \\[1ex]
&\ds\quad- \kappa_{\bu_{fh}}(\be_{\bu_{f}}^I, \be_{\bsi_f}^h)
\, -\, \kappa_{\be_{\bu_{f}}^I}(\bu_{f}, \be_{\bsi_f}^h) 
- l_{\bvarphi_h}(\be_{\bvarphi}^I,\be_{\bvarphi}^h) 
- l_{\be_{\bvarphi}^I}(\bvarphi,\be_{\bvarphi}^h) 
- b_\sk(\be_{\bsi_f}^h,\be_{\bgamma_f}^I) 
+ b_\sk(\be_{\bsi_f}^I,\be_{\bgamma_f}^h) \nonumber \\[1ex]
&\ds\quad - b_{\bn_f}(\be_{\bsi_f}^h,\be_{\bvarphi}^I)
\, +\, b_{\bn_f}(\be_{\bsi_f}^I,\be_{\bvarphi}^h)
- \alpha_p\,b_p(\partial_t\be_{\bbeta_p}^h,\be_{p_p}^I) 
+ \alpha_p\,b_p(\partial_t\be_{\bbeta_p}^I,\be_{p_p}^h) 
- b_{\bn_p}(\be_{\bu_p}^h, \be_{\lambda}^I) 
+ b_{\bn_p}(\be_{\bu_p}^I, \be_{\lambda}^h) \nonumber \\[1ex]
&\ds \quad - c_{\Gamma}(\partial_t\be_{\bbeta_p}^I,\be_{\bvarphi}^I;\be_{\lambda}^h) 
+ c_{\Gamma}(\partial_t\be_{\bbeta_p}^h,\be_{\bvarphi}^h;\be_{\lambda}^I)\,,
\end{align}
where, the right-hand side of \eqref{eq: error equation 1} has been simplified, 
since the projection properties \eqref{eq:projection1} and 
\eqref{eq:projection2}, and the fact that 
$\div(\bX_{ph})=\W_{ph}$,
$\bdiv(\bbX_{fh}) = \bV_{fh}$ (cf. \eqref{eq: div-prop}), 
imply that the following terms are zero:
\begin{equation*}
s_0 (\partial_t\be_{p_p}^I,\be_{p_p}^h)_{\Omega_p}, \,\, \rho_f (\partial_t\be_{\bu_f}^I,\be_{\bu_f}^h)_{\Omega_f}, \,\,  b_f(\be_{\bsi_f}^h,\be_{\bu_f}^I), \,\, b_f(\be_{\bsi_f}^I,\be_{\bu_f}^h),\,\,
b_p(\be_{\bu_p}^h,\be_{p_p}^I), \,\, b_p(\be_{\bu_p}^I,\be_{p_p}^h)\,.
\end{equation*}
Next, we discuss how to deal with the $\kappa_{\bw_f}$ and $l_{\bzeta}$ terms on the left-hand side of \eqref{eq: error equation 1}.
First, in analogy with \eqref{eq:mono f bound1}, we use the non-negativity of $a_f$ (cf. \eqref{eq:coercivity-af}), together with the continuity of $\kappa_{\bw_f}$ and $l_{\bzeta}$ (cf. \eqref{eq:continuity-cK-wf}, \eqref{eq:continuity-cL-zeta}) and the fact that $(\bu_f(t),\bvarphi(t)),(\bu_{fh}(t),\bvarphi_{h}(t)):[0,T]\to \bW_{\wh r_1,\wh r_2}$ (cf. \eqref{eq:wh-Wr-definition}), to deduce that
\begin{align}\label{eq:error-equation-4a}
&\ds
a_f(\be_{\bsi_f}^h,\be_{\bsi_f}^h) 
+ \kappa_{\bu_{fh}}(\be_{\bu_{f}}^h, \be_{\bsi_f}^h)
+ \kappa_{\be_{\bu_{f}}^h}(\bu_{f}, \be_{\bsi_f}^h)
+ l_{\bvarphi_h}(\be_{\bvarphi}^h,\be_{\bvarphi}^h) 
+ l_{\be_{\bvarphi}^h}(\bvarphi,\be_{\bvarphi}^h) \, \nonumber \\[1ex]
&\ds\quad \ge
  \ds\dfrac{1}{2\mu}\|(\be_{\bsi_f}^h)^{\rd}\|^2_{\bbL^2(\Omega_f)}   
- \frac{\wh{\beta}}{3}\,\|\be^h_{\bu_{f}}\|_{\bV_f}\|(\be_{\bsi_f}^h)^{\rd}\|_{\bbL^2(\Omega_f)}  
- \frac{\mu\,\wh{\beta}^2}{6}\,\|\be^h_{\bvarphi}\|^2_{\bLambda_f}.
\end{align}
Next, similarly to \eqref{eq:inf-sup-B1}, by employing the $\btau_{fh}$ row of \eqref{eq:NS-Biot-errorformulation-1} together with the discrete inf-sup condition for $B_f$ (cf. \eqref{eq:discrete inf-sup-vf-chif}), the continuity properties of $\kappa_{\bw_f}$ and $l_{\bzeta}$ (cf. \eqref{eq:continuity-cK-wf}, \eqref{eq:continuity-cL-zeta}), and the fact that $(\bu_f(t),\bvarphi(t)),(\bu_{fh}(t),\bvarphi_{h}(t)):[0,T]\to \bW_{\wh r_1,\wh r_2}$ (cf. \eqref{eq:wh-Wr-definition}), we obtain
\begin{align}\label{eq:error-equation-4b}
&\ds \dfrac{2\,\wh{\beta}}{3}\,\|(\be^h_{\bu_f}, \be^h_{\bgamma_f},\be^h_{\bvarphi})\| 
\,\leq\,
\big(\wh{\beta} - \frac{\rho_f\,n^{1/2}}{2\,\mu}(\|\bu_f\|_{\bV_f}+\|\bu_{fh}\|_{\bV_f})\big)\,\|(\be^h_{\bu_f}, \be^h_{\bgamma_f},\be^h_{\bvarphi})\| \nonumber \\[1ex]
&\ds\quad \leq\, \frac{1}{2\,\mu} \|(\be^h_{\bsi_f})^\rd\|_{\bbL^2(\Omega_f)}
+ C\,\Big( \|(\be^I_{\bsi_f})^\rd\|_{\bbL^2(\Omega_f)} + \|(\be^I_{\bu_f},\be^I_{\bgamma_{f}},\be^I_{\bvarphi})\| \Big) \,. 
\end{align}
Thus, by applying \eqref{eq:error-equation-4b} to bound both $\|\be^h_{\bu_{f}}\|_{\bV_f}$ and $\|\be^h_{\bvarphi}\|_{\bLambda_f}$ in \eqref{eq:error-equation-4a} by $\|(\be_{\bsi_f}^h)^{\rd}\|_{\bbL^2(\Omega_f)}$, using Young's inequality and performing some algebraic computations, we conclude that
\begin{align}\label{eq:error-equation-4}
  &\ds   a_f(\be_{\bsi_f}^h,\be_{\bsi_f}^h) 
+ \kappa_{\bu_{fh}}(\be_{\bu_{f}}^h, \be_{\bsi_f}^h)
+ \kappa_{\be_{\bu_{f}}^h}(\bu_{f}, \be_{\bsi_f}^h)
+ l_{\bvarphi_h}(\be_{\bvarphi}^h,\be_{\bvarphi}^h) 
+ l_{\be_{\bvarphi}^h}(\bvarphi,\be_{\bvarphi}^h) 
\nonumber \\
& \ds \quad \ge 
  \dfrac{1}{32\,\mu}\|(\be_{\bsi_f}^h)^{\rd}\|^2_{\bbL^2(\Omega_f)}   
-\, C\,\Big( \|(\be^I_{\bsi_f})^\rd\|^2_{\bbL^2(\Omega_f)} + \|(\be^I_{\bu_f},\be^I_{\bgamma_{f}},\be^I_{\bvarphi})\|^2 \Big)\,.
\end{align}
In turn, testing \eqref{eq:NS-Biot-errorformulation-1} with $(\btau_{fh}, \bv_{ph}) =(\be_{\bsi_f}^h, \be_{\bu_p}^h)$ and using the discrete inf-sup conditions \eqref{eq:discrete inf-sup-qp-xi}--\eqref{eq:discrete inf-sup-vf-chif} in Lemma \ref{lem: discrete inf-sup} along with the continuity of $\kappa_{\bw_f}$ (cf. \eqref{eq:continuity-cK-wf}) and the fact that $(\bu_f(t),\bvarphi(t)),(\bu_{fh}(t),\bvarphi_{h}(t)):[0,T]\to \bW_{\wh r_1,\wh r_2}$ (cf. \eqref{eq:wh-Wr-definition}), we get
\begin{align}
&\|(\be_{\bu_{f}}^h, \be_{p_{p}}^h, \be_{\bgamma_{f}}^h, \be_{\bvarphi}^h, \be_{\lambda}^h)\|
\,\leq\, C\,\Big( \|(\be_{\bsi_f}^h)^\rd\|_{\bbL^2(\Omega_f)} + \|\be_{\bu_p}^h\|_{\bL^2(\Omega_p)} +  \|(\be_{\bsi_f}^I)^\rd\|_{\bbL^2(\Omega_f)} + \|\be_{\bu_p}^I\|_{\bL^2(\Omega_p)}\nonumber \\
& \qquad  +\, \|\be_{\bu_f}^I\|_{\bV_f} + \|\be_{\bgamma_f}^I\|_{\bbQ_f} + \|\be_{\bvarphi}^I\|_{\bLambda_f} + \|\be_{p_p}^I\|_{\W_p} + \|\be_{\lambda}^I\|_{\Lambda_p} \Big)\,.
\label{eq: error equation 3}
\end{align}
Next, all terms on the right-hand side of \eqref{eq: error equation 1} are bounded using the continuity of the bilinear forms (cf. Lemma~\ref{lem:cont} and \eqref{eq:aep-coercivity-bound}), using also that $(\bu_f(t),\bvarphi(t))$, $(\bu_{fh}(t),\bvarphi_{h}(t)):[0,T]\to \bW_{\wh r_1,\wh r_2}$ (cf. \eqref{eq:wh-Wr-definition}) to bound the $\kappa_{\bw_f}$ and $l_{\bzeta}$ terms. Thus,
integrating \eqref{eq: error equation 1} from $0$ to $t \in (0,T]$, and using \eqref{eq:aep-coercivity-bound}, the non-negativity bounds in Lemma \ref{lem:coercivity-properties-A-E2}, together with the Cauchy--Schwarz and Young’s inequalities, and \eqref{eq:error-equation-4}--\eqref{eq: error equation 3}, we obtain
\begin{align*}
&\ds \|\be_{\bu_f}^h(t)\|^2_{\bL^2(\Omega_f)} 
+ s_0 \|\be_{p_p}^h(t)\|^2_{\W_p} 
+ \|\be_{\bbeta_p}^h(t)\|^2_{\bV_p} 
+ \|\partial_{t}\be_{\bbeta_p}^h(t)\|^2_{\bL^2(\Omega_p)} 
+\, \int_0^t \Big(\|(\be_{\bsi_f}^h)^{\rd}\|^2_{\bbL^2(\Omega_f)}   + \|\be_{\bu_p}^h\|^2_{\bL^2(\Omega_p)} \nonumber \\[1ex]
&\ds\quad 
+ \sum^{n-1}_{j=1} \|( \be_{\bvarphi}^h-\partial_t\be_{\bbeta_p}^h)\cdot\bt_{f,j}\|^2_{\L^2(\Gamma_{fp})} 
+ \|(\be_{\bu_{f}}^h, \be_{p_{p}}^h, \be_{\bgamma_{f}}^h, \be_{\bvarphi}^h, \be_{\lambda}^h)\|^2\Big)\,ds
\nonumber \\[1ex]
&\ds\leq\, C\int_0^t \Big(\|\be_{\ubsi}^I\|^2
+ \|\be_{\ubu}^I\|^2 
+ \sum^{n-1}_{j=1} \|( \be_{\bvarphi}^I-\partial_t\be_{\bbeta_p}^I)\cdot\bt_{f,j}\|^2_{\L^2(\Gamma_{fp})}  
+ \,\|\partial_{t}\be_{\bbeta_p}^I\|^2_{\bV_p} + \,\|\partial_{tt}\be_{\bbeta_p}^I\|^2_{\bL^2(\Omega_p)} \Big)\,ds  
\nonumber \\[1ex]
&\ds\quad + \delta_1 \int_0^t \Big(  \sum^{n-1}_{j=1} \|( \be_{\bvarphi}^h-\partial_t\be_{\bbeta_p}^h)\cdot\bt_{f,j}\|^2_{\L^2(\Gamma_{fp})} 
+ \|\be_{p_{p}}^h\|^2_{\W_p} 
+ \|\be_{\bgamma_{f}}^h\|^2_{\bbQ_f} 
+ \|\be_{\bvarphi}^h\|^2_{\bLambda_f} 
+ \|\be_{\lambda}^h\|^2_{\Lambda_p}  \Big)\,ds  
\nonumber \\[1ex]
&\ds\quad + \delta_2 \int_0^t \Big( \|\be_{\bsi_f}^h\|^2_{\bbX_f}   + \|\be_{\bu_p}^h\|^2_{\bX_p}  \Big)\,ds + C\,\Bigg( \|\be_{\bu_f}^h(0)\|^2_{\bL^2(\Omega_f)}
+ s_0\,\|\be_{p_p}^h(0)\|^2_{\W_p} 
+ \|\be_{\bbeta_p}^h(0)\|^2_{\bV_p}
\nonumber \\[1ex]
&\ds\quad +\,\|\partial_{t}\be_{\bbeta_p}^h(0)\|^2_{\bL^2(\Omega_p)}
+ \int_0^t \big(\|\partial_{tt}\be_{\bbeta_p}^I\|_{\bL^2(\Omega_p)}   
+ \|\be_{\bbeta_p}^I\|_{\bV_p}  
+  \|\be_{p_{p}}^I\|_{\W_p}  
+ \|\be_{\lambda}^I\|_{\Lambda_p} \big)\|\partial_{t}\be_{\bbeta_p}^h\|_{\bV_p}\,ds   
\Bigg) \,,
\end{align*}
and taking $\delta_1$ small enough, we get
\begin{align}\label{eq: error equation 4a}
&\ds \|\be_{\bu_f}^h(t)\|^2_{\bL^2(\Omega_f)} 
+ s_0 \|\be_{p_p}^h(t)\|^2_{\W_p} 
+ \|\be_{\bbeta_p}^h(t)\|^2_{\bV_p} 
+ \|\partial_{t}\be_{\bbeta_p}^h(t)\|^2_{\bL^2(\Omega_p)} 
+\, \int_0^t \Big(\|(\be_{\bsi_f}^h)^{\rd}\|^2_{\bbL^2(\Omega_f)}   + \|\be_{\bu_p}^h\|^2_{\bL^2(\Omega_p)} \nonumber \\[1ex]
&\ds\quad 
+ \sum^{n-1}_{j=1} \|( \be_{\bvarphi}^h-\partial_t\be_{\bbeta_p}^h)\cdot\bt_{f,j}\|^2_{\L^2(\Gamma_{fp})} 
+ \|(\be_{\bu_{f}}^h, \be_{p_{p}}^h, \be_{\bgamma_{f}}^h, \be_{\bvarphi}^h, \be_{\lambda}^h)\|^2\Big)\,ds
\nonumber \\[1ex]
&\ds\leq\, C\int_0^t \Big(\|\be_{\ubsi}^I\|^2
+ \|\be_{\ubu}^I\|^2 
+ \sum^{n-1}_{j=1} \|( \be_{\bvarphi}^I-\partial_t\be_{\bbeta_p}^I)\cdot\bt_{f,j}\|^2_{\L^2(\Gamma_{fp})}  
+ \,\|\partial_{t}\be_{\bbeta_p}^I\|^2_{\bV_p} + \,\|\partial_{tt}\be_{\bbeta_p}^I\|^2_{\bL^2(\Omega_p)} \Big)\,ds  
\nonumber \\[1ex]
&\ds\quad + \delta_2 \int_0^t \Big( \|\be_{\bsi_f}^h\|^2_{\bbX_f}   + \|\be_{\bu_p}^h\|^2_{\bX_p}  \Big)\,ds \,+\, C\,\Bigg( \|\be_{\bu_f}^h(0)\|^2_{\bL^2(\Omega_f)}
+ s_0\,\|\be_{p_p}^h(0)\|^2_{\W_p} 
+ \|\be_{\bbeta_p}^h(0)\|^2_{\bV_p}
\nonumber \\[1ex]
&\ds\quad +\,\|\partial_{t}\be_{\bbeta_p}^h(0)\|^2_{\bL^2(\Omega_p)}
+ \int_0^t \big(\|\partial_{tt}\be_{\bbeta_p}^I\|_{\bL^2(\Omega_p)}   
+ \|\be_{\bbeta_p}^I\|_{\bV_p}  
+  \|\be_{p_{p}}^I\|_{\W_p}  
+ \|\be_{\lambda}^I\|_{\Lambda_p} \big)\|\partial_{t}\be_{\bbeta_p}^h\|_{\bV_p}\,ds   
\Bigg) \,.
\end{align}
We emphasize the term $\ds\delta_2 \int_0^t \Big( \|\be_{\bsi_f}^h\|^2_{\bbX_f}   + \|\be_{\bu_p}^h\|^2_{\bX_p} \Big)\,ds$ on the right-hand side above, which appears due to the terms $b_\sk(\be_{\bsi_f}^h,\be_{\bgamma_f}^I)$, $b_{\bn_f}(\be_{\bsi_f}^h,\be_{\bvarphi}^I)$, and $b_{\bn_p}(\be_{\bu_p}^h, \be_{\lambda}^I)$ on the right-hand side of \eqref{eq: error equation 1}. Thus, to complete the bound, we need to control $\|\bdiv(\be_{\bsi_f}^h)\|_{\bL^{4/3}(\Omega_f)}$ and $\|\div(\be^h_{\bu_p})\|_{\L^2(\Omega_p)}$. To this end, from the equations in \eqref{eq:NS-Biot-errorformulation-1} corresponding to test functions $\bv_{fh}$ and $w_{ph}$, using the projection properties \eqref{eq:projection1} and \eqref{eq:projection2}, we find that
\begin{equation*}
\begin{array}{c}
\ds b_f(\be^h_{\bsi_f}, \bv_{fh}) =  \rho_f (\partial_t\be_{\bu_f}^h,\bv_{fh})_{\Omega_f} \quad \forall\,\bv_{fh}\in \bV_{fh}\,, \\ [2ex]
\ds b_p(\be^h_{\bu_p}, w_{ph}) = (s_0\,\partial_t\,\be^h_{p_p}, w_{ph})_{\Omega_p}  - \alpha_p\,b_p(\partial_t\be_{\bbeta_p}^h,w_{ph})- \alpha_p\,b_p(\partial_t\be_{\bbeta_p}^I,w_{ph}) \quad \forall\, w_{ph}\in \W_{ph}\,,
\end{array}
\end{equation*}
which, together with the fact that $\bdiv(\bbX_f) = (\bV_f)'$ and $\div(\bX_p) = (\W_p)'$, implies that
\begin{align}
&\|\bdiv(\be_{\bsi_f}^h)\|_{\bL^{4/3}(\Omega_f)} 
\,\leq\, \rho_f\,|\Omega_f|^{1/4}\,\|\partial_t\be_{\bu_f}^h\|_{\bL^2(\Omega_f)}\,, \nonumber \\[1ex]
& \|\div(\be^h_{\bu_p})\|_{\L^2(\Omega_p)} 
\,\leq\, s_0\,\|\partial_t\,\be^h_{p_p} \|_{\W_p} 
+ \alpha_p\,n^{1/2}\,\|\partial_t\be_{\bbeta_p}^h\|_{\bV_p}
+ \alpha_p\,n^{1/2}\,\|\partial_t\be_{\bbeta_p}^I\|_{\bV_p}\,.
\label{eq:error-analysis7}
\end{align}

\medskip
\noindent{\bf{Bounds on time derivatives on the right-hand side of \eqref{eq:error-analysis7}.}}

\noindent In order to bound the time derivative terms $\,\|\partial_t\be_{\bu_f}^h\|_{\bL^2(\Omega_f)}, s_0\,\|\partial_t\,\be^h_{p_p} \|_{\W_p}$, and $\|\partial_t\be_{\bbeta_p}^h\|_{\bV_p}$ in \eqref{eq:error-analysis7}, we differentiate in time the whole system \eqref{eq:NS-Biot-errorformulation-1} and, testing with $(\btau_{fh}, \bv_{ph}, \bxi_{ph}, \bv_{fh}, w_{ph}, \bchi_{fh}, \bpsi_{h}, \xi_{h}) = (\partial_t\be_{\bsi_f}^h, \partial_t\be_{\bu_p}^h,
\partial_{tt}\be_{\bbeta_p}^h, \partial_t\be_{\bu_f}^h,$ 
$\partial_t\be_{p_p}^h, \partial_t\be_{\bgamma_f}^h, \partial_t\be_{\bvarphi}^h, \partial_t\be_{\lambda}^h)$, similarly to \eqref{eq: error equation 1}, we get 
\begin{align}\label{eq: error equation 5}
&\ds\frac{1}{2}\,\partial_t\,\Big( \rho_f\|\partial_t\be_{\bu_f}^h\|^2_{\bL^2(\Omega_f)} 
+ s_0\,\|\partial_t\be_{p_p}^h\|^2_{\W_p} 
+ a^e_p(\partial_t\be_{\bbeta_p}^h,\partial_t\be_{\bbeta_p}^h) 
+ \rho_p\,\|\partial_{tt}\be_{\bbeta_p}^h\|^2_{\bL^2(\Omega_p)} \Big) 
\nonumber \\[1ex]
&\ds\quad +\, a_f(\partial_t\be_{\bsi_f}^h,\partial_t\be_{\bsi_f}^h) 
+ a^d_p(\partial_t\be_{\bu_p}^h,\partial_t\be_{\bu_p}^h) 
+ c_{\BJS}(\partial_{tt}\be_{\bbeta_p}^h, \partial_t\be_{\bvarphi}^h;\partial_{tt}\be_{\bbeta_p}^h, \partial_t\be_{\bvarphi}^h)
+ \kappa_{\bu_{fh}}(\partial_t\be_{\bu_{f}}^h, \partial_t\be_{\bsi_f}^h)
\nonumber \\[1ex]
&\ds\quad +\, \kappa_{\partial_t\be_{\bu_{f}}^h}(\bu_{f}, \partial_t\be_{\bsi_f}^h)
+ l_{\bvarphi_h}(\partial_t\be_{\bvarphi}^h,\partial_t\be_{\bvarphi}^h) 
+ l_{\partial_t\be_{\bvarphi}^h}(\bvarphi,\partial_t\be_{\bvarphi}^h)
\nonumber \\[1ex]
&\ds = -\, a^e_p(\partial_t\be_{\bbeta_p}^I,\partial_{tt}\be_{\bbeta_p}^h) 
-\rho_p(\partial_{ttt}\be_{\bbeta_p}^I,\partial_{tt}\be_{\bbeta_p}^h)_{\Omega_p}
- a_f(\partial_t\be_{\bsi_f}^I,\partial_t\be_{\bsi_f}^h)
-  a^d_p(\partial_t\be_{\bu_p}^I,\partial_t\be_{\bu_p}^h) 
\nonumber \\[1ex]
&\ds\quad - c_{\BJS}(\partial_{tt}\be_{\bbeta_p}^I, \partial_t\be_{\bvarphi}^I;\partial_{tt}\be_{\bbeta_p}^h, \partial_t\be_{\bvarphi}^h) 
-\, \kappa_{\partial_t\bu_{fh}}(\be_{\bu_{f}}^h, \partial_t\be_{\bsi_f}^h)
- \kappa_{\be_{\bu_{f}}^h}(\partial_t\bu_{f}, \partial_t\be_{\bsi_f}^h)
-l_{\partial_t\bvarphi_h}(\be_{\bvarphi}^h,\partial_t\be_{\bvarphi}^h)
\nonumber \\[1ex]
&\ds\quad -l_{\be_{\bvarphi}^h}(\partial_t\bvarphi,\partial_t\be_{\bvarphi}^h)  
- \kappa_{\partial_t\bu_{fh}}(\be_{\bu_{f}}^I, \partial_t\be_{\bsi_f}^h)  
-\, \kappa_{\partial_t\be_{\bu_{f}}^I}(\bu_{f}, \partial_t\be_{\bsi_f}^h)  
- \kappa_{\bu_{fh}}(\partial_t\be_{\bu_{f}}^I, \partial_t\be_{\bsi_f}^h) 
- \kappa_{\be_{\bu_{f}}^I}(\partial_t\bu_{f}, \partial_t\be_{\bsi_f}^h)
\nonumber \\[1ex]
&\ds\quad- l_{\partial_t\bvarphi_h}(\be_{\bvarphi}^I,\partial_t\be_{\bvarphi}^h) 
- l_{\partial_t\be_{\bvarphi}^I}(\bvarphi,\partial_t\be_{\bvarphi}^h) 
-\, l_{\bvarphi_h}(\partial_t\be_{\bvarphi}^I,\partial_t\be_{\bvarphi}^h) 
- l_{\be_{\bvarphi}^I}(\partial_t\bvarphi,\partial_t\be_{\bvarphi}^h)
- b_\sk(\partial_t\be_{\bsi_f}^h,\partial_t\be_{\bgamma_f}^I) 
\nonumber \\[1ex]
&\ds\quad + b_\sk(\partial_t\be_{\bsi_f}^I,\partial_t\be_{\bgamma_f}^h)  
- b_{\bn_f}(\partial_t\be_{\bsi_f}^h,\partial_t\be_{\bvarphi}^I) 
+\, b_{\bn_f}(\partial_t\be_{\bsi_f}^I,\partial_t\be_{\bvarphi}^h)
- \alpha_p\,b_p(\partial_{tt}\be_{\bbeta_p}^h,\partial_t\be_{p_p}^I)
+ \alpha_p\,b_p(\partial_{tt}\be_{\bbeta_p}^I,\partial_t\be_{p_p}^h) 
\nonumber \\[1ex]
&\ds\quad - b_{\bn_p}(\partial_t\be_{\bu_p}^h, \partial_t\be_{\lambda}^I)
+\, b_{\bn_p}(\partial_t\be_{\bu_p}^I, \partial_t\be_{\lambda}^h) 
- c_{\Gamma}(\partial_{tt}\be_{\bbeta_p}^I,\partial_t\be_{\bvarphi}^I;\partial_t\be_{\lambda}^h)
+ c_{\Gamma}(\partial_{tt}\be_{\bbeta_p}^h,\partial_t\be_{\bvarphi}^h;\partial_t\be_{\lambda}^I) \,,
\end{align}
where, using again the projection properties \eqref{eq:projection1} and 
\eqref{eq:projection2}, and the fact that 
$\div(\bX_{ph})=\W_{ph}$,
\newline
$\bdiv(\bbX_{fh}) = \bV_{fh}$ (cf. \eqref{eq: div-prop}), we have dropped from \eqref{eq: error equation 5} the following terms:
\begin{equation*}
\begin{array}{c}
s_0 (\partial_{tt}\be_{p_p}^I,\partial_{t}\be_{p_p}^h)_{\Omega_p}, \,\, \rho_f (\partial_{tt}\be_{\bu_f}^I,\partial_{t}\be_{\bu_f}^h)_{\Omega_f}, \,\, b_f(\partial_{t}\be_{\bsi_f}^h,\partial_{t}\be_{\bu_f}^I), \,\, b_f(\partial_{t}\be_{\bsi_f}^I,\partial_{t}\be_{\bu_f}^h)\,,  \\[2ex]
b_p(\partial_{t}\be_{\bu_p}^h,\partial_{t}\be_{p_p}^I), \,\, b_p(\partial_{t}\be_{\bu_p}^I,\partial_{t}\be_{p_p}^h)\,.
\end{array}
\end{equation*}
For the terms $\kappa_{\bw_f}$ and $l_{\bzeta}$ on the left-hand side of \eqref{eq: error equation 5}, similarly to \eqref{eq:error-equation-4}, using that $(\bu_f(t),\bvarphi(t))$, $(\bu_{fh}(t),\bvarphi_{h}(t)):[0,T]\to \bW_{\wh r_1,\wh r_2}$, we obtain
\begin{align}\label{eq:k-l-lhs}
  &\ds  a_f(\dt \be_{\bsi_f}^h,\dt \be_{\bsi_f}^h) 
+ \kappa_{\bu_{fh}}(\dt\be_{\bu_{f}}^h, \dt\be_{\bsi_f}^h)
+ \kappa_{\dt\be_{\bu_{f}}^h}(\bu_{f}, \dt\be_{\bsi_f}^h)
+ l_{\bvarphi_h}(\dt\be_{\bvarphi}^h,\dt\be_{\bvarphi}^h) 
+ l_{\dt\be_{\bvarphi}^h}(\bvarphi,\dt\be_{\bvarphi}^h) 
\nonumber \\
& \ds \quad \ge 
 \dfrac{1}{32\,\mu}\|\dt(\be_{\bsi_f}^h)^{\rd}\|^2_{\bbL^2(\Omega_f)}  
-\, C\,\Big( \|\dt(\be^I_{\bsi_f})^\rd\|^2_{\bbL^2(\Omega_f)} + \|(\dt\be^I_{\bu_f},\dt\be^I_{\bgamma_{f}},\dt\be^I_{\bvarphi})\|^2 \nonumber \\
& \ds \qquad + \big(\|\partial_t \bu_{fh}\|^2_{\bV_f} +  \|\partial_t \bu_{f}\|^2_{\bV_f})( \|\be^h_{\bu_f}\|^2_{\bV_f} + \|\be^I_{\bu_f}\|^2_{\bV_f}\big) \Big)\,.
\end{align}
We next bound the 12 $\kappa_{\bw_f}$ and $l_{\bzeta}$ terms on the right-hand side of \eqref{eq: error equation 5}. Using that $(\bu_f(t),\bvarphi(t))$, $(\bu_{fh}(t),\bvarphi_{h}(t)):[0,T]\to \bW_{\wh r_1,\wh r_2}$,
for four of the terms we have
\begin{equation}\label{eq:k-l-rhs-1}
\begin{array}{l}
\ds \kappa_{\partial_t\be_{\bu_{f}}^I}(\bu_{f}, \partial_t\be_{\bsi_f}^h)  
+ \kappa_{\bu_{fh}}(\partial_t\be_{\bu_{f}}^I, \partial_t\be_{\bsi_f}^h)
+ l_{\bvarphi_h}(\partial_t\be_{\bvarphi}^I,\partial_t\be_{\bvarphi}^h)
+ l_{\partial_t\be_{\bvarphi}^I}(\bvarphi,\partial_t\be_{\bvarphi}^h) \\[2ex] 
\ds\quad \leq\, \delta_3\,\Big(\|\partial_t(\be_{\bsi_f}^h)^\rd\|^2_{\bbL^2(\Omega_f)} + \|\partial_t\be_{\bvarphi}^h\|^2_{\bLambda_f} \Big) 
+ C\,\Big(\|\partial_{t}\be_{\bu_f}^I\|^2_{\bV_f} + \|\partial_t\be_{\bvarphi}^I\|^2_{\bLambda_f}\Big) \,,
\end{array}
\end{equation}
which follows from the Cauchy--Schwarz and Young's inequalities, while the remaining eight terms can be bounded by
\begin{equation}\label{eq:k-l-rhs-2}
\begin{array}{l}
\ds
\delta_3\,\Big(\|\partial_t(\be_{\bsi_f}^h)^\rd\|^2_{\bbL^2(\Omega_f)} + \|\partial_t\be_{\bvarphi}^h\|^2_{\bLambda_f} \Big) 
+ C\,\Big\{ ( \|\partial_{t}\bu_{fh}\|^2_{\bV_f} + \|\partial_{t}\bu_{f}\|^2_{\bV_f})(\|\be_{\bu_f}^I\|^2_{\bV_f} + \|\be_{\bu_f}^h\|^2_{\bV_f} ) \\[2ex]
\ds\quad +\, ( \|\partial_t\bvarphi_{h}\|^2_{\bLambda_f} + \|\partial_t\bvarphi\|^2_{\bLambda_f})(\|\be_{\bvarphi}^I\|^2_{\bLambda_f} + \|\be_{\bvarphi}^h\|^2_{\bLambda_f} ) \Big\} \,.
\end{array}
\end{equation} 

In turn, from the rows of $\btau_{fh}$ and $\bv_{ph}$ in \eqref{eq:NS-Biot-errorformulation-1}, differentiating in time and using discrete inf-sup conditions \eqref{eq:discrete inf-sup-qp-xi}--\eqref{eq:discrete inf-sup-vf-chif} in Lemma \ref{lem: discrete inf-sup} for $(\partial_t\be_{\bu_{f}}^h, \partial_t\be_{p_{p}}^h, \partial_t\be_{\bgamma_{f}}^h, \partial_t\be_{\bvarphi}^h, \partial_t\be_{\lambda}^h)$ along with the continuity of $\kappa_{\bw_f}$ (cf. \eqref{eq:continuity-cK-wf}) and the fact that $(\bu_f(t),\bvarphi(t)),(\bu_{fh}(t),\bvarphi_{h}(t)):[0,T]\to \bW_{\wh r_1,\wh r_2}$ (cf. \eqref{eq:wh-Wr-definition}), we get
\begin{align}
&\|(\partial_t\be_{\bu_{f}}^h, \partial_t\be_{p_{p}}^h, \partial_t\be_{\bgamma_{f}}^h, \partial_t\be_{\bvarphi}^h, \partial_t\be_{\lambda}^h)\|
\,\leq\, C\,\Big( \|(\partial_t\be_{\bsi_f}^h)^\rd\|_{\bbL^2(\Omega_f)} 
+ \|\partial_t\be_{\bu_p}^h\|_{\bL^2(\Omega_p)} + \|(\partial_t\be_{\bsi_f}^I)^\rd\|_{\bbL^2(\Omega_f)}  \nonumber \\
& \quad + \|\partial_t\be_{\bu_p}^I\|_{\bL^2(\Omega_p)} + \big(\|\partial_{t}\bu_{f}\|_{\bV_f} + \, \|\partial_{t}\bu_{fh}\|_{\bV_f}\big)\big(\|\be_{\bu_f}^h\|_{\bV_f} + \|\be_{\bu_f}^I\|_{\bV_f}\big)
+ \|\partial_t\be_{\bu_f}^I\|_{\bV_f}  
\nonumber \\
& \quad 
+ \,\|\partial_t\be_{\bgamma_f}^I\|_{\bbQ_f} 
+ \|\partial_t\be_{\bvarphi}^I\|_{\bLambda_f} + \|\partial_t\be_{p_p}^I\|_{\W_p} + \|\partial_t\be_{\lambda}^I\|_{\Lambda_p} \Big)\,.
\label{eq: error equation 5a}
\end{align}	
Then integrating \eqref{eq: error equation 5} from 0 to t $\in (0,T]$, using \eqref{eq:k-l-lhs}--\eqref{eq: error equation 5a} and the identities
  \begin{align*}
& \int^t_0 a^e_p(\partial_t\be_{\bbeta_p}^I,\partial_{tt}\be_{\bbeta_p}^h)\,ds 
\,=\, a^e_p(\partial_t\be_{\bbeta_p}^I,\partial_{t}\be_{\bbeta_p}^h)\Big|_0^t 
- \int^t_0 a^e_p(\partial_{tt}\be_{\bbeta_p}^I,\partial_{t}\be_{\bbeta_p}^h)\,ds,\nonumber \\
& \int^t_0 \rho_p(\partial_{ttt}\be_{\bbeta_p}^I,\partial_{tt}\be_{\bbeta_p}^h)_{\Omega_p} \,ds \,=\, \rho_p(\partial_{ttt}\be_{\bbeta_p}^I,\partial_{t}\be_{\bbeta_p}^h)_{\Omega_p}\Big|_0^t - \int^t_0 \rho_p(\partial_{tttt}\be_{\bbeta_p}^I,\partial_{t}\be_{\bbeta_p}^h)_{\Omega_p}  \,ds,\nonumber \\
&  \int^t_0  b_\sk(\partial_t\be_{\bsi_f}^h,\partial_t\be_{\bgamma_f}^I)\,ds 
\,=\,  b_\sk(\be_{\bsi_f}^h,\partial_t\be_{\bgamma_f}^I)\Big|_0^t 
- \int^t_0  b_\sk(\be_{\bsi_f}^h,\partial_{tt}\be_{\bgamma_f}^I)\,ds,\nonumber \\
&  \int^t_0 b_{\bn_f}(\partial_t\be_{\bsi_f}^h,\partial_t\be_{\bvarphi}^I)\,ds 
\,=\, b_{\bn_f}(\be_{\bsi_f}^h,\partial_t\be_{\bvarphi}^I)\Big|_0^t 
- \int^t_0 b_{\bn_f}(\be_{\bsi_f}^h,\partial_{tt}\be_{\bvarphi}^I)\,ds,\nonumber \\
&  \int^t_0 \,b_p(\partial_{tt}\be_{\bbeta_p}^h,\partial_t\be_{p_p}^I) ds 
\,=\, \,b_p(\partial_{t}\be_{\bbeta_p}^h,\partial_t\be_{p_p}^I) \Big|_0^t 
- \int^t_0 \,b_p(\partial_{t}\be_{\bbeta_p}^h,\partial_{tt}\be_{p_p}^I) ds ,\nonumber \\
&  \int^t_0 b_{\bn_p}(\partial_t\be_{\bu_p}^h, \partial_t\be_{\lambda}^I)  ds 
\,=\, \,b_{\bn_p}(\be_{\bu_p}^h, \partial_t\be_{\lambda}^I)  \Big|_0^t 
- \int^t_0 \,b_{\bn_p}(\be_{\bu_p}^h, \partial_{tt}\be_{\lambda}^I)  ds,\nonumber \\
&  \int^t_0  c_{\Gamma}(\partial_{tt}\be_{\bbeta_p}^h,\partial_t\be_{\bvarphi}^h;\partial_t\be_{\lambda}^I)  ds 
\,=\,  c_{\Gamma}(\partial_{t}\be_{\bbeta_p}^h,\be_{\bvarphi}^h;\partial_t\be_{\lambda}^I)\Big|_0^t 
- \int^t_0  c_{\Gamma}(\partial_{t}\be_{\bbeta_p}^h,\be_{\bvarphi}^h;\partial_{tt}\be_{\lambda}^I) ds, 
\end{align*}
and applying the non-negativity and continuity bounds of the bilinear forms involved (cf. Lemmas \ref{lem:cont} and \ref{lem:coercivity-properties-A-E2}), \eqref{eq: error equation 5a}, Cauchy--Schwarz and Young's inequalities, we obtain
\begin{align}\label{eq: error equation 6}
&\ds \|\partial_t\be_{\bu_f}^h(t)\|^2_{\bL^2(\Omega_f)} 
+ s_0\,\|\partial_t\be_{p_p}^h(t)\|^2_{\W_p} 
+ \|\partial_t\be_{\bbeta_p}^h(t)\|^2_{\bV_p} 
+ \,\|\partial_{tt}\be_{\bbeta_p}^h(t)\|^2_{\bL^2(\Omega_p)}  
+\, \int^t_0 \Big(\|\partial_t(\be_{\bsi_f}^h)^\rd\|^2_{\bbL^2(\Omega_f)} 
\nonumber \\[1ex]
&\ds\quad + \|\partial_t\be_{\bu_p}^h\|^2_{\bL^2(\Omega_p)} 
+ \sum^{n-1}_{j=1} \|( \partial_t\be_{\bvarphi}^h-\partial_{tt}\be_{\bbeta_p}^h)\cdot\bt_{f,j}\|^2_{\L^2(\Gamma_{fp})} + \|(\partial_t\be_{\bu_{f}}^h, \partial_t\be_{p_{p}}^h, \partial_t\be_{\bgamma_{f}}^h, \partial_t\be_{\bvarphi}^h, \partial_t\be_{\lambda}^h)\|^2\Big)\,ds
\nonumber \\[1ex]
&\ds \leq C \Bigg(\|\partial_{ttt}\be_{\bbeta_p}^I(t)\|^2_{\bL^2(\Omega_p)} 
+ \|\partial_t\be_{\bbeta_p}^I(t)\|^2_{\bV_p} 
+ \|\partial_t\be_{p_p}^I(t)\|^2_{\W_p} 
+ \|\partial_t\be_{\bgamma_f}^I(t)\|^2_{\bbQ_f} 
+ \|\partial_t\be_{\bvarphi}^I(t)\|^2_{\bLambda_f}
+ \|\partial_t\be_{\lambda}^I(t)\|^2_{\Lambda_p}
\nonumber \\[1ex]
&\ds\quad + \int_0^t \Big(\|\partial_{t}\be_{\bsi_f}^I\|^2_{\bbX_f} 
+ \|\partial_{tt}\be_{p_p}^I\|^2_{\W_p} 
+ \|\partial_t(\be_{\bsi_f}^I)^\rd\|^2_{\bbL^2(\Omega_f)}
+ \|\partial_t\be_{\bu_p}^I\|^2_{\bX_p}
+\, \|\partial_{tt}\be_{\bbeta_p}^I\|^2_{\bV_p}
+ \|\partial_{t}\be_{\bvarphi}^I\|^2_{\bLambda_f}  
\nonumber \\[1ex]
&\ds\quad +\, \|\partial_{tt}\be_{\bvarphi}^I\|^2_{\bLambda_f}    
+ \sum^{n-1}_{j=1} \|( \partial_t\be_{\bvarphi}^I-\partial_{tt}\be_{\bbeta_p}^I)\cdot\bt_{f,j}\|^2_{\L^2(\Gamma_{fp})} 
+ \|\partial_{tt}\be_{\lambda}^I\|^2_{\Lambda_p}
+ \|\partial_{tt}\be_{\bgamma_f}^I\|^2_{\bbQ_f} 
+ \|\partial_{t}\be_{\bu_f}^I\|^2_{\bV_f} \nonumber \\[1ex]
&\ds\quad +\, \|\partial_{t}\be_{\bgamma_f}^I\|^2_{\bQ_f} + \|\partial_t\be_{p_p}^I\|^2_{\W_p} + \|\partial_t\be_{\lambda}^I\|^2_{\Lambda_p} \Big) ds  + \|\partial_{ttt}\be_{\bbeta_p}^I(0)\|^2_{\bL^2(\Omega_p)}  +\, \|\partial_t\be_{\bbeta_p}^I(0)\|^2_{\bV_p} 
+ \|\partial_t\be_{\bvarphi}^I(0)\|^2_{\bLambda_f}
\nonumber \\[1ex]
&\ds\quad
+\, \|\partial_t\be_{p_p}^I(0)\|^2_{\W_p}  
+ \|\partial_t\be_{\bgamma_f}^I(0)\|^2_{\bbQ_f} 
+ \|\partial_t\be_{\lambda}^I(0)\|^2_{\Lambda_p}\Bigg)  +\, \delta_3\,\Bigg( \int_0^t \Big(\|\partial_t(\be_{\bsi_f}^h)^\rd\|^2_{\bbL^2(\Omega_f)}
+ \|\partial_t\be_{\bu_p}^h\|^2_{\bL^2(\Omega_p)} 
\nonumber \\[1ex]
&\ds\quad 
+\, \sum^{n-1}_{j=1} \|(\partial_t\be_{\bvarphi}^h-\partial_{tt}\be_{\bbeta_p}^h)\cdot\bt_{f,j}\|^2_{\L^2(\Gamma_{fp})} 
+ \|\partial_{t}\be^h_{p_p}\|^2_{\W_p}  +\, \|\be_{\bsi_f}^h\|^2_{\bbX_f} 
+ \|\be_{\bu_p}^h\|^2_{\bX_p} 
+ \|\be_{\bvarphi}^h\|^2_{\bLambda_f}
+ \|\partial_{t}\be_{\bvarphi}^h\|^2_{\bLambda_f}   
\nonumber \\[1ex] 
&\ds\quad                 
+\, \|\partial_{t}\be_{\bgamma_f}^h\|^2_{\bbQ_f} 
+ \|\partial_{t}\be_{\lambda}^h\|^2_{\Lambda_p}\Big)\,ds 
+ \|\be_{\ubsi}^h(t)\|^2
+ \|\be_{\bvarphi}^h(t)\|^2_{\bLambda_f} +\, \|\partial_{t}\be_{\bbeta_p}^h(t)\|^2_{\bV_p} \Bigg) 
+ C \Bigg( \|\partial_t\be_{\bu_f}^h(0)\|^2_{\bL^2(\Omega_f)} 
\nonumber \\[1ex] 
&\ds\quad +\, \|\be_{\bvarphi}^h(0)\|^2_{\bLambda_f}     
+ s_0\|\partial_t\be_{p_p}^h(0)\|^2_{\W_p}
+ \|\be_{\ubsi}^h(0)\|^2
+ \|\partial_t\be_{\bbeta_p}^h(0)\|^2_{\bV_p} 
+ \, \|\partial_{tt}\be_{\bbeta_p}^h(0)\|^2_{\bL^2(\Omega_p)}
\nonumber \\[1ex]
&\ds\quad +\, \int_0^t \big(\|\partial_{tttt}\be_{\bbeta_p}^I\|_{\bL^2(\Omega_p)} 
+ \|\partial_{tt}\be_{\bbeta_p}^I\|_{\bV_p}  
+ \|\partial_{tt}\be^I_{p_p} \|_{\W_p} 
+ \|\partial_{tt}\be_{\lambda}^I\|_{\Lambda_p} \big)
\|\partial_{t}\be_{\bbeta_p}^h\|_{\bV_p}\,ds
\nonumber \\[1ex]
&\ds\quad 
+ \int_0^t (\|\partial_{t}\bu_{fh}\|^2_{\bV_f} 
+ \|\partial_{t}\bu_{f}\|^2_{\bV_f})(\|\be_{\bu_f}^I\|^2_{\bV_f} + \|\be_{\bu_f}^h\|^2_{\bV_f})\,ds 
\nonumber \\[1ex]
&\ds\quad +\, \int_0^t (\|\partial_t\bvarphi_{h}\|^2_{\bLambda_f} 
+ \|\partial_t\bvarphi\|^2_{\bLambda_f})(\|\be_{\bvarphi}^I\|^2_{\bLambda_f} + \|\be_{\bvarphi}^h\|^2_{\bLambda_f})\,ds \Bigg) \,.
\end{align}
To bound the last two terms in \eqref{eq: error equation 6}, we use the assumption \eqref{eq:extra-solution-assumption} to obtain
\begin{align}\label{eq:eh-uf-varphi-bound}
&\ds \int_0^t \left\{ \big( \|\partial_t \bu_f\|_{\bV_f}^2 + \|\partial_t \bu_{fh}\|_{\bV_f}^2 \big) 
\big( \|\be_{\bu_f}^I\|_{\bV_f}^2 + \|\be_{\bu_f}^h\|_{\bV_f}^2 \big) 
+ \big(\|\partial_t\bvarphi_{h}\|^2_{\bLambda_f}{+}\|\partial_t\bvarphi\|^2_{\bLambda_f}\big)
\big(\|\be_{\bvarphi}^I\|^2_{\bLambda_f} + \|\be_{\bvarphi}^h\|^2_{\bLambda_f} \big) \right\} ds \nonumber \\[3ex]
&\ds\qquad \leq\, C\,\int_0^t\left(  
\|\be_{\bu_f}^h\|_{\bV_f}^2
+ \|\be_{\bvarphi}^h\|_{\bLambda_f}^2
+ \|\be_{\bu_f}^I\|_{\bV_f}^2
+ \|\be_{\bvarphi}^I\|_{\bLambda_f}^2 \right)\,ds \,,
\end{align}
where the first two terms on the right-hand side of \eqref{eq:eh-uf-varphi-bound} are bounded by \eqref{eq: error equation 4a}.

Next, we focus on bounding the terms $\|\be_{\ubsi}^h(t)\|^2 := \|\be_{\bsi_f}^h(t)\|^2_{\bbX_f} + \|\be_{\bu_p}^h(t)\|^2_{\bX_p}$ and $\|\be_{\bvarphi}^h(t)\|^2_{\bLambda_f}$ in \eqref{eq: error equation 6}.
In fact, recalling that we have control of $\|(\be_{\bsi_f}^h)^\rd\|^2_{\L^2(0,t;\bbL^2(\Omega_f))}$ and $\|\partial_t(\be_{\bsi_f}^h)^\rd\|^2_{\L^2(0,t;\bbL^2(\Omega_f))}$ by \eqref{eq: error equation 4a} and \eqref{eq: error equation 6}, respectively, and using the Sobolev embedding of $\H^1(0,t)$ into $\L^\infty(0,t)$, we obtain
\begin{equation}\label{eq:sigmaf-d}
\|(\be_{\bsi_f}^h)^\rd(t)\|^2_{\bbL^2(\Omega_f)} 
\,\leq\, C\,\int_0^t \Big(\|(\be_{\bsi_f}^h)^\rd\|^2_{\bbL^2(\Omega_f)} 
+ \|\partial_t(\be_{\bsi_f}^h)^\rd\|^2_{\bbL^2(\Omega_f)}\Big)\,ds \,,
\end{equation}
which together with the first equation of \eqref{eq:error-analysis7} and using \eqref{eq:tau-d-H0div-inequality}--\eqref{eq:tau-H0div-Xf-inequality}, implies
\begin{equation}\label{eq: error equation 8}
\|\be_{\bsi_f}^h(t)\|^2_{\bbX_f} \le C\,\int_0^t  \Big( \|(\be_{\bsi_f}^h)^\rd\|^2_{\bbL^2(\Omega_f)} 
+ \|\partial_t(\be_{\bsi_f}^h)^\rd\|^2_{\bbL^2(\Omega_f)}
\Big)\,ds
+ \|\partial_t\be_{\bu_f}^h(t)\|^2_{\bL^2(\Omega_f)}.
\end{equation} 
Similarly, by \eqref{eq: error equation 4a} and \eqref{eq: error equation 6}, we have control on $\|\be_{\bu_p}^h\|^2_{\L^2(0,t;\bL^2(\Omega_p))}$ and $\|\partial_t\be_{\bu_p}^h\|^2_{\L^2(0,t;\bL^2(\Omega_p))}$, respectively. 
Thus, using again the Sobolev embedding of $\H^1(0,t)$ into $\L^\infty(0,t)$, we obtain
\begin{equation*}
\|\be_{\bu_p}^h(t)\|^2_{\bL^2(\Omega_p)}
\,\leq\, C\,\int_0^t \Big(\|\be_{\bu_p}^h\|^2_{\bL^2(\Omega_p)} 
+ \|\partial_t\be_{\bu_p}^h\|^2_{\bL^2(\Omega_p)}\Big)\,ds \,,
\end{equation*}
which, combined with the second equation of \eqref{eq:error-analysis7}, implies
\begin{equation}\label{eq: error equation 9}
\|\be_{\bu_p}^h(t)\|^2_{\bX_p}
\,\leq\, C\left(\int_0^t \Big(\|\be_{\bu_p}^h\|^2_{\bL^2(\Omega_p)} 
+ \|\partial_t\be_{\bu_p}^h\|^2_{\bL^2(\Omega_p)}\Big)\,ds
+ s_0\,\|\partial_t\,\be^h_{p_p} \|^2_{\W_p} 
+ \|\partial_t\be_{\bbeta_p}^h\|^2_{\bV_p}
+ \|\partial_t\be_{\bbeta_p}^I\|^2_{\bV_p}\, \right).
\end{equation}
To control $\|\be_{\bvarphi}^h(t)\|^2_{\bLambda_f}$ in \eqref{eq: error equation 6}, we can use \eqref{eq:error-equation-4b} and \eqref{eq:sigmaf-d} to deduce that
\begin{equation}\label{eq: error equation 7}
\|\be_{\bvarphi}^h(t)\|_{\bLambda_f}^2
\leq C \bigg( \int_0^t \Big(\|(\be_{\bsi_f}^h)^\rd\|^2_{\bbL^2(\Omega_f)} 
+ \|\partial_t(\be_{\bsi_f}^h)^\rd\|^2_{\bbL^2(\Omega_f)}\Big)\,ds
+  \|(\be_{\bsi_f}^I)^\rd(t)\|_{\bbL^2(\Omega_f)}^2 + \|(\be_{\bu_f}^I,\be_{\bgamma_f}^I,\be_{\bvarphi}^I)(t)\|^2 \bigg).
\end{equation}

Thus, we combine \eqref{eq: error equation 4a} and \eqref{eq: error equation 6}--\eqref{eq: error equation 7} and apply Lemma \ref{xing-lemma} in the context of the non-negative functions $H= \|\partial_t\bbeta_p\|_{\bV_p}$ and 
$B=
\|\be_{\bbeta_p}^I\|_{\bV_p}  
+ \|\partial_{tt}\be_{\bbeta_p}^I\|_{\bV_p} 
+ \|\partial_{tttt}\be_{\bbeta_p}^I\|_{\bL^2(\Omega_p)}   
+  \|\be_{p_{p}}^I\|_{\W_p}  
+ \|\partial_{tt}\be^I_{p_p} \|_{\W_p} 
+ \|\be_{\lambda}^I\|_{\Lambda_p} 
+ \|\partial_{tt}\be_{\lambda}^I\|_{\Lambda_p}$, with $R$ and $A$ representing the remaining terms, to control $\int^t_0 B(s)\,H(s)\,ds$, 
along with \eqref{eq:error-analysis7} and the Sobolev embedding \eqref{sobolev-L-infty-L2} to bound the terms $\|\bdiv(\be_{\bsi_f}^h)\|_{\bL^{4/3}(\Omega_f)}$ and $\|\div(\be^h_{\bu_p})\|_{\L^2(\Omega_p)}$ in $\L^2(0,T)$, and their corresponding right-hand side in $\L^\infty(0,T)$. 
Then, employing \eqref{eq:tau-d-H0div-inequality}--\eqref{eq:tau-H0div-Xf-inequality}, choosing $\delta_2$ and $\delta_3$ small enough, we obtain
\begin{align}
&\ds \|\be_{\bu_f}^h(t)\|^2_{\bL^2(\Omega_f)} 
+ s_0\,\|\be_{p_p}^h(t)\|^2_{\W_p} + \|\be_{\bbeta_p}^h(t)\|^2_{\bV_p} 
+ \|\partial_{t}\be_{\bbeta_p}^h(t)\|^2_{\bL^2(\Omega_p)} 
+ \|\partial_t\be_{\bu_f}^h(t)\|^2_{\bL^2(\Omega_f)}  
+ s_0\,\|\partial_t\be_{p_p}^h(t)\|^2_{\W_p} \nonumber \\[1ex]
&\ds\quad +\, \|\partial_t\be_{\bbeta_p}^h(t)\|^2_{\bV_p} 
+ \|\partial_{tt}\be_{\bbeta_p}^h(t)\|^2_{\bL^2(\Omega_p)} 
+ \int^t_0 \Big(\|\be_{\ubsi}^h\|^2  
+ \|\partial_t(\be_{\bsi_f}^h)^\rd\|^2_{\bbL^2(\Omega_f)}  
+ \|\partial_t\be_{\bu_p}^h\|^2_{\bL^2(\Omega_p)}
\nonumber \\[1ex]
&\ds\quad +\, \sum^{n-1}_{j=1} \|( \be_{\bvarphi}^h-\partial_t\be_{\bbeta_p}^h)\cdot\bt_{f,j}\|^2_{\L^2(\Gamma_{fp})} 
+ \sum^{n-1}_{j=1} \|( \partial_t\be_{\bvarphi}^h-\partial_{tt}\be_{\bbeta_p}^h)\cdot\bt_{f,j}\|^2_{\L^2(\Gamma_{fp})}
+ \|(\be_{\bu_{f}}^h, \be_{p_{p}}^h, \be_{\bgamma_{f}}^h, \be_{\bvarphi}^h, \be_{\lambda}^h)\|^2  
\nonumber \\[1ex]
&\ds\quad + \|(\partial_t\be_{\bu_{f}}^h, \partial_t\be_{p_{p}}^h, \partial_t\be_{\bgamma_{f}}^h, \partial_t\be_{\bvarphi}^h, \partial_t\be_{\lambda}^h)\|^2\Big)\,ds
\nonumber \\[1ex]
&\ds \leq\, C\,T\,\Bigg( \int_0^t \Big(\|\be_{\ubsi}^I\|^2
+ \|\be_{\ubu}^I\|^2 
+ \sum^{n-1}_{j=1} \|( \be_{\bvarphi}^I-\partial_t\be_{\bbeta_p}^I)\cdot\bt_{f,j}\|^2_{\L^2(\Gamma_{fp})} 
+ \|\partial_{t}\be_{\bbeta_p}^I\|^2_{\bV_p} 
+ \|\partial_{tt}\be_{\bbeta_p}^I\|^2_{\bV_p}
\nonumber \\[1ex]
&\ds\quad 
+\, \|\partial_{t}\be_{\bsi_f}^I\|^2_{\bbX_f} 
+ \|\partial_{tt}\be_{p_p}^I\|^2_{\W_p} 
+ \|\partial_t(\be_{\bsi_f}^I)^\rd\|^2_{\bbL^2(\Omega_f)} 
+ \|\partial_t\be_{\bu_p}^I\|^2_{\bX_p}
+ \|\partial_{t}\be_{\bvarphi}^I\|^2_{\bLambda_f}  
+ \|\partial_{tt}\be_{\bgamma_f}^I\|^2_{\bbQ_f} 
+ \|\partial_{tt}\be_{\bvarphi}^I\|^2_{\bLambda_f}  
\nonumber \\[1ex]
&\ds\quad 
+\, \sum^{n-1}_{j=1} \|( \partial_t\be_{\bvarphi}^I-\partial_{tt}\be_{\bbeta_p}^I)\cdot\bt_{f,j}\|^2_{\L^2(\Gamma_{fp})}   
+ \|\partial_{tt}\be_{\lambda}^I\|^2_{\Lambda_p} 
+ \|\partial_{tttt}\be_{\bbeta_p}^I\|^2_{\bL^2(\Omega_p)} 
+ \|\partial_{t}\be_{\bu_f}^I\|^2_{\bV_f}  + \|\partial_{t}\be_{\bgamma_f}^I\|^2_{\bQ_f} \nonumber \\[1ex]
&\ds\quad + \|\partial_t\be_{p_p}^I\|^2_{\W_p} + \|\partial_t\be_{\lambda}^I\|^2_{\Lambda_p} \Big)\,ds 
+ \|\partial_{ttt}\be_{\bbeta_p}^I(t)\|^2_{\bL^2(\Omega_p)} 
+ \|\partial_t\be_{\bbeta_p}^I(t)\|^2_{\bV_p} 
+ \|\partial_t\be_{p_p}^I(t)\|^2_{\W_p} 
+ \|\partial_t\be_{\bvarphi}^I(t)\|^2_{\bLambda_f}
\nonumber \\[1ex]
&\ds\quad 
+ \|\partial_t\be_{\bgamma_f}^I(t)\|^2_{\bbQ_f} 
+ \|\partial_t\be_{\lambda}^I(t)\|^2_{\Lambda_p} 
+ \|\partial_{ttt}\be_{\bbeta_p}^I(0)\|^2_{\bL^2(\Omega_p)} 
+ \|\partial_t\be_{\bbeta_p}^I(0)\|^2_{\bV_p} 
+ \|\partial_t\be_{p_p}^I(0)\|^2_{\W_p}  
+ \|\partial_t\be_{\bvarphi}^I(0)\|^2_{\bLambda_f}
\nonumber \\[1ex]
&\ds\quad 
+ \|\partial_t\be_{\bgamma_f}^I(0)\|^2_{\bbQ_f} 
+ \|\partial_t\be_{\lambda}^I(0)\|^2_{\Lambda_p}
+ \|\be_{\bu_f}^h(0)\|^2_{\bL^2(\Omega_f)} 
+ \|\be_{\bbeta_p}^h(0)\|^2_{\bV_p} 
+\, s_0\,\|\be^h_{p_p}(0)\|^2_{\W_p} 
+ \|\partial_t\be_{\bu_f}^h(0)\|^2_{\bL^2(\Omega_f)}
\nonumber \\[1ex]
&\ds\quad  
+ s_0\,\|\partial_t\be_{p_p}^h(0)\|^2_{\W_p}
+ \|\partial_t\be_{\bbeta_p}^h(0)\|^2_{\bV_p}
+ \,\|\partial_{tt}\be_{\bbeta_p}^h(0)\|^2_{\bL^2(\Omega_p)}
+ \|\be_{\ubsi}^h(0)\|^2 
+ \|\be_{\bvarphi}^h(0)\|^2_{\bLambda_f}  \Bigg)\,. 
\label{eq: error equation 10}
\end{align}

\medskip
\noindent{\bf{Bounds on initial data.}}

\noindent Finally, to bound the initial data terms in \eqref{eq: error equation 10}, we first recall from Theorems \ref{thm:unique soln} and
\ref{thm: well-posedness main result semi} that 
$(\bsi_f(0),\bu_p(0), \bbeta_p(0), \bu_f(0), p_p(0), \bgamma_f(0), \bvarphi(0), \lambda(0)) = (\bsi_{f,0},\bu_{p,0}, \bbeta_{p,0}, \bu_{f,0}, p_{p,0}, \bgamma_{f,0}, \bvarphi_{0}, \lambda_{0})$ 
and 
$(\bsi_{fh}(0),$ $\bu_{ph}(0),\bbeta_{ph}(0),
\bu_{fh}(0), p_{ph}(0), \bgamma_{fh}(0), \bvarphi_{h}(0), \lambda_{h}(0)) = (\bsi_{fh,0},\bu_{ph,0}, \bbeta_{ph,0}, \bu_{fh,0}, p_{ph,0}, \bgamma_{fh,0}, \bvarphi_{h,0}, \lambda_{h,0})$, respectively.
Recall also that discrete initial data satisfy \eqref{eq:system-discrete-sol0-1}. Then, in a way similar to \eqref{eqn:discrete-soln-1-bound}, we obtain
\begin{align}\label{eq: error initial 1}
& \|\be_{\bsi_f}^h(0)\|_{\bbX_f} + \|\be_{\bu_p}^h(0)\|_{\bX_p} + \|P_h^{\V}(\bsi_{e,0})-\widetilde\bsi_{eh,0}\|_{\bSigma_{e}} +\|\be_{\bu_f}^h(0)\|_{\bL^2(\Omega_f)} + \|\be^h_{p_p}(0)\|_{\W_p} + \|\be_{\bgamma_f}^h(0)\|_{\bbQ_f}  \nonumber \\
&\quad +\, \|\partial_t\be_{\bbeta_p}^h(0)\|_{\bV_p} +  \|\be_{\bvarphi}^h(0)\|_{\bLambda_f} + \|\be_{\lambda}^h(0)\|_{\Lambda_p} 
\,\leq\, C\,\Big( \|\be_{\bsi_f}^I(0)\|_{\bbX_f} + \|\be_{\bu_p}^I(0)\|_{\bX_p} +\|\be_{\bu_f}^I(0)\|_{\bV_f} \nonumber \\
&\quad +\, \|\partial_t\be_{\bbeta_p}^I(0)\|_{\bV_p}  +  \|\be_{p_p}^I(0)\|_{\W_p}  +  \|\be_{\bvarphi}^I(0)\|_{\bLambda_f} + \|\be_{\lambda}^I(0)\|_{\Lambda_p}
+ \|\bsi_{e}(0) - P_h^{\V}(\bsi_{e,0})\|_{\bSigma_{e}}\Big) 
\end{align}
and
\begin{align}\label{eq: error initial 2a}
& \|P_h^{\V}(\bsi_{e,0})-\bsi_{eh,0}\|_{\bSigma_{e}} + \|S^{\bV_p}_h(\bbeta_{p,0})-\widetilde\bbeta_{ph,0}\|_{\bV_p} 
\,\leq\, C\,\Big( \|\partial_t\be_{\bbeta_p}^h(0)\|_{\bV_p} +  \|\be_{\bvarphi}^h(0)\|_{\bLambda_f} + \|\be_{\lambda}^h(0)\|_{\Lambda_p} \nonumber \\
& \quad + \|\be^h_{p_p}(0)\|_{\W_p} +\, \|\bsi_{e}(0) - P_h^{\V}(\bsi_{e,0})\|_{\bSigma_{e}}
+ \|\partial_t\be_{\bbeta_p}^I(0)\|_{\bV_p} +  \|\bbeta_p(0)-S^{\bV_p}_h(\bbeta_{p,0})\|_{\bV_p} + \|\be_{\bvarphi}^I(0)\|_{\bLambda_f} \nonumber \\
& \quad + \|\be_{\lambda}^I(0)\|_{\Lambda_p} + \|\be^I_{p_p}(0)\|_{\W_p} \Big) \,.
\end{align}
In addition, to bound $\|\be_{\bbeta_p}^h(0)\|_{\bV_p}$, we first multiply the equation \eqref{eq:system-sol0-4} by $\be(\bxi_{ph})$ to obtain
\begin{equation}\label{etap0-cont}
a^e_p(\bbeta_{p,0},\bxi_{ph}) 
= - b_s(\bsi_{e,0},\bxi_{ph})\quad \forall\,\bxi_{ph} \in \bV_{ph} \,,
\end{equation}
and then, subtracting \eqref{discrete etap0-bound} from \eqref{etap0-cont}, using coercivity and continuity of the bilinear forms involved, we deduce
\begin{align}\label{eq: error initial 2b}
\|\be^h_{\bbeta_p}(0)\|_{\bV_p} \,\leq\, C\,\Big( \|\be^I_{\bbeta_p}(0)\|_{\bV_p} + \|P_h^{\V}(\bsi_{e,0})-\bsi_{eh,0}\|_{\bSigma_{e}} + \, \|\bsi_{e}(0) - P_h^{\V}(\bsi_{e,0})\|_{\bSigma_{e}} \Big)\,.
\end{align}
In turn, for the terms $\|\partial_t\be_{\bu_f}^h(0)\|^2_{\bL^2(\Omega_f)}$ and 
$s_0\,\|\partial_t\be_{p_p}^h(0)\|^2_{\W_p}$ in \eqref{eq: error equation 10}, we consider the error equation \eqref{eq:NS-Biot-errorformulation-1} at $t=0$, 
use \eqref{eq:semi-discrete-weak-formulation-1a}--\eqref{eq:semi-discrete-weak-formulation-1c} and \eqref{eq:semi-discrete-weak-formulation-1f}--\eqref{eq:semi-discrete-weak-formulation-1i} and take $(\bv_{fh}, w_{ph}) = (\partial_t\be_{\bu_f}^h(0), \partial_t\be_{p_p}^h(0))$, to get
\begin{align}\label{eq: error initial 3a}
&\rho_f\|\partial_t\be_{\bu_f}^h(0)\|^2_{\bL^2(\Omega_f)} 
+ s_0\,\|\partial_t\be_{p_p}^h(0)\|^2_{\W_p}  = 0 \,,
\end{align}
where, the right-hand side of \eqref{eq: error initial 3a} has been simplified, since the projection property \eqref{eq:projection1}, 
imply that the terms :
$s_0 (\partial_t\be_{p_p}^I(0),\partial_t\be_{p_p}^h(0))_{\Omega_p}$ and $\rho_f (\partial_t\be_{\bu_f}^I(0),\partial_t\be_{\bu_f}^h(0))_{\Omega_f}$ are zero.

Finally, to bound $\|\partial_{tt}\be_{\bbeta_p}^h(0)\|^2_{\bL^2(\Omega_p)}$, we consider the error equation \eqref{eq:NS-Biot-errorformulation-1} at $t=0$ for the test function $\bxi_{ph}$, such that
\begin{equation}\label{error-etap0-2}
\begin{array}{l}
\ds\rho_p(\partial_{tt}\be_{\bbeta_p}(0),\bxi_{ph})_{\Omega_p} \\[1ex] 
\ds\quad \,=\, - a^e_p(\be_{\bbeta_p}(0),\bxi_{ph})
- \alpha_p\,b_p(\bxi_{ph},\be_{p_p}(0))  
- c_{\BJS}(\partial_t\be_{\bbeta_p}(0), \be_{\bvarphi}(0);\bxi_{ph}, \0)
+ c_{\Gamma}(\bxi_{ph},\0;\be_{\lambda}(0)) \,, 
\end{array}
\end{equation}
where, using \eqref{discrete etap0-bound} and \eqref{etap0-cont}, we deduce that $a^e_p(\be_{\bbeta_p}(0),\bxi_{ph}) = -b_s(\be_{\bsi_e}(0),\bxi_{ph})$, which, together with \eqref{eq:semi-discrete-weak-formulation-1e}, implies that the right-hand side of \eqref{error-etap0-2} is $(\be_{\bu_{s}}(0),\bxi_{ph})_{\Omega_p}=(\partial_t\be_{\bbeta_{p}}(0),\bxi_{ph})_{\Omega_p}$. Then, using the decomposition of the error $\partial_{tt}\be_{\bbeta_p}(0) = \partial_{tt}\be^h_{\bbeta_p}(0) + \partial_{tt}\be^I_{\bbeta_p}(0)$, and applying Cauchy–Schwarz and Young's inequalities, we obtain
%
%
\begin{align}\label{eq: error initial 3b}
& \|\partial_{tt}\be_{\bbeta_p}^h(0)\|^2_{\bL^2(\Omega_p)}  
\leq C\,\Big( \|\partial_{t}\be_{\bbeta_p}^h(0)\|^2_{\bL^2(\Omega_p)}
+ \|\partial_{t}\be_{\bbeta_p}^I(0)\|^2_{\bL^2(\Omega_p)}
+ \|\partial_{tt}\be_{\bbeta_p}^I(0)\|^2_{\bL^2(\Omega_p)} \Big) \,.
\end{align}
Hence, combining \eqref{eq: error equation 10} with \eqref{eq: error initial 1}, \eqref{eq: error initial 2a}, \eqref{eq: error initial 2b}, \eqref{eq: error initial 3a} and \eqref{eq: error initial 3b}, making use of triangle inequality and the approximation properties \eqref{eq:approx-property1}, \eqref{eq:approx-property2}, and \eqref{eq: approx property 3}, we obtain \eqref{eq:errror-rate-of-convergence} and complete the proof.
\end{proof}


\section{Numerical results}\label{sec:numerical}

For the fully discrete scheme utilized in the numerical tests we employ the backward Euler method for the time discretization. 
Let $\Delta t$ be the time step, $T=N \Delta t$, $t_m = m\,\Delta t$, $m=0,\dots, N$. 
Let $d_t\, u^m := (\Delta t)^{-1}(u^m - u^{m-1})$ be the first order (backward) discrete 
time derivative and  $d_{tt} \, u^m := (\Delta t)^{-2}(u^m - 2u^{m-1} + u^{m-2})$ be the second order (backward) discrete 
time derivative, where $u^m := u(t_m)$. Then the fully discrete model reads: 
Given $(\bsi_{fh}^0,\bu_{ph}^0, \bbeta_{ph}^0, \bu_{fh}^0, p_{ph}^0, \bgamma_{fh}^0, \bvarphi_{h}^0, \lambda_{h}^0)=(\bsi_{fh,0},\bu_{ph,0}, \bbeta_{ph,0}, \bu_{fh,0}, p_{ph,0},$
$ \bgamma_{fh,0}, \bvarphi_{h,0}, \lambda_{h,0})$, find $(\bsi_{fh}^m,\bu_{ph}^m, \bbeta_{ph}^m, \bu_{fh}^m, p_{ph}^m, \bgamma_{fh}^m, \bvarphi_{h}^m, \lambda_{h}^m)\in\bbX_{fh}\times \bX_{ph}\times \bV_{ph}\times \bV_{fh}\times \W_{ph}\times \bbQ_{fh}\times \bLambda_{fh}\times \Lambda_{ph}$, $m=1, \dots, N$, such that 
\begin{align}
& \rho_f (d_t\,\bu_{fh}^m,\bv_{fh})_{\Omega_f}
+ a_f(\bsi_{fh}^m,\btau_{fh})
+ b_{\bn_f}(\btau_{fh},\bvarphi_{h}^m) 
+ b_f(\btau_{fh},\bu_{fh}^m) 
+ b_\sk(\bgamma_{fh}^m,\btau_{fh})
\nonumber\\ 
&\ds\quad +\, \kappa_{\bu_{fh}^m}(\bu_{fh}^m, \btau_{fh}) 
- b_f(\bsi_{fh}^m,\bv_{fh}) 
- b_\sk(\bsi_{fh}^m,\bchi_{fh})  
\,=\, (\f_{f}^m,\bv_{fh})_{\Omega_f}\,,  
\nonumber\\ 
& \rho_p(d_{tt}\bbeta_{ph}^m,\bxi_{ph})_{\Omega_p} 
+ a^e_p(\bbeta_{ph}^m,\bxi_{ph})
+ \alpha_p\,b_p(\bxi_{ph}^m,p_{ph}) 
+ c_{\BJS}(d_t\,\bbeta_{ph}^m, \bvarphi_{h}^m;\bxi_{ph}, \bpsi_{h}) 
\nonumber\\ 
&\ds\quad -\, c_{\Gamma}(\bxi_{ph},\bpsi_{h};\lambda_{h}^m) 
- b_{\bn_f}(\bsi_{fh}^m,\bpsi_{h})
+ l_{\bvarphi_{h}^m}(\bvarphi_{h}^m,\bpsi_{h}) 
\,=\, (\f_{p}^m,\bxi_{ph})_{\Omega_p}\,, 
\nonumber\\ 
& s_0\,(d_t\,p_{ph}^m,w_{ph})_{\Omega_p} 
+ a^d_p(\bu_{ph}^m,\bv_{ph}) 
+ b_p(\bv_{ph},p_{ph}^m)
+ b_{\bn_p}(\bv_{ph},\lambda_{h}^m) 
- \alpha_p\,b_p(d_t\,\bbeta_{ph}^m,w_{ph})
\nonumber\\ 
&\ds\quad -\, b_p(\bu_{ph}^m,w_{ph}) 
\,=\, (q_{p}^m,w_{ph})_{\Omega_p},  
\nonumber\\ 
& c_{\Gamma}(d_t\,\bbeta_{ph}^m,\bvarphi_{h}^m;\xi_{h})
- b_{\bn_p}(\bu_{ph}^m,\xi_{h})
\,=\,0,  
\label{eq: discrete formulation}
\end{align}
for all $(\btau_{fh}, \bv_{ph}, \bxi_{ph}, \bv_{fh}, w_{ph}, \bchi_{fh}, \bpsi_{h}, \xi_{h})\in \bbX_{fh}\times \bX_{ph}\times \bV_{ph}\times \bV_{fh}\times \W_{ph}\times \bbQ_{fh}\times \bLambda_{fh}\times \Lambda_{ph}$.
The fully discrete method results in the solution of a nonlinear algebraic system at each time step.

Next, we present numerical results that illustrate the behavior of the numerical method. We use the Newton--Rhapson method to solve this nonlinear algebraic system at each time step. 
The implementation is based on a {\tt FreeFEM} code \cite{Hecht} on triangular grids. 
For spatial discretization we use the following finite element spaces: $\bbBDM_1-\bP_0-\bbP_0$ for stress-velocity-vorticity in Navier--Stokes, 
$\bP_1$ for displacement in elasticity,
$\bBDM_1-\rP_0$ for Darcy velocity-pressure, and $\bP_1 - \rP_1$ for the the traces of fluid velocity and Darcy pressure. 

The examples considered in this section are described next.
Example 1 is used to corroborate the rates of convergence. 
In Example 2 we present a simulation of air flow through a filter. 


\subsection{Example 1: convergence test}

In this test we study the convergence for the space discretization using an analytical solution.
The domain is $\Omega = (-1,1)\times(0,1)$ with $\Omega_f = (0,1)\times (0,1)$, $\Omega_p = (0,1)\times (-1,0)$, and $\Gamma_{fp} = (0,1)\times\{0\}$; i.e.,
the upper half is associated with the Navier--Stokes flow, while the lower half represents the poroelastic medium governed by the Biot system, see Figure~\ref{fig:example1} (left). The analytical solution is given in Figure~\ref{fig:example1} (right). It satisfies the appropriate interface conditions along the interface $\Gamma_{fp}$.
\begin{figure}[htb!]
\begin{minipage}{.49\textwidth}
\centering\includegraphics[scale=1]{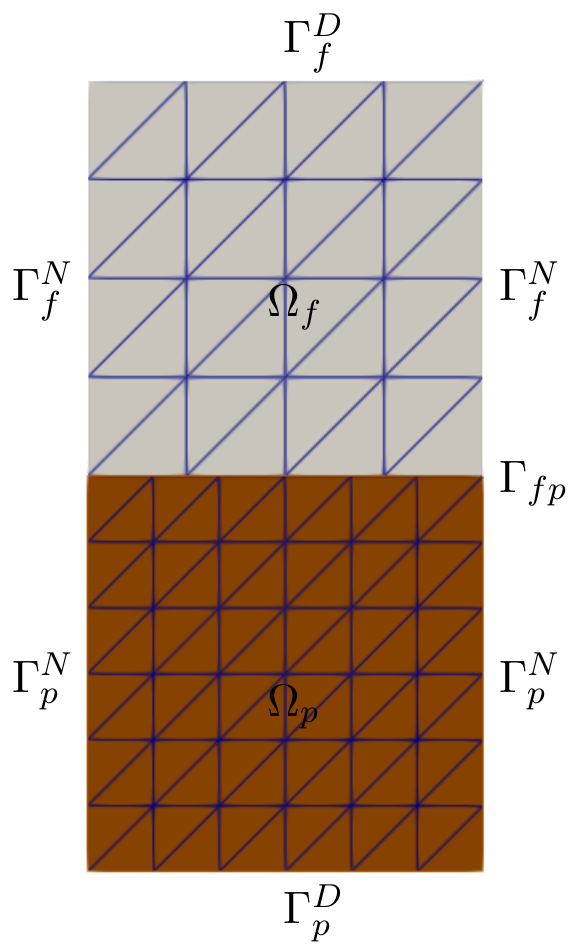}
\end{minipage}
\hfill
\begin{minipage}{.5\textwidth}

Solution in the Navier-Stokes region:

  \bigskip
$ \ds \bu_f = \pi \cos(\pi t)
\begin{pmatrix}
\ds -3x + \cos(y) \\[1ex] \ds y+1
\end{pmatrix}$

\smallskip
$ \ds p_f = \exp(t)\,\sin(\pi x)\cos\Big(\frac{\pi y}{2}\Big) + 2\pi \cos(\pi t)$

\bigskip
\bigskip
Solution in the Biot region:

\bigskip
$\ds p_p = \exp(t)\,\sin(\pi x)\cos\Big(\frac{\pi y}{2}\Big)$

$\ds \bu_p = -\frac{1}{\mu} \bK \nabla p_p $

$ \ds \bbeta_p = \sin(\pi t) \begin{pmatrix} \ds -3x+\cos(y) \\[1ex] \ds y+1 \end{pmatrix}$
\end{minipage}
\caption{[Example 1] Left: computational domain at the coarsest mesh level. Right: analytical solution.}
\label{fig:example1}
\end{figure}

The model parameters are
$$
\mu = 1, \quad \rho_f = 1,  \quad \rho_p = 1, \quad \lambda_p = 1, \quad \mu_p = 1, \quad s_0 = 1, \quad \bK = \bI, \quad \alpha_p = 1, \quad \alpha_{\BJS} = 1\,.
$$
The right-hand side functions $\f_f, \f_p$ and $q_p$ are computed from \eqref{eq:Navier-Stokes-1}--\eqref{eq:Biot-model} using the analytical solution. 
Note that, we do not take the analytical solution of fluid velocity to be divergence-free and then we define $q_f:= \div(\bu_f)$ and modify accordingly \eqref{eq: discrete formulation} to account for this data. 
The model problem is complemented with the appropriate boundary conditions, as shown in Figure~\ref{fig:example1} (left), and initial data.
Notice that the boundary conditions are not homogeneous and therefore the right-hand side of the resulting system must be modified accordingly.
The total simulation time for this test case is $T=0.01$ and the time step is $\Delta t=10^{-3}$. The time step is sufficiently small, so that the time discretization error does not affect the spatial convergence rates.
Table~\ref{table1-example1} shows the convergence history for a sequence of quasi-uniform mesh refinements with non-matching grids along the interface employing conforming spaces for the Lagrange multipliers \eqref{defn-Lambda-h}. The grids on the coarsest level are shown in Figure~\ref{fig:example1} (left). In the table, $h_f$ and $h_p$ denote the mesh sizes in $\Omega_f$ and $\Omega_p$, respectively, while the mesh sizes for their traces on $\Gamma_{fp}$ are $h_{tf}$ and $h_{tp}$. 
We note that the Navier--Stokes pressure at $t_m$ is recovered by post-processing, using the formula $\ds p_{fh}^m = -\frac{1}{n} \left( \tr(\bsi_{fh}^m)  + \rho_f\,\tr(\bu_{fh}^m\otimes\bu_{fh}^m) - 2\mu q_f^m \right)$ (cf. \eqref{eq:pseudostress-pressure-formulae}).
The results illustrate that at least the optimal spatial rate of convergence $\cO(h)$ established in Theorem~\ref{thm: error analysis} is attained for all subdomain variables.
The Lagrange multiplier variables, which are approximated in $\bP_1-\rP_1$, exhibit a rate of convergence $\cO(h^{3/2})$ in the $\H^{1/2}$-norm on $\Gamma_{fp}$, which is consistent with the order of approximation.
\begin{table}[ht]
\begin{center}
\small{			
\begin{tabular}{|c||cc|cc|cc|cc||c||cc|}
\hline
& \multicolumn{2}{|c|}{$\|\be_{\bsi_{f}}\|_{\ell^2(0,T;\bbX_f)}$}  
& \multicolumn{2}{|c|}{$\|\be_{\bu_f}\|_{\ell^2(0,T;\bV_f)}$} 
& \multicolumn{2}{|c|}{$\|\be_{\bgamma_f}\|_{\ell^2(0,T;\bbQ_f)}$}  
& \multicolumn{2}{|c||}{$\|\be_{p_f}\|_{\ell^2(0,T;\L^2(\Omega_f))}$} 
&
& \multicolumn{2}{|c|}{$\|\be_{\bvarphi}\|_{\ell^2(0,T;\bLambda_f)}$} \\
$h_f$  & error & rate & error & rate & error & rate & error & rate & $h_{tf}$ & & \\  \hline
0.354 & 7.5E-01 &   --  & 2.4E-02 &   --  & 6.3E-02 &   --  & 1.4E-01 &   --  & 0.500 & 5.0E-03 &   -- \\ 
0.177 & 2.4E-01 & 1.656 & 1.2E-02 & 0.999 & 3.3E-02 & 0.927 & 6.1E-02 & 1.164 & 0.250 & 1.5E-03 & 1.722 \\
0.088 & 8.7E-02 & 1.453 & 6.1E-03 & 1.000 & 1.7E-02 & 0.978 & 3.0E-02 & 1.025 & 0.125 & 5.7E-04 & 1.424 \\
0.044 & 3.7E-02 & 1.233 & 3.1E-03 & 1.000 & 8.4E-03 & 0.995 & 1.5E-02 & 1.004 & 0.063 & 2.1E-04 & 1.419 \\
0.022 & 1.8E-02 & 1.085 & 1.5E-03 & 1.000 & 4.2E-03 & 0.999 & 7.5E-03 & 1.001 & 0.031 & 7.4E-05 & 1.521 \\
0.011 & 8.6E-03 & 1.027 & 7.7E-04 & 1.000 & 2.1E-03 & 1.000 & 3.7E-03 & 1.000 & 0.016 & 2.4E-05 & 1.608 \\
\hline 
\end{tabular}
		
\medskip

\begin{tabular}{|c||cc|cc|cc||c||cc||c|}
\hline
& \multicolumn{2}{|c|}{$\|\be_{\bu_p}\|_{\ell^2(0,T;\bX_p)}$} 
& \multicolumn{2}{|c|}{$\|\be_{p_p}\|_{\ell^\infty(0,T;\W_p)}$}   
& \multicolumn{2}{|c||}{$\|\be_{\bbeta_p}\|_{\ell^\infty(0,T;\bV_p)}$}
&
& \multicolumn{2}{|c||}{$\|\be_{\lambda}\|_{\ell^2(0,T;\Lambda_p)}$}  & \\ 
$h_p$  & error  & rate & error & rate & error & rate & $h_{tp}$ & error & rate & iter  \\  \hline
0.236 & 5.5E-02 &   --  & 6.9E-02 &   --  & 1.3E-04 &   --  & 0.333 & 4.6E-03 &   --  & 2.2 \\
0.118 & 2.0E-02 & 1.454 & 3.5E-02 & 0.998 & 6.5E-05 & 1.001 & 0.167 & 1.6E-03 & 1.532 & 2.2 \\
0.059 & 7.9E-03 & 1.347 & 1.7E-02 & 0.999 & 3.2E-05 & 0.999 & 0.083 & 5.6E-04 & 1.515 & 2.2 \\
0.029 & 3.5E-03 & 1.167 & 8.6E-03 & 1.000 & 1.6E-05 & 0.998 & 0.042 & 2.0E-04 & 1.504 & 2.2 \\
0.015 & 1.7E-03 & 1.046 & 4.3E-03 & 1.000 & 8.1E-06 & 0.996 & 0.021 & 7.0E-05 & 1.501 & 2.2 \\
0.007 & 8.4E-04 & 1.008 & 2.2E-03 & 1.000 & 4.1E-06 & 0.995 & 0.010 & 2.5E-05 & 1.500 & 2.2 \\
\hline 
\end{tabular}
\caption{[Example 1] Mesh sizes, errors, rates of convergences and average Newton iterations for the fully discrete system $(\bbBDM_1-\bP_0-\bbP_0)-\bP_1-(\bBDM_1-\rP_0) - (\bP_1 - \rP_1)$ approximation for the Navier--Stokes--Biot model in non-matching grids.}\label{table1-example1}
}
\end{center}
\end{table}


\subsection{Example 2: Air flow through a filter}

The setting in this example is similar to the 
one presented in \cite{swgbh2020}, where the Navier--Stokes--Darcy model is considered. The domain is a two-dimensional rectangular channel with length $2.5$m and width $0.25$m, which on the bottom center is partially blocked by a square poroelastic filter of length and width $0.2$m, see Figure \ref{fig:filter} (top).
\begin{figure}
\includegraphics[width=\textwidth]{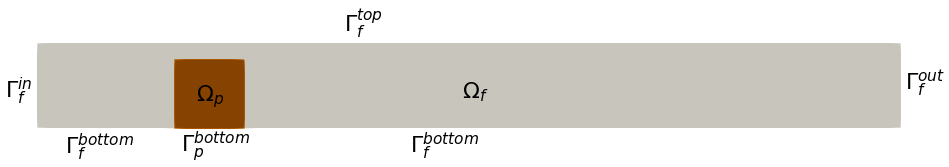}
\begin{equation*}
\begin{array}{c}
\ds \bT_f \bn_f = -p_{in} \bn_f \qon \Gamma_f^{in},\quad \bT_f \bn_f = -p_{out} \bn_f \qon \Gamma_f^{out} \\[0.5ex]
\ds p_{in}= p_{ref} + 2\times 10^{-6} \text{kPa},\quad p_{out} = p_{ref},\quad \bu_f = 0 \qon \Gamma_f^{top} \cup \Gamma_f^{bottom} \\[0.5ex]
\ds \bbeta_p = \0 \qan \bu_p \cdot \bn_p = 0 \qon \Gamma_p^{bottom}
\end{array}
\end{equation*}
\vspace{-0.5cm}
\caption{[Example 2] Top: computational domain and boundaries; channel $\Omega_{f}$ in gray, filter $\Omega_{p}$ in brown. Bottom: boundary conditions with $p_{ref}=100\,\text{kPa}$.}
\label{fig:filter}
\end{figure}
The model parameters are set as
\begin{equation*}
\begin{array}{c}
\mu=1.81 \times 10^{-8} \text{ kPa s}, \quad \rho_f =1.225 \times 10^{-3} \text{ Mg}/\text{m}^3, \quad s_0=7 \times 10^{-2} \text{ kPa}^{-1}, \\[1ex]
\bK= [ 0.505, -0.495; -0.495,0.505] \times 10^{-6} \text{ m}^2,  \quad \rho_p =1.601 \times 10^{-2} \text{ Mg}/\text{m}^3, \quad \alpha_{\BJS}=1.0, \quad \alpha_p = 1.0.
\end{array}
\end{equation*}
Note that $\mu$ and $\rho_f$ are parameters for air. The permeability tensor $\bK$ is obtained by rotating the identity tensor by a $45^\circ$ rotation angle in order to consider the effect of material anisotropy on the flow. We further consider a stiff material in the poroelastic region with parameters:
$\lambda_p=1{\times}10^4\,\text{kPa}$ and $\mu_p = 1{\times}10^5\,\text{kPa}$. 
The top and bottom of the domain are rigid, impermeable walls. The flow is driven by a pressure difference $\Delta p = 2\times 10^{-6}\,\text{kPa}$ between the left and right boundary, see Figure \ref{fig:filter} (bottom) for the boundary conditions. The body force terms $\f_f$ and $\f_p$ and external source $q_p$ are set to zero. For the initial conditions, we consider
$$ p_{p,0}=100  \ \text{kPa}, \quad \bbeta_{p,0} = \0  \ \text{m},  \quad \bu_{s,0} = \0 \ \text{m/s},\quad \bu_{f,0} = \0  \ \text{m/s}. $$
The computational matching grid along $\Gamma_{fp}$ has a characteristic parameter $h=\max\{ h_f, h_p \} = 0.018$. The total simulation time is $T=400$s with $\Delta t=1$s.

The computed magnitude of the velocities and pressures are displayed in Figures \ref{fig:filter-velocity} and \ref{fig:filter-pressure}, respectively. We observe high velocity in the narrow open channel above the filter. Vortices develop behind the obstacle, which travel with the fluid and are smoothed out at later times. A sharp pressure gradient is observed in the region above the filter, as well as within the filter, where the permeability anisotropy affects both the pressure and velocity fields. This example illustrates the ability of the mixed method to produce oscillation-free solution in a regime of challenging physical parameters, including small viscosity and permeability and large Lam\'e coefficients. 
\begin{figure}[ht]
\includegraphics[width=\textwidth]{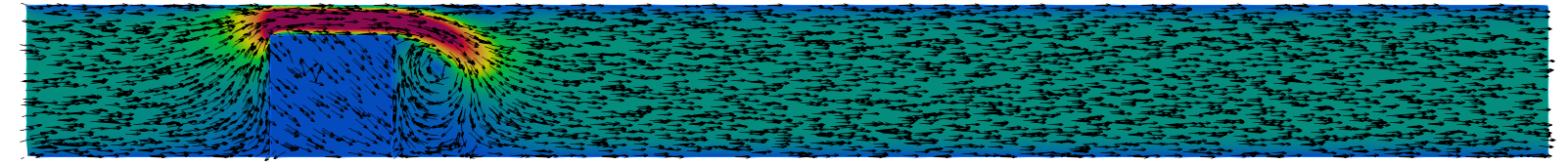}
\includegraphics[width=\textwidth]{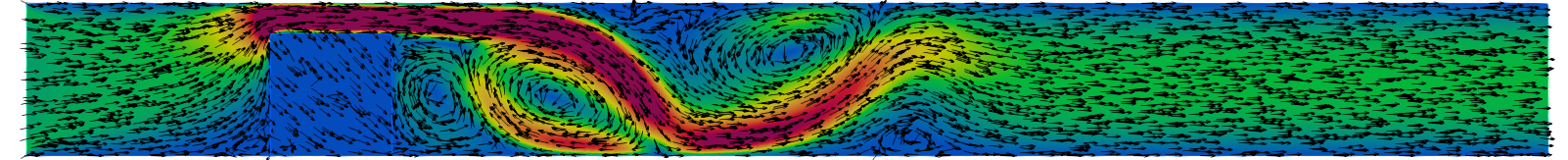}
\includegraphics[width=\textwidth]{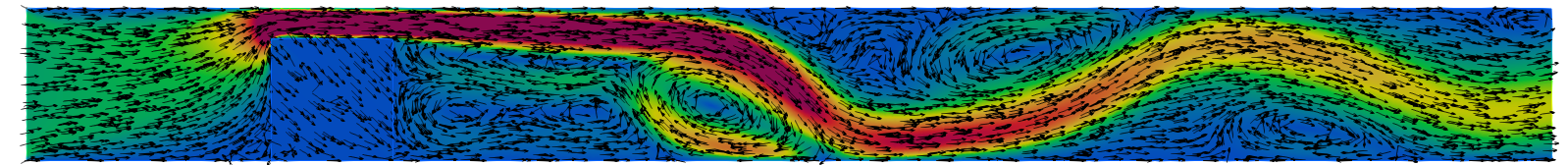}
\includegraphics[width=\textwidth]{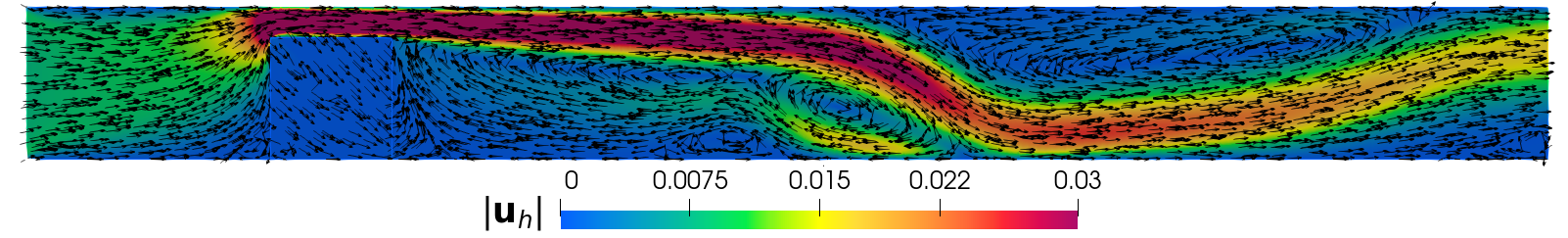}
\vspace{-0.4cm}
\caption{[Example 2] Computed velocities $\bu_{fh}$ and $\bu_{ph}$ (arrows not scaled) and their magnitudes at times $t\in \{ 20, 80, 200, 400 \}$ (from top to bottom).} \label{fig:filter-velocity}
\end{figure}

\begin{figure}[ht]
\includegraphics[width=\textwidth]{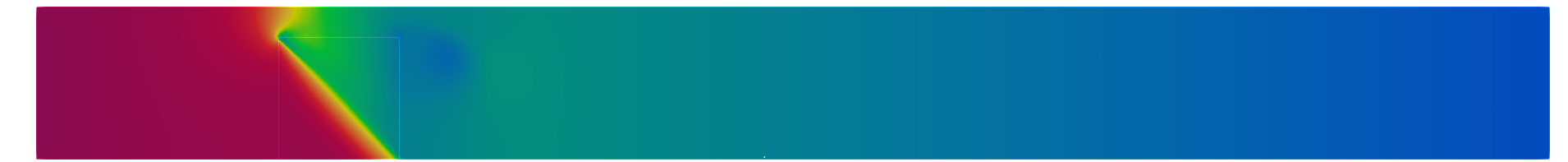}
\includegraphics[width=\textwidth]{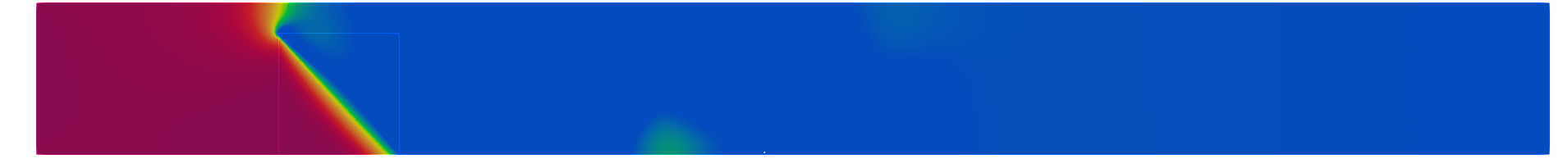}
\includegraphics[width=\textwidth]{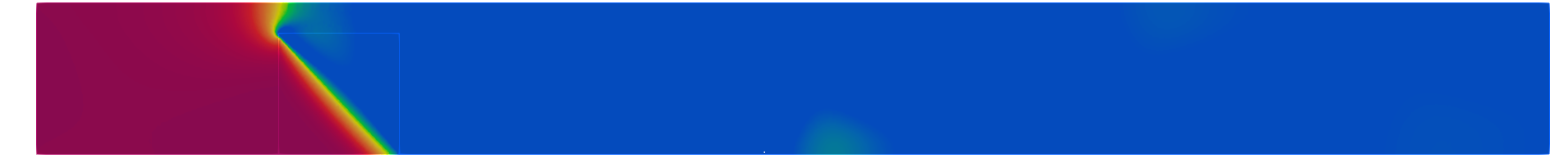}	
\includegraphics[width=\textwidth]{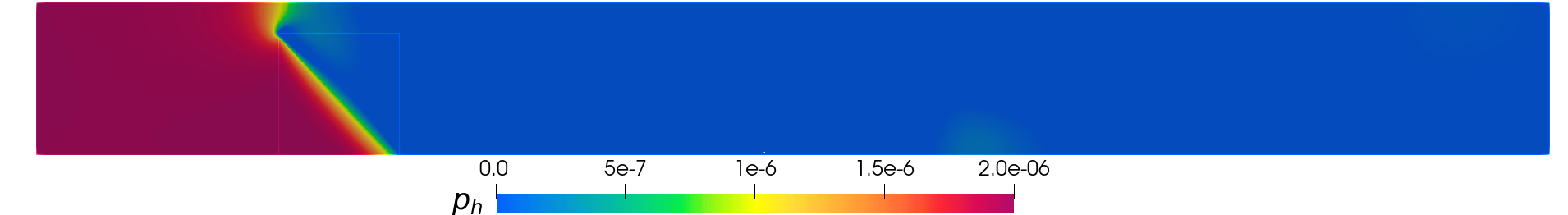}
\vspace{-0.4cm}
\caption{[Example 2] Computed pressures $p_{fh} - p_{ref}$ and $p_{ph} - p_{ref}$ at times $t\in \{ 20, 80, 200, 400 \}$ (from top to bottom).} \label{fig:filter-pressure}
\end{figure}


\section{Conclusions}\label{sec:concl}

We presented a new formulation for the fully dynamic Navier--Stokes--Biot problem, where the variables are pseudostress-velocity-vorticity for Navier--Stokes, velocity-pressure for Darcy, and displacement for elasticity. At the discrete level this formulation provides local fluid momentum conservation and local mass conservation for the poroelastic flow with continuous normal fluid stress and Darcy velocity. We established well-posedness of the weak formulation in a Banach space framework, which avoids the need for stabilization and/or augmentation terms. To handle the nonlinear advective term in the Navier--Stokes equations, under a small data assumption, we utilized results from nonlinear semigroup theory for monotone operators and the Banach fixed-point theory. We further studied the numerical approximation of the model using $\bH(\div)$-conforming finite element spaces for the Navier--Stokes and Darcy flows and $\bH^1$-conforming space for the displacement. We developed well posedness, stability, and error analysis for the semidiscrete formulation, establishing optimal rates of convergence. Numerical experiments were presented to verify the convergence rates and illustrate the ability of the method to simulate a realistic flow application of air flow around a filter
with flow and elasticity parameters in a challenging physical regime.

\bibliographystyle{abbrv}
\bibliography{caucao-dalal-yotov}

\end{document}